\documentclass[12pt,a4paper]{amsart}
\usepackage{tikz}
\tikzset{
	op/.style = {shape=rectangle, rounded corners, draw=black, align=center,
		fill=blue!30},base align tikzpicture/.style={execute at end picture={
			\path (current bounding box.north) -- (current bounding box.south)
			node[midway](Xphantom){\phantom{X}};},baseline={(Xphantom.base)}}
}
\usepackage{tikz-cd}

\usepackage{graphicx} 
\usepackage{hyperref}
\usepackage[notref,notcite]{}
\hypersetup{
	colorlinks=true,
	linkcolor=blue,
	filecolor=magenta,     
	urlcolor=cyan,
	pdftitle={Codimbraid},
}
\usepackage{amssymb,amsmath,amsthm,amsfonts,mathrsfs,bbm}
\usepackage{array,amssymb,amsxtra} 
\usepackage{geometry}
\usepackage[shortlabels]{enumitem} 
\usepackage[normalem,normalbf]{ulem}
\usepackage[all]{xy} 
\usepackage{fancybox}
\usepackage{diagbox}
\usepackage{thmtools}
\usepackage{cite}
\usepackage{scalerel}
\usepackage{stackengine,wasysym}
\usepackage{stmaryrd}
\usepackage{MnSymbol}

\newcommand{\Z}{\mathbb{Z}} 

\newcommand{\N}{\mathbb{N}}

\newcommand{\B}{\mathcal{B}}

\newcommand{\KK}{\mathcal{K}}
\newcommand{\VV}{\mathcal{V}}
\newcommand{\YD}{\mathcal{YD}}

\newcommand{\Dpr}{D^{\bar{p}, \bar{r}}}
\newcommand{\Dp}{D^{\bar{p}}}
\newcommand{\Opr}{\mathcal{O}_{q}^{\bar{p}, \bar{r}}({\rm GL}_n(k))}
\newcommand{\MM}{\mathcal{M}}

\newcommand{\rr}{\mathbf{r}}

\newcommand{\Oo}{\mathcal{O}_{q}(\normalfont\text{SL}_{2}(k))}

\DeclareMathOperator{\id}{id}

\usepackage{thmtools}

\DeclareMathOperator{\eq}{eq}
\DeclareMathOperator{\ob}{ob}
\DeclareMathOperator{\Cc}{\textbf{\textup{C}}}

\DeclareMathOperator{\CC}{\textbf{\textup{C}}}
\DeclareMathOperator{\GL}{GL}
\newcommand{\Oq}{\mathcal{O}_{q}(\mathrm{M}_{n}(k))}

\newcommand{\OQ}{\mathcal{O}_{q}(\mathrm{SL}_{n}(k))}
\newcommand{\Opgl}{\mathcal{O}^{\bar{p}}_{q}(\mathrm{GL}_{n}(k))}
\newcommand{\Ogl}{\mathcal{O}_{q}(\mathrm{GL}_{n}(k))}

\newtheorem{theorem}{Theorem }[section]
\newcommand{\bt}{\begin{theorem}}
	\newcommand{\et}{\end{theorem}}

\definecolor{navajowhite}         {rgb}{1.00,0.87,0.68}
\newtheorem{prop}[theorem]{Proposition}
\newcommand{\bp}{\begin{prop}}
	\newcommand{\ep}{\end{prop}}

\newtheorem{dep}[theorem]{Proposition-Definition}
\newcommand{\bdep}{\begin{dep}}
	\newcommand{\edep}{\end{dep}}

\newtheorem{tp}[theorem]{Theorem-Definition}
\newcommand{\btp}{\begin{tp}}
	\newcommand{\etp}{\end{tp}}

\newtheorem{cor}[theorem]{Corollary}
\newcommand{\bc}{\begin{cor}}
	\newcommand{\ec}{\end{cor}}
\definecolor{babyblueeyes}{rgb}    {0.63, 0.79, 0.95}
\newtheorem{lemma}[theorem]{Lemma}
\newcommand{\bl}{\begin{lemma}}
	\newcommand{\el}{\end{lemma}}
\theoremstyle{definition}

\newtheorem*{lemma*}{Lemma}
\newtheorem*{proposition*}{Proposition}

\newtheorem{dfn}[theorem]{Definition}
\newcommand{\bd}{\begin{dfn}}
	\newcommand{\ed}{\end{dfn}}

\theoremstyle{remark}

\newtheorem{rmq}[theorem]{Remark}
\newcommand{\brq}{\begin{rmq}}
	\newcommand{\erq}{\end{rmq}}

\theoremstyle{definition}

\newtheorem{ex}[theorem]{Example}
\newcommand{\bex}{\begin{ex}}
	\newcommand{\eex}{\end{ex}}

\newcommand{\bpf}{\begin{proof}}
	\newcommand{\epf}{\end{proof}}

\newcommand{\smallbraid}[2]{\draw (#1,#2) ..controls +(0,-0.5) and
+(0,0.5).. (#1-1,#2-1);
\draw[white][line width = 5pt] (#1-1,#2) ..controls +(0,-0.5) and
+(0,0.5).. (#1,#2-1);
\draw (#1-1,#2) ..controls +(0,-0.5) and +(0,0.5).. (#1,#2-1);}

\newcommand{\produit}[3]{\draw (#1,#2) arc (180:360:#3); \draw(#1 + #3, #2 -#3)--(#1 + #3, #2 -#3 -0.5);}
\newcommand{\coproduit}[3]{\draw (#1,#2) arc (0:180:#3); \draw(#1 - #3, #2 +#3 + 0.5)--(#1 - #3, #2 + #3);}

\newcommand{\gauche}[2]{\draw (#1,#2)..controls+(0,-1)and+(1,0)..(#1-1,#2-1);

\draw (#1-2,#2)..controls+(0,-1)and+(-1,0)..(#1-1,#2-1); \draw(#1-1.5, #2 - 1)--(#1-1.5, #2 - 1.5);}

\newcommand{\coright}[2]{\draw (#1,#2)..controls+(0, 1)and+(1,0)..(#1- 1,#2+1);

\draw (#1-2,#2)..controls+(0,1)and+(-1, 0)..(#1-1,#2+1); \draw(#1-1.5, #2 + 1.5)--(#1-1.5, #2 + 1);}

\newcommand{\coleft}[2]{\draw (#1,#2)..controls+(0, 1)and+(1,0)..(#1- 1,#2+1);

\draw (#1-2,#2)..controls+(0,1)and+(-1, 0)..(#1-1,#2+1); \draw(#1-.5, #2 + 1.5)--(#1-.5, #2 + 1);}

\title{Braided cogroupoids}
\author{Thi Hoa Emilie Nguyen}
\address{Universit\'e Clermont Auvergne, CNRS, LMBP, F-63000 CLERMONT-FERRAND, FRANCE}
\email{thi\_hoa.nguyen@uca.fr}
\begin{document}
\begin{abstract}
We introduce and develop the theory of braided cogroupoids, a class of algebraic structures generalizing cogroupoids in a braided setting. We show that braided cogroupoids induce monoidal equivalences between the associated comodule categories, and we generalize Majid’s transmutation and bosonization of braided Hopf algebras to the cogroupoid setting. Several examples are studied in detail, including the braided SL$_{n}$ cogroupoid and the braided bilinear cogroupoid.
\end{abstract}

\maketitle
\tableofcontents
\section{Introduction}
Takeuchi \cite{MR472967} introduced an analogue of Morita contexts for coalgebras and comodules, providing criteria for when two coalgebras have equivalent comodule categories; this is a dual version of classical Morita theory for algebras. In this foundational work, Takeuchi showed that, under suitable conditions, a bicomodule between two coalgebras gives rise to an equivalence of comodule categories when it is part of what is now called a Morita–Takeuchi context. Building on the work of Ulbrich \cite{Ulbrich1989}, Schauenburg \cite{Schauenburg1996} later placed these ideas into a monoidal framework in his study of Hopf–Galois and bi-Galois extensions. He proved that bi-Galois objects classify monoidal equivalences between comodule categories of Hopf algebras, thereby providing a monoidal version of Morita–Takeuchi equivalence. Developing this perspective further, Bichon introduced the notion of cogroupoid to encode collections of bi-Galois objects and to organize monoidal equivalences across multiple Hopf algebras. In particular, he showed that a connected cogroupoid induces bi-Galois objects and hence monoidal equivalences between the comodule categories of the associated Hopf algebras \cite{MR3285340}.

A systematic theory of Hopf algebras in braided monoidal categories was developed by Majid, who introduced the bosonization (or biproduct) construction, relating braided Hopf algebras to ordinary Hopf algebras \cite{Maj94,Rad85}. This construction has become a fundamental tool in the theory of quantum groups and Nichols algebras \cite{MR1913436, heckenberger2020hopf}, and it highlights the importance of braidings in transferring algebraic structures between different categorical contexts.

The aim of the present paper is to introduce and develop the theory of \emph{braided cogroupoids}, extending the notion of cogroupoid to a braided monoidal setting. Roughly speaking, a braided cogroupoid consists of a family of algebras equipped with cogroupoid structure maps that are compatible with a given braiding. This framework simultaneously generalizes ordinary cogroupoids and braided Hopf algebras, and provides a natural setting in which to study families of quantum symmetries linked by tensor equivalences. In a similar spirit, Schauenburg  \cite{MR2294763} has developed a theory of braided bi-Galois objects.

A central result in this work is Theorem~\ref{Thm: monoid func}, which shows that braided cogroupoids induce monoidal equivalences between their associated categories of comodules. Other highlights include the generalization, at the cogroupoid level, of some important constructions from braided Hopf algebra theory due to Majid:
\begin{itemize}
\item we generalize bosonization to cogroupoids by associating an ordinary cogroupoid with a braided cogroupoid in a category of Yetter--Drinfeld modules, and conversely we associate a braided cogroupoid with an ordinary cogroupoid equipped with a kind of projection \cite{Rad85};
\item we extend the transmutation procedure \cite{Maj93} to braided cogroupoids, producing braided cogroupoids in categories of comodules over a coquasitriangular Hopf algebra.
\end{itemize}

We illustrate the theory with explicit examples, including the braided
\( \mathrm{SL}_n \) cogroupoid, constructed via multiparametric deformations of the quantum algebra
\( \mathcal{O}_q(\mathrm{GL}_n) \), and we also consider the transmutation of the bilinear cogroupoid \( \mathcal{B}\) \cite{ MR1998031}.

The paper is organized as follows. In Section~2, we recall the necessary background on flat regular monoidal categories, braidings, and (co)algebraic structures in monoidal categories, with particular emphasis on equalizers and cotensor products. Section~3 deals with Takeuchi categorical coalgebras and cocategories, and introduces the notion of braided cogroupoids together with their main properties. In Section~4, we study coinvariants and bosonizations of braided cogroupoids, illustrating the theory with the example of the braided \( \mathrm{SL}_n \) cogroupoid. Finally, Section~5 is devoted to the transmutation construction for cogroupoids, along with explicit examples.

\vspace*{0.5cm}
\paragraph{\textbf{Notations and conventions}}
Throughout this paper we work over a fixed base field denoted by \( k \). We assume that the reader is familiar with Hopf algebras and monoidal categories.
\section{Preliminaries}\label{sec:prelim}
\subsection{Equalizers}
Let $\VV$ be a category. Let $X, Y$ be objects and $f, g : X \to Y$ be morphisms of $\VV$. A morphism $eq : E \to X$ in $\VV$ is said to be an \textit{equalizer} for the pair $(f, g)$ if $f \circ eq = g \circ eq$ and if $(E, eq)$ satisfies the following universal property: for every morphism $j : O \to X $ in $\VV$ such that $ f \circ j = g \circ j$, there exists a unique morphism $u : O \to E$ such that $eq \circ u = j$.
\begin{center}
\begin{tikzcd}[column sep= 2cm,row sep= 1cm]
	E \ar[r, "eq"] &X \ar[r, shift left = 1, "f" ] \ar[r, shift right = 1, "g" ']&Y\\
	O \ar[u, dashrightarrow, "u"] \ar[ru, "j"]&&
\end{tikzcd}
\end{center}
In this case, the object $E$ is unique up to isomorphism; we sometimes write $E = \textbf{Eq}(f, g)$ and refer to $eq$, which is always a monomorphism.
\subsection{Flat regular monoidal categories and braided categories}
A \textit{monoidal category} is a category $\VV$ equipped with a tensor product bifunctor $\otimes \colon \VV \times \VV \to \VV$, a \textit{unit object} $I \in \ob(\VV)$, and natural isomorphisms $\alpha: (X \otimes Y) \otimes Z \simeq X \otimes (Y \otimes Z)$, $r: X \otimes I \simeq X$ and $l: I \otimes X \simeq X$, for any $X, Y, Z \in \ob(\VV),$ satisfying coherence conditions (Mac Lane’s pentagon and triangle axioms) ensuring all associativity and unit diagrams commute. It is well known, due to the coherence theorem of Mac Lane's (see e.g \cite[Section 1.5]{bulacu2019quasi}), that every monoidal category is monoidally equivalent to a strict monoidal category, (a monoidal category such that all associativity and unit isomorphisms are identities). Thus, without loss of generality, from now on, we will only consider strict monoidal categories.

An object $A$ in $\VV$ is said to be \textit{flat} if tensoring with $A$ preserves equalizers. That is, if
\[\begin{tikzcd}[column sep= 1cm,row sep= 1cm]
E \ar[r, "eq"] &X \ar[r, shift left = 1, "f" ] \ar[r, shift right = 1, "g" ']&Y
\end{tikzcd}\]
is an equalizer in $\VV$, then so are
\[\begin{tikzcd}[column sep= 2cm,row sep= 1cm]
A \otimes E \ar[r, "1_{A} \otimes eq"] & A \otimes X \ar[r, shift left = 1, " 1_{A} \otimes f" ] \ar[r, shift right = 1, "1_{A} \otimes g" ']& A \otimes Y
\end{tikzcd}\] and 
\[\begin{tikzcd}[column sep= 2cm,row sep= 1cm]
 E \otimes A \ar[r, "eq \otimes 1_{A}"] & X \otimes A \ar[r, shift left = 1, "f \otimes 1_{A}" ] \ar[r, shift right = 1, "g \otimes 1_{A}" ']& Y \otimes A.
\end{tikzcd}\]
A flat object $A$ is said to be \textit{faithfully flat} if if tensoring with $A$ reflects isomorphisms.

A monoidal category \(\mathcal{V}\) is called \textit{flat regular} \cite[Definition 2.1.1]{aguiar1997internal} if it has equalizers and every object in $\VV$ is flat.
	\bex
	The category $_{k}\MM$ of $k$-vector spaces is flat regular because all vector spaces are flat.
	\eex
	\bex \label{Ex : regular}
	An \textit{abelian $k$-linear monoidal category} is a $k$-linear abelian category, that is, an abelian category in which each $\operatorname{Hom}$ set is a $k$-vector space and composition is $k$-bilinear, equipped with a monoidal structure such that the tensor product functor $-\otimes -$ is $k$-bilinear and exact in each variable. It follows immediately that an abelian $k$-linear monoidal category is a flat regular monoidal category.
	\eex

	A \textit{braided category} is a monoidal category $\VV$ endowed with a braiding, that is a natural isomorphism
	\[c_{X, Y} : X \otimes Y \longrightarrow Y \otimes X,\]
	satisfying, for all objects $X, Y, Z$ in $\mathcal V$,
	\[c_{X, Y\otimes Z} = (\id_{Y} \otimes c_{X,Z}) \circ (c_{X,Y} \otimes \id_{Z}), \quad
	c_{X\otimes Y, Z} = (c_{X, Z} \otimes  \id_{Y}) \circ (\id_{X} \otimes  c_{Y, Z}), \quad  
	c_{X,I}=\id_X = c_{I,X}.
	\] 
	More details on monoidal categories and braided categories can be found, for example, in \cite{bulacu2019quasi}. We use the familiar graphical calculus and denote by
	\[\begin{tikzpicture}[scale=0.5, base align tikzpicture]
		\draw[above](0.5,0) node{$X$};
		\draw(0.5,0) -- (0.5,-1.2);
		\draw (0,0) -- ( 1,0);
		\draw(0,-1.2) -- (1,-1.2);
		\draw[below](0.5,-1.5) node{$X$};
	\end{tikzpicture}
	\quad \textrm{and} \quad \begin{tikzpicture}[scale=0.6, base align tikzpicture]
		\draw[above](0.5,0) node{$X$};
		\draw(0.5,0) -- (0.5,-1.2);
		\draw (0,0) -- ( 1,0);
		\draw(0,-1.2) -- (1,-1.2);
		\draw[below](0.5,-1.5) node{$Y$};
		\draw [fill=white] (0.5,-0.5) circle (0.4);
		\draw (0.5,-0.5) node {\footnotesize{$f$}};
	\end{tikzpicture}\]
	the identity morphism $\id_{X}: X \to X$ and a morphism $f : X \to Y$ in $\VV$. The braiding $c_{X, Y}$ and its inverse $c^{-1}_{X, Y}$ are denoted by
	\[ \begin{tikzpicture}[scale=0.5, base align tikzpicture]
		\draw[above](1,0) node{$X$};
		\draw[above](2,0) node{$Y$};
		\draw (0.5,0) -- (2.5,0);
		\draw (1,0) -- (1, -0.25);
		\draw (2,0) -- (2, -0.25);
		\smallbraid{2}{-0.25}
		\draw (1,-1.25) -- (1, -1.5);
		\draw (2,-1.25) -- (2, -1.5);
		\draw(0.5,-1.5) -- (2.5,-1.5);
		\draw[below](1,-1.5) node{$Y$};
		\draw[below](2,-1.5) node{$X$};
	\end{tikzpicture} \quad \text{and}\quad
	\begin{tikzpicture}[scale=0.5, base align tikzpicture]
		\draw[above](1,0) node{$Y$};
		\draw[above](2,0) node{$X$};
		\draw (0.5,0) -- (2.5,0);
		\draw (1,0) -- (1, -0.25);
		\draw (2,0) -- (2, -0.25);
		\draw (1,-0.25) ..controls +(0,-0.5) and +(0,0.5).. (2,-1.25);
		\draw[white][line width = 5pt] (1.45, -0.5) ..controls +(0,-0.5) and
		+(0,0.5).. (1.55,-1);
		\draw (2, -0.25) ..controls +(0,-0.5) and +(0,0.5).. (1,-1.25);
		\draw (1,-1.25) -- (1, -1.5);
		\draw (2,-1.25) -- (2, -1.5);
		\draw(0.5,-1.5) -- (2.5,-1.5);
		\draw(0.5,-1.5) -- (2.5,-1.5);
		\draw[below](1,-1.5) node{$X$};
		\draw[below](2,-1.5) node{$Y$};
	\end{tikzpicture}.\] 
	Then the braiding axioms become \begin{align}
		\begin{tikzpicture}[scale=0.5, base align tikzpicture]
			\draw[above](1,0) node{$X$};
			\draw[above](2,0) node{$Y$};
			\draw[above](3,0) node{$Z$};
			\smallbraid{2}{0};
			\smallbraid{3}{-1};
			\draw (0.5,0) -- ( 3.5,0);
			\draw(3, 0) -- (3,-1);
			\draw(1, -1) -- (1, -2);
			\draw(0.5,-2) -- (3.5,-2);
			\draw[below](1,-2.5) node{$Y$};
			\draw[below](2,-2.5) node{$Z$};
			\draw[below](3,-2.5) node{$X$};
			\draw(-1.5,-1.5) node{$c_{X, Y \otimes Z} =$};
		\end{tikzpicture} \quad
		\text{and} \quad
		\begin{tikzpicture}[scale=0.5, base align tikzpicture]
			\draw[above](1,0) node{$X$};
			\draw[above](2,0) node{$Y$};
			\draw[above](3,0) node{$Z$};
			\smallbraid{3}{0};
			\smallbraid{2}{-1};
			\draw (0.5,0) -- ( 3.5,0);
			\draw(3, -1) -- (3,-2);
			\draw(1, 0) -- (1, -1);
			\draw(0.5,-2) -- (3.5,-2);
			\draw[below](1,-2.5) node{$Z$};
			\draw[below](2,-2.5) node{$X$};
			\draw[below](3,-2.5) node{$Y$};
			\draw(-1.5,-1.5) node{$c_{X \otimes Y, Z} =$};
		\end{tikzpicture}. \label{(braid)}
	\end{align}
	\subsubsection{Coquasitriangular Hopf algebras} \label{Ex: MH}
	The basic example of a braided category is the category $_{k}\MM$ of $k$-vector spaces over our base field $k$, with the braiding given by the flip (or symmetry) map. More generally, the category of comodules over a commutative bialgebra forms a braided category.
	
	Recall that a coquasitriagular Hopf algebra is a Hopf algebra $H$ equipped with a convolution-invertible linear form \textbf{r} $\colon H \otimes H \to k$ (called a universal $r$-form) such that, for any $x, y, z \in H$,
	\begin{equation}
		yx = \textbf{r}(x_{(1)},y_{(1)}) x_{(2)}y_{(2)} \textbf{r}^{-1}(x_{(3)},y_{(3)}) \label{r1}
	\end{equation}
	\begin{equation}
		\textbf{r}(xy,z)= \textbf{r}(x,z_{(1)}) \textbf{r}(y,z_{(2)}), \quad  \textbf{r}(x,yz)= \textbf{r}(x_{(1)}, z) \textbf{r}(x_{(2)},y) \label{r2}.
	\end{equation}
	Then the category of $H$-comodules, $\MM^{H}$, is a braided category with the braiding given, for any $H$-comodules $V$ and $W$, by
	\begin{align}
		\textbf{r}_{V, W}(v \otimes w) = \textbf{r}(v_{(1)}, w_{(1)})w_{(0)} \otimes v_{(0)} \label{rVW}.
	\end{align}
	Sometimes we denote this category by $\MM^{H, \textbf{r}}$ to remember the braided structure.
	\bex
	Let $\Gamma$ be an abelian group. Then the universal $r$-forms on the group algebra $k\Gamma$ correspond to the bicharacters $\Gamma\times \Gamma \to k^{*}$, i.e the maps $\psi$ such that
	\[ \psi(xy, z) = \psi(x, z)\psi(y,z); \,\psi(x, yz) = \psi(x, y)\psi(x, z) \quad \text{for} \quad x, y, z \in \Gamma.\]
	Recall that $\mathcal M^{k\Gamma}$ identifies with the category of $\Gamma$-graded vector spaces as follows: if $V = (V, \alpha)$ is a right $k\Gamma$-comodule, put, for $g \in \Gamma$, $V_{g} = \{v \in V \mid \alpha(v) = v \otimes g\}$. Then $V = \bigoplus_{g \in \Gamma} V_{g}$ defines a $\Gamma$-grading on $V$. Conversely, if $V = \bigoplus_{g \in \Gamma} V_{g}$ is $\Gamma$-graded, putting $\alpha(v) = v\otimes g$ for $v \in V_{g}$,  defines a structure of $k\Gamma$-comodule on $V$.
	
	Given a bicharacter $\psi$, the category $\mathcal M^{k\Gamma}$ is braided with braiding: 
	\begin{align*}
		c_{V, W} \colon V \otimes W &\longrightarrow W \otimes V\\
		v \otimes w \in V_{g} \otimes W_{h} &\mapsto \psi(g, h)w \otimes v
	\end{align*}
	When $\Gamma = \mathbb Z  =\langle z \rangle$ is the infinite cyclic group with a fixed generator $z$, a bicharacter is uniquely determined by $\xi= \psi(z,z)$. We denote by $\mathcal M^{k\mathbb Z,\xi}$ the resulting braided category. 
	\eex
	
	In the case of a coquasitriangular Hopf algebra, we recall the following result that will be needed later, in particular in Section~\ref{Sec : transmutation}.
	
	\bp{\cite[Propositions 10.2, 10.3]{KSBook97}} \label{Prop : r3}
	Let $H$ be a coquasitriangular Hopf algebra with universal $r$-form $\mathbf{r}$. Define linear maps 
	$f, \bar{f} \colon H \longrightarrow k$ by 
	\[f(x) = \mathbf{r}[x_{(1)}, S_{H}(x_{(2)})]\quad \text{and} \quad \overline{f}(x) = \mathbf{r}^{-1}[S_{H}(x_{(1)}), x_{(2)}], \quad x \in H.\]
	Then we have
	\begin{enumerate}
		\item for any $x, y \in H$,
		\[\mathbf{r}\big(S_{H}(x), y\big) = \mathbf{r}^{-1}(x, y), \quad\mathbf{r}(x, y) = \mathbf{r}\big(S_{H}(x), S_{H}(y)\big),\]
		\item for $x \in H$,
		\[S^{2}_{H}(x_{(2)})f(x_{(1)}) = x_{(1)}f(x_{(2)})\quad \text{and} \quad \overline{f}(x_{(2)})S^{2}_{H}(x_{(1)}) = \overline{f}(x_{(1)})x_{(2)}.\]
	\end{enumerate}
	\ep
	Note that the map $\overline{f}$ is the inverse of $f$ with respect to the convolution product.
	\subsubsection{Category of Yetter–Drinfeld modules}\label{sec : YD}
	Let $H$ be a fixed Hopf algebra. Recall that a Yetter-Drinfeld module over $H$ is a right $H$-comodule and a right $H$-module $V$ satisfying the condition that, for $v \in V$ and $x \in H$,
	\begin{align}
		(v \leftarrow x)_{(0)} \otimes (v \leftarrow x)_{(1)} = v_{(0)} \leftarrow x_{(2)} \otimes S_{H}(x_{(1)})v_{(1)}x_{(3)}.\label{compatibiliy}
	\end{align}
	The category of Yetter-Drinfeld modules over $H$ is denoted by $\YD^{H}_{H}$: the morphisms are the $H$-linear and $H$-colinear maps.
	
	When $S_{H}$ is bijective, the category $\YD^{H}_{H}$ of Yetter-Drinfeld modules over $H$ is a braided category with the following monoidal and braided structure. Let $V, W \in \YD^{H}_{H}$, the tensor product of vector spaces $V \otimes W$ is also a Yetter-Drinfeld module with the usual diagonal action
	\begin{align*}
		\mu_{V \otimes W} \colon V \otimes W \otimes H &\longrightarrow V \otimes W\\
		v \otimes w \otimes x &\longmapsto (v \leftarrow x_{(1)} )\otimes (w \leftarrow x_{(2)})
	\end{align*}
	and coaction
	\begin{align*}
		\beta_{V \otimes W} \colon V \otimes W &\longrightarrow V \otimes W \otimes H\\
		v \otimes w&\longmapsto v_{(0)}w_{(0)}\otimes v_{(1)}w_{(1)}.
	\end{align*}
	The unit object is the field $k$ with the trivial $H$-comodule and $H$-module structure, where $\beta(1) = 1 \otimes 1$, and $1 \leftarrow x = \varepsilon(x)$ for all $x \in H$. 
	The braiding is given by, for $V, W \in \YD^{H}_{H}$,
	\begin{align} \label{braid YD}
		c_{V, W} \colon V \otimes W &\longrightarrow W \otimes V\\
		v \otimes w &\longmapsto w_{(0)} \otimes (v\leftarrow w_{(1)}). \nonumber
	\end{align}
	and its inverse
	\begin{align*}
		c^{-1}_{V, W} \colon W \otimes V &\longrightarrow V \otimes W\\
		w \otimes v &\longmapsto \big(v\leftarrow S_{H}^{-1}(w_{(1)})\big)\otimes w_{(0)}.
	\end{align*}
	\brq[\textit{Connection to Coquasitriangular Structure}] \label{rk:YD-coquasi}
	Let $H$ be a coquasitriangular Hopf algebra, with the universal $r$-form  $\mathbf{r} \colon H \otimes H \to k$. Given a right $H$-comodule $V$, we defines a right $H$-action on $V$ using the coquasitriangular structure $\mathbf{r}$ as follows:
	\[v \leftarrow x = \mathbf{r}(v_{(1)}, x)v_{(0)}, \quad v \in V, x \in H\]
	This expression defines a bilinear map $V  \otimes H \to V$ which, using the properties of  $\mathbf{r}$ satisfies the axioms of an $H$-module. Moreover, this action is compatible with the comodule structure of $V$ given by \eqref{compatibiliy}, thus endowing $V$ with the structure of a Yetter–Drinfeld module.
	
	This construction yields a braided monoidal functor:
	\[\MM^{H} \hookrightarrow \YD^{H}_{H}.\]
	In particuliar, this functor embeds the braided monoidal category $\MM^{H}$ into the category of Yetter-Drinfeld modules over $H$, providing an explicit correspondence between the braiding induced by the coquasitriangular form  $\mathbf{r}$ and that arising from the weak center construction,  which is known to be equivalent to the category of Yetter–Drinfeld modules $\YD_{H}^{H}$ \cite{schauenburg2004hopf}. Thus, the Yetter–Drinfeld braiding given by \eqref{braid YD} generalizes and recovers the coquasitriangular braiding.
	\erq
	
	\subsection{Algebras, Modules, Coalgebras, Comodules in monoidal categories} The familiar notions of algebras, modules,  coalgebras and comodules  in vector space categories have direct generalizations in monoidal categories. Let $\VV$ be a monoidal category.
	\subsubsection{} An \textit{algebra in $\VV$} is a triple $(A, m_{A}, \eta_{A})$, where $A \in \text{ob}(\VV)$, and $m_{A}: A \otimes A \to A$ and $\eta_A : I \to A$ are morphisms such that
	\[m_{A} \circ (m_{A}\otimes \id_{A}) =  m_{A} \circ (\id_A \otimes m_{A}), \ m_{A} \circ (\eta_{A}\otimes \id_A)= \id_A = m_{A} \circ (\id_A \otimes \eta_{A}).\]
	Denoting the multiplication and the unit by
	\[m_A=\begin{tikzpicture}[scale=0.4, base align tikzpicture]
		\draw(0.5,0)--(3.5,0);
		\draw(0.5,-1.5)--(3.5,-1.5);
		\draw[above](1,0) node{$A$};
		\draw[above](3,0) node{$A$};
		\produit{1}{0}{1}
		\draw[below](2,-1.5) node{$A$};
	\end{tikzpicture} \quad \textrm{and} \quad 
	\eta_A = \begin{tikzpicture}[scale=0.4, base align tikzpicture]
		\draw(8,0)--(9,0);
		\draw(8,-1.5)--(9,-1.5);
		\draw[above](8.5,0) node{$I$};
		\draw(8.5,-0.5) node{$\bullet$};
		\draw(8.5,-0.5) -- (8.5,-1.5);
		\draw[below](8.5,-1.5) node{$A$};
	\end{tikzpicture}\]
	the associativity and unit axioms above read
	\begin{equation}
		\begin{tikzpicture}[scale=0.35, base align tikzpicture]\label{associatif}
			\draw(0,0)--(5,0);
			\draw(0,-4.5)--(5,-4.5);
			\draw(0.5,0)--(0.5,-3);
			\draw[above](4.5, 0) node{$A$};
			\draw[above](0.5,0) node{$A$};
			\draw[above](2.5,0) node{$A$};
			\produit{2.5}{0}{1}
			\produit{0.5}{-3}{1}
			\draw[below](1.5,-4.5) node{$A$};
			\draw(3.5,-1.5) ..controls +(0,-0.7) and +(0,0.7).. (2.5,-3);
			\draw(6.5,-2) node{$=$};
			\draw(8,0)--(13,0);
			\draw(8,-4.5)--(13,-4.5);
			\draw(12.5,0)--(12.5,-3);
			\draw[above](12.5, 0) node{$A$};
			\draw[above](8.5,-0) node{$A$};
			\draw[above](10.5,-0) node{$A$};
			\produit{10.5}{-3}{1}
			\produit{8.5}{0}{1}
			\draw[below](11.5, -4.5) node{$A$};
			\draw(9.5,-1.5) ..controls +(0,-0.7) and +(0,0.7).. (10.5,-3);
			\draw(15,-2) node{and};
			\draw(17,-0.5)--(20,-0.5);
			\draw(17,-3)--(20,-3);
			\draw(19.5,-0.5)--(19.5,-1.5);
			\draw(17.5,-1.25)--(17.5,-1.5);
			\draw[above](19.5, -0.5) node{$A$};
			\draw(17.5,-1.25) node{$\bullet$};
			\produit{17.5}{-1.5}{1}
			\draw[below](18.5, -3) node{$A$};
			\draw(21,-2) node{$=$};
			\draw(22,-1)--(23,-1);
			\draw(22,-2.5)--(23,-2.5);
			\draw[above](22.5, -1) node{$A$};
			\draw[below](22.5, -2.5) node{$A$};
			\draw(22.5,-1)--(22.5,-2.5);
			\draw(24,-2) node{$=$};
			\draw(25,-0.5)--(28,-0.5);
			\draw(25,-3)--(28,-3);
			\draw(25.5,-0.5)--(25.5,-1.5);
			\draw(25.5,-1.25)--(25.5,-1.5);
			\draw[above](27.5, -0.5) node{$A$};
			\draw(27.5,-1.25) node{$\bullet$};
			\draw(27.5,-1.25)--(27.5,-1.5);
			\produit{25.5}{-1.5}{1}
			\draw[below](26.5, -3) node{$A$};
			\draw(29,-2) node{.};
		\end{tikzpicture}
	\end{equation}
	If  $A$, $B$ are  algebras in the monoidal category $\VV$, an algebra morphism $f : A \to B$ is a morphism in $\VV$ such that
	\[f \circ m_A = m_B\circ (f\otimes f) \quad \textrm{and} \quad f\circ\eta_A = \eta_B\]
	Graphically, this means 
	\begin{equation*}
		\begin{tikzpicture}[scale=0.55]
			\draw(0, 0)--(2,0);
			\draw(0, -3)--(2,-3);
			\produit{0.5}{0}{0.5}
			\draw(0.5, -2)--(0.5,-3);
			\draw(1,-1) ..controls +(0,-0.7) and +(0,0.7).. (0.5,-2);
			\draw [fill=white] (0.5,-2.25) circle (0.4);
			\draw (0.5,-2.25) node {\footnotesize{$f$}};
			\draw[above](0.5,0) node{$A$};
			\draw[above](1.5,0) node{$A$};
			\draw[below](0.5,-3) node{$B$};
			\draw (2.5,-1.5) node{$=$};
			\draw[above](4.5,-.5) node{$A$};
			\draw[above](5.5,-.5) node{$A$};
			\draw (4,-.5) -- ( 6,-0.5);
			\draw (5.5,-0.5) -- (5.5,-1.2);
			\draw (4.5,-.5) -- (4.5,-1.2);
			\draw (5.5,-1.1) -- (5.5,-2.1);
			\draw (4.5,-1.1) -- (4.5,-2.1);
			\draw [fill=white] (4.5,-1.73) circle (0.4);
			\draw (4.5,-1.73) node {\footnotesize{$f$}};
			\draw [fill=white] (5.5,-1.73) circle (0.4);
			\draw (5.5,-1.73) node {\footnotesize{$f$}};
			\draw(4, -3.1) -- (6,-3.1);
			\produit{4.5}{-2.1}{0.5}
			\draw[below](5,-3.1) node{$B$};
			\draw (7,-1.5) node{and};
			\draw(8, -0.5)--(9.5,-0.5);
			\draw(8, -2.5)--(9.5,-2.5);
			\draw(8.75, -1) node {$\bullet$};
			\draw(8.75, -1)--(8.75,-2.5);
			\draw [fill=white] (8.75,-1.75) circle (0.4);
			\draw (8.75,-1.75) node {\footnotesize{$f$}};
			\draw[above](8.75,-0.5) node{$I$};
			\draw[below](8.75,-2.5) node{$A$};
			\draw (10.5,-1.5) node{$=$};
			\draw(11.5, -0.75)--(13,-0.75);
			\draw(11.5, -2.25)--(13,-2.25);
			\draw(12.25, -1.35) node {$\bullet$};
			\draw(12.25, -1.35)--(12.25,-2.25);
			\draw[above](12.25,-0.75) node{$I$};
			\draw[below](12.25,-2.25) node{$A$};
			\draw (15,-1.5) node{$.$};
		\end{tikzpicture}\label{morphism}
	\end{equation*}
	
	Let $A$ be an algebra in $\VV$. 
	A \textsl{left $A$-module} $M$ (in $\VV$) is an object $M$ in $\VV$ together with a morphism $\mu_{M}^{l}: A \otimes M \to M$, denoted by 
	\[\begin{tikzpicture}[scale=0.35, base align tikzpicture]
		\draw(0.5,0)--(3.5,0);
		\draw(0.5,-1.5)--(3.5,-1.5);
		\draw[above](3, 0) node{$M$};
		\draw[above](1,-0) node{$A$};
		\gauche{3}{0}{1}
		\draw[below](1.5,-1.5) node{$M$};
	\end{tikzpicture},\]
	such that
	\begin{align}
		\begin{tikzpicture}[scale=0.35, base align tikzpicture]
			\draw(0,0)--(5,0);
			\draw(0,-4)--(5,-4);
			\draw(0.5,0)--(0.5,-2.5);
			\draw[above](4.5, 0) node{$M$};
			\draw[above](0.5,0) node{$A$};
			\draw[above](2.5,0) node{$A$};
			\gauche{4.5}{0}{1}
			\gauche{2.5}{-2.5}{1}
			\draw[below](1,-4) node{$M$};
			\draw(3,-1.5) ..controls +(0,-0.7) and +(0,0.7).. (2.5,-2.5);
			\draw(6.5,-2) node{$=$};
			\draw(8,0)--(13,0);
			\draw(8,-4)--(13,-4);
			\draw(12.5,0)--(12.5,-2.5);
			\draw[above](12.5, 0) node{$M$};
			\draw[above](8.5,-0) node{$A$};
			\draw[above](10.5,-0) node{$A$};
			\gauche{12.5}{-2.5}{1}
			\produit{8.5}{0}{1}
			\draw[below](11, -4) node{$M$};
			\draw(9.5,-1.5) ..controls +(0,-0.7) and +(0,0.7).. (10.5,-2.5);
			\draw(15,-2) node{and};
			\draw(17,-0.5)--(20,-0.5);
			\draw(17,-3)--(20,-3);
			\draw(19.5,-0.5)--(19.5,-1.5);
			\draw(17.5,-1.25)--(17.5,-1.5);
			\draw[above](19.5, -0.5) node{$M$};
			\draw(17.5,-1.25) node{$\bullet$};
			\gauche{19.5}{-1.5}{1}
			\draw[below](18, -3) node{$M$};
			\draw(21,-2) node{$=$};
			\draw(22,-1)--(23,-1);
			\draw(22,-2.5)--(23,-2.5);
			\draw[above](22.5, -1) node{$M$};
			\draw[below](22.5, -2.5) node{$M$};
			\draw(24,-2) node{.};
			\draw(22.5,-1)--(22.5,-2.5);
		\end{tikzpicture}\label{l}
	\end{align}
	The category of left $A$-modules (in $\VV$) is denoted by $_{A}\VV$, and its morphisms are the left $A$-linear maps, defined just as in the classical case. 
	The  category  $\VV_{A}$  of right $A$-modules is defined similarly.
	
	\subsubsection{}As in the ordinary case of vector spaces, the definition of a coalgebra in a monoidal category is dual to that of an algebra. More precisely, a \textsl{coalgebra in} $\VV$ is a triple $(C, \Delta_{C}, \varepsilon_{C})$, where $\Delta_{C}: C \to C \otimes C$ and  $\varepsilon_{C}: C \to I$ are morphisms, denoted by
	\begin{equation*}
		\begin{tikzpicture}[scale=0.35, base align tikzpicture]
			\draw(0.5,0)--(3.5,0);
			\draw(0.5,-1.5)--(3.5,-1.5);
			\draw[below](3,-1.5) node{$C$};
			\draw[below](1,-1.5) node{$C$};
			\coproduit{3}{-1.5}{1}
			\draw[above](2,-0) node{$C$};
		\end{tikzpicture}
		\quad  \textrm{and} \quad  \begin{tikzpicture}[scale=0.35, base align tikzpicture] 
			\draw(8,0)--(9,0);
			\draw(8,-1.5)--(9,-1.5);
			\draw[above](8.5,0) node{$C$};
			\draw(8.5,-1) node{$\bullet$};
			\draw(8.5,0) -- (8.5,-1);
			\draw[below](8.5,-1.5) node{$I$};
		\end{tikzpicture}, 
	\end{equation*}
	satisfying the coassociativity and counit conditions:
	\begin{equation} \label{cogebra} 
		\begin{tikzpicture}[scale=0.35, base align tikzpicture]
			\draw(0,0)--(5,0);
			\draw(0,-4.5)--(5,-4.5);
			\draw(4.5,-1.5)--(4.5,-4.5);
			\draw[below](4.5, -4.5) node{$C$};
			\draw[below](0.5,-4.5) node{$C$};
			\draw[below](2.5,-4.5) node{$C$};
			\coproduit{4.5}{-1.5}{1}
			\coproduit{2.5}{-4.5}{1}
			\draw[above](3.5,0) node{$C$};
			\draw(2.5,-1.5) ..controls +(0,-0.7) and +(0,0.7).. (1.5,-3);
			\draw(6.5,-2) node{$=$};
			\draw(8,0)--(13,0);
			\draw(8,-4.5)--(13,-4.5);
			\draw(8.5,-1.5)--(8.5,-4.5);
			\draw[below](12.5, -4.5) node{$C$};
			\draw[below](8.5,-4.5) node{$C$};
			\draw[below](10.5,-4.5) node{$C$};
			\coproduit{10.5}{-1.5}{1}
			\coproduit{12.5}{-4.5}{1}
			\draw[above](9.5,0) node{$C$};
			\draw(10.5,-1.5) ..controls +(0,-0.7) and +(0,0.7).. (11.5,-3);
			\draw(15,-2) node{and};
			\draw(17,-0.5)--(20,-0.5);
			\draw(17,-3)--(20,-3);
			\draw(19.5,-2)--(19.5,-3);
			\draw(17.5,-2)--(17.5,-2.5);
			\draw[below](19.5, -3) node{$C$};
			\draw(17.5,-2.5) node{$\bullet$};
			\coproduit{19.5}{-2}{1}
			\draw[above](18.5, -0.5) node{$C$};
			\draw(21,-2) node{$=$};
			\draw(22,-1)--(23,-1);
			\draw(22,-2.5)--(23,-2.5);
			\draw[above](22.5, -1) node{$C$};
			\draw[below](22.5, -2.5) node{$C$};
			\draw(22.5,-1)--(22.5,-2.5);
			\draw(24,-2) node{$=$};
			\draw(25,-0.5)--(28,-0.5);
			\draw(25,-3)--(28,-3);
			\draw(27.5,-2)--(27.5,-3);
			\draw(25.5,-2)--(25.5,-2.5);
			\draw[below](25.5, -3) node{$C$};
			\draw(25.5,-2.5) node{$\bullet$};
			\coproduit{27.5}{-2}{1}
			\draw[above](26.5, -0.5) node{$C$};
			\draw(29,-2) node{.};
		\end{tikzpicture}
	\end{equation}
	The definition of a coalgebra morphism is a straightforward adaptation of the ordinary one.
	
	Likewise, by duality, we obtain the notion of a comodule over a coalgebra: let $C$ be a coalgebra in $\VV$. A right $C$-\textit{comodule} $M$ in $\VV$ is an object $M$ in $\VV$ together with a morphism $\beta_{M} : M \to C \otimes M$, denoted by
	\[\begin{tikzpicture}[scale=0.44, base align tikzpicture]
		\tiny
		\draw(0,-0.5) -- (3,-0.5);
		\draw(0,-2) -- (3,-2);
		\draw[above](1,-0.5) node{$M$};
		\coright{2.5}{-2}
		\draw[below](0.5,-2) node{$M$};
		\draw[below](2.5,-2) node{$C$};
	\end{tikzpicture},\]
	such that
	\[\begin{tikzpicture}[scale=0.44, base align tikzpicture]
		\tiny
		\draw(-0.5,0) -- (4,-0);
		\draw(-0.5,-2.8) -- (4,-2.8);
		\draw[above](1,0) node{$M$};
		\coright{2.5}{-1.5}
		\coright{2}{-2.8}
		\draw(2.5,-1.5) ..controls +(0,-0.7) and +(0,0.7).. (3.2,-2.5);
		\draw(3.2,-2.5) -- (3.2,-2.8);
		\draw[below](0,-2.8) node{$M$};
		\draw[below](2,-2.8) node{$C$};
		\draw[below](3.2,-2.8) node{$C$};
		\draw(6.25,-2.5) node{$=$};
		\draw(8,0) -- (13,-0);
		\draw(8,-2.8) -- (13,-2.8);
		\draw[above](9.5,0) node{$M$};
		\coright{11}{-1.5}
		\coproduit{11.8}{-2.8}{0.8}
		\draw(9,-1.5) -- (9,-2.8);
		\draw[below](9,-2.8) node{$M$};
		\draw[below](10.2,-2.8) node{$C$};
		\draw[below](11.8,-2.8) node{$C$};
		\draw(14,-2.5) node{.};
	\end{tikzpicture}\]
	The category of right $C$-comodules in $\VV$ is denoted by $\VV^{C}$. In a similar manner we can define $^{C}\VV$, the category of left $C$-comodules and left $C$-colinear morphisms in $\VV$.
	
	A $C$-\textit{bicomodule} in $\VV$ is an object $M$ in $\VV$ which is a left $C$-comodule and right $C$-comodule and such that 
	\[\begin{tikzpicture}[scale=0.44, base align tikzpicture]
		\tiny
		\draw(-1,0) -- (5.75,-0);
		\draw(-1,-4) -- (5.75,-4);
		\draw[above](2,0) node{$M$};
		\coleft{2.5}{-1.5}
		\coright{5}{-4}
		\draw(2.5,-1.5) ..controls +(0,-0.7) and +(0,0.7).. (3.5,-2.5);
		\draw(0.5,-1.5) -- (0.5,-4);
		\draw[below](0.5,-4) node{$C$};
		\draw[below](5,-4) node{$C$};
		\draw[below](3,-4) node{$M$};
		\draw(6.75,-3) node{$=$};
		\draw(8,0) -- (14,-0);
		\draw(8,-4) -- (14,-4);
		\draw[above](12,0) node{$M$};
		\coright{13.5}{-1.5}
		\coleft{11}{-4}{1.5}
		\draw(11.5,-1.5) ..controls +(0,-0.7) and +(0,0.7).. (10.5,-2.5);
		\draw(13.5,-1.5) -- (13.5,-4);
		\draw[below](9,-4) node{$C$};
		\draw[below](11,-4) node{$M$};
		\draw[below](13.5,-4) node{$C$};
		\draw(15,-3) node{.};
	\end{tikzpicture}\]
	The category of $C$-bicomodules in $\VV$ is denoted by $\,^{C}\VV^{C}$.
	\subsubsection{}\label{Hopf}Beyond algebras and coalgebras in a monoidal category, we can consider \textit{braided Hopf algebras}, namely Hopf algebras in appropriate braided categories.
	Let $\VV=(\mathcal V, c)$ be a braided monoidal category, and let $A$, $B$ be algebras in $\VV$. 
	The braiding $c$ of $\VV$ gives rise to an algebra structure on the object $A \otimes B$ with multiplication given by
	\begin{align}
		\begin{tikzpicture}[scale=0.5, base align tikzpicture]
			\draw(3.5,0)--(3.5,-1.5);
			\draw[above](1.5, 0) node{$B$};
			\draw[above](0.5, 0) node{$A$};
			\draw[above](2.5, 0) node{$A$};
			\draw[above](3.5, 0) node{$B$};
			\smallbraid{2.5}{0}
			\draw(0,0)--(4,0);
			\draw(0,-2.5)--(3.5,-2.5);
			\produit{.5}{-1.5}{0.5}
			\produit{2.5}{-1.5}{0.5}
			\draw(0.5,0)--(0.5,-1.5);
			\draw(1.5,-1)--(1.5,-1.5);
			\draw(2.5,-1)--(2.5,-1.5);
			\draw(-1.5,-1.5) node{$m_{A \otimes_{c} B} = $};
			\draw[below](1, -2.5) node{$A$};
			\draw[below](3, -2.5) node{$B$};
		\end{tikzpicture}\label{cross}
	\end{align}
	and unit $\eta_A\otimes \eta_B$.
	The resulting algebra in $\VV$ is denoted by $A \otimes_{c} B$ and is called the \textsl{braided tensor product algebra of $A$ and $B$}. This allows us to define a \textit{bialgebra} in $\VV$ and then a Hopf algebra in $\VV$, as in the classical case: a \textit{Hopf algebra} $H = (H, m_H, \eta_{H}, \Delta_{H}, \varepsilon_{H}, S_{H})$ in $\VV$ is an algebra ($H, m_H, \eta_H$) and a coalgebra $(H, \Delta_{H}, \varepsilon_{H})$ in $\VV$, such that 
	$$\Delta_{H} : H \to H \otimes_{c} H; \quad \varepsilon_{H} : H \to I$$
	are algebra morphisms, and $S_{H} : H \to H$ is a morphism in $\VV$ satisfying $S_{H} * 1_{H} = \eta_{H} \circ \varepsilon_{H} = 1_{H} * S_{H}$, where $*$ is the convolution product.
	
	We omit the details and refer the reader to \cite[Chapter $2$]{bulacu2019quasi} for more information.

	\subsection{Cotensor product} Let $\VV$ be a flat regular monoidal category. Let $C$ be some coalgebra in $\VV$. The cotensor product of a right $C$-comodule $(V, \beta_{V})$ and a left $C$-comodule $(W, \alpha_{W})$ is the following object of $\VV$: 
	\begin{center}
		\begin{tikzcd}
			V \Box_{C} W = \textbf{Eq} \big( V \otimes W \arrow[r, shift left, "\beta_{V} \otimes 1"] \ar[r, shift right, "1 \otimes \alpha_{W}" ']& V \otimes C \otimes W\big).
		\end{tikzcd}
	\end{center}
	
	From now on, we make the convention that
	$\alpha$  and $\beta$ denote the left and right comodule structures in $\VV$, respectively. We begin with a result that can be found, for example, in \cite{aguiar1997internal}: 
	\bp \label{Prop : cotensor morphism}
	Let $C$ be a coalgebra in $\VV$. Let $V, V' \in \VV^{C}$ and let $f : V \to V'$ be a morphism in $\VV^{C}$. Let $W, W' \in \,^{C}\VV$ and let $g : W \to W'$ be a morphism in $\,^{C}\VV$. Then there exists a unique morphism 
	$f \Box_{C} g : V \Box_{C} W \longrightarrow V' \Box_{C} W'$ fitting into the following diagram:
	\begin{equation}
		\begin{tikzcd}[column sep= 1.5cm,row sep= 1cm]
			V \Box_{C} W \ar[d, "f \Box_{C} g" '] \ar[r, "eq"] &V \otimes W \ar[d,"f \otimes g"]\\
			V' \Box_{C} W' \ar[r, "eq"] &V' \otimes W'.
		\end{tikzcd}
	\end{equation}
	\ep
	\bpf
	This is an easy verification using the equalizer property.
	\epf
	\bp \label{Prop : forget right}
	Let $C$ be a coalgebra in $\VV$ and $X$ be a left $C$-comodule in $\VV$. Let $V, W \in \VV^{C}$ and let $f : V \longrightarrow W$ be a morphism of $\VV^{C}$. Then there exists a unique morphism $$f \Box_{C} 1_{X} : V \Box_{C} X \longrightarrow W \Box_{C} X$$ such that the following diagram commutes: 
	\begin{equation} \label{morphism left}
		\begin{tikzcd}[column sep= 2cm,row sep= 1.5cm]
			V \Box_{C} X \ar[d, dashed, "f \Box_{C} 1_{X}"]  \ar[r, "eq"] &V \otimes X \ar[d, "f \otimes 1_{X}"]  \\
			W \Box_{C} X \ar[r, "eq" '] &W \otimes X
		\end{tikzcd}
	\end{equation}
	and this defines a functor: 
	\begin{align*}
		\Upsilon_{r}^{X} : \VV^{C} &\longrightarrow \VV\\
		V &\longmapsto V \Box_{C} X
	\end{align*}
	\ep
	\bpf
	This result follows directly from Proposition \ref{Prop : cotensor morphism}, by replacing the identity morphism $1_{X}$ with $g$. Functoriality follows from the uniqueness of $f \Box_C 1_{X}$.
	\epf
	We also have the right version of Proposition $\ref{Prop : forget right}$ as follows:
	\bp\label{Prop: forget left}
	Let $C$ be a coalgebra in $\VV$ and $X$ be a right $C$-comodule in $\VV$. Let $V, W \in \,^{C}\VV$ and let $g : V \longrightarrow W$ be a morphism of $\,^{C}\VV$. Then there exists a unique morphism $$1_{X} \Box_{C} \, g : X \Box_{C} V \longrightarrow X \Box_{C} W$$ such that the following diagram commutes: 
	\begin{equation} \label{morphism right}
		\begin{tikzcd}[column sep= 2cm,row sep= 1.5cm]
			X \Box_{C} V \ar[d, dashed, "1_{X}\Box_{C} \,g"]\ar[r, "eq"] &X \otimes V \ar[d, "1_{X} \otimes g"] \\
			X \Box_{C} W  \ar[r, "eq" '] &X \otimes W
		\end{tikzcd}
	\end{equation}
	and this defines a functor: 
	\begin{align*}
		\Upsilon_{l}^{X} : \, ^{C}\VV &\longrightarrow \VV\\
		V &\longmapsto X \Box_{C} V.
	\end{align*}
	\ep
	\bp \label{Prop:funct}
	Let $C, D, E$ be some coalgebras in $\VV$. Any $C$-$D$-bicomodule $X$ over $\VV$ defines functors:
	\begin{align*}
		F^{X} \colon \VV^{C} &\longrightarrow \VV^{D} \qquad \qquad \quad \text{and}
		&
		\overline{F}^{X} \colon {}^{E}\VV^{C} &\longrightarrow {}^{E}\VV^{D}
		\\
		V &\longmapsto V \Box_{C} X
		&
		W &\longmapsto W \Box_{C} X .
	\end{align*}
	\ep
	\bpf
	First, for a right $C$-comodule $V$ in $\VV$, we see that $1_{V} \otimes \beta^{D}_{X}$ endows $V \otimes X$ with a structure of right $D$-comodule; coassociativity and counitality follow from those of $(X, \beta^{D}_{X})$. We now show that $V \Box_{C} X$ is a $D$-subcomodule of $V \otimes X$. By the flatness of $\VV$, we have $$(V \Box_{C} X) \otimes D = \textbf{Eq}\big(\beta^{C}_{V} \otimes 1_{W} \otimes 1_{D}, 1_{V} \otimes \alpha^{C}_{X} \otimes 1_{D}\big).$$ We also have
	\begin{align*}
		\big(\beta^{C}_{V} \otimes 1_{W} \otimes 1_{D}\big)(1_{V} \otimes \beta^{D}_{X})eq &= \big(1_{V} \otimes 1_{C} \otimes \beta^{D}_{X}\big)(\beta^{C}_{V}\otimes 1_{X})eq\\
		&= \big(1_{V} \otimes 1_{C} \otimes \beta^{D}_{X}\big)(1_{V} \otimes \alpha^{C}_{X})eq \quad \text{($V \Box_{C} X$ is an equalizer)}\\
		&= \big(1_{V} \otimes \alpha^{C}_{X}\otimes 1_{D}\big)(1_{V} \otimes \beta^{D}_{X})eq \quad \text{($X$ is a $C$-$D$-bicomodule)}.
	\end{align*}
	Hence, there exists a unique morphism $\beta^{D}_{V \Box_{C} X} \colon V \Box_{C} X \to (V \Box_{C} X) \otimes D$ such that $(1_{V} \otimes \beta^{D}_{X})eq = (eq \otimes 1_{D})\beta^{D}_{V \Box_{C} X}$, i.e. the following diagram commutes
	\begin{equation} \label{coaction}
		\begin{tikzcd}[column sep= 1.5cm,row sep= 1cm]
			V \Box_{C} X \ar[d, dashrightarrow, "\beta^{D}_{V \Box_{C} X}" ']\ar[r, "eq"] &V \otimes X \ar[d, "1_{V} \otimes \beta^{D}_{X}" ']\arrow[r, shift left, "\beta^{C}_{V} \otimes 1_{X}"] \ar[r, shift right, "1_{V} \otimes \alpha^{C}_{X}" ']& V \otimes C \otimes X\\
			(V \Box_{C} X)  \otimes D \ar[r, "eq \otimes 1_{D}" '] &V \otimes X \otimes D \arrow[r, shift left, "\beta^{C}_{V} \otimes 1_{X} \otimes 1_{D}"] \ar[r, shift right, "1_{V} \otimes \alpha^{C}_{X} \otimes 1_{D}" ']& V \otimes C \otimes X \otimes D.
		\end{tikzcd}
	\end{equation}
	Since
	\begin{align*}
		(eq \otimes 1_{D} \otimes 1_{D})&\big((1_{V} \Box_{C} 1_{X}) \otimes \Delta_{D}\big)\beta^{D}_{V \Box_{C} X} \\
		&= (1_{V} \otimes 1_{X} \otimes \Delta_{D})(eq \otimes 1_{D})\beta^{D}_{V \Box_{C} X}\\
		&= (1_{V} \otimes 1_{X} \otimes \Delta_{D})(1_{V} \otimes \beta^{D}_{X})eq \quad (\text{by \ref{coaction}})\\
		&= (1_{V} \otimes \beta^{D}_{X} \otimes 1_{D})(1 \otimes \beta^{D}_{X})eq \quad (\text{$X$ is a right $D$-comodule})\\
		&= (1_{V} \otimes \beta^{D}_{X} \otimes 1_{D})(eq \otimes 1_{D})\beta^{D}_{V \Box_{C} X} \quad (\text{by \ref{coaction}})\\
		&= (eq \otimes 1_{D} \otimes 1_{D})\big(\beta^{D}_{V \Box_{C} X} \otimes 1_{D}) \beta^{D}_{V \Box_{C} X}
	\end{align*}
	and $eq \otimes 1_{D} \otimes 1_{D}$ is a monomorphism (because it is an equalizer), it follows that $\big((1_{V} \Box_{C} 1_{X}) \otimes \Delta_{D}\big)\beta^{D}_{V \Box_{C} X} = \big(\beta^{D}_{V \Box_{C} X} \otimes 1_{D}) \beta^{D}_{V \Box_{C} W}.$ Hence, $V \Box_{C} X$ is a $D$-subcomodule of $V \otimes X$ via $eq$, i.e. the following diagram commutes
	\[\begin{tikzcd}[column sep= 2cm,row sep= 1cm]
		V \Box_{C} X \ar[d, "eq" ']\ar[r, "\beta^{D}_{V \Box_{C} X}"] &(V \Box_{C} X) \otimes D \ar[d, "eq \otimes 1_{D}" ']\arrow[r, shift left, "(1_{V} \Box_{C} 1_{X}) \otimes \Delta_{D} "] \ar[r, shift right, "\beta^{D}_{V \Box_{C} X} \otimes 1_{D}" ']& (V \Box_{C} X) \otimes D \otimes D \ar[d, "eq \otimes 1_{D} \otimes 1_{D}"]\\
		V \otimes X \ar[r, "1 \otimes \beta^{D}_{X}" '] &V \otimes X \otimes D \arrow[r, shift left, "1_{V} \otimes 1_{X} \otimes \Delta_{D}"] \ar[r, shift right, "1_{V} \otimes \beta^{D}_{X} \otimes 1_{D}" ']& V \otimes X \otimes D \otimes D .
	\end{tikzcd}\]
	Let $f \colon V \to W$ be a morphism of $C$-comodules. By Proposition \ref{Prop : forget right}, there exists a unique morphism $f \Box_{C} 1_{X} \colon V \Box_{C} X \longrightarrow W \Box_{C} X$ in $\VV$.
	
	
Finally we must check that $f \Box_{C} 1_{X}$ is $D$-colinear, i.e. that the following diagram commutes:
	\begin{center}
		\begin{tikzcd}[column sep= 1.5cm,row sep= 1cm]
			V \Box_{C} X \arrow[rrr, "f \Box_{C} 1_{X}"] \arrow[ddd, "\beta^{D}_{V \Box_{C} X}" '] \arrow[rd, "eq_{V}"] &                       &             & W \Box_{C} X \arrow[ddd, "\beta^{D}_{W \Box_{C} X}"] \arrow[ld, "eq_{W}" '] \\
			& V \otimes X \arrow[r, "f \otimes 1_{X}"] \arrow[d, "1_V \otimes \beta^{D}_{X}" '] & W \otimes X \arrow[d, "1_{W} \otimes \beta^{D}_{X}"] &                          \\
			& V \otimes X \otimes D \arrow[r, " f \otimes 1_{X} \otimes 1_{D}" ']           & W \otimes X \otimes D       &                          \\
			(V \Box_{C} X) \otimes D \arrow[rrr, "(f \Box_{D} 1_{X}) \otimes 1_{D}" '] \arrow[ru, "eq_{V} \otimes 1_{D}"]             &                       &             & (W \Box_{C} X) \otimes D \arrow[lu, "eq_{W} \otimes 1_{D}" ']            
		\end{tikzcd}
	\end{center}
	Using the diagram \eqref{morphism left} and the fact that $-\otimes D$ preserves equalizers, we obtain:
	\begin{align*}
		(eq_{W} \otimes 1_D)(f \Box_{D} 1_{X} \otimes 1_{D})\beta^{D}_{V \Box_{C} X} &= (f \otimes 1_{X} \otimes 1_{D})(eq_{V} \otimes 1_D)\beta^{D}_{V \Box_{C} X} \\
		&= (f \otimes 1_{X} \otimes 1_{D})(1_V \otimes \beta^{D}_{X})eq_{V} \hspace*{1.25cm} \text{by diagram \eqref{coaction}}\\
		&= (1_W \otimes \beta^{D}_{X}) (f \otimes 1_X) eq_{V} \\
		&= (1_W \otimes \beta^{D}_{X}) eq_{W} (f \Box_{C} 1_{X}) \hspace*{1.75cm} \text{by diagram \eqref{morphism}}\\
		&= (eq_{W} \otimes 1_D)\beta^{D}_{W \Box_{C} X}(f \Box_{C} 1_{X}) \quad \text{by diagram \eqref{coaction}}.
	\end{align*}
	Since $eq_{W} \otimes 1_{D}$ is a monomorphism, we have \[(f \Box_{D} 1_{X} \otimes 1_{D})\beta^{D}_{V \Box_{C} X} = \beta^{D}_{W \Box_{C} X}(f \Box_{C} 1_{X}).\]
	Hence, this yields the functor $F^{X}$.
	
	Now, let $W$ be an $E$-$C$-bicomodule. By the same argument as above, the object $W \otimes X$ naturally carries the structure of an $E$-$D$-bicomodule, and one checks that $W \Box_{C} X$ inherits an $E$-$D$-bicomodule structure. Thus, we obtain the functor $\overline{F}^{X},$ as desired.
	\epf
	\brq \label{remark: commute tensor}
	Let $C$ be a coalgebra in $\VV$, let $(X, \beta_{X}) \in \VV^{C}$ and $(Y, \alpha_{Y}) \in \,^{C}\VV$. Let $V \in \ob(\VV)$. By the flatness of $V$, we have
	$V \otimes (X \Box_{C} Y) = \textbf{Eq}\big(1_{V} \otimes \beta_{X} \otimes 1_{Y}, 1_{V} \otimes 1_{X} \otimes \alpha_{Y}\big).$
	Moreover, we know that $1_{V} \otimes \beta_{X}$ endows $V \otimes X$ with the structure of a right $C$-comodule; thus, we also have
	\[(V \otimes X) \Box_{C} Y = \textbf{Eq}\big(1_{V} \otimes \beta_{X} \otimes 1_{Y}, 1_{V} \otimes 1_{X} \otimes \alpha_{Y}\big).\]
	Hence $V \otimes (X \Box_{C} Y) \simeq (V \otimes X) \Box_{C} Y $.
	
	\erq
	The following result is an immediate consequence of Propositions \ref{Prop : forget right} and \ref{Prop:funct}.
	\bp
	Let $C, D$ be some coalgebras in $\VV$, and let $X \in \,^{C}\VV^{D}$. Then the functor $\Upsilon^{X}_{r}$ in Proposition \ref{Prop : forget right} can be factorized through the functor $F^{X}$ of Propostion \ref{Prop:funct} as follows: 
	\begin{center}
		\begin{tikzcd}[column sep= 1cm,row sep= 1cm]
			\VV^{C} \ar[rd, "F^{X}" '] \ar[rr , "\Upsilon^{X}_{r}"] &&\VV\\
			&\VV^{D} \ar[ru, "\Omega^{D}" '] &
		\end{tikzcd}
	\end{center}
	where $\Omega^{D} : \VV^{D} \longrightarrow \VV$ is the forgetful functor.
	\ep
	The next two propositions can also be found in \cite[Propositions 2.2.3 and 2.2.2]{aguiar1997internal}:
	\bp\label{Prop : isotrivial}
	Let $C$ be a coalgebra in  $\VV$. Let $(V, \beta) \in \VV^{C}$. Then we have an isomorphism $\bar{\beta} : V \to V \Box_{C} C$ of right $C$-comodules in $\VV$, fitting in the following commutative diagram
	\begin{center}
		\begin{tikzcd}[column sep= 2cm,row sep= 1cm]
			V \ar[r, "\beta"] \ar[rd, dashed,"\bar{\beta}" ', "\sim"]&V \otimes C \\
			&V \Box_{C} C \ar[u, "eq"]
		\end{tikzcd}
	\end{center}
	with inverse $1_{V} \Box_{C}\varepsilon_{C} : V \Box_{C} C \longrightarrow V$.
	\ep
	\bp \label{Prop:associativity}
	Let $C, D, E$ be coalgebras in $\VV$. Let $X \in \VV^{C}, Y \in \,^{C}\VV^{D}$, and $Z \in \,^{D}\VV^{E}$. Then there is a canonical isomorphism of $E$-comodules:
	\[\begin{tikzcd}
		\theta: (X \Box_{C} Y) \Box_{D} Z \ar[r, "\sim"] &X \Box_{C} (Y \Box_{D} Z),
	\end{tikzcd}\]
	fitting in the following commutative diagram
	\[	\begin{tikzcd}
		X \Box_{C} (Y \Box_{D} Z )\ar[r, "eq_{C}"] & X \otimes (Y \Box_{D} Z) \ar[r, "1 \otimes eq_{D}"] & X \otimes Y \otimes Z\\
		( X \Box_{C} Y) \Box_{D} Z \ar[u, "\theta"]\ar[r, "eq_{D}"] & (X \Box_{C} Y) \otimes Z \ar[ru, "eq_{C} \otimes 1" '] &.
	\end{tikzcd}\]
	\ep
	\brq Building on this result, we can define the object $X \Box_C Y \Box_{D} Z$ as the equalizer of the following diagram: \begin{center}
		\begin{tikzcd}[column sep= 2cm,row sep= 1cm]
			X \otimes Y \otimes Z \ar[r, shift left, "\beta^C_X \otimes \beta^D_Y \otimes 1_Z"] \ar[r, shift right, "1_X \otimes \alpha^C_Y \otimes \alpha^D_Z" '] &X \otimes C \otimes Y \otimes D \otimes Z
		\end{tikzcd}
	\end{center}
	We can then express the associativity property as:
	\[(X \Box_{C} Y) \Box_{D} Z \simeq X \Box_{C} Y \Box_{D} Z \simeq X \Box_{C} (Y \Box_{D} Z).\]
	\erq
	The following proposition establishes the associativity of the morphisms:
	\bp \label {Prop: Morphism asso}
	Let $C, D, E$ be coalgebras in $\VV$. Let $X, X' \in \VV^{C}; Y, Y' \in \,^{C}\VV^{D}$, and $Z, Z' \in \,^{D}\VV^{E}$. Let $f : X \to X', g : Y \to Y'$, and $h : Z \to Z'$ be morphisms in $\VV^{C}, \,^{C}\VV^{D}$ and in $\,^{D}\VV^{E}$, respectively. Then the following diagram commutes:
	\begin{center}
		\begin{tikzcd}[column sep= 2cm,row sep= 1cm]
			(X \Box_{C} Y) \Box_{D} Z\ar[r, "\sim", "\theta" ']  \ar[d, "(f \Box_{C} g) \Box_{D} h" ']&X \Box_{C} (Y \Box_{D} Z)  \ar[d, "f \Box_{C} (g \Box_{D} h)"]\\
			(X' \Box_{C} Y') \Box_{D} Z'\ar[r, "\sim", "\theta' " '] &X' \Box_{C} (Y' \Box_{D} Z') 
		\end{tikzcd}
	\end{center}
	where $\theta$ and $\theta'$ come from Proposition \ref{Prop:associativity}.
	\ep
	\bpf
	The following diagram commutes:
	\begin{center}
		\begin{tikzcd}[column sep= 2cm,row sep= 1.5cm]
			&X \otimes Y \otimes Z \ar[ddd, "f \otimes g \otimes h", ""{name=D}]\\
			(X \Box_{C} Y) \Box_{D} Z \arrow[ru, "(eq_{C} \otimes 1)eq_{D}", ""{name = U}]\arrow[d, "(f \Box_{C} g) \Box_{D} h" '] \ar[rr, dashed, "\theta", pos = 0.6]&& X \Box_{C} (Y \Box_{D} Z) \ar[d, "f \Box_{C} (g \Box_{D} h)", ""{name = K}] \ar[lu, "(1 \otimes eq_{D})eq_{C}" '] \\
			(X' \Box_{C} Y') \Box_{D} Z' \arrow[rd, "(eq_{C} \otimes 1)eq_{D}"] \ar[rr, dashed, "\theta' ", pos = 0.6]&& X' \Box_{C} (Y' \Box_{D} Z') \ar[ld, "(1 \otimes eq_{D})eq_{C}"] \\
			&X' \otimes Y' \otimes Z'
			\arrow[to path={(D) node[left, scale=1]{$\underset{\tiny{\text{\hspace*{-1.5cm}(by Proposition \ref{Prop : cotensor morphism})}}}{\circlearrowleft \hspace*{2cm}}$} (U)}]{}
			\arrow[to path={(K) node[left, scale=1]{$\underset{\tiny{\text{\hspace*{-1.5cm}(by Proposition \ref{Prop : cotensor morphism})}}}{\circlearrowleft \hspace*{2cm}}$} (U)}]{}
		\end{tikzcd}
	\end{center}
	Then
	\begin{align*}
		(1 \otimes eq_{D})eq_{C} \theta' \big[(f \Box_{C} g) \Box_{D} h\big]&= (eq_{C} \otimes 1)eq_{D} \big[(f \Box_{C} g) \Box_{D} h\big]\quad (\text{by Proposition \ref{Prop:associativity}})\\
		&=( f \otimes g \otimes h )(eq_{C} \otimes 1)eq_{D} \quad (\text{by Proposition \ref{Prop : cotensor morphism}}) \\
		&= ( f \otimes g \otimes h )(1 \otimes eq_{D})eq_{C} \theta \quad (\text{by Proposition \ref{Prop:associativity}})\\
		&= (1 \otimes eq_{D})eq_{C}(f \Box_{C} (g \Box_{D} h)) \theta.
	\end{align*}
	Since $(1 \otimes eq_D)eq_C$ is a monomorphism, the result holds.
	\epf 
	\section{Takeuchi Cocategories and braided cogroupoids}
	Let $\VV = (\VV, \otimes, I$) be a monoidal category. In this section, we first present the notion of a categorical $\VV$-coalgebra and, when $\VV$ is braided, the notions of $\VV$-cocategory and $\VV$-cogroupoid. We then show that these structures induce appropriate equivalences between the associated comodule categories. This generalizes the results of \cite{MR3285340}.
	\subsection{Categorical \texorpdfstring{$\VV$}{V}-coalgebras} 
	\bd \label{Def: Cocat}
	A \textbf{categorical} $\VV$-\textbf{coalgebra} \textbf{C} consists of: 
	\begin{itemize}
		\item a set of objects ob(\textbf{C});
		\item for any $X, Y \in$ ob(\textbf{C}), an object $\Cc(X, Y)$ of $\VV$;
		\item for any $X, Y, Z \in \ob(\Cc)$, morphisms of $\VV$
		\[\Delta^{Z}_{X, Y}: \Cc(X, Y) \to \Cc(X, Z) \otimes \Cc(Z, Y) \quad \text{and} \quad \varepsilon_{X}: \Cc(X, X) \to I,\]
		such that, for any $X, Y, Z, T \in$ ob(\textbf{C}), the following diagrams commute:
		\begin{equation} \label{asso}
			\begin{tikzcd}
				\Cc(X, Y) \ar[d, "\Delta^{T}_{X, Y}" '] \ar[r, "\Delta^{Z}_{X, Y}"] &\Cc(X, Z) \otimes \Cc(Z, Y) \ar[d, "\Delta^{T}_{X, Z} \otimes 1"]\\
				\Cc(X, T) \otimes \Cc(T, Y) \ar[r, "1 \otimes \Delta^{Z}_{T, Y}" '] &\Cc(X, T) \otimes \Cc(T, Z) \otimes \Cc(Z, Y)
			\end{tikzcd}
		\end{equation}
		\begin{equation}\label{counit}
			\begin{tikzcd}
				\Cc(X, Y)\ar[rd, equals, shift left = 1ex] \ar[d, "\Delta^{Y}_{X, Y}" '] &&\Cc(X, Y) \ar[d, "\Delta^{X}_{X, Y}"] \ar[rd, equals, shift left = 1ex]&\\
				\Cc(X, Y) \otimes \Cc(Y, Y) \ar[r, "1 \otimes \varepsilon_{Y}" '] &\Cc(X, Y) &\Cc(X, X) \otimes \Cc(X, Y) \ar[r, "\varepsilon_{X} \otimes 1" '] &\Cc(X, Y).
			\end{tikzcd}
		\end{equation}  
	\end{itemize}
	\ed

	We denote the morphisms $\Delta^{Z}_{X, Y}$ and $\varepsilon_{X}$ by
	\[\begin{tikzpicture}[scale=0.44, base align tikzpicture]
		\tiny
		\draw[above](1,0) node{$\Cc(X, Y)$};
		\draw(1,0) -- (1,-1);
		\coproduit{2.5}{-2.5}{1.5}
		\draw (-1.5,0) -- (4,0);
		\draw(-1.5,-3) -- (4,-3);
		\draw [fill=white] (1,-1.25) circle (0.8);
		\draw (1,-1.25) node {\tiny $\Delta^{Z}_{X, Y}$};
		\draw[below](-0.5,-3) node{$\Cc(X,Z)$};
		\draw[below](3,-3) node{$\Cc(Z,Y)$};
		\draw (-0.5,-2.5) -- (-0.5,-3);
		\draw (2.5,-2.5) -- (2.5,-3);
	\end{tikzpicture}
	\quad \textrm{and} \quad \begin{tikzpicture}[scale=0.44, base align tikzpicture]
		\tiny
		\draw[above](1,0) node{$\Cc(X, X)$};
		\draw(1,0) -- (1,-2);
		\draw (0,0) -- ( 2,0);
		\draw(0,-2.5) -- (2,-2.5);
		\draw[below](1,-2.5) node{$I$};
		\draw [fill=white] (1,-1) circle (0.6);
		\draw (1,-1) node {\footnotesize{$\varepsilon_{X}$}};
		\draw (1,-2) node {$\bullet$};
		\draw(3,-2) node{$,$};
	\end{tikzpicture}\]
	and the axioms \eqref{asso} and \eqref{counit} read 
	\begin{equation*}
		\begin{tikzpicture}[scale=0.44, base align tikzpicture]
			\tiny
			\draw(-1.5,0)--(6,0);
			\draw(-1.5,-5.5)--(6,-5.5);
			\draw(4.5,-2.5)--(4.5,-5.5);
			\draw(3,0)--(3,-1);
			\draw[below](5.5, -5.5) node{$\Cc(Z, Y)$};
			\draw[below](-0.5,-5.5) node{$\Cc(X, T)$};
			\draw[below](2.5,-5.5) node{$\Cc(T, Z)$};
			\coproduit{4.5}{-2.5}{1.5}
			\coproduit{2.5}{-5.5}{1.5}
			\draw [fill=white] (3,-1.25) circle (0.8);
			\draw (3,-1.25) node {\tiny $\Delta^{Z}_{X, Y}$};
			\draw[above](3.5,0) node{$\Cc(X, Y)$};
			\draw(1.5,-2.5) ..controls +(0,-0.7) and +(0,0.7).. (1,-3.5);
			\draw [fill=white] (1,-3.75) circle (0.8);
			\draw (1,-3.75) node {\tiny $\Delta^{T}_{X, Z}$};
			\draw(7,-4) node{$=$};
			\draw(8,0)--(16,0);
			\draw(8,-5.5)--(16,-5.5);
			\draw(9,-2.5)--(9,-5.5);
			\draw(10.5,0)--(10.5,-1);
			\draw[below](14.75, -5.5) node{$\Cc(Z, Y)$};
			\draw[below](9,-5.5) node{$\Cc(X, T)$};
			\draw[below](11.75,-5.5) node{$\Cc(T, Z)$};
			\coproduit{12}{-2.5}{1.5}
			\coproduit{14}{-5.5}{1.5}
			\draw[above](10.5,0) node{$\Cc(X, Y)$};
			\draw(12,-2.5) ..controls +(0,-0.7) and +(0,0.7).. (12.5,-3.5);
			\draw [fill=white] (10.5,-1.25) circle (0.8);
			\draw (10.5,-1.25) node {\tiny $\Delta^{T}_{X, Y}$};
			\draw [fill=white] (12.5,-3.75) circle (0.8);
			\draw (12.5,-3.75) node {\tiny $\Delta^{Z}_{T, Y}$};
			\draw(17,-4) node{;};
			\draw(19,-0.5)--(23,-0.5);
			\draw(19,-5)--(23,-5);
			\draw(20.5,-0.5)--(20.5,-1);
			\draw(22,-3)--(22,-5);
			\draw(19,-3)--(19,-4.25);
			\draw[below](22, -5) node{$\Cc(X, Y)$};
			\draw(19,-4.25) node{$\bullet$};
			\coproduit{22}{-3}{1.5}
			\draw [fill=white] (20.5,-1.75) circle (0.8);
			\draw (20.5,-1.75) node {\tiny $\Delta^{X}_{X, Y}$};
			\draw [fill=white] (19,-3.25) circle (0.6);
			\draw (19,-3.25) node {\footnotesize{$\varepsilon_{X}$}};
			\draw[above](20.5, -0.5) node{$\Cc(X, Y)$};
			\draw(24,-4) node{$=$};
			\draw(26,-2)--(26,-4.5);
			\draw(25,-4.5)--(27,-4.5);
			\draw[above](26, -2) node{$\Cc(X, Y)$};
			\draw[below](26, -4.5) node{$\Cc(X, Y)$};
			\draw(25,-2)--(27,-2);
			\draw(28.5,-4) node{$=$};
			\draw(29.5,-0.5)--(33.5,-0.5);
			\draw(29.5,-5)--(33.5,-5);
			\draw(31.5,-0.5)--(31.5,-1);
			\coproduit{33}{-3}{1.5}
			\draw[above](31.5, -0.5) node{$\Cc(X, Y)$};
			\draw [fill=white] (31.5,-1.75) circle (0.8);
			\draw (31.5,-1.75) node {\tiny $\Delta^{Y}_{X, Y}$};
			\draw(33,-3.5)--(33,-4.25);
			\draw [fill=white] (33,-3.25) circle (0.6);
			\draw (33,-3.25) node {\footnotesize{$\varepsilon_{Y}$}};
			\draw(33,-4.25) node{$\bullet$};
			\draw(30,-3)--(30,-5);
			\draw[below](30.5, -5) node{$\Cc(X, Y)$};
			\draw(34.5,-4) node{.};
		\end{tikzpicture}
	\end{equation*}

	If $\Cc$ is a categorical $\VV$-coalgebra and $X, Y \in \ob(\Cc)$, then it follows immediately that $\Cc(X, X)$ is a coalgebra in $\VV$, and that $\Cc(X, Y)$ is a $\Cc(X, X)$-$\Cc(Y, Y)$ bicomodule in $\VV$ via $\Delta^X_{X, Y}$ and $\Delta^{Y}_{X, Y}$ respectively. 

	\bdep \label{Prop : Deltabar}
	Let $\VV$ be a flat regular monoidal category. Let $\Cc$ be a categorical $\VV$-coalgebra, and let $X, Y, Z \in \ob(\Cc)$. Then $\Delta^{Z}_{X, Y}$ induces a morphism in $\,^{\Cc(X, X)}\VV^{\Cc(Y, Y)}$:
	\[\bar{\Delta}^{Z}_{X, Y} : \Cc(X, Y) \longrightarrow \Cc(X, Z) \Box_{\Cc(Z, Z)} \Cc(Z, Y).\]
	We say that $\Cc$ is a \textbf{Takeuchi categorical $\VV$-coalgebra} if, for every $X, Y, Z \in \ob(\Cc)$, the morphism $\bar{\Delta}^{Z}_{X, Y}$ is an isomorphism.
	\edep
	\bpf
	We consider the following equalizer
	\begin{center}
		\begin{tikzcd}
			\Cc(X, Z) \Box_{\Cc(Z, Z)} \Cc(Z, Y) \arrow[r, "eq"] & \Cc(X, Z) \otimes \Cc(Z, Y)  \arrow[dd, shift right, "\Delta^{Z}_{X, Z} \otimes 1" '] \arrow[dd, shift left, "1 \otimes \Delta^{Z}_{Z, Y}"]\\
			& \\
			\Cc(X, Y) \arrow[uu, dashed, "\bar{\Delta}^{Z}_{X, Y}"] \arrow[ruu, "\Delta^{Z}_{X, Y}"]    & \Cc(X, Z) \otimes \Cc(Z, Z) \otimes \Cc(Z, Y) 
		\end{tikzcd}
	\end{center}
	Since $(\Delta^{Z}_{X, Z} \otimes 1)\Delta^{Z}_{X, Y} = (1 \otimes \Delta^{Z}_{Z, Y})\Delta^{Z}_{X, Y}$, there exists a unique morphism $\bar{\Delta}^{Z}_{X, Y}$ fitting in the above diagram.
	
	As in Proposition \ref{Prop:funct}, we see that $\Cc(X, Z) \Box_{\Cc(Z, Z)} \Cc(Z, Y)$ is a $\Cc(X, X)$-$\Cc(Y, Y)$-bicomodule via ($\Delta^{X}_{X, Z} \Box_{\Cc(Z, Z)} 1$) and ($1 \Box_{\Cc(Z, Z)} \Delta^{Y}_{Z, Y}$). It is clear that $\Delta^{Z}_{X, Y}$ is $\Cc(X, X)$-$\Cc(Y, Y)$-bicolinear. Since the equalizer map $eq$ is a monomorphism, it follows that $\overline{\Delta}^{Z}_{X, Y}$ is also $\Cc(X, X)$-$\Cc(Y, Y)$-bicolinear.
	\epf

	When $\VV = \,_{k}\MM$, this corresponds to the notion of a Morita-Takeuchi context introduced in \cite{MR472967}. There, it is shown that the analogues of Morita equivalence hold for categories of comodules as well. In the following theorem, we will show that if $X, Y$ are objects in a Takeuchi categorical $\VV$-coalgebra, where $\VV$ is an arbitrary flat regular monoidal category, then the categories of comodules over $\Cc(X, X)$ and $\Cc(Y, Y)$ are equivalent.
	
	\bt \label{Thm : equi}
	Let $\VV$ be a flat regular monoidal category. Let $\Cc$ be a Takeuchi categorical $\VV$-coalgebra. Then for any $X, Y \in \ob(\Cc)$, we have equivalences of categories that are inverse to each other
	\begin{align*}
		\VV^{\Cc(X, X)} \cong & \, \, \VV^{\Cc(Y, Y)} \quad \quad &\VV^{\Cc(Y, Y)} \cong & \, \, \VV^{\Cc(X, X)}\\
		V \longmapsto &V \Box_{\Cc(X, X)} \Cc(X, Y) &V \longmapsto &  V \Box_{\Cc(Y, Y)} \Cc(Y, X).
	\end{align*}
	\et
	\bpf
	For any $X, Y \in \ob(\CC)$, $\CC(X, Y)$ is a $\CC(X, X)$-$\CC(Y, Y)$-bicomodule via $\Delta^{X}_{X, Y}$ and $\Delta^{Y}_{X, Y}$. Then by Proposition \ref{Prop:funct}, we have the following functors:
	\begin{align*}
		F \colon \, \VV^{\Cc(X, X)} &\longrightarrow \VV^{\Cc(Y, Y)}\\
		V &\longmapsto V \Box_{\Cc(X, X)} \Cc(X, Y)
	\end{align*}
	and
	\begin{align*}
		G \colon \, \VV^{\Cc(Y, Y)} &\longrightarrow \VV^{\Cc(X, X)}\\
		V &\longmapsto V \Box_{\Cc(Y, Y)} \Cc(Y, X).
	\end{align*}
	Let $V \in \VV^{\Cc(X, X)}$. We have an isomorphism of $\Cc(X, X)$-comodules $\theta_V \colon V \simeq GF(V)$ given by the composition
	\begin{align*}
		V &\simeq V \Box_{\Cc(X, X)} \Cc(X, X) \hspace*{4cm} (\text{by Proposition \ref{Prop : isotrivial}})\\
		&\simeq V \Box_{\Cc(X, X)} \big(\Cc(X, Y) \Box_{\Cc(Y, Y)} \Cc(Y, X)\big) \quad \big(\text{$\bar{\Delta}^{Y}_{X, X}$ is an isomorphism in $\,^{\Cc(X, X)}\VV^{\Cc(Y, Y)}$}\big)\\
		&\simeq \big(V \Box_{\Cc(X, X)} \Cc(X, Y)\big) \Box_{\Cc(Y, Y)} \Cc(Y, X)  \quad (\text{by Proposition \ref{Prop:associativity}})\\
		&= G \circ F(V).
	\end{align*}
To conclude, we must verify the naturality of $\theta$. To do so, let $f : V \longrightarrow W$ be a morphism in $\VV^{\Cc(X, X)}$. We observe the following diagram:
	\begin{center}
		\tiny
		\begin{tikzcd}[column sep= 2.1cm,row sep= 1.5cm]
			V \otimes \Cc(X, X) \ar[rrr, bend left = 12, "f \otimes 1_{\Cc(X, X)}"]& V \arrow[d, "\bar{\beta}_{V}", ""{name = D}] \arrow[r, "f"] \arrow[l, "\beta_V" ']            & W\arrow[d, "\bar{\beta}_{W}"]  \arrow[r, "\beta_W"]            & W \otimes \Cc(X, X) \\
			& V \Box_{\Cc(X, X)} \Cc(X, X)\arrow[r, "f \Box_{\Cc(X, X)} 1_{\Cc(X, X)}"] \arrow[lu, "eq"] & W \Box_{\Cc(X, X)} \Cc(X, X) \arrow[ru, "eq", ""{name = U}] &  
			\arrow[to path={(D) node[left, scale=1]{$\underset{\tiny{}}{\circlearrowleft \hspace*{-5cm}}$} (U)}]{}
		\end{tikzcd}
	\end{center}
	We have
	\begin{align*}
		eq (f \Box_{\Cc(X, X)} 1_{\Cc(X, X)}) \bar{\beta}_{V} &= (f \otimes 1_{\Cc(X, X)})eq \bar{\beta}_{V} &(\text{by Proposition \ref{Prop : forget right}})\\
		&= (f \otimes 1_{\Cc(X, X)})\beta_{V} &(\text{by Proposition \ref{Prop : isotrivial}})\\
		&= \beta_{W} f &(\text{$f$ is $\Cc(X, X)$-colinear})\\
		&= eq \bar{\beta}_{W} f  &(\text{by Proposition \ref{Prop : isotrivial}}).
	\end{align*}
	Since $eq$ is a monomorphism, we deduce that 
	\begin{equation}
		(f \Box_{\Cc(X, X)} 1_{\Cc(X, X)}) \bar{\beta}_{V} = \bar{\beta}_{W} f \label{(3.2.1)}. 
	\end{equation}
	Next, we observe the following diagram (with identity morphisms abbreviated as $1$ for simplicity):
	\begin{center}
		\tiny
		\begin{tikzcd}[column sep= 2.2cm,row sep= 1.5cm]
			V \otimes \Cc(X, X) \ar[ddd, bend right = 80, "f \otimes 1" ' ]\ar[r, "1_{V} \otimes \bar{\Delta}^{Y}_{X,X}"] &V \otimes \big(\Cc(X, Y) \Box_{\Cc(Y, Y)} \Cc(Y,X) \big)\ar[ddd, bend left = 80, "f \otimes (1 \Box_{\Cc(Y, Y)} 1)"]\\
			V \Box_{\Cc(X, X)} \Cc(X, X) \ar[d, "f \Box_{\Cc(X, X)} 1", ""name = D]\ar[u, "eq"] \ar[r, "1_{V} \Box_{\Cc(X, X)} \bar{\Delta}^{Y}_{X, X}"] &V \Box_{\Cc(X, X)} \big(\Cc(X, Y) \Box_{\Cc(Y, Y)} \Cc(Y,X)\big) \ar[u, "eq"] \ar[d, "f \Box_{\Cc(X, X)} (1 \Box_{\Cc(Y, Y)} 1)" ']\\
			W \Box_{\Cc(X, X)} \Cc(X, X) \ar[d, "eq"] \ar[r, "1_{W} \Box_{\Cc(X, X)} \bar{\Delta}^{Y}_{X, X}"] &W \Box_{\Cc(X, X)} \big(\Cc(X, Y) \Box_{\Cc(Y, Y)} \Cc(Y,X)\big)\ar[d, "eq", ""name = U]\\
			W \otimes \Cc(X, X) \ar[r, "1_{W} \otimes \bar{\Delta}^{Y}_{X, X}"]& W \otimes \big(\Cc(X, Y) \Box_{\Cc(Y, Y)} \Cc(Y,X)\big)
			\arrow[to path={(D) node[left, scale=1]{$\underset{\tiny{}}{\circlearrowleft \hspace*{-5cm}}$} (U)}]{}
		\end{tikzcd}
	\end{center}
	We have, by Proposition \ref{Prop: forget left},
	\begin{align*}
		&eq \big(1_{W} \Box_{\Cc(X, X)} \bar{\Delta}^{Y}_{X, X}\big) \big(f \Box_{\Cc(X, X)} 1_{\Cc(X, X)}\big) = \big(1_{W} \otimes \bar{\Delta}^{Y}_{X, X} \big) eq \big(f \Box_{\Cc(X, X)} 1_{\Cc(X, X)}\big)\\
		&= \big(1_{W} \otimes \bar{\Delta}^{Y}_{X, X} \big)(f \otimes 1_{\Cc(X, X)})eq \hspace*{4cm} (\text{by Proposition \ref{Prop : forget right}})\\
		&= \big(f \otimes (1_{\Cc(X, Y)} \Box_{\Cc(Y, Y)} 1_{\Cc(Y, X)})\big)\big(1_{V} \otimes \bar{\Delta}^{Y}_{X, X}\big)eq\\
		&= \big[f \otimes \big(1_{\Cc(X, Y)} \Box_{\Cc(Y, Y)} 1_{\Cc(Y, X)}\big)\big] eq \big(1_{V} \Box_{\Cc(X, X)} \bar{\Delta}^{Y}_{X, X}\big)  \quad (\text{by Proposition \ref{Prop: forget left}})\\
		&= eq \big[f \Box_{\Cc(X, X)} \big(1_{\Cc(X, Y)} \Box_{\Cc(Y, Y)} 1_{\Cc(Y, X)}\big)\big]\big(1_{V} \Box_{\Cc(X, X)} \bar{\Delta}^{Y}_{X, X}\big) \quad (\text{by Proposition \ref{Prop : forget right}}).
	\end{align*}
	Once more, the monomorphism property of $eq$ implies that
	\begin{equation}
		\scriptsize
		\big(1_{W} \Box_{\Cc(X, X)} \bar{\Delta}^{Y}_{X, X}\big) \big(f \Box_{\Cc(X, X)} 1_{\Cc(X, X)}\big) = \big[f \Box_{\Cc(X, X)} \big(1_{\Cc(X, Y)} \Box_{\Cc(Y, Y)} 1_{\Cc(Y, X)}\big)\big]\big(1_{V} \Box_{\Cc(X, X)} \bar{\Delta}^{Y}_{X, X}\big), \label{(3.2.2)}
	\end{equation}
	which means that the square in the center commutes.\\
	By Proposition \ref{Prop: Morphism asso}, we also have that the following diagram commutes:
	\begin{equation}
		\scriptsize
		\begin{tikzcd}[column sep= 2.2cm,row sep= 1.5cm]
			V \Box_{\Cc(X, X)} \big(\Cc(X, Y) \Box_{\Cc(Y, Y)} \Cc(Y,X)\big) \ar[d, "f \Box_{\Cc(X, X)} (1 \Box_{\Cc(X, X)} 1)", ""name = D]\ar[r, "\sim"] &\big(V \Box_{\Cc(X, X)} \Cc(X, Y)\big) \Box_{\Cc(Y, Y)} \Cc(Y,X) \ar[d, "(f \Box_{\Cc(X, X)} 1) \Box_{\Cc(X, X)} 1", ""name = U]\\
			W \Box_{\Cc(X, X)} \big(\Cc(X, Y) \Box_{\Cc(Y, Y)} \Cc(Y,X)\big) \ar[r, "\sim"] &\big(W \Box_{\Cc(X, X)} \Cc(X, Y)\big) \Box_{\Cc(Y, Y)} \Cc(Y,X) 
			\arrow[to path={(D) node[left, scale=1]{$\underset{\tiny{}}{\circlearrowleft \hspace*{-8cm}}$} (U)}]{}
		\end{tikzcd} \label{(3.2.3)}
	\end{equation}
	By combining \eqref{(3.2.1)}, \eqref{(3.2.2)} and \eqref{(3.2.3)}, the following diagram commutes
	\begin{equation}
		\scriptsize
		\begin{tikzcd}[column sep= 2.2cm,row sep= 1.5cm]
			V  \ar[d, "f" ']\ar[r, "\theta_V"] &GF(V) \ar[d, "f"]\\
			W  \ar[r, "\theta_{W}"] &GF(W)
		\end{tikzcd}
	\end{equation}
	and we conclude that $\theta$ is natural. Similarly, we also have $\id \simeq FG$, and thus $F$ and $G$ are inverse equivalences. 
	\epf
	\subsection{\texorpdfstring{$\VV$}{V}-cocategories} 
	Our goal here is to generalize the concept of $k$-cocategory in \cite{MR3285340} to the braided setting, taking into account the algebra structure on the tensor product.
	
	We now let $\VV$ be a flat regular braided category endowed with a braiding $c$.
	\bd
	A $\VV$-\textbf{cocategory} \textbf{C} consists of: 
	\begin{itemize}
		\item a set of objects ob(\textbf{C});
		\item for any $X, Y \in$ ob(\textbf{C}), an \textit{algebra} $\Cc(X, Y)$ in $\VV$ with product and unit denoted respectively by
		\[ m = \begin{tikzpicture}[scale=0.43, base align tikzpicture]
			\tiny
			\draw(-1,0)--(4,0);
			\draw(-1,-3)--(4,-3);
			\draw(1.5,-2)--(1.5,-3);
			\draw[above](0,0) node{$\Cc(X, Y)$};
			\draw[above](3.25,0) node{$\Cc(X, Y)$};
			\produit{0}{0}{1.5}
			\draw[below](1.5,-3) node{$\Cc(X, Y)$};
		\end{tikzpicture} 
		\qquad \text{and} \qquad u_{X, Y} = \begin{tikzpicture}[scale=0.43, base align tikzpicture]
			\tiny
			\draw(0,1)--(3,1);
			\draw(0,-3)--(3,-3);
			\draw(1.5, 0.25)--(1.5,-3);
			\draw(1.5,0.25) node{$\bullet$};
			\draw [fill=white] (1.5,-1.5) circle (0.85);
			\draw (1.5,-1.5) node {\footnotesize{$u_{X, Y}$}};
			\draw[above](1.5,1) node{$I$};
			\draw[below](1.5,-3) node{$\Cc(X, Y)$};
		\end{tikzpicture};\]
		\item for any $X, Y, Z \in \ob(\Cc)$, \textit{algebra morphisms} in $\VV$
		\[\Delta^{Z}_{X, Y}: \Cc(X, Y) \to \Cc(X, Z) \otimes_{c} \Cc(Z, Y) \quad \text{and} \quad \varepsilon_{X}: \Cc(X, X) \to I\] that satisfy the categorical coalgebra axioms \eqref{asso} and \eqref{counit}.
	\end{itemize}
	%

	%
	A $\VV$-cocategory $\CC$ is called a \textbf{Takeuchi $\VV$-cocategory} if, for every objects $X, Y, Z$ in $\Cc$, the morphism
	\begin{align*}
		\bar{\Delta}^{Z}_{X, Y} \colon \Cc(X, Y) &\longrightarrow \Cc(X, Z) \Box_{\Cc(Z, Z)} \Cc(Z, Y)
	\end{align*}
	is a $\Cc(X, X)-\Cc(Y,Y)$-bicolinear isomorphism.
	\ed
	It follows from the definition that a Takeuchi $\VV$-cocategory is, in particular, a Takeuchi categorical $\VV$-coalgebra and therefore induces an equivalence of categories. Our aim is to check that these equivalences can be made into monoidal equivalences.
	The morphism appearing in the following result can be found, for example, in \cite[Lemma 2.3]{MR2294763}; however, for the sake of completeness, we include a proof of its construction here.
	
	\begin{lemma}\label{Lemma: Tensor Compa}
		Let $\CC$ be a Takeuchi $\VV$-cocategory, and let \( V, W \in \VV^{\Cc(X, X)}\) for any $X \in \ob(\CC)$. There exists a natural isomorphism
		\[
		\widetilde{F}_{V, W} : \big(V \Box_{\Cc(X, X)} \Cc(X, Y)\big) \otimes \big(W \Box_{\Cc(X, X)} \Cc(X, Y)\big) \longrightarrow (V \otimes W) \Box_{\Cc(X, X)} \Cc(X, Y).
		\]
	\end{lemma}
	\bpf 
	\textbf{Step 1:} We first construct $	\widetilde{F}_{V, W}$. We consider the following equalizer
	\begin{center}
		\footnotesize
		\begin{tikzcd}[column sep= 2cm,row sep= 1cm]
			(V \otimes W) \Box_{\Cc(X, X)} \Cc(X, Y) \ar[r, "eq"] &V \otimes W \otimes \Cc(X, Y) \ar[r, shift left, "\beta_{V \otimes W} \otimes 1_{\Cc(X, Y)}"] \ar[r, shift right, "1_{V \otimes W} \otimes \Delta^{X}_{X, Y}" '] & V \otimes W \otimes \Cc(X, X) \otimes \Cc(X, Y).
		\end{tikzcd}
	\end{center}
	Then we define $f : (V \Box_{\Cc(X, X)} \Cc(X, Y)) \otimes (W \Box_{\Cc(X, X)} \Cc(X, Y)) \longrightarrow V \otimes W \otimes \Cc(X, Y)$ as follows: 
	\begin{center}
		\begin{tikzcd}
			(V \Box_{\Cc(X, X)} \Cc(X, Y)) \otimes (W \Box_{\Cc(X, X)} \Cc(X, Y) ) \arrow[d, "f" ']\arrow[r, "eq \otimes eq"] & V \otimes \Cc(X, Y)\otimes W \otimes \Cc(X, Y) \arrow[d, "1 \otimes c \otimes 1"] \\
			V \otimes W \otimes \Cc(X, Y)  & V \otimes W \otimes \Cc(X, Y) \otimes \Cc(X, Y)\arrow[l, "1 \otimes 1 \otimes m"]
		\end{tikzcd}
	\end{center}
	and we want to show that $(\beta_{V \otimes W} \otimes 1_{\Cc(X, Y)}) f = (1_{V \otimes W} \otimes \Delta^{X}_{X, Y}) f$. To do this, we have the following computations: 
	\[\begin{tikzpicture}[scale=0.45]
		\tiny
		\draw(-1.5,0)--(6,0);
		\draw(-1.5,-7.5)--(6,-7.5);
		\draw(1.5,0)--(1.5,-1);
		\draw(4.5,0)--(4.5,-1);
		\draw[above](0.5,0) node{$V \Box \Cc(X, Y)$};
		\draw[above](5,0) node{$W \Box \Cc(X, Y)$};
		\coproduit{2.5}{-2}{1}
		\coproduit{5.5}{-2}{1}
		\draw [fill=white] (1.5,-1) circle (0.5);
		\draw (1.5,-1) node {$\eq$};
		\draw [fill=white] (4.5,-1) circle (0.5);
		\draw (4.5,-1) node {$\eq$};
		\produit{4.5}{-3.5}{0.5}
		\draw(5.5,-2)--(5.5,-3.5);
		\draw(2.5,-3)--(2.5,-4);
		\coright{4}{-5.5}{0.5}
		\coright{1}{-5.5}{0.5}
		\smallbraid{3.5}{-2}
		\smallbraid{2}{-5.5}
		\produit{2}{-6.5}{0.5}
		\draw(4,-5.5) ..controls +(0,-0.7) and +(0,0.7).. (3,-6.5);
		\draw(3.5,-3) ..controls +(0,-0.3) and +(0,0.3).. (4.5,-3.5);
		\draw(0.5,-2) ..controls +(0,-1) and +(0,1).. (-0.5,-4);
		\draw(-1,-5.5)--(-1,-7.5);
		\draw(1,-6.5) ..controls +(0,-0.7) and +(0,0.7).. (0.5,-7.5);
		\draw(5,-4.5)--(5,-7.5);
		\draw[below](-1,-7.5) node{$V$};
		\draw[below](0.5,-7.5) node{$W$};
		\draw[below](2.5,-7.5) node{$\Cc(X, X)$};
		\draw[below](5.5,-7.5) node{$\Cc(X, Y)$};
		\draw(7, -3.5) node{$=$};
		\draw(8,0)--(15.5,0);
		\draw(8, -6.5)--(15.5,-6.5);
		\draw(10,0)--(10,-1);
		\draw(14,0)--(14,-1);
		\draw[above](9.5,0) node{$V \Box \Cc(X, Y)$};
		\draw[above](14.5,0) node{$W \Box \Cc(X, Y)$};
		\coproduit{11}{-2}{1}
		\coproduit{15}{-2}{1}
		\draw [fill=white] (10,-1) circle (0.5);
		\draw (10,-1) node { $\eq$};
		\draw [fill=white] (14,-1) circle (0.5);
		\draw (14,-1) node {$\eq$};
		\coright{14.5}{-3.5}{0.5}
		\coright{10.5}{-3.5}{0.5}
		\smallbraid{12.5}{-3.5}
		\draw(11,-2) ..controls +(0,-0.7) and +(0,0.7).. (11.5,-3.5);
		\smallbraid{13.5}{-4.5}
		\draw(14.5,-3.5) ..controls +(0,-0.7) and +(0,0.7).. (13.5,-4.5);
		\draw(15,-2)--(15,-4.5);
		\produit{13.5}{-5.5}{0.5}
		\draw(15,-4.5) ..controls +(0,-0.7) and +(0,0.7).. (14.5,-5.5);
		\smallbraid{11.5}{-4.5}
		\produit{11.5}{-5.5}{0.5}
		\draw(10.5,-3.5)--(10.5,-4.5);
		\draw(8.5,-3.5)--(8.5,-6.5);
		\draw(10.5,-5.5) ..controls +(0,-0.7) and +(0,0.7).. (9.5,-6.5);
		\draw[below](8.5,-6.5) node{$V$};
		\draw[below](9.5,-6.5) node{$W$};
		\draw[below](11.5,-6.5) node{$\Cc(X, X)$};
		\draw[below](14.5,-6.5) node{$\Cc(X, Y)$};
		\draw(16.5, -3.5) node{$=$};
		\draw[above](16.5, -3.5) node{$(*)$};
		\draw(18,0)--(27,0);
		\draw(18, -7.5)--(27,-7.5);
		\draw[above](20,0) node{$V \Box \Cc(X, Y)$};
		\draw[above](25,0) node{$W \Box \Cc(X, Y)$};
		\coproduit{21}{-2}{1}
		\coproduit{25.5}{-2}{1}
		\draw(20,0)--(20,-1);
		\draw(24.5,0)--(24.5,-1);
		\draw [fill=white] (20,-1) circle (0.5);
		\draw (20,-1) node { $\eq$};
		\draw [fill=white] (24.5,-1) circle (0.5);
		\draw (24.5,-1) node {$\eq$};
		\coproduit{22.5}{-4.5}{1.5}
		\coproduit{27}{-4.5}{1.5}
		\smallbraid{24.5}{-5.5}
		\draw [fill=white] (21,-2.75) circle (0.8);
		\draw (21,-2.75) node { $\Delta^{X}_{X,Y}$};
		\draw [fill=white] (25.5,-2.75) circle (0.8);
		\draw (25.5,-2.75) node { $\Delta^{X}_{X,Y}$};
		\smallbraid{23.5}{-4.5}
		\draw(23.5,-2)--(23.5,-4.5);
		\produit{24.5}{-6.5}{0.5}
		\draw(24,-4.5) ..controls +(0,-0.7) and +(0,0.7).. (24.5,-5.5);
		\smallbraid{22.5}{-5.5}
		\draw(19.5,-4.5) ..controls +(0,-0.7) and +(0,0.7).. (21.5,-5.5);
		\produit{22.5}{-6.5}{0.5}
		\draw(27,-4.5) ..controls +(0,-1) and +(0,1).. (25.5,-6.5);
		\draw(19,-2) ..controls +(0,-1) and +(0,1).. (18.5,-4);
		\draw(18.5,-4)--(18.5,-7.5);
		\draw(21.5,-6.5) ..controls +(0,-0.7) and +(0,0.7).. (20,-7.5);
		\draw[below](18.5,-7.5) node{$V$};
		\draw[below](20,-7.5) node{$W$};
		\draw[below](22.5,-7.5) node{$\Cc(X, X)$};
		\draw[below](25.5,-7.5) node{$\Cc(X, Y)$};
	\end{tikzpicture}\]
	\[\begin{tikzpicture}[scale=0.45]
		\tiny
		\draw(6, -5) node{$=$};
		\draw(7,0)--(17,0);
		\draw(7, -9)--(17,-9);
		\draw(9,0)--(9,-1);
		\draw(14,0)--(14,-1);
		\draw[above](10,0) node{$V \Box \Cc(X, Y)$};
		\draw[above](14.5,0) node{$W \Box \Cc(X, Y)$};
		\coproduit{10}{-2}{1}
		\coproduit{15}{-2}{1}
		\draw [fill=white] (9,-1) circle (0.5);
		\draw (9,-1) node {$\eq$};
		\draw [fill=white] (14,-1) circle (0.5);
		\draw (14,-1) node {$\eq$};
		\smallbraid{11}{-3}
		\draw(13,-2) ..controls +(0,-0.7) and +(0,0.7).. (11,-3);
		\draw(10,-2)--(10,-3);
		\coproduit{16.5}{-6.5}{1.5}
		\coproduit{12.5}{-6.5}{1.5}
		\draw [fill=white] (11,-4.75) circle (0.8);
		\draw (11,-4.75) node { $\Delta^{X}_{X,Y}$};
		\draw [fill=white] (15,-4.75) circle (0.8);
		\draw (15,-4.75) node { $\Delta^{X}_{X,Y}$};
		\draw(15,-2)--(15,-4);
		\smallbraid{13.5}{-6.5}
		\produit{13.5}{-7.5}{1}
		\produit{10.5}{-7.5}{1}
		\draw(16.5,-6.5) ..controls +(0,-0.7) and +(0,0.7).. (15.5,-7.5);
		\draw(9.5,-6.5) ..controls +(0,-0.7) and +(0,0.7).. (10.5,-7.5);
		\draw(10,-4) ..controls +(0,-0.5) and +(0,0.5).. (9,-5);
		\draw(8,-2) ..controls +(0,-1) and +(0,1).. (7.5,-4);
		\draw(7.5,-4)--(7.5,-9);
		\draw(9,-5)--(9,-9);
		\draw[below](7.5,-9) node{$V$};
		\draw[below](9,-9) node{$W$};
		\draw[below](11.5,-9) node{$\Cc(X, X)$};
		\draw[below](15,-9) node{$\Cc(X, Y)$};
		\draw(18, -5) node{$=$};
		\draw[above](18, -5) node{$(**)$};
		\draw(20,0)--(28,0);
		\draw(20, -7)--(28,-7);
		\draw(22,0)--(22,-1);
		\draw(25,0)--(25,-1);
		\draw[above](21.5,0) node{$V \Box \Cc(X, Y)$};
		\draw[above](26,0) node{$W \Box \Cc(X, Y)$};
		\coproduit{23}{-2}{1}
		\coproduit{26}{-2}{1}
		\draw [fill=white] (22,-1) circle (0.5);
		\draw (22,-1) node { $\eq$};
		\draw [fill=white] (25,-1) circle (0.5);
		\draw (25,-1) node {$\eq$};
		\smallbraid{24}{-2}
		\produit{24}{-3}{1}
		\coproduit{26.5}{-7}{1.5}
		\draw [fill=white] (25,-5.25) circle (0.8);
		\draw (25,-5.25) node {$\Delta^{X}_{X, Y}$};
		\draw(23,-3) ..controls +(0,-0.7) and +(0,0.7).. (22,-4);
		\draw(22,-4)--(22,-7);
		\draw(21,-2)--(21,-7);
		\draw(26,-2)--(26,-3);
		\draw[below](21,-7) node{$V$};
		\draw[below](22,-7) node{$W$};
		\draw[below](24,-7) node{$\Cc(X, X)$};
		\draw[below](27,-7) node{$\Cc(X, Y)$};
		\draw(28.5, -5) node{$,$};
	\end{tikzpicture}\]
	where the equality $(*)$ follows from the definition of $V \Box_{C} \Cc(X, Y)$ and $W \Box_{C} \Cc(X, Y)$ as equalizers, i.e., $(\beta_{V} \otimes 1_{\Cc(X, Y)}) eq = (1_{V} \otimes \Delta^{X}_{X, Y}) eq$ and $(\beta_{W} \otimes 1_{\Cc(X, Y)}) eq = (1_{W} \otimes \Delta^{X}_{X, Y}) eq$; the equality $(**)$ uses the fact that $\Delta^{X}_{X, Y}$ is a morphism of algebras. Thus, we obtain \[(\beta_{V \otimes W} \otimes 1_{\Cc(X, Y)}) f = (1_{V \otimes W} \otimes \Delta^{X}_{X, Y}) f,\] and there exists a unique morphism \[\footnotesize \tilde{F}_{V, W} : (V \Box_{\Cc(X, X)} \Cc(X, Y)) \otimes (W \Box_{\Cc(X, X)} \Cc(X, Y)) \longrightarrow (V \otimes W) \Box_{\Cc(X, X)} \Cc(X, Y)\] such that the following diagram commutes:
	\begin{center}
		\scriptsize
		\begin{tikzcd}[column sep= 0.4cm,row sep= 1cm]
			(V \otimes W) \Box_{\Cc(X, X)} \Cc(X, Y) \ar[r, "eq"] &V \otimes W \otimes \Cc(X, Y) \ar[r, shift left, "\beta_{V \otimes W} \otimes 1"] \ar[r,shift right, "1_{V \otimes W} \otimes \Delta^{X}_{X, Y}" '] & V \otimes W \otimes \Cc(X, X) \otimes \Cc(X, Y)\\
			& (V \Box_{C} \Cc(X, Y)) \otimes (W \Box_{\Cc(X, X)} \Cc(X, Y)) \ar[u, "f"] \ar[lu, dashed, "\tilde{F}_{V, W}"]
		\end{tikzcd}
	\end{center}
	
	\textbf{Step 2:} Our objective is to verify that, for all $V, W \in \VV^{\Cc(X, X)}$, the morphism $\widetilde{F}_{V, W}$ is indeed an isomorphism. The key point of this step is to verify that the following diagram commutes:
	\begin{equation} \label{iso5}
		\begin{tikzcd}[column sep=2cm,row sep= 1cm]
			V \otimes W \ar[d, "\theta_V \otimes \theta_{W}" '] \ar[r, "\theta_{V \otimes W}" ] &GF(V \otimes W)\\
			GF(V) \otimes CF(W) \ar[r, "\tilde{G}_{F(V), F(W)}"] &G\big(F(V) \otimes F(W)\big) \ar[u, "G(\tilde{F}_{V, W})" ']
		\end{tikzcd}
	\end{equation}
	where $F$, $G$ and $\theta$ are the functors and the natural isomorphism introduced in the proof of Theorem \ref{Thm : equi}, and $\widetilde{G}$ is defined in a similar way by
	\[\footnotesize \tilde{G}_{V', W'} : (V' \Box_{\Cc(Y, Y)} \Cc(Y, X)) \otimes (W' \Box_{\Cc(Y, Y)} \Cc(Y, X)) \longrightarrow (V' \otimes W') \Box_{\Cc(Y, Y)} \Cc(Y, X).\]
	We observe the following commutative diagram:
	\begin{center}
		\tiny
		\begin{tikzcd}[column sep= 2cm,row sep= 1.5cm]
			\big(V \Box \Cc(X, Y)\big) \Box\Cc(Y, X) \otimes \big(W \Box \Cc(X, Y)\big) \Box\Cc(Y, X) \ar[d, "eq \otimes eq" '] \ar[r, "\widetilde{G}_{F(V), F(W)}"] &\big[\big(V \Box \Cc(X, Y)\big) \otimes \big(W \Box \Cc(X, Y)\big)\big] \Box\Cc(Y, X) \ar[lddd,bend left, "eq" ']\ar[ddd, "G(\tilde{F}_{V, W})", ""{name=D}]\\
			\big(V \Box \Cc(X, Y)\big) \otimes \Cc(Y, X) \otimes \big(W \Box \Cc(X, Y)\big) \otimes\Cc(Y, X) \ar[d, "1_{V \Box \Cc(X, Y)} \otimes \, c\,  \otimes 1_{\Cc(Y, X)}" ', ""{name=U}] &\\
			\big(V \Box \Cc(X, Y)\big) \otimes \big(W \Box \Cc(X, Y)\big) \otimes \Cc(Y, X)  \otimes\Cc(Y, X) \ar[d, "1_{V \Box \Cc(X, Y)} \otimes 1_{W \Box \Cc(X, Y)}\otimes m" '] &\arrow[to path={(D) node[left, scale=1]{$\overset{(\text{Definition of $\widetilde{G}$})\hspace*{3cm}}{\circlearrowleft \hspace*{3cm}}$} (U)}]{}\\
			\big[\big(V \Box \Cc(X, Y)\big) \otimes \big(W \Box \Cc(X, Y)\big)\big] \otimes\Cc(Y, X)\ar[d, "eq \otimes eq \otimes 1_{\Cc(Y, X)}" ']&\big[(V \otimes W) \Box \Cc(X, Y)\big] \Box \Cc(Y, X) \ar[d, "eq", ""{name=K}]\\
			V \otimes \Cc(X, Y) \otimes W \otimes \Cc(X, Y)\otimes\Cc(Y, X) \ar[d, " 1_{V} \otimes c \, \otimes 1_{\Cc(X, Y)} \otimes 1_{\Cc(Y, X)}" ', ""{name=L}] &\big[(V \otimes W) \Box \Cc(X, Y)\big] \otimes \Cc(Y, X) \ar[d, "eq \otimes 1_{\Cc(Y, X)}"]
			\arrow[to path={(K) node[left, scale=1]{$\overset{(\text{Definition of $\widetilde{F}$})\hspace*{3cm}}{\circlearrowleft \hspace*{3cm}}$} (L)}]{}\\
			V \otimes W \otimes \Cc(X, Y) \otimes \Cc(X, Y) \otimes C(Y, X) \ar[r, "1_{V \otimes W} \otimes m \otimes 1 _{\Cc(Y, X)}"] &V \otimes W \otimes \Cc(X, Y) \otimes C(Y, X)
		\end{tikzcd}
	\end{center}
	This helps us to follow the computations outlined below:
	\[\begin{tikzpicture}[scale=0.45]
		\tiny
		\draw(0,0)--(8.5,0);
		\draw(0,-10.5)--(8.5,-10.5);
		\draw(1.5,0)--(1.5,-1);
		\draw(6.5,0)--(6.5,-1);
		\draw[above](1.5,0) node{$V \Box \Cc(X, X)$};
		\draw[above](7,0) node{$W \Box \Cc(X, X)$};
		\draw [fill=white] (0,-1) rectangle (3, -2);
		\draw [fill=white] (5,-1) rectangle (8, -2);
		\draw (1.5,-1.5) node {$1 \Box \bar{\Delta}^{Y}_{X, X}$};
		\draw (6.5,-1.5) node {$1 \Box \bar{\Delta}^{Y}_{X, X}$};
		\draw [fill=white] (2.1,-3) rectangle (5.9, -4);
		\draw (4.1,-3.5) node {$\widetilde{G}_{F(V), F(W)}$};
		\draw [fill=white] (2.5,-5) rectangle (5.5, -6);
		\draw (4,-5.5) node {$G(\widetilde{F}_{V, W})$};
		\draw(6.5,-2) ..controls +(0,-0.7) and +(0,0.7).. (5.9,-3.1);
		\draw(1.5,-2) ..controls +(0,-0.7) and +(0,0.7).. (2.1,-3.1);
		\draw(4,-4)--(4,-5);
		\draw(4,-6)--(4,-6.5);
		\draw [fill=white] (4,-7) circle (0.5);
		\draw (4,-7) node {$\eq$};
		\draw(4,-7.5)--(4,-8);
		\draw [fill=white] (3,-8) rectangle (5, -9);
		\draw (4,-8.5) node {$\eq \otimes 1$};
		\draw(3,-9) ..controls +(0,-0.7) and +(0,0.7).. (0.5,-10.5);
		\draw(4,-9) ..controls +(0,-0.7) and +(0,0.7).. (2.5,-10.5);
		\draw(4,-9) ..controls +(0,-0.7) and +(0,0.7).. (5,-10.5);
		\draw(5,-9) ..controls +(0,-0.7) and +(0,0.7).. (7.5,-10.5);
		\draw[below](0.5,-10.5) node{$V$};
		\draw[below](2.5,-10.5) node{$W$};
		\draw[below](4.75,-10.5) node{$\Cc(X, Y)$};
		\draw[below](7.75,-10.5) node{$\Cc(Y, X)$};
		\draw(10.5, -5) node{$=$};
		\draw(12.5,0)--(21,0);
		\draw(12.5,-10)--(21,-10);
		\draw(15,0)--(15,-1);
		\draw(19,0)--(19,-1);
		\draw[above](15,0) node{$V \Box \Cc(X, X)$};
		\draw[above](19.5,0) node{$W \Box \Cc(X, X)$};
		\draw [fill=white] (13.5,-1) rectangle (16.5, -2);
		\draw [fill=white] (17.5,-1) rectangle (20.5, -2);
		\draw (15,-1.5) node {$1 \Box \bar{\Delta}^{Y}_{X, X}$};
		\draw (19,-1.5) node {$1 \Box \bar{\Delta}^{Y}_{X, X}$};
		\draw(15,-2)--(15,-2.5);
		\draw(19,-2)--(19,-2.5);
		\coproduit{16}{-4}{1}
		\coproduit{20}{-4}{1}
		\draw [fill=white] (15,-3) circle (0.5);
		\draw (15,-3) node {$\eq$};
		\draw [fill=white] (19,-3) circle (0.5);
		\draw (19,-3) node {$\eq$};
		\draw(18,-4)--(18,-5);
		\smallbraid{18}{-5}
		\draw(16,-4) ..controls +(0,-0.7) and +(0,0.7).. (17,-5);
		\draw(17,-6)--(17,-6.5);
		\coproduit{18}{-8}{1}
		\draw [fill=white] (17,-7) circle (0.5);
		\draw (17,-7) node {$\eq$};
		\coproduit{15}{-8}{1}
		\draw [fill=white] (14,-7) circle (0.5);
		\draw (14,-7) node {$\eq$};
		\smallbraid{16}{-8}
		\draw(14,-4)--(14,-6.5);
		\draw(20,-4)--(20,-7);
		\produit{16}{-9}{0.5}
		\produit{19}{-7}{0.5}
		\draw(18,-8) ..controls +(0,-0.7) and +(0,0.7).. (17,-9);
		\draw(18,-6) ..controls +(0,-0.7) and +(0,0.7).. (19,-7);
		\draw(13,-8)--(13,-10);
		\draw(15,-9)--(15,-10);
		\draw(19.5,-8)--(19.5,-10);
		\draw[below](13,-10) node{$V$};
		\draw[below](15,-10) node{$W$};
		\draw[below](17.25,-10) node{$\Cc(X, Y)$};
		\draw[below](20.25,-10) node{$\Cc(Y, X)$};
		\draw(23, -5) node{$=$};
		\draw(25,0)--(33.5,0);
		\draw(25,-9.5)--(33.5,-9.5);
		\draw(27.5,0)--(27.5,-1);
		\draw(31.5,0)--(31.5,-1);
		\draw[above](27.5,0) node{$V \Box \Cc(X, X)$};
		\draw[above](32,0) node{$W \Box \Cc(X, X)$};
		\draw [fill=white] (26,-1) rectangle (29, -2);
		\draw [fill=white] (30,-1) rectangle (33, -2);
		\draw (27.5,-1.5) node {$1 \Box \bar{\Delta}^{Y}_{X, X}$};
		\draw (31.5,-1.5) node {$1 \Box \bar{\Delta}^{Y}_{X, X}$};
		\draw(27.5,-2)--(27.5,-2.5);
		\draw(31.5,-2)--(31.5,-2.5);
		\coproduit{28.5}{-4}{1}
		\coproduit{32.5}{-4}{1}
		\draw [fill=white] (27.5,-3) circle (0.5);
		\draw (27.5,-3) node {$\eq$};
		\draw [fill=white] (31.5,-3) circle (0.5);
		\draw (31.5,-3) node {$\eq$};
		\coproduit{31.5}{-5.5}{1}
		\draw [fill=white] (30.5,-4.5) circle (0.5);
		\draw (30.5,-4.5) node {$\eq$};
		\smallbraid{29.5}{-5.5}
		\smallbraid{30.5}{-6.5}
		\draw(28.5,-4)--(28.5,-5.5);
		\coproduit{27.5}{-5.5}{1}
		\smallbraid{28.5}{-7.5}
		\draw [fill=white] (26.5,-4.5) circle (0.5);
		\draw (26.5,-4.5) node {$\eq$};
		\draw(27.5,-5.5)--(27.5,-7.5);
		\draw(28.5,-6.5)--(28.5,-7.5);
		\draw(29.5,-7.5)--(29.5,-8.5);
		\draw(31.5,-5.5) ..controls +(0,-0.7) and +(0,0.7).. (30.5,-6.5);
		\draw(30.5,-7.5) ..controls +(0,-0.7) and +(0,0.7).. (31.5,-8.5);
		\produit{28.5}{-8.5}{0.5}
		\produit{31.5}{-8.5}{0.5}
		\draw(25.5,-5.5)--(25.5,-9.5);
		\draw(27.5,-8.5)--(27.5,-9.5);
		\draw(32.5,-4)--(32.5,-8.5);
		\draw[below](25.5,-9.5) node{$V$};
		\draw[below](27.5,-9.5) node{$W$};
		\draw[below](29.5,-9.5) node{$\Cc(X, Y)$};
		\draw[below](32.25,-9.5) node{$\Cc(Y, X)$};
	\end{tikzpicture}\]
	\[\begin{tikzpicture}[scale=0.45]
		\tiny
		\draw(-1, -7) node{$=$};
		\draw(0,0)--(8.5,0);
		\draw(0,-9.5)--(8.5,-9.5);
		\draw(1.5,0)--(1.5,-1);
		\draw(6.5,0)--(6.5,-1);
		\draw[above](1.5,0) node{$V \Box \Cc(X, X)$};
		\draw[above](7,0) node{$W \Box \Cc(X, X)$};
		\draw [fill=white] (0,-1) rectangle (3, -2);
		\draw [fill=white] (5,-1) rectangle (8, -2);
		\draw (1.5,-1.5) node {$1 \Box \bar{\Delta}^{Y}_{X, X}$};
		\draw (6.5,-1.5) node {$1 \Box \bar{\Delta}^{Y}_{X, X}$};
		\draw(1.5,-2)--(1.5,-2.5);
		\draw(6.5,-2)--(6.5,-2.5);
		\coproduit{2.5}{-4}{1}
		\coproduit{7.5}{-4}{1}
		\draw [fill=white] (1.5,-3) circle (0.5);
		\draw (1.5,-3) node {$\eq$};
		\draw [fill=white] (6.5,-3) circle (0.5);
		\draw (6.5,-3) node {$\eq$};
		\coproduit{3.5}{-5.5}{1}
		\coproduit{8.5}{-5.5}{1}
		\draw [fill=white] (2.5,-4.5) circle (0.5);
		\draw (2.5,-4.5) node {$\eq$};
		\draw [fill=white] (7.5,-4.5) circle (0.5);
		\draw (7.5,-4.5) node {$\eq$};
		\smallbraid{4.5}{-5.5}
		\smallbraid{5.5}{-6.5}
		\draw(5.5,-4) ..controls +(0,-0.7) and +(0,0.7).. (4.5,-5.5);
		\draw(6.5,-5.5) ..controls +(0,-0.7) and +(0,0.7).. (5.5,-6.5);
		\smallbraid{3.5}{-6.5}
		\produit{3.5}{-7.5}{0.5}
		\draw(1.5,-5.5) ..controls +(0,-0.7) and +(0,0.7).. (2.5,-6.5);
		\produit{6.5}{-8.5}{0.5}
		\draw(8.5,-5.5)--(8.5,-7);
		\draw(8.5,-7) ..controls +(0,-0.7) and +(0,0.7).. (7.5,-8.5);
		\draw(5.5,-7.5) ..controls +(0,-0.7) and +(0,0.7).. (6.5,-8.5);
		\draw(4,-8.5)--(4,-9.5);
		\draw(0.5,-4)--(0.5,-9.5);
		\draw(2.5,-7.5)--(2.5,-9.5);
		\draw[below](0.5,-9.5) node{$V$};
		\draw[below](2.5,-9.5) node{$W$};
		\draw[below](4.75,-9.5) node{$\Cc(X, Y)$};
		\draw[below](7.75,-9.5) node{$\Cc(Y, X)$};
		\draw(10.5, -7) node{$=$};
		\draw(12.5,0)--(21,0);
		\draw(12.5,-9.5)--(21,-9.5);
		\draw(14,0)--(14,-1);
		\draw(19,0)--(19,-1);
		\draw[above](14.5,0) node{$V \Box \Cc(X, X)$};
		\draw[above](19.5,0) node{$W \Box \Cc(X, X)$};
		\coproduit{20}{-2.5}{1}
		\coproduit{15}{-2.5}{1}
		\draw [fill=white] (19,-1.5) circle (0.5);
		\draw (19,-1.5) node {$\eq$};
		\draw [fill=white] (19,-1.5) circle (0.5);
		\draw (19,-1.5) node {$\eq$};
		\draw [fill=white] (14,-1.5) circle (0.5);
		\draw (14,-1.5) node {$\eq$};
		\coproduit{16.5}{-5}{1.5}
		\coproduit{21.5}{-5}{1.5}
		\draw [fill=white] (15,-3.25) circle (0.8);
		\draw (15,-3.25) node {$\Delta^{Y}_{X,X}$};
		\draw [fill=white] (20,-3.25) circle (0.8);
		\draw (20,-3.25) node {$\Delta^{Y}_{X,X}$};
		\smallbraid{17.5}{-5}
		\smallbraid{18.5}{-6}
		\smallbraid{16.5}{-7.5}
		\produit{16.5}{-8.5}{0.5}
		\produit{19.5}{-8.5}{0.5}
		\draw(16.5,-6)--(16.5,-7.5);
		\draw(17.5,-7)--(17.5,-8.5);
		\draw(18.5,-5)--(18.5,-6);
		\draw(21.5,-5)--(21.5,-7);
		\draw(13.5,-5) ..controls +(0,-1) and +(0,1).. (15.5,-7.5);
		\draw(18,-2.5) ..controls +(0,-0.7) and +(0,0.7).. (17.5,-5);
		\draw(21.5,-7) ..controls +(0,-0.7) and +(0,0.7).. (20.5,-8.5);
		\draw(18.5,-7) ..controls +(0,-0.7) and +(0,0.7).. (19.5,-8.5);
		\draw(15.5,-8.5) ..controls +(0,-1) and +(0,1).. (14.5,-9.5);
		\draw(13,-2.5)--(13,-9.5);
		\draw[below](13,-9.5) node{$V$};
		\draw[below](15,-9.5) node{$W$};
		\draw[below](17.25,-9.5) node{$\Cc(X, Y)$};
		\draw[below](20.25,-9.5) node{$\Cc(Y, X)$};
		\draw(23, -7) node{$,$};
	\end{tikzpicture}\]
	where the third equality is a consequence of the associativity property in Proposition \ref{Prop:associativity} and the last equality follows from Proposition \ref{Prop: forget left} and the definition of $\overline{\Delta}$ in Proposition \ref{Prop : Deltabar}.
	Then, by Proposition \ref{Prop : isotrivial},
	$(eq \otimes 1_{\Cc(Y, X)}) \circ eq \circ G(\widetilde{F}_{V,W}) \circ \widetilde{G}_{F(V), F(W)} \circ (\theta_V \otimes \theta_{W})$ is equal to
	\[\begin{tikzpicture}[scale=0.45]
		\tiny
		\draw(-1,0)--(7,0);
		\draw(-1,-8.5)--(7,-8.5);
		\draw[above](0,0) node{$V$};
		\draw[above](4.5,0) node{$W$};
		\coright{6}{-1.5}
		\coright{1.5}{-1.5}
		\coproduit{3}{-4}{1.5}{1}
		\coproduit{7.5}{-4}{1.5}{1}
		\draw [fill=white] (1.5,-2.25) circle (0.8);
		\draw (1.5,-2.25) node {$\Delta^{Y}_{X,X}$};
		\draw [fill=white] (6,-2.25) circle (0.8);
		\draw (6,-2.25) node {$\Delta^{Y}_{X,X}$};
		\smallbraid{4}{-4}
		\smallbraid{5}{-5}
		\smallbraid{2}{-6}
		\draw(4,-1.5)--(4,-4);
		\draw(4.5,-4) ..controls +(0,-0.7) and +(0,0.7).. (5,-5);
		\produit{5.5}{-7}{1}
		\produit{2}{-7}{1}
		\draw(3,-5) ..controls +(0,-0.7) and +(0,0.7).. (2,-6);
		\draw(0,-4) ..controls +(0,-0.7) and +(0,0.7).. (1,-6);
		\draw(5,-6) ..controls +(0,-0.7) and +(0,0.7).. (5.5,-7);
		\draw(-0.5,-1.5)--(-0.5,-8.5);
		\draw(7.5,-4)--(7.5,-7);
		\draw(4,-6)--(4,-7);
		\draw(1,-7)--(1,-8.5);
		\draw[below](-0.5,-8.5) node{$V$};
		\draw[below](1,-8.5) node{$W$};
		\draw[below](3, -8.5) node{\tiny$\Cc(X, Y)$};
		\draw[below](6.5, -8.5) node{\tiny$\Cc(Y, X)$};
		\draw(8.5, -5) node{$=$};
		\draw(10,0)--(18,0);
		\draw(10,-8.5)--(18,-8.5);
		\draw[above](11,0) node{$V$};
		\draw[above](15.5,0) node{$W$};
		\coright{17}{-1.5}
		\coright{12.5}{-1.5}
		\smallbraid{13.5}{-2.5}
		\coproduit{18.5}{-4}{1.5}
		\coproduit{15}{-6}{1.5}
		\draw [fill=white] (13.5,-4.25) circle (0.8);
		\draw (13.5,-4.25) node {$\Delta^{Y}_{X,X}$};
		\draw [fill=white] (17,-2.25) circle (0.8);
		\draw (17,-2.25) node {$\Delta^{Y}_{X,X}$};
		\smallbraid{16}{-6}
		\produit{16}{-7}{1}
		\produit{13}{-7}{1}
		\draw(18.5,-4)--(18.5,-6);
		\draw(18.5,-6) ..controls +(0,-0.7) and +(0,0.7).. (18,-7);
		\draw(12.5,-1.5)--(12.5,-2.5);
		\draw(15,-1.5) ..controls +(0,-0.7) and +(0,0.7).. (13.5,-2.5);
		\draw(10.5,-1.5)--(10.5,-8.5);
		\draw(12,-6) ..controls +(0,-0.7) and +(0,0.7).. (13,-7);
		\draw(15.5,-4) ..controls +(0,-0.7) and +(0,0.7).. (16,-6);
		\draw(12.5,-3.5) ..controls +(0,-0.7) and +(0,0.7).. (11.5,-5);
		\draw(11.5,-5)--(11.5,-8.5);
		\draw[below](10.5,-8.5) node{$V$};
		\draw[below](11.5,-8.5) node{$W$};
		\draw[below](14, -8.5) node{\tiny$\Cc(X, Y)$};
		\draw[below](17.5, -8.5) node{\tiny$\Cc(Y, X)$};
		\draw(20, -5) node{$=$};
		\draw(22,0)--(28.5, 0);
		\draw(22,-6.5)--(28.5,-6.5);
		\draw[above](23,0) node{$V$};
		\draw[above](26,0) node{$W$};
		\coright{27.5}{-1.5}{0.5}
		\coright{24.5}{-1.5}{0.5}
		\smallbraid{25.5}{-1.5}
		\coproduit{28}{-6.5}{1.5}
		\produit{25.5}{-2.5}{1}
		\draw(22.5,-1.5)--(22.5,-6.5);
		\draw(23.5,-3.5)--(23.5,-6.5);
		\draw(27.5,-1.5)--(27.5,-2.5);
		\draw(24.5,-2.5) ..controls +(0,-0.7) and +(0,0.7).. (23.5,-3.5);
		\draw [fill=white] (26.5,-4.75) circle (0.8);
		\draw (26.5,-4.75) node {$\Delta^{Y}_{X,X}$};
		\draw[below](22.5,-6.5) node{$V$};
		\draw[below](23.5,-6.5) node{$W$};
		\draw[below](25.5, -6.5) node{\tiny$\Cc(X, Y)$};
		\draw[below](28.5, -6.5) node{\tiny$\Cc(Y, X)$};
		\draw(29, -5) node{$.$};
	\end{tikzpicture}\]
	This means that the following diagram commutes:
	\begin{equation*}
		\tiny
		\begin{tikzcd}[column sep=0.15cm,row sep= 1cm]
			V \otimes W \ar[rd, "\cong"]\ar[d, "\cong" ', ]\ar[rr, "\beta_{V \otimes W}"] &&V \otimes W \otimes \Cc(X, X) \ar[dddddd, "(1_{V} \otimes 1_{W}) \otimes \Delta^{Y}_{X, X}" , ""{name=D}] \\
			V \Box \Cc(X, X) \otimes W \Box \Cc(X, X) \ar[dd, "1_{V} \Box \bar{\Delta}^{Y}_{X, X} \otimes 1_{W} \Box \bar{\Delta}^{Y}_{X, X} " ' , ""{name=K}]&(V \otimes W) \Box C(X, X) \ar[ru, "eq"] \ar[d, "(1_{V} \otimes 1_{W}) \Box \bar{\Delta}^{Y}_{X, X}"]&\\
			&(V \otimes W) \Box \big[\Cc(X, Y) \Box \Cc(Y, X)\big]\ar[lddddd, near start, "\simeq" ', bend left= 35,  ""{name=U}] & \arrow[to path={(D) node[left, scale=1]{$\underset{{(**)\hspace*{2cm}}}{\circlearrowleft \hspace*{2cm}}$} (U)}]{}\\
			V \Box \big(\Cc(X, Y) \Box\Cc(Y, X)\big) \otimes W \Box \big(\Cc(X, Y) \Box\Cc(Y, X)\big) \ar[dd, "\cong" '] &&\\
			&&\\
			\big(V \Box \Cc(X, Y)\big) \Box\Cc(Y, X) \otimes \big(W \Box \Cc(X, Y)\big) \Box\Cc(Y, X) \ar[d, "\tilde{G}_{F(V), F(W)}" ']&&\\
			\big[\big(V \Box \Cc(X, Y)\big) \otimes \big(W \Box \Cc(X, Y)\big)\big] \Box\Cc(Y, X) \ar[d, "G(\tilde{F}_{V,W})" ',  ""{name=D}] &&V \otimes W \otimes \Cc(X, Y) \otimes C(Y, X)\\
			\big[(V \otimes W) \Box \Cc(X, Y)\big] \Box \Cc(Y, X) \ar[r, "eq" ']&\big[(V \otimes W) \Box \Cc(X, Y)\big] \otimes \Cc(Y, X) \ar[ru, "eq \otimes 1_{\Cc(Y, X)}" '] & \arrow[to path={(U) node[left, scale=1]{$\underset{{(*)\hspace*{2cm}}}{\circlearrowleft \hspace*{2cm}}$} (K)}]{}
		\end{tikzcd}
	\end{equation*}
Inside this diagram, we can see that the diagram $(**)$ commutes by Proposition \ref{Prop: forget left}, and by the associativity property of Proposition \ref{Prop:associativity}, as follows: 
	\begin{equation} \label{iso4}
		\scriptsize
		\begin{tikzcd}[column sep=0.8cm,row sep= 1.5cm]
			(V \otimes W) \Box \Cc(X, X) \arrow[r, "eq"] \arrow[d, "1_{V} \otimes 1_{W} \otimes \bar{\Delta}^{Y}_{X, X}", ""{name = D} '] &V \otimes W \otimes \Cc(X, X)   \arrow[r, "1_{V} \otimes 1_{W} \otimes \Delta^{Y}_{X, X}"] \arrow[d, "1_{V} \otimes 1_{W} \otimes \bar{\Delta}^{Y}_{X, X}" ', ""{name = K}] & V \otimes W \otimes \Cc(X, Y) \otimes \Cc(Y, X)\\
			(V \otimes W) \Box \big(\Cc(X, Y) \Box \Cc(Y,X)\big) \arrow[rd, "\simeq", ""{name = T}] \arrow[r, "eq"]                             & (V \otimes W) \otimes \big(\Cc(X, Y) \Box \Cc(Y,X)\big)  \arrow[ru, "1_{V} \otimes 1_{W} \otimes eq" ', ""{name = U}]          & 	\arrow[to path={(D) node[left, scale=1]{$\underset{\tiny{}}{\circlearrowleft \hspace*{-6cm}}$} (U)}]{} 
			\arrow[to path={(K) node[left, scale=1]{$\underset{\tiny{}}{\circlearrowleft \hspace*{-4cm}}$} (U)}]{}\\
			& \big((V \otimes W) \Box \Cc(X, Y)\big) \Box \Cc(Y, X) \arrow[r, "eq"]          & \big((V \otimes W) \Box \Cc(X, Y)\big) \otimes \Cc(Y, X) \arrow[uu, "eq \otimes 1_{\Cc(Y, X)}" ', ""{name = L}]
			\arrow[to path={(T) node[left, scale=1]{$\underset{{(\text{Proposition \ref{Prop:associativity}})\hspace*{-7cm}}}{\circlearrowleft \hspace*{-7cm}}$} (L)}]{}
		\end{tikzcd}
	\end{equation}
	Hence the diagram $(*)$ also commutes. This shows that 
	\[\theta_{V \otimes W} = G(\widetilde{F}_{V, W}) \circ \widetilde{G}_{F(V), F(W)} \circ (\theta_V \otimes \theta_{W}).\]
	\textbf{Step 3: }Since $\theta_V, \theta_{W}$, and $\theta_{V \otimes W}$ are isomorphisms, it follows that $G(\widetilde{F}_{V, W}) \circ \widetilde{G}_{F(V), F(W)}$ is also an isomorphism, therefore there exists a morphism $\psi : GF(V \otimes W) \longrightarrow GF(V) \otimes GF(W)$ such that 
	\[\psi \circ G(\widetilde{F}_{V, W}) \circ \widetilde{G}_{F(V), F(W)} = \id\]
	hence
	\[F(\psi) \circ FG(\widetilde{F}_{V, W}) \circ F(\widetilde{G}_{F(V), F(W)} )= \id\]
	Similarly, we can also show that $F(\widetilde{G}_{F(V), F(W)}) \circ \widetilde{F}_{GF(V), GF(W)}$ is an isomorphism, Therefore, there exists a morphism $\phi : FG(F(V) \otimes F(W)) \longrightarrow FGF(V) \otimes FGF(W)$ such that \[F(\widetilde{G}_{F(V), F(W)}) \circ \widetilde{F}_{GF(V), GF(W)} \circ \phi = \id.\]
	By composing both sides with $F(\psi) \circ FG(\widetilde{F}_{V, W})$, we obtain
	\[\widetilde{F}_{GF(V), GF(W)} \circ \phi = F(\psi) \circ FG(\widetilde{F}_{V, W})\]
	and \[F(\widetilde{G}_{F(V), F(W)}) \circ F(\psi) \circ FG(\widetilde{F}_{V, W})= \id\]
	which proves that $F(\widetilde{G}_{F(V), F(W)})$ is an isomorphism. Since $F$ and $G$ are equivalences of categories, it follows that $\widetilde{G}_{F(V), F(W)}$ is also an isomorphism, and therefore, $\widetilde{F}_{V, W}$ is an isomorphism.
	\epf
	\begin{lemma} \label{Lem : monoid_compa}
		Let $\CC$ be a Takeuchi $\VV$-cocategory.
		Let $U, V, W \in \VV^{\Cc(X, X)}$ for any $X \in \ob(\Cc)$. The following diagram is commutative
		\[\tiny \begin{tikzcd}[column sep= 0.75cm,row sep= 1.5cm]
			U \Box_{\Cc(X, X)} \Cc(X, Y) \otimes V \Box_{\Cc(X, X)} \Cc(X, Y) \otimes W \Box_{\Cc(X, X)} \Cc(X, Y)\ar[r, "\widetilde{F}_{U, V} \otimes 1"] \ar[d, "1 \otimes \widetilde{F}_{V, W}" ']& (U \otimes V) \Box_{\Cc(X, X)} \Cc(X, Y) \otimes W \Box_{\Cc(X, X)} \Cc(X, Y)\ar[d, "\widetilde{F}_{U \otimes V, W}"]\\
			U \Box_{\Cc(X, X)} \Cc(X, Y) \otimes (V \otimes W) \Box_{\Cc(X, X)} \Cc(X, Y) \ar[r, "\widetilde{F}_{U, V \otimes W}"] & (U \otimes V \otimes W) \Box_{\Cc(X, X)} \Cc(X, Y).
		\end{tikzcd}
		\]
	\end{lemma}
	\bpf
	Since $eq$ is a monomorphism, it is enough to show that 
	\[eq \circ \widetilde{F}_{U \otimes V, W} \circ (\widetilde{F}_{U, V} \otimes 1) = eq \circ \widetilde{F}_{U, V \otimes W} \circ (1 \otimes \widetilde{F}_{V, W}).\]
	In fact, we have
	\[\begin{tikzpicture}[scale=0.45]
		\tiny
		\draw(0,0)--(12,0);
		\draw(2,-8)--(12,-8);
		\draw(1.5,0) ..controls +(0,-0.7) and +(0,0.7).. (2.5,-1.5);
		\draw(5.5,0) ..controls +(0,-0.7) and +(0,0.7).. (4.5,-1.5);
		\draw(10,0)--(10,-1.5);
		\draw(7.5,-4)--(7.5,-6);
		\draw[above](1.5,0) node{$U \Box \Cc(X, Y)$};
		\draw[above](6,0) node{$V \Box \Cc(X, Y)$};
		\draw[above](10.5,0) node{$W \Box \Cc(X, Y)$};
		\draw [fill=white] (2.5,-1.5) rectangle (4.5, -2.5);
		\draw (3.5,-2) node {$\widetilde{F}_{U, V}$};
		\draw [fill=white] (6,-3.5) rectangle (9, -4.5);
		\draw (7.5,-4) node {$\widetilde{F}_{U \otimes V, W}$};
		\draw(3.5,-2.5) ..controls +(0,-0.7) and +(0,0.7).. (6,-3.5);
		\draw(10,-1.5) ..controls +(0,-0.7) and +(0,0.7).. (9,-3.5);
		\draw(7.5,-6) ..controls +(0,-0.7) and +(0,0.7).. (8,-8);
		\draw(7.5,-6) ..controls +(0,-0.7) and +(0,0.7).. (4.5,-8);
		\draw(7.5,-6) ..controls +(0,-0.7) and +(0,0.7).. (6.5,-8);
		\draw(7.5,-6) ..controls +(0,-0.7) and +(0,0.7).. (10.5,-8);
		\draw [fill=white] (7.5,-6) circle (0.5);
		\draw (7.5,-6) node {$\eq$};
		\draw[below](4.5,-8) node{$U$};
		\draw[below](6.5,-8) node{$V$};
		\draw[below](8,-8) node{$W$};
		\draw[below](10.5,-8) node{$\Cc(X, Y)$};
		\draw(13, -5) node{$=$};
		\draw(14,0)--(26,0);
		\draw(14,-8)--(26,-8);
		\draw(15.5,0) ..controls +(0,-0.7) and +(0,0.7).. (16.5,-1);
		\draw(19.5,0) ..controls +(0,-0.7) and +(0,0.7).. (18.5,-1);
		\draw(24,0)--(24,-1.5);
		\draw(17.5,-2)--(17.5,-3);
		\draw[above](15.5,0) node{$U \Box \Cc(X, Y)$};
		\draw[above](20,0) node{$V \Box \Cc(X, Y)$};
		\draw[above](24.5,0) node{$W \Box \Cc(X, Y)$};
		\draw [fill=white] (16.5,-1) rectangle (18.5, -2);
		\draw (17.5,-1.5) node {$\widetilde{F}_{U, V}$};
		\draw(24,-1.5) ..controls +(0,-1) and +(0,1).. (22.5,-3);
		\coproduit{24}{-5}{1.5}
		\draw(17.5,-3) ..controls +(0,-0.7) and +(0,0.7).. (15,-5);
		\draw(17.5,-3) ..controls +(0,-0.7) and +(0,0.7).. (20,-5);
		\draw [fill=white] (17.5,-3) circle (0.5);
		\draw (17.5,-3) node {$\eq$};
		\draw [fill=white] (22.5,-3.25) circle (0.5);
		\draw (22.5,-3.25) node {$\eq$};
		\smallbraid{21}{-5}
		\produit{21}{-6}{1.5}
		\draw(24,-5)--(24,-6);
		\draw(20,-6) ..controls +(0,-1) and +(0,1).. (19,-8);
		\draw(17.5,-3.5)--(17.5,-8);
		\draw(15,-5)--(15,-8);
		\draw[below](15,-8) node{$U$};
		\draw[below](17.5,-8) node{$V$};
		\draw[below](19,-8) node{$W$};
		\draw[below](22.5,-8) node{$\Cc(X, Y)$};
	\end{tikzpicture}\]
	\[\begin{tikzpicture}[scale=0.45]
		\tiny
		\draw(0, -5) node{$=$};
		\draw(0,0)--(12,0);
		\draw(0,-8.5)--(12,-8.5);
		\draw(10,0)--(10,-1.5);
		\draw[above](1.5,0) node{$U \Box \Cc(X, Y)$};
		\draw[above](6,0) node{$V \Box \Cc(X, Y)$};
		\draw[above](10.5,0) node{$W \Box \Cc(X, Y)$};
		\draw(2,0)--(2,-1);
		\draw(6,0)--(6,-1);
		\coproduit{3}{-2.5}{1}
		\coproduit{7}{-2.5}{1}
		\coproduit{10}{-4.5}{1}
		\draw(3,-2.5) ..controls +(0,-0.7) and +(0,0.7).. (3.5,-3.5);
		\draw(5,-2.5) ..controls +(0,-0.7) and +(0,0.7).. (4.5,-3.5);
		\smallbraid{4.5}{-3.5}
		\produit{4.5}{-4.5}{1}
		\draw(7,-2.5) ..controls +(0,-1) and +(0,1).. (6.5,-4.5);
		\smallbraid{6.5}{-6}
		\produit{6.5}{-7}{1}
		\draw(1,-2.5)--(1,-8.5);
		\draw(3.5,-4.5)--(3.5,-8.5);
		\draw(5.5,-7)--(5.5,-8.5);
		\draw(10,-1.5) ..controls +(0,-1) and +(0,1).. (9,-3);
		\draw(8,-4.5) ..controls +(0,-0.7) and +(0,0.7).. (6.5,-6);
		\draw(10,-4.5) ..controls +(0,-1) and +(0,1).. (8.5,-7);
		\draw [fill=white] (2,-1.5) circle (0.5);
		\draw (2,-1.5) node {$\eq$};
		\draw [fill=white] (6,-1.5) circle (0.5);
		\draw (6,-1.5) node {$\eq$};
		\draw [fill=white] (9,-3.5) circle (0.5);
		\draw (9,-3.5) node {$\eq$};
		\draw[below](1,-8.5) node{$U$};
		\draw[below](3.5,-8.5) node{$V$};
		\draw[below](5.5,-8.5) node{$W$};
		\draw[below](8,-8.5) node{$\Cc(X, Y)$};
		\draw(13, -5) node{$=$};
		\draw(14,0)--(26,0);
		\draw(14,-8)--(26,-8);
		\draw(23.5,0)--(23.5,-1.5);
		\draw[above](15.5,0) node{$U \Box \Cc(X, Y)$};
		\draw[above](20,0) node{$V \Box \Cc(X, Y)$};
		\draw[above](24.5,0) node{$W \Box \Cc(X, Y)$};
		\draw(16,0)--(16,-1);
		\draw(19,0)--(19,-1);
		\coproduit{17}{-2.5}{1}
		\coproduit{20}{-2.5}{1}
		\coproduit{24.5}{-2.5}{1}
		\smallbraid{18}{-2.5}
		\smallbraid{22}{-3.5}
		\smallbraid{21}{-4.5}
		\produit{22}{-4.5}{1}
		\produit{21}{-6}{1}
		\draw(15,-2.5)--(15,-8);
		\draw(17,-3.5)--(17,-8);
		\draw(21,-5.5)--(21,-6);
		\draw(22,-7.5)--(22,-8);
		\draw(20,-5.5)--(20,-8);
		\draw(18,-3.5) ..controls +(0,-0.7) and +(0,0.7).. (20,-4.5);
		\draw(20,-2.5) ..controls +(0,-0.7) and +(0,0.7).. (21,-3.5);
		\draw(22.5,-2.5) ..controls +(0,-0.7) and +(0,0.7).. (22,-3.5);
		\draw(24.5,-2.5) ..controls +(0,-1) and +(0,1).. (24,-4.5);
		\draw [fill=white] (16,-1.5) circle (0.5);
		\draw (16,-1.5) node {$\eq$};
		\draw [fill=white] (19,-1.5) circle (0.5);
		\draw (19,-1.5) node {$\eq$};
		\draw [fill=white] (23.5,-1.5) circle (0.5);
		\draw (23.5,-1.5) node {$\eq$};
		\draw[below](15,-8) node{$U$};
		\draw[below](17,-8) node{$V$};
		\draw[below](20,-8) node{$W$};
		\draw[below](22.5,-8) node{$\Cc(X, Y)$};
	\end{tikzpicture}\]
	\[\begin{tikzpicture}[scale=0.45]
		\tiny
		\draw(-1.5, -15) node{$=$};
		\draw(0,-12)--(12,-12);
		\draw(0,-20)--(10,-20);
		\draw(1.5,-12)--(1.5,-13.5);
		\draw(10.5,-12) ..controls +(0,-0.7) and +(0,0.7).. (9,-13);
		\draw(5.5,-12) ..controls +(0,-0.7) and +(0,0.7).. (7,-13);
		\draw[above](1.5,-12) node{$U \Box \Cc(X, Y)$};
		\draw[above](6,-12) node{$V \Box \Cc(X, Y)$};
		\draw[above](10.5,-12) node{$W \Box \Cc(X, Y)$};
		\draw [fill=white] (7,-13) rectangle (9, -14);
		\draw (8,-13.5) node {$\widetilde{F}_{V, W}$};
		\draw [fill=white] (3,-16.5) rectangle (6, -15.5);
		\draw (4.5,-16) node {$\widetilde{F}_{U \otimes V, W}$};
		\draw(8,-14) ..controls +(0,-0.7) and +(0,0.7).. (6,-15.5);
		\draw(1.5,-13.5) ..controls +(0,-0.7) and +(0,0.7).. (3,-15.5);
		\draw(4.5,-17.5) ..controls +(0,-0.7) and +(0,0.7).. (2,-20);
		\draw(4.5,-17.5) ..controls +(0,-0.7) and +(0,0.7).. (4,-20);
		\draw(4.5,-17.5) ..controls +(0,-0.7) and +(0,0.7).. (5.5,-20);
		\draw(4.5,-17.5) ..controls +(0,-0.7) and +(0,0.7).. (8,-20);
		\draw [fill=white] (4.5,-17.5) circle (0.5);
		\draw (4.5,-17.5) node {$\eq$};
		\draw(4.5,-16.5)--(4.5,-17);
		\draw[below](2,-20) node{$U$};
		\draw[below](4,-20) node{$V$};
		\draw[below](5.5,-20) node{$W$};
		\draw[below](8,-20) node{$\Cc(X, Y)$};
		\draw(13, -18) node{$.$}; 
	\end{tikzpicture}\]
	This gives us the desired equality.
	\epf 

	By \cite[Definition~2.4.1]{EGNO}, a monoidal functor consists of a functor
	$F:\mathcal{C} \to \mathcal{D}$ together with natural isomorphisms
	\[
	\widetilde{F}_{X,Y}: F(X)\otimes F(Y)\xrightarrow{\sim} F(X\otimes Y),
	\]
	and an isomorphism $F(1_{\mathcal{C}}) \simeq 1_{\mathcal{D}}$,
	satisfying the usual coherence conditions. Hence, if $F$ is already an equivalence of categories with quasi-inverse $G:\mathcal{D}\to\mathcal{C}$, and if we are given the natural isomorphisms $\widetilde{F}$, then $F$ is automatically a monoidal equivalence. Indeed, since we are working with strict monoidal categories, we have
	\begin{align*}
		F(1_{\mathcal{C}})
		&\simeq F(1_{\mathcal{C}})\otimes 1_{\mathcal{D}} \\
		&\simeq F(1_{\mathcal{C}})\otimes FG(1_{\mathcal{D}})
		\qquad (\text{since } F \text{ is an equivalence})\\
		&\simeq F\bigl(1_{\mathcal{C}}\otimes G(1_{\mathcal{D}})\bigr)
		\qquad (\widetilde{F}_{1_{\mathcal{C}},\,G(1_{\mathcal{D}})} \text{ is an isomorphism})\\
		&\simeq FG(1_{\mathcal{D}}) \simeq 1_{\mathcal{D}}.
	\end{align*}
	
	In light of this observation, let $\Cc$ be a Takeuchi $\VV$-cocategory. Since we already have an equivalence of categories $F$ (from Theorem \ref{Thm : equi}), and Lemma~\ref{Lemma: Tensor Compa} provides the required monoidal structure $\widetilde{F}$, we obtain the following theorem:
	\bt \label{Cor : monoi equi}
	Let $\VV$ be a flat regular braided category. Let $\Cc$ be a Takeuchi $\VV$-cocategory. Then for any $X, Y \in \ob(\Cc)$, we have an equivalence of monoidal categories
	\begin{align*}
		\VV^{\Cc(X, X)} &\cong^{\otimes} \VV^{\Cc(Y, Y)}\\
		V &\mapsto V \Box_{\Cc(X, X)} \Cc(X, Y).
	\end{align*}
	\et

	\subsection{Braided cogroupoids}
	Let $\VV$ be a braided category endowed with a braiding $c$.
	\bd
	A $\VV$-\textbf{cogroupoid} \textbf{C} consists of a $\VV$-cocategory $\CC$ together with, for any $X, Y \in \ob(\Cc),$ morphisms in $\VV$
	\[S_{X, Y} : \Cc(X, Y) \longrightarrow \Cc(Y, X)\]
	such that the following diagrams commute
	\begin{equation} \label{antipode}
		\begin{tikzcd}
			\Cc(X, X) \ar[d, "\Delta^{Y}_{X, X}"]\ar[r, "\varepsilon_{X}"] &I \ar[r, "u_{Y, X}"] &\Cc(Y, X)\\
			\Cc(X, Y) \otimes \Cc(Y, X) \ar[rr, "S_{X, Y} \otimes 1"] &&\Cc(Y, X) \otimes \Cc(Y, X) \ar[u, "m"]
		\end{tikzcd}
	\end{equation}
	\[\begin{tikzcd}
		\Cc(X, X) \ar[d, "\Delta^{Y}_{X, X}"]\ar[r, "\varepsilon_{X}"] &I \ar[r, "u_{X, Y}"] &\Cc(X, Y)\\
		\Cc(X, Y) \otimes \Cc(Y, X) \ar[rr, "1 \otimes S_{Y, X}"] &&\Cc(X, Y) \otimes \Cc(X, Y) \ar[u, "m"]
	\end{tikzcd},\]
	which, in diagrammatic notation, means that
	\begin{equation}
		\begin{tikzpicture}[scale=0.43, base align tikzpicture]
			\tiny
			\draw(-1,0)--(4,0);
			\draw[above](1.5,0) node{$\Cc(X, X)$};
			\draw(1.5,0)--(1.5,-1);
			\coproduit{3}{-3}{1.5}
			\produit{0}{-4}{1.5}
			\draw [fill=white] (1.5,-1.5) circle (0.8);
			\draw (1.5,-1.5) node {\tiny $\Delta^{Y}_{X, X}$};
			\draw [fill=white] (0,-3.5) circle (0.8);
			\draw (0,-3.5) node {$S_{X, Y}$};
			\draw(3,-3)--(3,-4);
			\draw(-1,-6)--(4,-6);
			\draw[below](1.5,-6) node{$\Cc(Y, X)$};
			\draw(5.5, -3.5) node{$=$};
			\draw(7,0)--(10,0);
			\draw[above](8.5,0) node{$\Cc(X, X)$};
			\draw(8.5,0)--(8.5,-2.5);
			\draw [fill=white] (1.5,-1.5) circle (0.8);
			\draw (1.5,-1.5) node {\tiny $\Delta^{Y}_{X, X}$};
			\draw [fill=white] (8.5,-1.5) circle (0.6);
			\draw (8.5,-1.5) node {$\varepsilon_{X}$};
			\draw(8.5,-3.25)--(8.5,-6);
			\draw [fill=white] (8.5,-4.75) circle (0.85);
			\draw (8.5,-4.75) node {\footnotesize{$u_{Y, X}$}};
			\draw(7,-6)--(10,-6);
			\draw(8.5,-2.5) node{$\bullet$};
			\draw(8.5,-3.25) node{$\bullet$};
			\draw[below](8.5,-6) node{$\Cc(Y, X)$};
		\end{tikzpicture} \qquad \text{and} \qquad \begin{tikzpicture}[scale=0.43, base align tikzpicture]
			\tiny
			\draw(-1,0)--(4,0);
			\draw[above](1.5,0) node{$\Cc(X, X)$};
			\draw(1.5,0)--(1.5,-1);
			\coproduit{3}{-3}{1.5}
			\produit{0}{-4}{1.5}
			\draw [fill=white] (1.5,-1.5) circle (0.8);
			\draw (1.5,-1.5) node {\tiny $\Delta^{Y}_{X, X}$};
			\draw [fill=white] (3,-3.5) circle (0.8);
			\draw (3,-3.5) node {$S_{Y, X}$};
			\draw(0,-3)--(0,-4);
			\draw(-1,-6)--(4,-6);
			\draw[below](1.5,-6) node{$\Cc(X, Y)$};
			\draw(5.5, -3.5) node{$=$};
			\draw(7,0)--(10,0);
			\draw[above](8.5,0) node{$\Cc(X, X)$};
			\draw(8.5,0)--(8.5,-2.5);
			\draw [fill=white] (1.5,-1.5) circle (0.8);
			\draw (1.5,-1.5) node {\tiny $\Delta^{Y}_{X, X}$};
			\draw [fill=white] (8.5,-1.5) circle (0.6);
			\draw (8.5,-1.5) node {$\varepsilon_{X}$};
			\draw(8.5,-3.25)--(8.5,-6);
			\draw [fill=white] (8.5,-4.75) circle (0.85);
			\draw (8.5,-4.75) node {\footnotesize{$u_{X, Y}$}};
			\draw(7,-6)--(10,-6);
			\draw(8.5,-2.5) node{$\bullet$};
			\draw(8.5,-3.25) node{$\bullet$};
			\draw[below](8.5,-6) node{$\Cc(X, Y)$};
			\draw(11, -4.5) node{$.$};
		\end{tikzpicture} \label{antipode1}
	\end{equation}
	A $\VV$-cogroupoid $\Cc$ is said to be \textit{faithfully flat} if, for all $X, Y \in \ob(\Cc)$, the object $\Cc(X, Y)$ is faithfully flat.
	\ed
	It follows from the definition that if $\Cc$ is a $\VV$-cogroupoid and $X \in \ob(\Cc)$, then $\Cc(X, X)$ is a Hopf algebra in $\VV$ (a braided Hopf algebra).
	\brq
	When $\VV = \MM_{k}$, $\VV = \MM^{H}$ or $\VV = \YD^{H}_{H}$, where $H$ is a coquasitriangular Hopf algebra, a faithfully flat $\VV$-cogroupoid coincides with a connected cogroupoid in the sense of \cite[Definition 2.3]{MR3285340}; that is, $\Cc(X, Y)$ is a nonzero algebra for all $X, Y \in \ob(\VV)$.
	\erq

	Next, we show some basic properties of braided cogroupoids. The following result generalizes  \cite[Proposition 2.13]{MR3285340}, which states that the ``antipodes" in a braided cogroupoid are anti-morphisms of algebras and coalgebras in the associated braided category.
	
	\bp\label{Proprosition: S_X,Y}
	Let $\Cc$ be a $\VV$-cogroupoid and let $X, Y \in \ob(\Cc)$. Then
	\begin{itemize}
		\item [$(1)$]$S_{X, Y}: \Cc(X, Y) \longrightarrow \Cc(Y, X)^{\operatorname{op}}$ is an algebra morphism of $\VV$; i.e.
		\vspace*{-0.5cm}
		\begin{equation*}
			\begin{tikzpicture}[scale=0.43]
				\tiny
				\draw(0.5,0)--(6,0);
				\draw(0.5, -5)--(6,-5);
				\draw[above](2,0) node{$\Cc(X, Y)$};
				\draw[above](5.25,0) node{$\Cc(X, Y)$};
				\produit{2}{0}{1.5}
				\draw(3.5,-1.5)--(3.5,-5);
				\draw [fill=white] (3.5,-3) circle (0.8);
				\draw (3.5,-3) node {$S_{X, Y}$};
				\draw[below](3.5,-5) node{$\Cc(Y, X)$};
				\draw(7, -4) node{$=$};
				\draw(8,0)--(15,-0);
				\draw(8,-7.5)--(15,-7.5);
				\draw[above](9.5,0) node{$\Cc(X, Y)$};
				\draw[above](13,0) node{$\Cc(X, Y)$};
				\draw(12.5,-0) ..controls +(0,-0.7) and +(0,0.7).. (11.5,-1);
				\draw(9.5,-0) ..controls +(0,-0.7) and +(0,0.7).. (10.5,-1);
				\smallbraid{11.5}{-1}
				\draw(11.5,-2) ..controls +(0,-0.7) and +(0,0.7).. (12.5,-3.25);
				\draw(10.5,-2) ..controls +(0,-0.7) and +(0,0.7).. (9.5,-3.25);
				\produit{10}{-6}{1}
				\draw(9.5,-4.75) ..controls +(0,-0.7) and +(0,0.7).. (10,-6);
				\draw(12.5,-4.75) ..controls +(0,-0.7) and +(0,0.7).. (12,-6);
				\draw [fill=white] (9.5,-4) circle (0.8);
				\draw (9.5,-4) node {$S_{X, Y}$};
				\draw [fill=white] (12.5,-4) circle (0.8);
				\draw (12.5,-4) node {$S_{X, Y}$};
				\draw[below](11,-7.5) node{$\Cc(Y, X)$};
				\draw(15, -4.5) node{$.$};
			\end{tikzpicture}
		\end{equation*}

		\item [$(2)$]For any $Z \in \ob(\Cc)$, we have
		\[\begin{tikzpicture}[scale=0.43]
			\tiny
			\draw(-1,0)--(4,0);
			\draw[above](1.5,0) node{$\Cc(X, Y)$};
			\draw(1.5,0)--(1.5,-4);
			\draw [fill=white] (1.5,-2) circle (0.8);
			\draw (1.5,-2) node {$S_{X, Y}$};
			\coproduit{3}{-6}{1.5}
			\draw [fill=white] (1.5,-4.5) circle (0.8);
			\draw (1.5,-4.5) node {\tiny $\Delta^{Z}_{Y, X}$};
			\draw(-1,-6)--(4,-6);
			\draw[below](0,-6) node{$\Cc(Y, Z)$};
			\draw[below](3,-6) node{$\Cc(Z, X)$};
			\draw(5.5, -5.5) node{$=$};
			\draw(6.5, 0)--(12,0);
			\draw(6.5,-8)--(12,-8);
			\draw(9,0)--(9,-1.5);
			\draw[above](9,0) node{$\Cc(X, Y)$};
			\coproduit{10.5}{-3}{1.5}
			\draw [fill=white] (9,-1.5) circle (0.8);
			\draw (9,-1.5) node {\tiny $\Delta^{Z}_{X, Y}$};
			\draw(7.5,-3) ..controls +(0,-0.7) and +(0,0.7).. (8.5,-4);
			\draw(10.5,-3) ..controls +(0,-0.7) and +(0,0.7).. (9.5,-4);
			\smallbraid{9.5}{-4}
			\draw(8.5,-5) ..controls +(0,-0.7) and +(0,0.7).. (7.5,-6);
			\draw(9.5,-5) ..controls +(0,-0.7) and +(0,0.7).. (10.5,-6);
			\draw(7.5,-7.5)--(7.5,-8);
			\draw(10.5,-7.5)--(10.5,-8);
			\draw [fill=white] (7.5,-6.75) circle (0.8);
			\draw (7.5,-6.75) node {$S_{Z, Y}$};
			\draw [fill=white] (10.5,-6.75) circle (0.8);
			\draw (10.5,-6.75) node {$S_{X, Z}$};
			\draw[below](7.5,-8) node{$\Cc(Y, Z)$};
			\draw[below](10.5,-8) node{$\Cc(Z, X)$};
			\draw(13, -6.5) node{$.$};
		\end{tikzpicture}\]
	\end{itemize}
	\ep
	\bpf
	First, we have 
	\begin{align*}
		S_{X, Y} u_{X, Y} = m  \big(S_{X, Y} u_{X, Y} \otimes u_{X, Y}\big)
		&= m \big(S_{X, Y} \otimes 1\big)(u_{X, Y} \otimes u_{X, Y})\\
		&= m \big(S_{X, Y} \otimes 1\big)\Delta^{Z}_{X,Y}u_{X, X}\\
		&= u_{Y, X} \varepsilon_{X} u_{X, X}\\
		&= u_{Y, X}.
	\end{align*}
	We also have
	
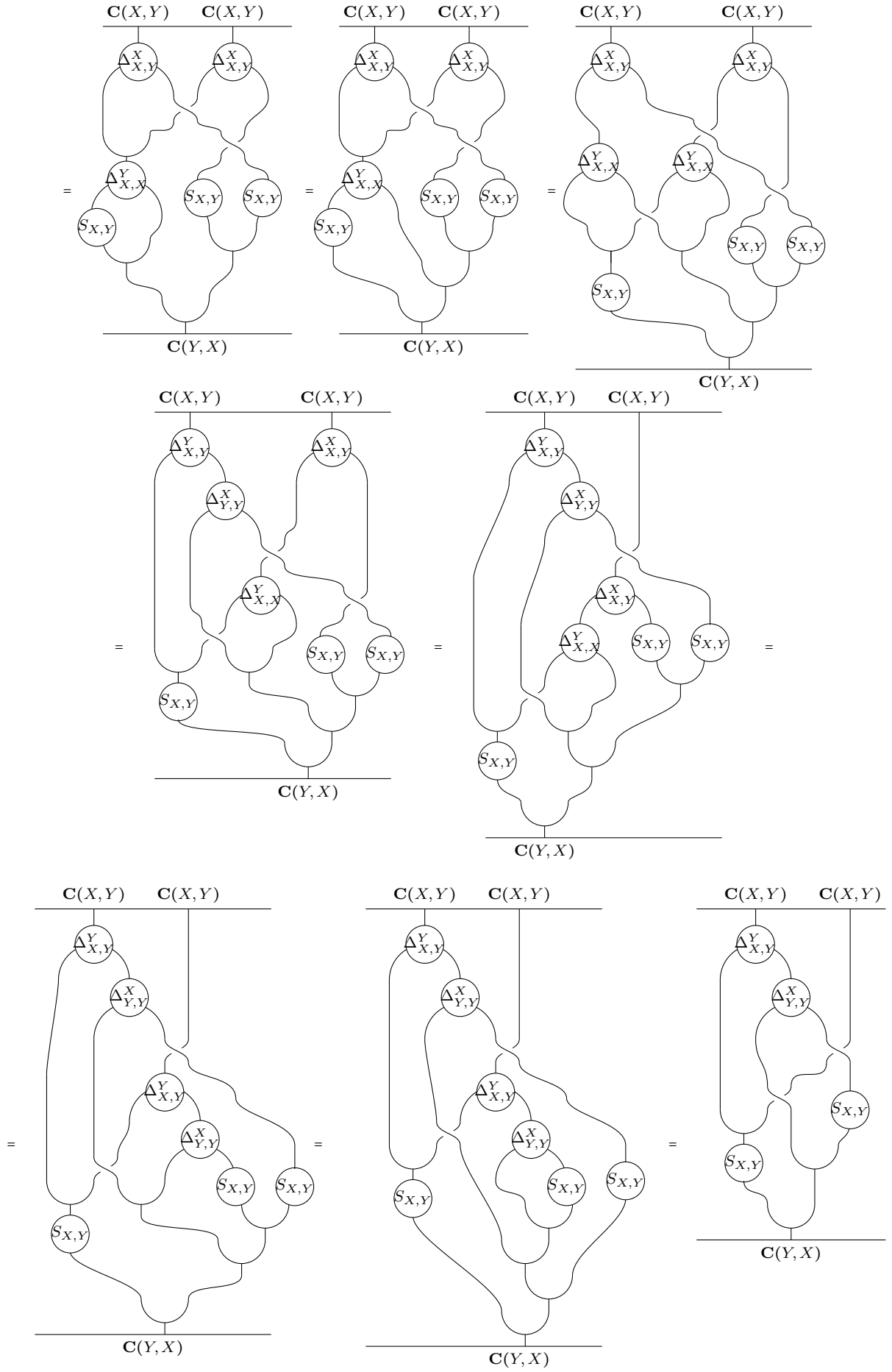
\[\begin{tikzpicture}[scale=0.45]
		\tiny
		\draw(0,0)--(6,0);
		\draw(0,-7.5)--(6,-7.5);
		\draw[above](1,0) node{$\Cc(X, Y)$};
		\draw[above](4.5,0) node{$\Cc(X, Y)$};
		\draw(4.5,-0) ..controls +(0,-0.7) and +(0,0.7).. (3.5,-1);
		\draw(1.5,-0) ..controls +(0,-0.7) and +(0,0.7).. (2.5,-1);
		\smallbraid{3.5}{-1}
		\draw(3.5,-2) ..controls +(0,-0.7) and +(0,0.7).. (4.5,-3.25);
		\draw(2.5,-2) ..controls +(0,-0.7) and +(0,0.7).. (1.5,-3.25);
		\produit{2}{-6}{1}
		\draw(1.5,-4.75) ..controls +(0,-0.7) and +(0,0.7).. (2,-6);
		\draw(4.5,-4.75) ..controls +(0,-0.7) and +(0,0.7).. (4,-6);
		\draw [fill=white] (1.5,-4) circle (0.8);
		\draw (1.5,-4) node {$S_{X, Y}$};
		\draw [fill=white] (4.5,-4) circle (0.8);
		\draw (4.5,-4) node {$S_{X, Y}$};
		\draw[below](3,-7.5) node{$\Cc(Y, X)$};
		\draw(7, -4.5) node{$=$};
		\draw(9,0)--(17,0);
		\draw(9,-10)--(17,-10);
		\draw(10.5,0)--(10.5,-1);
		\draw(14.5,0)--(14.5,-1);
		\draw[above](10.5,0) node{$\Cc(X, Y)$};
		\draw[above](14.5,0) node{$\Cc(X, Y)$};
		\coproduit{16}{-3}{1.5}
		\coproduit{12}{-3}{1.5}
		\smallbraid{13}{-3}
		\draw(9,-3)--(9, -6);
		\draw [fill=white] (14.5,-1.5) circle (0.8);
		\draw (14.5,-1.5) node {\tiny $\Delta^{X}_{X, Y}$};
		\draw [fill=white] (10.5,-1.5) circle (0.8);
		\draw (10.5,-1.5) node {\tiny $\Delta^{X}_{X, Y}$};
		\smallbraid{15}{-4.5}
		\draw(16,-3) ..controls +(0,-0.7) and +(0,0.7).. (15,-4.5);
		\draw(13,-4) ..controls +(0,-0.3) and +(0,0.3).. (14,-4.5);
		\draw(14,-5.5) ..controls +(0,-0.7) and +(0,0.7).. (13,-6.5);
		\draw(15,-5.5) ..controls +(0,-0.7) and +(0,0.7).. (16,-6.5);
		\draw(13,-7) ..controls +(0,-0.7) and +(0,0.7).. (13.5,-8.5);
		\draw(16,-7) ..controls +(0,-0.7) and +(0,0.7).. (15.5,-8.5);
		\draw [fill=white] (13.25,-7.25) circle (0.8);
		\draw (13.25,-7.25) node {$S_{X, Y}$};
		\draw [fill=white] (15.75,-7.25) circle (0.8);
		\draw (15.75,-7.25) node {$S_{X, Y}$};
		\produit{13.5}{-8.5}{1}
		\draw(12,-4) ..controls +(0,-0.7) and +(0,0.7).. (11,-5);
		\draw [fill=white] (9,-5) circle (0.6);
		\draw (9,-5) node {$\varepsilon_{X}$};
		\draw(11.5,-5.5)--(11.5, -6);
		\draw [fill=white] (11.5,-5) circle (0.6);
		\draw (11.5,-5) node {$\varepsilon_{X}$};
		\draw(9,-6) node{$\bullet$};
		\draw(11.5,-6) node{$\bullet$};
		\draw[below](15,-10) node{$\Cc(Y, X)$};
		\draw(18, -7) node{$=$};
		\draw(20,0)--(28,0);
		\draw(20, -12.5)--(28,-12.5);
		\draw(21.5,0)--(21.5,-1);
		\draw(25.5,0)--(25.5,-1);
		\draw[above](21.5,0) node{$\Cc(X, Y)$};
		\draw[above](25.5,0) node{$\Cc(X, Y)$};
		\coproduit{27}{-3}{1.5}
		\coproduit{23}{-3}{1.5}
		\smallbraid{24}{-3}
		\produit{21}{-4}{1}
		\draw(20,-3) ..controls +(0,-0.7) and +(0,0.7).. (21,-4);
		\draw [fill=white] (25.5,-1.5) circle (0.8);
		\draw (25.5,-1.5) node {\tiny $\Delta^{X}_{X, Y}$};
		\draw [fill=white] (21.5,-1.5) circle (0.8);
		\draw (21.5,-1.5) node {\tiny $\Delta^{X}_{X, Y}$};
		\smallbraid{26}{-4.5}
		\draw(22, -5.5)--(22,-7);
		\draw [fill=white] (22,-6) circle (0.6);
		\draw (22,-6) node {$\varepsilon_{X}$};
		\draw(22,-7) node{$\bullet$};
		\draw(22,-7.5)--(22, -10);
		\draw [fill=white] (22,-8.75) circle (0.75);
		\draw (22,-8.75) node {$u_{Y, X}$};
		\draw(22,-7.5) node{$\bullet$};
		\draw(27,-3) ..controls +(0,-0.7) and +(0,0.7).. (26,-4.5);
		\draw(24,-4) ..controls +(0,-0.3) and +(0,0.3).. (25,-4.5);
		\draw(25,-5.5) ..controls +(0,-0.7) and +(0,0.7).. (24,-6.5);
		\draw(26,-5.5) ..controls +(0,-0.7) and +(0,0.7).. (27,-6.5);
		\draw(24,-7) ..controls +(0,-0.7) and +(0,0.7).. (24.5,-8.5);
		\draw(27,-7) ..controls +(0,-0.7) and +(0,0.7).. (26.5,-8.5);
		\draw [fill=white] (24.25,-7.25) circle (0.8);
		\draw (24.25,-7.25) node {$S_{X, Y}$};
		\draw [fill=white] (26.75,-7.25) circle (0.8);
		\draw (26.75,-7.25) node {$S_{X, Y}$};
		\produit{24.5}{-8.5}{1}
		\produit{23}{-11}{1}
		\draw(22,-10) ..controls +(0,-0.7) and +(0,0.7).. (23,-11);
		\draw(25.5,-10) ..controls +(0,-0.7) and +(0,0.7).. (25,-11);
		\draw[below](24.5,-12.5) node{$\Cc(Y, X)$};
	\end{tikzpicture}\]
	\[\begin{tikzpicture}[scale=0.425]
		\tiny
		\draw(-1.5, -7) node{$=$};
		\draw(0,0)--(8,0);
		\draw(0, -13)--(8,-13);
		\draw(1.5,0)--(1.5,-1);
		\draw(5.5,0)--(5.5,-1);
		\draw[above](1.5,0) node{$\Cc(X, Y)$};
		\draw[above](5.5,0) node{$\Cc(X, Y)$};
		\coproduit{7}{-3}{1.5}
		\coproduit{3}{-3}{1.5}
		\smallbraid{4}{-3}
		\produit{0}{-4.5}{1}
		\draw(3,-4) ..controls +(0,-0.7) and +(0,0.3).. (2,-4.5);
		\draw [fill=white] (5.5,-1.5) circle (0.8);
		\draw (5.5,-1.5) node {\tiny $\Delta^{X}_{X, Y}$};
		\draw [fill=white] (1.5,-1.5) circle (0.8);
		\draw (1.5,-1.5) node {\tiny $\Delta^{X}_{X, Y}$};
		\smallbraid{6}{-4.5}
		\coproduit{2.5}{-8}{1.5}
		\draw(0, -3)--(0,-4.5);
		\draw [fill=white] (1,-6.5) circle (0.8);
		\draw (1,-6.5) node {\tiny $\Delta^{Y}_{X, X}$};
		\produit{0}{-9}{1}
		\draw(-0.5,-8) ..controls +(0,-0.7) and +(0,0.3).. (0,-9);
		\draw [fill=white] (-0.25,-8.5) circle (0.8);
		\draw (-0.25,-8.5) node {$S_{X, Y}$};
		\draw(2.5,-8) ..controls +(0,-0.7) and +(0,0.3).. (2,-9);
		\draw(7,-3) ..controls +(0,-0.7) and +(0,0.7).. (6,-4.5);
		\draw(4,-4) ..controls +(0,-0.3) and +(0,0.3).. (5,-4.5);
		\draw(5,-5.5) ..controls +(0,-0.7) and +(0,0.7).. (4,-6.5);
		\draw(6,-5.5) ..controls +(0,-0.7) and +(0,0.7).. (7,-6.5);
		\draw(4,-7) ..controls +(0,-0.7) and +(0,0.7).. (4.5,-8.5);
		\draw(7,-7) ..controls +(0,-0.7) and +(0,0.7).. (6.5,-8.5);
		\draw [fill=white] (4.25,-7.25) circle (0.8);
		\draw (4.25,-7.25) node {$S_{X, Y}$};
		\draw [fill=white] (6.75,-7.25) circle (0.8);
		\draw (6.75,-7.25) node {$S_{X, Y}$};
		\produit{4.5}{-8.5}{1}
		\produit{2.5}{-11.5}{1}
		\draw(1,-10.5) ..controls +(0,-0.7) and +(0,0.7).. (2.5,-11.5);
		\draw(5.5,-10) ..controls +(0,-0.7) and +(0,0.7).. ( 4.5,-11.5);
		\draw[below](4,-13) node{$\Cc(Y, X)$};
		\draw(8.75, -7) node{$=$};
		\draw(10,0)--(18,0);
		\draw(10, -13)--(18,-13);
		\draw(11.5,0)--(11.5,-1);
		\draw(15.5,0)--(15.5,-1);
		\draw[above](11.5,0) node{$\Cc(X, Y)$};
		\draw[above](15.5,0) node{$\Cc(X, Y)$};
		\coproduit{17}{-3}{1.5}
		\coproduit{13}{-3}{1.5}
		\smallbraid{14}{-3}
		\produit{10}{-4.5}{1}
		\draw(13,-4) ..controls +(0,-0.7) and +(0,0.3).. (12,-4.5);
		\draw [fill=white] (15.5,-1.5) circle (0.8);
		\draw (15.5,-1.5) node {\tiny $\Delta^{X}_{X, Y}$};
		\draw [fill=white] (11.5,-1.5) circle (0.8);
		\draw (11.5,-1.5) node {\tiny $\Delta^{X}_{X, Y}$};
		\smallbraid{16}{-4.5}
		\coproduit{12.5}{-8}{1.5}
		\draw(10, -3)--(10,-4.5);
		\draw [fill=white] (11,-6.5) circle (0.8);
		\draw (11,-6.5) node {\tiny $\Delta^{Y}_{X, X}$};
		\produit{13.5}{-10}{1}
		\draw(12.5,-8) ..controls +(0,-0.7) and +(0,0.3).. (13.5,-10);
		\draw(9.5,-8) ..controls +(0,-0.7) and +(0,0.3).. (10,-9);
		\draw(9.75, -9)--(9.75,-10);
		\draw [fill=white] (9.75,-8.5) circle (0.8);
		\draw (9.75,-8.5) node {$S_{X, Y}$};
		\draw(17,-3) ..controls +(0,-0.7) and +(0,0.7).. (16,-4.5);
		\draw(14,-4) ..controls +(0,-0.3) and +(0,0.3).. (15,-4.5);
		\draw(15,-5.5) ..controls +(0,-0.7) and +(0,0.7).. (14,-6.5);
		\draw(16,-5.5) ..controls +(0,-0.7) and +(0,0.7).. (17,-6.5);
		\draw(14,-7) ..controls +(0,-0.7) and +(0,0.7).. (14.5,-8.5);
		\draw(17,-7) ..controls +(0,-0.7) and +(0,0.7).. (16.5,-8.5);
		\draw [fill=white] (14.25,-7.25) circle (0.8);
		\draw (14.25,-7.25) node {$S_{X, Y}$};
		\draw [fill=white] (16.75,-7.25) circle (0.8);
		\draw (16.75,-7.25) node {$S_{X, Y}$};
		\produit{14.5}{-8.5}{1}
		\produit{12.5}{-11.5}{1}
		\draw(9.75,-10) ..controls +(0,-0.7) and +(0,0.7).. (12.5,-11.5);
		\draw[below](14,-13) node{$\Cc(Y, X)$};
		\draw(19, -7) node{$=$};
		\draw(20,0)--(30,0);
		\draw(20, -14.5)--(30,-14.5);
		\draw(21.5,0)--(21.5,-1);
		\draw(27.5,0)--(27.5,-1);
		\draw[above](21.5,0) node{$\Cc(X, Y)$};
		\draw[above](27.5,0) node{$\Cc(X, Y)$};
		\coproduit{29}{-3}{1.5}
		\coproduit{23}{-3}{1.5}
		\smallbraid{26}{-4}
		\draw(20,-3) ..controls +(0,-0.7) and +(0,0.3).. (21,-4.5);
		\draw(23,-3) ..controls +(0,-0.7) and +(0,0.3).. (25,-4);
		\draw(21, -4.5)--(21,-5);
		\draw(26, -3)--(26,-4);
		\draw [fill=white] (27.5,-1.5) circle (0.8);
		\draw (27.5,-1.5) node {\tiny $\Delta^{X}_{X, Y}$};
		\draw [fill=white] (21.5,-1.5) circle (0.8);
		\draw (21.5,-1.5) node {\tiny $\Delta^{X}_{X, Y}$};
		\smallbraid{29}{-6.5}
		\coproduit{26.5}{-7.5}{1.5}
		\coproduit{22.5}{-7.5}{1.5}
		\draw(29, -3)--(29,-6.5);
		\draw(26,-5) ..controls +(0,-0.7) and +(0,0.7).. (28,-6.5);
		\draw [fill=white] (21,-5.75) circle (0.8);
		\draw (21,-5.75) node {\tiny $\Delta^{Y}_{X, X}$};
		\draw [fill=white] (25,-5.75) circle (0.8);
		\draw (25,-5.75) node {\tiny $\Delta^{Y}_{X, X}$};
		\smallbraid{23.5}{-7.5}
		\produit{23.5}{-8.5}{1}
		\produit{20.5}{-8.5}{1}
		\draw(26.5,-7.5) ..controls +(0,-0.7) and +(0,0.7).. (25.5,-8.5);
		\draw(19.5,-7.5) ..controls +(0,-0.7) and +(0,0.3).. (20.5,-8.5);
		\draw(21.5, -9.5)--(21.5,-10.5);
		\draw [fill=white] (21.5,-11.25) circle (0.8);
		\draw (21.5,-11.25) node {$S_{X, Y}$};
		\draw(21.5,-12) ..controls +(0,-0.7) and +(0,0.7).. (25.5,-13);
		\draw(28,-7.5) ..controls +(0,-0.7) and +(0,0.7).. (27,-8.5);
		\draw(29,-7.5) ..controls +(0,-0.7) and +(0,0.7).. (30,-8.5);
		\draw [fill=white] (27.25,-9.25) circle (0.8);
		\draw (27.25,-9.25) node {$S_{X, Y}$};
		\draw [fill=white] (29.75,-9.25) circle (0.8);
		\draw (29.75,-9.25) node {$S_{X, Y}$};
		\produit{27.5}{-10}{1}
		\produit{26.5}{-11.5}{1}
		\produit{25.5}{-13}{1}
		\draw(24.5,-10) ..controls +(0,-0.7) and +(0,0.7).. (26.5,-11.5);
		\draw[below](26.5,-14.5) node{$\Cc(Y, X)$};
	\end{tikzpicture}\]
	\vspace*{-1cm}
	\[\begin{tikzpicture}[scale=0.425]
		\tiny
		\draw(-1.5, -10) node{$=$};
		\draw(0,0)--(10,0);
		\draw(0, -15.5)--(10,-15.5);
		\draw(1.5,0)--(1.5,-1);
		\draw(7.5,0)--(7.5,-1);
		\draw[above](1.5,0) node{$\Cc(X, Y)$};
		\draw[above](7.5,0) node{$\Cc(X, Y)$};
		\coproduit{9}{-3}{1.5}
		\coproduit{3}{-3}{1.5}
		\draw [fill=white] (7.5,-1.5) circle (0.8);
		\draw (7.5,-1.5) node {\tiny $\Delta^{X}_{X, Y}$};
		\draw [fill=white] (1.5,-1.5) circle (0.8);
		\draw (1.5,-1.5) node {\tiny $\Delta^{Y}_{X, Y}$};
		\coproduit{4.5}{-5.5}{1.5}
		\draw [fill=white] (3,-3.75) circle (0.8);
		\draw (3,-3.75) node {\tiny $\Delta^{X}_{Y, Y}$};
		\smallbraid{5.5}{-5.5}
		\coproduit{6}{-9}{1.5}
		\draw(4.5, -6.5)--(4.5,-7);
		\draw [fill=white] (4.5,-7.75) circle (0.8);
		\draw (4.5,-7.75) node {\tiny $\Delta^{Y}_{X, X}$};
		\smallbraid{3}{-9}
		\draw(1.5, -5.5)--(1.5,-8);
		\draw(1.5,-8) ..controls +(0,-0.7) and +(0,0.7).. (2,-9);
		\produit{3}{-10}{1}
		\produit{0}{-10}{1}
		\draw(0, -3)--(0,-10);
		\draw(6,-9) ..controls +(0,-0.7) and +(0,0.3).. (5,-10);
		\draw(1, -11)--(1,-12);
		\draw [fill=white] (1,-12.25) circle (0.8);
		\draw (1,-12.25) node {$S_{X, Y}$};
		\draw(5.5,-6.5) ..controls +(0,-0.7) and +(0,0.7).. (8,-7.5);
		\smallbraid{9}{-7.5}
		\draw(8,-8.5) ..controls +(0,-0.7) and +(0,0.7).. (7,-9.5);
		\draw(9,-8.5) ..controls +(0,-0.7) and +(0,0.7).. (10,-9.5);
		\draw [fill=white] (7.25,-10.25) circle (0.8);
		\draw (7.25,-10.25) node {$S_{X, Y}$};
		\draw [fill=white] (9.75,-10.25) circle (0.8);
		\draw (9.75,-10.25) node {$S_{X, Y}$};
		\produit{7.5}{-11}{1}
		\produit{6.5}{-12.5}{1}
		\produit{5.5}{-14}{1}
		\draw(9, -3)--(9,-7.5);
		\draw(6, -3)--(6,-4.5);
		\draw(6,-4.5) ..controls +(0,-0.7) and +(0,0.7).. (5.5,-5.5);
		\draw(4,-11.5) ..controls +(0,-0.7) and +(0,0.7).. (6.5,-12.5);
		\draw(1,-13) ..controls +(0,-0.7) and +(0,0.7).. (5.5,-14);
		\draw[below](6.5,-15.5) node{$\Cc(Y, X)$};
		\draw(12, -10) node{$=$};
		\draw(14,0)--(24,0);
		\draw(14, -18)--(24,-18);
		\draw(16.5,0)--(16.5,-1);
		\draw(20.5,0)--(20.5,-5.5);
		\draw[above](16.5,0) node{$\Cc(X, Y)$};
		\draw[above](20.5,0) node{$\Cc(X, Y)$};
		\coproduit{18}{-3}{1.5}
		\draw [fill=white] (16.5,-1.5) circle (0.8);
		\draw (16.5,-1.5) node {\tiny $\Delta^{Y}_{X, Y}$};
		\coproduit{19.5}{-5.5}{1.5}
		\draw [fill=white] (18,-3.75) circle (0.8);
		\draw (18,-3.75) node {\tiny $\Delta^{X}_{Y, Y}$};
		\smallbraid{20.5}{-5.5}
		\coproduit{21}{-9}{1.5}
		\draw(19.5, -6.5)--(19.5,-7);
		\draw [fill=white] (19.5,-7.75) circle (0.8);
		\draw (19.5,-7.75) node {\tiny $\Delta^{X}_{X, Y}$};
		\coproduit{19.5}{-11.5}{1.5}
		\draw [fill=white] (18,-9.75) circle (0.8);
		\draw (18,-9.75) node {\tiny $\Delta^{Y}_{X, X}$};
		\smallbraid{16.5}{-11.5}
		\produit{16.5}{-12.5}{1}
		\draw [fill=white] (21,-9.75) circle (0.8);
		\draw (21,-9.75) node {$S_{X, Y}$};
		\draw [fill=white] (23.5,-9.75) circle (0.8);
		\draw (23.5,-9.75) node {$S_{X, Y}$};
		\produit{21.25}{-10.5}{1}
		\draw(20.5,-6.5) ..controls +(0,-0.7) and +(0,1).. (23.5,-8);
		\draw (23.5,-8)-- (23.5,-9);
		\draw(19.5,-11.5) ..controls +(0,-0.7) and +(0,0.3).. (18.5,-12.5);
		\produit{13.5}{-12.5}{1}
		\draw(16.5,-5.5) ..controls +(0,-0.7) and +(0,0.7).. (15.5,-9);
		\draw(15,-3) ..controls +(0,-0.7) and +(0,0.7).. (13.5,-7);
		\draw(15.5, -9)--(15.5,-11.5);
		\draw(13.5, -7)--(13.5,-12.5);
		\draw(22.25,-12) ..controls +(0,-0.7) and +(0,0.7).. (19.5,-14);
		\draw(14.5,-15.5) ..controls +(0,-0.7) and +(0,0.3).. (15.5,-16.5);
		\draw(18.5,-15.5) ..controls +(0,-0.7) and +(0,0.3).. (17.5,-16.5);
		\draw [fill=white] (14.5,-14.75) circle (0.8);
		\draw (14.5,-14.75) node {$S_{X, Y}$};
		\produit{17.5}{-14}{1}
		\produit{15.5}{-16.5}{1}
		\draw[below](16.5,-18) node{$\Cc(Y, X)$};
		\draw(26, -10) node{$=$};
	\end{tikzpicture}\]
	\[\begin{tikzpicture}[scale=0.425]
		\tiny
		\draw(-1, -10) node{$=$};
		\draw(0,0)--(10,0);
		\draw(0, -18)--(10,-18);
		\draw(2.5,0)--(2.5,-1);
		\draw(6.5,0)--(6.5,-5.5);
		\draw[above](2.5,0) node{$\Cc(X, Y)$};
		\draw[above](6.5,0) node{$\Cc(X, Y)$};
		\coproduit{4}{-3}{1.5}
		\draw [fill=white] (2.5,-1.5) circle (0.8);
		\draw (2.5,-1.5) node {\tiny $\Delta^{Y}_{X, Y}$};
		\coproduit{5.5}{-5.5}{1.5}
		\draw [fill=white] (4,-3.75) circle (0.8);
		\draw (4,-3.75) node {\tiny $\Delta^{X}_{Y, Y}$};
		\smallbraid{6.5}{-5.5}
		\coproduit{7}{-9}{1.5}
		\draw(5.5, -6.5)--(5.5,-7);
		\draw [fill=white] (5.5,-7.75) circle (0.8);
		\draw (5.5,-7.75) node {\tiny $\Delta^{Y}_{X, Y}$};
		\coproduit{8.5}{-11.5}{1.5}
		\draw [fill=white] (7,-9.75) circle (0.8);
		\draw (7,-9.75) node {\tiny $\Delta^{X}_{Y, Y}$};
		\smallbraid{3.5}{-10.5}
		\produit{3.5}{-11.5}{1}
		\draw [fill=white] (8.5,-11.75) circle (0.8);
		\draw (8.5,-11.75) node {$S_{X, Y}$};
		\draw [fill=white] (11,-11.75) circle (0.8);
		\draw (11,-11.75) node {$S_{X, Y}$};
		\produit{8.75}{-12.5}{1}
		\draw(6.5,-6.5) ..controls +(0,-0.7) and +(0,1).. (8.5,-8);
		\draw(8.5,-8) ..controls +(0,-0.7) and +(0,1).. (11,-10);
		\draw (11,-10)-- (11,-11);
		\produit{0.5}{-11.5}{1}
		\draw(4,-9) ..controls +(0,-0.7) and +(0,0.7).. (3.5,-10.5);
		\draw(1,-3) ..controls +(0,-0.7) and +(0,0.7).. (0.5,-7);
		\draw(0.5, -7)--(0.5,-11.5);
		\draw(2.5, -5.5)--(2.5,-10.5);
		\draw(4.5,-13) ..controls +(0,-0.7) and +(0,0.7).. (7.75,-14);
		\draw(1.5,-14.5) ..controls +(0,-0.7) and +(0,0.7).. (4.5,-16.5);
		\draw(8.75,-15.5) ..controls +(0,-0.7) and +(0,0.7).. (6.5,-16.5);
		\draw [fill=white] (1.5,-13.75) circle (0.8);
		\draw (1.5,-13.75) node {$S_{X, Y}$};
		\produit{7.75}{-14}{1}
		\produit{4.5}{-16.5}{1}
		\draw[below](5,-18) node{$\Cc(Y, X)$};
		\draw(12, -10) node{$=$};
		\draw(14,0)--(24,0);
		\draw(14, -18.5)--(24,-18.5);
		\draw(16.5,0)--(16.5,-1);
		\draw(20.5,0)--(20.5,-5.5);
		\draw[above](16.5,0) node{$\Cc(X, Y)$};
		\draw[above](20.5,0) node{$\Cc(X, Y)$};
		\coproduit{18}{-3}{1.5}
		\draw [fill=white] (16.5,-1.5) circle (0.8);
		\draw (16.5,-1.5) node {\tiny $\Delta^{Y}_{X, Y}$};
		\coproduit{19.5}{-5.5}{1.5}
		\draw [fill=white] (18,-3.75) circle (0.8);
		\draw (18,-3.75) node {\tiny $\Delta^{X}_{Y, Y}$};
		\smallbraid{20.5}{-5.5}
		\coproduit{21}{-9}{1.5}
		\draw(19.5, -6.5)--(19.5,-7);
		\draw [fill=white] (19.5,-7.75) circle (0.8);
		\draw (19.5,-7.75) node {\tiny $\Delta^{Y}_{X, Y}$};
		\coproduit{22.5}{-11.5}{1.5}
		\draw [fill=white] (21,-9.75) circle (0.8);
		\draw (21,-9.75) node {\tiny $\Delta^{X}_{Y, Y}$};
		\smallbraid{18}{-9}
		\produit{19.75}{-14}{1}
		\draw [fill=white] (22.5,-11.75) circle (0.8);
		\draw (22.5,-11.75) node {$S_{X, Y}$};
		\draw(25,-12.25) ..controls +(0,-1) and +(0,1).. (22.75,-15.5);
		\draw (25,-10)-- (25,-11);
		\draw [fill=white] (25,-11.5) circle (0.8);
		\draw (25,-11.5) node {$S_{X, Y}$};
		\produit{20.75}{-12.5}{1}
		\draw(18,-10) ..controls +(0,-1) and +(0,1).. (19.75,-14);
		\draw(19.5,-11.5) ..controls +(0,-0.7) and +(0,0.7).. (20.75,-12.5);
		\draw(22.5,-8) ..controls +(0,-0.7) and +(0,1).. (25,-10);
		\produit{15}{-10}{1}
		\draw(16.5,-5.5) ..controls +(0,-0.7) and +(0,0.7).. (17,-9);
		\draw(15, -3)--(15,-10);
		\draw(20.5,-6.5) ..controls +(0,-0.7) and +(0,0.7).. (22.5,-8);
		\draw(16,-13) ..controls +(0,-1) and +(0,1).. (19.75,-17);
		\draw [fill=white] (16,-12.25) circle (0.8);
		\draw (16,-12.25) node {$S_{X, Y}$};
		\produit{20.75}{-15.5}{1}
		\produit{19.75}{-17}{1}
		\draw[below](20.5,-18.5) node{$\Cc(Y, X)$};
		\draw(27, -10) node{$=$};
		\draw(28,0)--(36,0);
		\draw(28, -14)--(36,-14);
		\draw(30.5,0)--(30.5,-1);
		\draw(34.5,0)--(34.5,-5.5);
		\draw[above](30.5,0) node{$\Cc(X, Y)$};
		\draw[above](34.5,0) node{$\Cc(X, Y)$};
		\coproduit{32}{-3}{1.5}
		\draw [fill=white] (30.5,-1.5) circle (0.8);
		\draw (30.5,-1.5) node {\tiny $\Delta^{Y}_{X, Y}$};
		\coproduit{33.5}{-5.5}{1.5}
		\draw [fill=white] (32,-3.75) circle (0.8);
		\draw (32,-3.75) node {\tiny $\Delta^{X}_{Y, Y}$};
		\smallbraid{34.5}{-5.5}
		\smallbraid{32}{-7.5}
		\produit{29}{-8.5}{1}
		\draw(30.5,-5.5) ..controls +(0,-1) and +(0,1).. (31,-7.5);
		\draw(33.5,-6.5) ..controls +(0,-1) and +(0,1).. (32,-7.5);
		\draw(34.5,-6.5)--(34.5,-8);
		\draw(34.5,-9) ..controls +(0,-0.7) and +(0,0.3).. (34,-10);
		\draw [fill=white] (34.5,-8.5) circle (0.8);
		\draw (34.5,-8.5) node {$S_{X, Y}$};
		\draw(32,-8.5)--(32,-10);
		\draw(29,-3)--(29,-8.5);
		\produit{32}{-10}{1}
		\draw [fill=white] (30,-10.75) circle (0.8);
		\draw (30,-10.75) node {$S_{X, Y}$};
		\produit{31}{-12.5}{1}
		\draw(33,-11.5)--(33,-12.5);
		\draw(30,-11.5) ..controls +(0,-0.7) and +(0,0.7).. (31,-12.5);
		\draw[below](32,-14) node{$\Cc(Y, X)$};
	\end{tikzpicture}\]
	\[\begin{tikzpicture}[scale=0.45]
		\tiny
		\draw(-1, -7) node{$=$};
		\draw(0,0)--(6,0);
		\draw(0, -10)--(6,-10);
		\draw(2.5,0)--(2.5,-1);
		\draw(5,0)--(5,-3);
		\draw[above](2,0) node{$\Cc(X, Y)$};
		\draw[above](5.25,0) node{$\Cc(X, Y)$};
		\coproduit{4}{-3}{1.5}
		\draw [fill=white] (2.5,-1.5) circle (0.8);
		\draw (2.5,-1.5) node {\tiny $\Delta^{Y}_{X, Y}$};
		\smallbraid{5}{-3}
		\draw(5,-4)--(5,-6.25);
		\draw [fill=white] (5,-5) circle (0.6);
		\draw (5,-5) node {\tiny $\varepsilon_{Y}$};
		\draw(5,-6.25) node{$\bullet$};
		\draw(5,-7) node{$\bullet$};
		\draw(5,-7)--(5,-7.5);
		\produit{1}{-5}{1}
		\draw(4,-4) ..controls +(0,-0.7) and +(0,0.7).. (3,-5);
		\draw(2,-8)--(2,-8.5);
		\draw(1,-3)--(1,-5);
		\draw [fill=white] (2,-7.25) circle (0.8);
		\draw (2,-7.25) node {$S_{X, Y}$};
		\produit{2}{-8.5}{1}
		\draw(5,-7.5) ..controls +(0,-0.7) and +(0,0.7).. (4,-8.5);
		\draw[below](3.5,-10) node{$\Cc(Y, X)$};
		\draw(7, -7) node{$=$};
		\draw(9,0)--(15,0);
		\draw(9, -7.5)--(15,-7.5);
		\draw(11,0)--(11,-1);
		\draw(13.5,0)--(13.5,-1);
		\draw[above](11,0) node{$\Cc(X, Y)$};
		\draw[above](14.25,0) node{$\Cc(X, Y)$};
		\produit{11}{-1}{1.25}
		\draw(11.5,-6.5)--(11.5,-7.5);
		\draw [fill=white] (11.5,-5.75) circle (0.8);
		\draw (11.5,-5.75) node {$S_{X, Y}$};
		\draw(12.25,-2.75) ..controls +(0,-0.7) and +(0,0.7).. (11.5,-5);
		\draw[below](11.5,-7.5) node{$\Cc(Y, X)$};
		\draw(16, -7) node{$.$};
	\end{tikzpicture}\]
	which proves $(1)$. The proof of $(2)$ follows by applying the same diagrammatic reasoning as above, with the diagrams inverted.
	\epf
	\bp \label{Prop : Takeuchi}
	Let \(\VV\) be a flat regular braided category. Let $\Cc$ be a faithfully flat $\VV$-cogroupoid. Then $\Cc$ is a Takeuchi $\VV$-cocategory.
	\ep

	\bpf
	By definition, it suffices to show that, for any $X, Y, Z \in \ob(\Cc)$, the map \[\bar{\Delta}^{Z}_{X, Y}: \Cc(X, Y) \longrightarrow \Cc(X, Z) \Box_{\Cc(Z, Z)} \Cc(Z, Y)\] is an isomorphism of $\VV$.
	To do so, we will check that 
	\[1_{\Cc(Z, X)} \otimes \bar{\Delta}^{Z}_{X, Y} : \Cc(Z, X) \otimes \Cc(X, Y) \longrightarrow \Cc(Z, X) \otimes \Cc(X, Z) \Box_{\Cc(Z, Z)} \Cc(Z, Y)\] is an isomorphism.
	We define the morphism $$f \colon \Cc(Z, X) \otimes \Cc(X, Z) \Box_{\Cc(Z, Z)} \Cc(Z, Y) \longrightarrow \Cc(Z, X) \otimes \Cc(X, Y)$$ as below
	\begin{center}
		\footnotesize
		\begin{tikzcd}[column sep= 2cm,row sep= 1cm]
			\Cc(Z, X) \otimes \Cc(X, Z) \Box_{\Cc(Z, Z)} \Cc(Z, Y) \arrow[r, "f"]\arrow[d, "1 \otimes eq" ']  &   \Cc(Z, X) \otimes \Cc(X, Y)\\
			\Cc(Z, X) \otimes \Cc(X, Z) \otimes \Cc(Z, Y) \ar[d, "1 \otimes 1 \otimes \Delta^{X}_{Z, Y}" ',  ""{name=U}] &\Cc(Z, X) \otimes \Cc(Z, X) \otimes \Cc(X, Y)\arrow[u, "m \otimes 1", ""{name=D}] \\
			\Cc(Z, X) \otimes \Cc(X, Z) \otimes \Cc(Z, X) \otimes \Cc(X, Y) \ar[r, "1 \otimes S_{X, Z} \otimes 1 \otimes 1"] &\Cc(Z, X) \otimes \Cc(Z, X) \otimes \Cc(Z, X) \otimes \Cc(X, Y) \ar[u, "1 \otimes m \otimes 1"]
			\arrow[to path={(D) node[left, scale=1]{$\underset{\tiny{\hspace*{2cm}}}{\circlearrowleft \hspace*{4cm}}$} (U)}]{}
		\end{tikzcd}
	\end{center}
	
	We first have
	\[\begin{tikzpicture}[scale=0.45]
		\tiny
		\draw(-1.5,0)--(8,0);
		\draw(-1.5,-10.5)--(8,-10.5);
		\draw(0,0)--(0,-7);
		\draw(5,0)--(5,-3.5);
		\draw[above](0,0) node{$\Cc(Z, X)$};
		\draw[above](5,0) node{$\Cc(X, Y)$};
		\draw [fill=white] (5,-2) circle (0.8);
		\draw (5,-2) node {\tiny $\bar{\Delta}^{Z}_{X, Y}$};
		\coproduit{6}{-5}{1}
		\draw [fill=white] (5,-4) circle (0.5);
		\draw (5,-4) node {$\eq$};
		\coproduit{7.5}{-7.5}{1.5}
		\draw [fill=white] (6,-5.75) circle (0.8);
		\draw (6,-5.75) node {\tiny $\Delta^{X}_{Z, Y}$};
		\draw(4,-5) ..controls +(0,-0.7) and +(0,0.7).. (2.5,-6);
		\draw [fill=white] (2.5,-6.75) circle (0.75);
		\draw (2.5,-6.75) node {$S_{X, Z}$};
		\produit{2.5}{-7.5}{1}
		\produit{1.5}{-9}{1}
		\draw(0,-7) ..controls +(0,-0.7) and +(0,0.7).. (1.5,-9);
		\draw(7.5,-7.5)--(7.5,-10.5);
		\draw[below](2.5,-10.5) node{$\Cc(Z, X)$};
		\draw[below](7,-10.5) node{$\Cc(X, Y)$};
		\draw(9, -5.5) node{$=$};
		\draw(10,0)--(18,0);
		\draw(10,-9.25)--(18,-9.25);
		\draw(10.5,0)--(10.5,-6);
		\draw(14.5,0)--(14.5,-2.5);
		\draw[above](11,0) node{$\Cc(Z, X)$};
		\draw[above](14.5,0) node{$\Cc(X, Y)$};
		\coproduit{16}{-3.75}{1.5}
		\draw [fill=white] (14.5,-2) circle (0.8);
		\draw (14.5,-2) node {\tiny $\Delta^{Z}_{X, Y}$};
		\coproduit{17.5}{-6.25}{1.5}
		\draw [fill=white] (16,-4.5) circle (0.8);
		\draw (16,-4.5) node {\tiny $\Delta^{X}_{Z, Y}$};
		\draw(13,-3.75) ..controls +(0,-0.7) and +(0,0.7).. (12.5,-5.5);
		\draw [fill=white] (12.5,-5.5) circle (0.75);
		\draw (12.5,-5.5) node {$S_{X, Z}$};
		\produit{12.5}{-6.25}{1}
		\produit{11.5}{-7.75}{1}
		\draw(10.5,-6) ..controls +(0,-0.7) and +(0,0.7).. (11.5,-7.75);
		\draw(17.5,-6.25)--(17.5,-9.25);
		\draw[below](13,-9.25) node{$\Cc(Z, X)$};
		\draw[below](17,-9.25) node{$\Cc(X, Y)$};
		\draw(19, -5.5) node{$=$};
		\draw(20,0)--(28,0);
		\draw(20,-10.25)--(28,-10.25);
		\draw(20.5,0)--(20.5,-6.5);
		\draw(26,0)--(26,-2);
		\draw[above](21,0) node{$\Cc(Z, X)$};
		\draw[above](26,0) node{$\Cc(X, Y)$};
		\coproduit{27.5}{-3.75}{1.5}
		\draw [fill=white] (26,-2) circle (0.8);
		\draw (26,-2) node {\tiny $\Delta^{X}_{X, Y}$};
		\coproduit{26}{-6.25}{1.5}
		\draw [fill=white] (24.5,-4.5) circle (0.8);
		\draw (24.5,-4.5) node {\tiny $\Delta^{Z}_{X, X}$};
		\draw(26,-6.25) ..controls +(0,-0.7) and +(0,0.7).. (25,-7.25);
		\draw [fill=white] (23,-6.5) circle (0.75);
		\draw (23,-6.5) node {$S_{X, Z}$};
		\produit{23}{-7.25}{1}
		\produit{22}{-8.75}{1}
		\draw(20.5,-6.5) ..controls +(0,-0.7) and +(0,0.7).. (22,-8.75);
		\draw(27.5,-3.75)--(27.5,-10.25);
		\draw[below](23,-10.25) node{$\Cc(Z, X)$};
		\draw[below](27,-10.25) node{$\Cc(X, Y)$};
		\draw(28.5, -5.5) node{$=$};
		\draw(29.5,0)--(35,0);
		\draw(29.5,-10)--(35,-10);
		\draw(30,0)--(30,-6.5);
		\draw(33,0)--(33,-2);
		\draw[above](30.25,0) node{$\Cc(Z, X)$};
		\draw[above](33.5,0) node{$\Cc(X, Y)$};
		\coproduit{34.5}{-3.75}{1.5}
		\draw [fill=white] (33,-2) circle (0.8);
		\draw (33,-2) node {\tiny $\Delta^{X}_{X, Y}$};
		\draw(31.5,-5)--(31.5,-5.65);
		\draw [fill=white] (31.5,-4.25) circle (0.7);
		\draw (31.5,-4.25) node {$\varepsilon_{X}$};
		\draw (31.5,-5.65) node {$\bullet$};
		\draw (31.5,-6.25) node {$\bullet$};
		\draw(31.5,-6.25)--(31.5,-6.85);
		\draw [fill=white] (31.5,-7.65) circle (0.8);
		\draw (31.5,-7.65) node {\tiny $u_{Z,X}$};
		\produit{29.5}{-8.5}{1}
		\draw(30,-6.5) ..controls +(0,-0.7) and +(0,0.7).. (29.5,-8.5);
		\draw(34.5,-3.75)--(34.5,-10);
		\draw[below](30.5,-10) node{$\Cc(Z, X)$};
		\draw[below](34.5,-10) node{$\Cc(X, Y)$};
	\end{tikzpicture}\]
	which means
	\begin{align*}
		f \circ (1 \otimes \bar{\Delta}^{Z}_{X, Y}) &= \big(m \otimes 1\big)\big(1 \otimes m \otimes 1\big)\big(1 \otimes S_{X, Z} \otimes 1_{2}\big)\big(1_{2} \otimes \Delta^{X}_{Z, Y}\big)\big(1 \otimes eq\big)\big(1 \otimes \bar{\Delta}^{Z}_{X, Y}\big)\\
		&= \big(m \otimes 1\big)\big(1 \otimes m \otimes 1\big)\big(1 \otimes S_{X, Z} \otimes 1_{2}\big)\big(1_{2} \otimes \Delta^{X}_{Z, Y}\big)\big(1 \otimes \Delta^{Z}_{X, Y}\big)\\
		&= \big(m \otimes 1\big)\big(1 \otimes m \otimes 1\big)\big(1 \otimes S_{X, Z} \otimes 1_{2}\big)\big(1 \otimes \Delta^{Z}_{X, X} \otimes 1\big)\big(1 \otimes \Delta^{X}_{X, Y}\big) &(\text{by \eqref{asso}})\\
		&=  \big(m \otimes 1\big)(1 \otimes u \otimes 1)(1 \otimes \varepsilon_{X} \otimes 1)\big(1 \otimes \Delta^{X}_{X, Y}\big) &\text{(by \eqref{antipode})}\\
		&= \id_{\Cc(Z, X) \otimes \Cc(X, Y)} &(\text{by \eqref{counit}}).
	\end{align*}
	On the other hand,
	
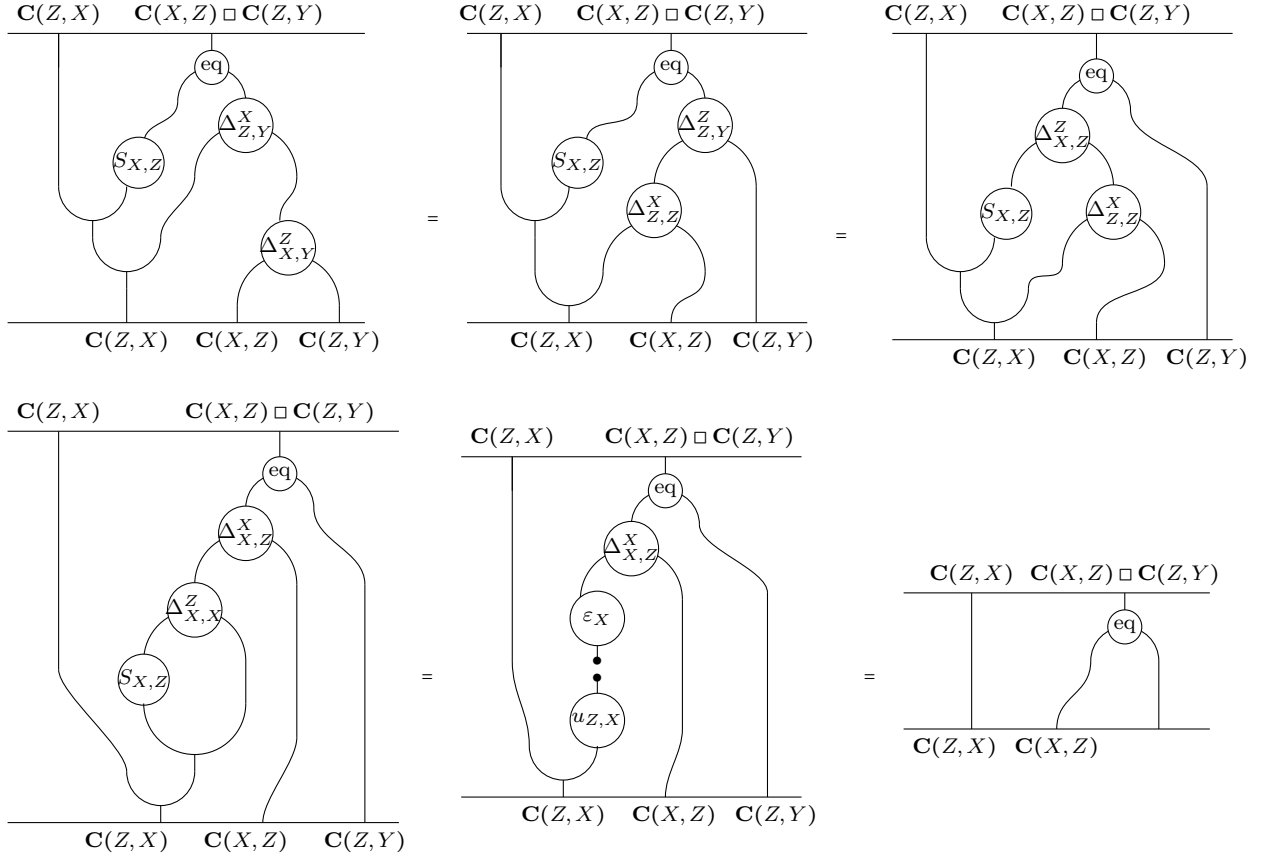
\[\begin{tikzpicture}[scale=0.45]
		\tiny
		\draw(-1.5,0)--(9,0);
		\draw(-1.5,-8.5)--(9,-8.5);
		\draw(0,0)--(0,-1);
		\draw(4.5,0)--(4.5,-1);
		\draw[above](0,0) node{$\Cc(Z, X)$};
		\draw[above](5,0) node{$\Cc(X, Z) \Box \Cc(Z, Y)$};
		\coproduit{5.5}{-2}{1}
		\draw [fill=white] (4.5,-1) circle (0.5);
		\coproduit{7}{-4.2}{1.5}
		\draw [fill=white] (5.5,-2.7) circle (0.8);
		\draw (5.5,-2.7) node {\tiny $\Delta^{X}_{Z, Y}$};
		\draw (4.5,-1) node {$\eq$};
		\draw(3.5,-2) ..controls +(0,-0.7) and +(0,0.7).. (2.5,-3.25);
		\draw [fill=white] (2.35,-3.8) circle (0.75);
		\draw (2.35,-3.8) node {$S_{X, Z}$};
		\produit{0}{-4.5}{1}
		\draw(0,0)--(0,-4.5);
		\produit{1}{-6}{1}
		\draw(4,-4.2) ..controls +(0,-0.7) and +(0,0.7).. (3,-6);
		\coproduit{8.25}{-8}{1.5}
		\draw [fill=white] (6.75,-6.3) circle (0.8);
		\draw (6.75,-6.3) node {\tiny $\Delta^{Z}_{X, Y}$};
		\draw(7,-4.2) ..controls +(0,-0.7) and +(0,0.7).. (6.5,-5.5);
		\draw(2,-7.5)--(2,-8.5);
		\draw(5.25,-8)--(5.25,-8.5);
		\draw(8.25,-8)--(8.25,-8.5);
		\draw[below](2,-8.5) node{$\Cc(Z, X)$};
		\draw[below](5.25,-8.5) node{$\Cc(X, Z)$};
		\draw[below](8.25,-8.5) node{$\Cc(Z, Y)$};
		\draw(11, -5.5) node{$=$};
		\draw(12,0)--(22,0);
		\draw(12,-8.5)--(22,-8.5);
		\draw(13,0)--(13,-1);
		\draw(18,0)--(18,-1);
		\draw[above](13,0) node{$\Cc(Z, X)$};
		\draw[above](18,0) node{$\Cc(X, Z) \Box \Cc(Z, Y)$};
		\coproduit{19}{-2}{1}
		\draw [fill=white] (18,-1) circle (0.5);
		\draw (18,-1) node {$\eq$};
		\coproduit{20.5}{-4.5}{1.5}
		\draw [fill=white] (19,-2.7) circle (0.8);
		\draw (19,-2.7) node {\tiny $\Delta^{Z}_{Z, Y}$};
		\draw(17,-2) ..controls +(0,-1) and +(0,1).. (15.5,-3.25);
		\draw [fill=white] (15.25,-3.8) circle (0.75);
		\draw (15.25,-3.8) node {$S_{X, Z}$};
		\produit{13}{-4.5}{1}
		\draw(13,0)--(13,-4.5);
		\produit{14}{-7}{1}
		\coproduit{19}{-7}{1.5}
		\draw [fill=white] (17.5,-5.2) circle (0.8);
		\draw (17.5,-5.2) node {\tiny $\Delta^{X}_{Z, Z}$};
		\draw(14,-6)--(14,-7);
		\draw(20.5,-4.5)--(20.5,-8.5);
		\draw(19,-7) ..controls +(0,-1) and +(0,1).. (18,-8.5);
		\draw[below](14.5,-8.5) node{$\Cc(Z, X)$};
		\draw[below](18, -8.5) node{$\Cc(X, Z)$};
		\draw[below](21,-8.5) node{$\Cc(Z, Y)$};
		\draw(23, -6) node{$=$};
		\draw(24.5,0)--(34.5,0);
		\draw(24.5,-9)--(34.5,-9);
		\draw(25.5, 0)--(25.5,-6);
		\draw(30.5,0)--(30.5,-1);
		\draw[above](25.5,0) node{$\Cc(Z, X)$};
		\draw[above](30.5,0) node{$\Cc(X, Z) \Box \Cc(Z, Y)$};
		\coproduit{31.5}{-2.25}{1}
		\draw [fill=white] (30.5,-1.25) circle (0.5);
		\draw (30.5,-1.25) node {$\eq$};
		\draw(31.5,-2.25) ..controls +(0,-1) and +(0,1).. (33.75,-4.5);
		\coproduit{31}{-4.5}{1.5}
		\draw [fill=white] (29.5,-3) circle (0.8);
		\draw (29.5,-3) node {\tiny $\Delta^{Z}_{X, Z}$};
		
		\draw [fill=white] (27.85,-5.3) circle (0.75);
		\draw (27.85,-5.3) node {$S_{X, Z}$};
		\produit{25.5}{-6}{1}
		\produit{26.5}{-7.5}{1}
		\draw(29.5,-6.75) ..controls +(0,-0.7) and +(0,0.7).. (28.5,-7.5);
		\coproduit{32.5}{-6.75}{1.5}
		\draw [fill=white] (31,-5.25) circle (0.8);
		\draw (31,-5.25) node {\tiny $\Delta^{X}_{Z, Z}$};
		\draw(32.5,-6.75) ..controls +(0,-1) and +(0,1).. (30.5,-8.5);
		\draw(30.5,-8.5)--(30.5,-9);
		\draw(33.75,-4.5)--(33.75,-9);
		\draw[below](27.5,-9) node{$\Cc(Z, X)$};
		\draw[below](30.75,-9) node{$\Cc(X, Z)$};
		\draw[below](33.75,-9) node{$\Cc(Z, Y)$};
	\end{tikzpicture}\]
	\[\begin{tikzpicture}[scale=0.45, base align tikzpicture]
		\tiny
		\draw(0,-12)--(11.5 ,-12);
		\draw(0,-23.5)--(11.5,-23.5);
		\draw(1.5,-12)--(1.5,-19);
		\draw(8,-12)--(8,-13);
		\draw[above](1.5,-12) node{$\Cc(Z, X)$};
		\draw[above](8,-12) node{$\Cc(X, Z) \Box \Cc(Z, Y)$};
		\coproduit{9}{-14.25}{1}
		\draw [fill=white] (8,-13.25) circle (0.5);
		\draw (8,-13.25) node {$\eq$};
		\draw(9,-14.25) ..controls +(0,-1) and +(0,1).. (10.5,-16.5);
		\coproduit{7}{-18.75}{1.5}
		\coproduit{8.5}{-16.5}{1.5}
		\draw [fill=white] (7,-15) circle (0.8);
		\draw (7,-15) node {\tiny $\Delta^{X}_{X, Z}$};
		\draw(1.5,-19) ..controls +(0,-0.7) and +(0,0.7).. (3.5,-22);
		\draw [fill=white] (4,-19.3) circle (0.75);
		\draw (4,-19.3) node {$S_{X, Z}$};
		\produit{3.5}{-22}{1}
		\produit{4}{-20}{1.5}
		\draw [fill=white] (5.5,-17.25) circle (0.8);
		\draw (5.5,-17.25) node {\tiny $\Delta^{Z}_{X, X}$};
		\draw(7,-18.75)--(7,-20);
		\draw(10.5,-16.5)--(10.5,-23.5);
		\draw(8.5,-16.5)--(8.5,-21);
		\draw(8.5,-21) ..controls +(0,-0.7) and +(0,0.7).. (7.5,-23.5);
		\draw[below](3.5,-23.5) node{$\Cc(Z, X)$};
		\draw[below](7,-23.5) node{$\Cc(X, Z)$};
		\draw[below](10.5,-23.5) node{$\Cc(Z, Y)$};
	\end{tikzpicture}
	\begin{tikzpicture}[scale=0.45, base align tikzpicture]
		\tiny
		\draw(-4, -6.5) node{$=$};
		\draw(-3,0)--(7,0);
		\draw(-3,-10)--(7,-10);
		\draw(-1.5,0)--(-1.5,-1);
		\draw(3,0)--(3,-1);
		\draw[above](-1.5,0) node{$\Cc(Z, X)$};
		\draw[above](4,0) node{$\Cc(X, Z) \Box \Cc(Z, Y)$};
		\coproduit{4}{-2}{1}
		\draw [fill=white] (3,-1) circle (0.5);
		\draw (3,-1) node {$\eq$};
		\coproduit{3.5}{-4.2}{1.5}
		\draw [fill=white] (2,-2.7) circle (0.8);
		\draw (2,-2.7) node {\tiny $\Delta^{X}_{X, Z}$};
		\draw(-1.5,0)--(-1.5,-6);
		\draw(-1.5,-6) ..controls +(0,-1) and +(0,1).. (-1,-8.5);
		\draw(1,-5)--(1,-6);
		\draw [fill=white] (1,-4.75) circle (0.8);
		\draw (1,-4.75) node {$\varepsilon_{X}$};
		\draw (1,-6) node {$\bullet$};
		\draw (1,-6.5) node {$\bullet$};
		\draw(1,-6.5)--(1,-7);
		\draw [fill=white] (1,-7.75) circle (0.8);
		\draw (1,-7.75) node {$u_{Z,X}$};
		\produit{-1}{-8.5}{1}
		\draw(4,-2) ..controls +(0,-0.7) and +(0,0.7).. (6,-4);
		\draw(3.5,-4.2)--(3.5,-8);
		\draw(3.5,-8) ..controls +(0,-0.7) and +(0,0.7).. (3,-10);
		\draw(6,-4)--(6,-10);
		\draw[below](0,-10) node{$\Cc(Z, X)$};
		\draw[below](3.25,-10) node{$\Cc(X, Z)$};
		\draw[below](6.25,-10) node{$\Cc(Z, Y)$};
		\draw(9, -6.5) node{$=$};
		\draw(10,-4)--(19,-4);
		\draw(10,-8)--(19,-8);
		\draw(12,-4)--(12,-8);
		\draw(16.5,-4)--(16.5,-5);
		\draw[above](12,-4) node{$\Cc(Z, X)$};
		\draw[above](16.5,-4) node{$\Cc(X, Z) \Box \Cc(Z, Y)$};
		\coproduit{17.5}{-6}{1}
		\draw [fill=white] (16.5,-5) circle (0.5);
		\draw (16.5,-5) node {$\eq$};
		\draw(15.5,-6) ..controls +(0,-1) and +(0,1).. (14.5,-8);
		\draw(17.5,-6)--(17.5,-8);
		\draw[below](11.5,-8) node{$\Cc(Z, X)$};
		\draw[below](14.5,-8) node{$\Cc(X, Z)$};
		\draw[below](27.5,-8) node{$\Cc(Z, Y)$};
	\end{tikzpicture}\]
	which means
	\begin{align*}
		&(1 \otimes \Delta^{Z}_{X, Y}) \circ f = (1 \otimes \Delta^{Z}_{X, Y})\big(m \otimes 1\big)\big(1 \otimes m \otimes 1\big)\big(1 \otimes S_{X, Z} \otimes 1 \otimes 1\big)\big(1 \otimes 1 \otimes \Delta^{X}_{Z, Y}\big)\big(1 \otimes eq\big)\\
		&=\big(m \otimes 1 \otimes  1\big)\big(1 \otimes m \otimes 1 \otimes 1\big)\big(1 \otimes S_{X, Z} \otimes 1 \otimes 1 \otimes 1 \big) (1 \otimes 1 \otimes 1 \otimes \Delta^{Z}_{X, Y}) \big(1 \otimes 1 \otimes \Delta^{X}_{Z, Y}\big)\big(1 \otimes eq\big)\\
		&= \big(m \otimes 1 \otimes  1\big)\big(1 \otimes m \otimes 1 \otimes 1\big)\big(1 \otimes S_{X, Z} \otimes 1 \otimes 1 \otimes 1 \big) (1 \otimes 1 \otimes \Delta^{X}_{Z, Z} \otimes 1) \big(1 \otimes 1 \otimes \Delta^{Z}_{Z, Y}\big)\big(1 \otimes eq\big)\\
		&= \big(m \otimes 1 \otimes  1\big)\big(1 \otimes m \otimes 1 \otimes 1\big)\big(1 \otimes S_{X, Z} \otimes 1 \otimes 1 \otimes 1 \big) (1 \otimes 1 \otimes \Delta^{X}_{Z, Z} \otimes 1) \big(1 \otimes \Delta^{Z}_{X, Z} \otimes 1\big)\big(1 \otimes eq\big)\\
		&\hspace{1cm}(\text{by the definition of $\Cc(X, Z) \Box_{\Cc(Z, Z)} \Cc(Z, Y)$ as an equalizer})\\
		&= \big(m \otimes 1 \otimes  1\big)\big(1 \otimes m \otimes 1 \otimes 1\big)\big(1 \otimes S_{X, Z} \otimes 1 \otimes 1 \otimes 1 \big) (1 \otimes \Delta^{Z}_{X, X} \otimes 1 \otimes 1) \big(1 \otimes \Delta^{X}_{X, Z} \otimes 1\big)\big(1 \otimes eq\big)\\
		&\hspace*{1cm} \text{(by \eqref{asso})}\\
		&= (m \otimes 1 \otimes 1)(1 \otimes u \otimes 1 \otimes 1)(1 \otimes \varepsilon_{X} \otimes 1 \otimes 1)\big(1 \otimes \Delta^{X}_{X, Z} \otimes 1\big)\big(1 \otimes eq\big) \quad \text{by \eqref{antipode}}\\
		&= 1 \otimes eq \hspace*{9.75cm} \text{by \eqref{counit}}.
	\end{align*}
	Thus
	$(1 \otimes eq)(1 \otimes \bar{\Delta}^{Z}_{X, Y}) \circ f = 1 \otimes eq$. Since $(1 \otimes eq)$ is also a monomorphism, we deduce that $(1 \otimes \bar{\Delta}^{Z}_{X, Y}) \circ f = \id$. We obtain that $1_{\CC(Z, X)} \otimes \bar{\Delta}^{Z}_{X, Y}$ is an isomorphism
	\[\Cc(Z, X) \otimes \Cc(X, Y) \cong \Cc(Z, X) \otimes \Cc(X, Z) \Box_{\Cc(Z, Z)} \Cc(Z, Y).\]
Using the faithful flatness hypothesis, we conclude that $\bar{\Delta}^{Z}_{X, Y}$ is an isomorphism. This completes the proof.
	\epf
We now obtain the main result of this section:
	\bt\label{Thm: monoid func}
	Let $\VV$ be a flat regular braided category, and let $\Cc$ be a faithfully flat $\VV$-cogroupoid. For any $X, Y \in \ob(\Cc)$, we have equivalences of monoidal categories
	\[\VV^{\Cc(X, X)} \cong^{\otimes} \VV^{\Cc(Y, Y)}.\]
	\et
	\bpf
	By Proposition $\ref{Prop : Takeuchi}$, we know that $\Cc$ is a Takeuchi $\VV$-cocategory. Thus, by Corollary \ref{Cor : monoi equi}, we obtain equivalences of monoidal categories that are inverse to each other
	\begin{align*}
		F : \VV^{\Cc(X, X)} &\cong^{\otimes} \VV^{\Cc(Y, Y)} \quad &G : \VV^{\Cc(Y, Y)} &\cong^{\otimes} \VV^{\Cc(X, X)}\\
		V &\longmapsto V \Box_{\Cc(X, X)} \Cc(X, Y) &\quad V &\longmapsto V \Box_{\Cc(Y, Y)} \Cc(Y, X). \qedhere
	\end{align*}
	\epf
	\brq
	Let $\Cc$ be a $\VV$-cogroupoid. For any $X, Y \in \ob(\Cc)$, the object $\Cc(X, Y)$ is a right $\Cc(Y, Y)$-Galois object in the sense of Schauenburg \cite[Definition 3.1]{MR2294763}, since the morphism
	\begin{center}
		\footnotesize
		\begin{tikzcd}[column sep= 1.5cm,row sep= 1cm]
			\kappa_{r} \colon \Cc(X, Y) \otimes \Cc(X, Y)\ar[r, "1 \otimes \Delta^{Y}_{X, Y}"] &\Cc(X, Y) \otimes \Cc(X, Y) \otimes \Cc(Y, Y) \ar[r, "m \otimes 1"] &\Cc(X, Y) \otimes \Cc(Y, Y)
		\end{tikzcd}
	\end{center}
	is an isomorphism with inverse $\kappa_{r}^{-1}$ given as follows:

	\begin{center}
		\footnotesize
		\begin{tikzcd}[column sep= 1.7cm,row sep= 1cm]
			\Cc(X, Y) \otimes \Cc(Y, Y) \ar[r, "\kappa_{r}^{-1}", ""{name = D}]\ar[d, "1 \otimes \Delta^{X}_{Y, Y}" '] &\Cc(X, Y) \otimes \Cc(Y, Y)\\
			\Cc(X, Y) \otimes \Cc(Y, X) \otimes \Cc(X, Y) \ar[r, "1 \otimes S_{Y, X} \otimes 1"] &\Cc(X, Y) \otimes \Cc(X, Y) \otimes \Cc(X, Y) \ar[u, "m \otimes 1" ', ""{name = U}].
			\arrow[to path={(U) node[left, scale=1]{$\underset{{\hspace*{2cm}}}{\circlearrowleft \hspace*{3cm}}$} (K)}]{}
		\end{tikzcd}
	\end{center}
	Hence, Lemma \ref{Lemma: Tensor Compa} can be deduced, under the additional assumption of faithful flatness, from \cite[Lemmas 2.3 and 3.6]{MR2294763}.
	
	Conversely, it can be shown that a faithfully flat right $H$-Galois object $A$ always gives rise to a cogroupoid. Indeed, setting $L := (A \otimes A)^{coH}$ and $A^{-1} := (H \otimes A)^{coH}$, Schauenburg shows that $L$ is a Hopf algebra, that $A$ is $L-H$-bi-Galois object in $\VV$, and that  $A^{-1}$ is an $H$–$L$ bi-Galois object. Working as Grunspan did in \cite{Grunspan2}, we can reconstruct a full $\VV$-cogroupoid structure. Since this viewpoint is not essential for the main results of the present paper, we omit the details.
	
Notice however that, in a faithfully flat cogroupoid, we have isomorphisms \[\Cc(X, Y) \simeq \big(\Cc(X, Z) \otimes \Cc(Y, Z)\big)^{co\Cc(Z, Z)}\] fitting into the commutative diagram
\begin{center}
	\tiny
	\begin{tikzcd}[column sep=1.7cm, row sep=1cm]
		\Cc(X, X) \ar[rr, "\sim"] \ar[rd, "\overline{\Delta}^{Z}_{X, X}" '] && \big(\Cc(X, Z) \otimes \Cc(Y, Z)\big)^{co\Cc(Z, Z)} \\
		& \Cc(X, Z) \Box_{\Cc(Z, Z)} \Cc(Z, Y) \ar[ru, "1 \Box_{\Cc(Z, Z)} S_{Z, Y}" '] &
	\end{tikzcd}
\end{center}
which shows that, if we apply the Schauenburg–Grunspan construction above to a braided cogroupoid, we recover (essentially) the same object.
	\erq

	\section{Cogroupoids of coinvariants and Bosonizations} \label{Sec : Bosonization}
	In this section, we construct a braided cogroupoid over the category of Yetter–Drinfeld modules from a cogroupoid endowed with a kind of projection. Conversely, as in the case of Hopf algebras, we show that the bosonization of a braided cogroupoid in a category of Yetter–Drinfeld modules yields an ordinary cogroupoid. To this end, we first introduce the following definition of a cogroupoid triple, which generalizes the notion of a Hopf algebra triple \cite{heckenberger2020hopf}.
	\bd \label{Def : cogroup triple}
Let $H$ be a Hopf algebra. \textbf{A cogroupoid triple over} $H$ is a triple $(\mathcal{K}, \pi, i)$, where $\mathcal{K}$ is a $k$-cogroupoid, and
\begin{enumerate}
	\item  $i := \{i_{X, Y} \colon H \longrightarrow \mathcal{K}(X, Y) \mid X, Y \in \ob(\KK)\}$ is a family of algebra maps satisfying, for any $X, Y, Z \in \ob(\KK)$, \[(i_{X, Z} \otimes i_{Z, Y}) \circ \Delta_{H} = \Delta^{Z}_{X, Y} \circ i_{X, Y} \quad \text{and} \quad \varepsilon_{X} \circ i_{X,X} = \varepsilon_{H},\]
	\item $\pi := \{\pi_{X} \colon \KK(X, X) \to H \mid X \in \ob(\KK)\}$ is a family of Hopf algebra maps, such that \[\pi_{X}i_{X, X} = 1_{H}\] for every $X \in \ob(\KK)$.
\end{enumerate}
\ed
It follows that if $(\mathcal{K}, \pi, i)$ is a cogroupoid triple over a Hopf algebra $H$, then for any $X \in \ob(\KK)$, $\big(\KK(X, X), \pi_{X}, i_{X, X}\big)$ is a Hopf algebra triple over $H$ in the category of $k$-vector spaces, in the sense of  \cite[Section 3.10]{heckenberger2020hopf}. The results obtained in this section generalize the work of Majid \cite{Maj94} and Radford \cite{Rad85} in the context of ordinary Hopf algebras.
We also adopt Sweedler’s notation when working with a cogroupoid: let $\KK$ be a $k$-cogroupoid and $X, Y \in \ob(\KK)$, for $h \in H, a^{X, Y} \in \KK(X, Y)$, we write
\[\Delta_{H}(h) = h_{(1)} \otimes h_{(2)}\]
\[\Delta^{Z}_{X, Y}(a^{X, Y}) = a^{X, Z}_{[1]} \otimes a^{Z, Y}_{[2]}.\]
Then the conditions for the maps $i$ read
\begin{equation}
	\Delta^{Z}_{X, Y}\big(i_{X, Y}(h)\big) = i_{X, Z}(h_{(1)}) \otimes i_{Z, Y}(h_{(2)}) \label{condition i}.
\end{equation}
	\subsection{Cogroupoid of coinvariants} \label{Sec: coinv}
	Let $H$ be a fixed Hopf algebra with bijective antipode. Let $(\KK, \pi, i)$ be a cogroupoid triple over $H$. For any $X, Y \in \ob(\KK),$ let
\[^{co\pi_{X}}\KK(X, Y) = \{a^{X, Y} \in \KK(X, Y) \mid (\pi_{X} \otimes 1)\Delta^{X}_{X, Y}(a^{X, Y}) = 1 \otimes a^{X, Y}\}.\]
\bl
Endowed with the following structures, $^{co\pi_{X}}\KK(X, Y)$ is an algebra in the category $\YD_{H}^{H}$:
\begin{enumerate}
	\item the action of $H$ on $^{co\pi_{X}}\KK(X, Y)$ is given by
	\[a^{X, Y} \leftarrow h = i_{X, Y}S_{H}(h_{(1)}) a^{X, Y} i_{X, Y}(h_{(2)});\]
	\item the $H$-coaction is
	\begin{equation}
		a^{X, Y} \longmapsto a^{X, Y}_{(0)} \otimes a^{X, Y}_{(1)} := a^{X, Y}_{[1]} \otimes \pi_{Y}(a^{Y,Y}_{[2]}). \label{H-coaction coin}
	\end{equation}
\end{enumerate}
	\el
	\bpf
It is not difficult to check that the formulas above define, respectively, an $H$-module structure and an $H$-comodule structure on $\,^{co\pi_{X}}\KK(X, Y)$. Now, we verify the compatibility condition: for $a^{X, Y} \in \,^{co\pi_{X}}\KK(X, Y)$ and $h \in H$,
\[(a^{X, Y} \leftarrow h)_{(0)} \otimes (a^{X, Y} \leftarrow h)_{(1)} = (a^{X, Y}_{(0)} \leftarrow h_{(2)}) \otimes S_{H}(h_{(1)}) a^{X, Y}_{(1)} h_{(3)}.\]
The left-hand side equals
\begin{align*}
	&\big(i_{X, Y}S_{H}(h_{(1)}) a^{X, Y} i_{X, Y}(h_{(2)})\big)_{(0)} \otimes \big(i_{X, Y}S_{H}(h_{(1)}) a^{X, Y} i_{X, Y}(h_{(2)})\big)_{(1)}\\
	&= (1 \otimes \pi_{Y})\Delta^{Y}_{X,Y}\Big(i_{X, Y}S_{H}(h_{(1)}) a^{X, Y} i_{X, Y}(h_{(2)})\Big)\\
	&= (1 \otimes \pi_{Y})\Big(i_{X, Y}S_{H}(h_{(2)}) a_{[1]}^{X, Y} i_{X, Y}(h_{(3)}) \otimes i_{Y, Y}S_{H}(h_{(1)}) a_{[2]}^{Y, Y} i_{Y, Y}(h_{(4)})\Big)\\
	&\quad (\text{by using \eqref{condition i} for $i_{Y, X}\big(S_{H}(h_{(1)})\big)$ and $i_{X, Y}(h_{(2)})$})\\
	&= i_{X, Y}S_{H}(h_{(2)}) a_{[1]}^{X, Y} i_{X, Y}(h_{(3)}) \otimes \pi_{Y}\big[i_{Y, Y}S_{H}(h_{(1)}) a_{[2]}^{Y, Y} i_{Y, Y}(h_{(4)})\big]\\
	&= i_{X, Y}S_{H}(h_{(2)}) a_{[1]}^{X, Y} i_{X, Y}(h_{(3)}) \otimes S_{H}(h_{(1)}) \pi_{Y}\big(a^{Y, Y}_{[2]}\big) h_{(4)}\\
	&= (a^{X, Y}_{(0)} \leftarrow h_{(2)}) \otimes S_{H}(h_{(1)}) a^{X, Y}_{(1)} h_{(3)}
\end{align*}
which is exactly the right-hand side.

A straightforward computation shows that the multiplication is $H$-linear as well as $H$-colinear, thereby completing the proof.
\epf

	\bl \label{Lem : prop of v}
Let $\vartheta \colon \KK(X, Y) \longrightarrow \,^{co\pi_{X}}\KK(X, Y)$ be the map given by
\[\vartheta(a^{X, Y}) =( i_{X, Y} \pi_{X} S_{X, X} \otimes 1)\Delta^{X}_{X, Y}(a^{X, Y}) = i_{X, Y} \pi_{X} S_{X, X}(a_{[1]}^{X, X}) a_{[2]}^{X, Y}.\]
Then for any $X, Y \in \ob(\KK)$, $a^{X, Y}, b^{X, Y} \in \KK(X, Y)$ and $h \in H$, we have
\begin{enumerate}
	\item $\vartheta(a^{X, Y}b^{X, Y}) = i_{X, Y} \pi_{X} S_{X, X}(b_{[1]}^{X, X}) \vartheta(a^{X,Y})b_{[2]}^{X, Y}$,
	\item $\vartheta(i_{X, Y}(h)) = \varepsilon_{H}(h)$,
	\item $\Delta^{Z}_{X, Y}\big(\vartheta(a^{X, Y})\big) = \vartheta(a_{[2]}^{X, Z}) \otimes i_{Z, Y} \big(\pi_{X} S_{X, X}(a_{[1]}^{X, X})\big) a_{[3]}^{Z, Y},$
	\item $(1 \otimes \vartheta)\Delta^{Z}_{X, Y}\big(\vartheta(a^{X, Y})\big) = \vartheta(a_{[1]}^{X, Z}) \otimes \vartheta(a_{[2]}^{Z, Y}).$
\end{enumerate}
\el
	\bpf
	First, we see that $\vartheta$ is well-defined, since
	\begin{align*}
		&(\pi_{X} \otimes 1)\Delta^{X}_{X, Y}\vartheta(a^{X, Y}) = (\pi_{X} \otimes 1)\Delta^{X}_{X, Y}\Big(i_{X, Y} \pi_{X} S_{X, X}(a_{[1]}^{X, X}) a_{[2]}^{X, Y}\Big)\\
		&=  (\pi_{X} \otimes 1)(i_{X, X} \otimes i_{X, Y})\Big(\pi_{X} S_{X, X}(a_{[2]}^{X, X}) a_{[3]}^{X, X} \otimes \pi_{X} S_{X, X}(a_{[1]}^{X, X}) a_{[4]}^{X, Y}\Big)\\
		& = 1 \otimes i_{X, Y}\pi_{X} S_{X, X}(a_{[1]}^{X, X}) a_{[2]}^{X, Y}\\
		&= 1 \otimes \vartheta(a^{X, Y}).
	\end{align*}
	The first two statements then follow from a straightforward computation using the definition of $\vartheta$.
	We prove $(3)$:
	\begin{align*}
		&\Delta^{Z}_{X, Y}\big(\vartheta(a^{X, Y})\big) = \Delta^{Z}_{X, Y}\big[i_{X, Y} \pi_{X} S_{X, X}(a_{[1]}^{X, X}) a_{[2]}^{X, Y}\big]\\
		&= i_{X, Z} \big((\pi_{X} S_{X, X}(a_{[1]}^{X, X}))_{(1)}\big) a_{[2]}^{X, Z} \otimes i_{Z, Y} \big((\pi_{X} S_{X, X}(a_{[1]}^{X, X}))_{(2)}\big) a_{[3]}^{Z, Y}\\
		&= i_{X, Z} \big(\pi_{X} S_{X, X}(a_{[2]}^{X, X})\big) a_{[3]}^{X, Z} \otimes i_{Z, Y} \big(\pi_{X} S_{X, X}(a_{[1]}^{X, X})\big) a_{[4]}^{Z, Y}\\
		&= \vartheta(a^{X, Z}_{[2]}) \otimes i_{Z, Y} \big(\pi_{X} S_{X, X}(a_{[1]}^{X, X})\big) a_{[3]}^{Z, Y}.
	\end{align*}
	
	Next, for $a^{X, Y} \in \,^{co\pi_{X}}\KK(X, Y)$,
	\begin{align*}
		(1 \otimes \vartheta)	\Delta^{Z}_{X, Y}\big(\vartheta(a^{X, Y})\big) &= \vartheta(a^{X, Z}_{[2]}) \otimes \vartheta[i_{Z, Y} \big(\pi_{X} S_{X, X}(a_{[1]}^{X, X})\big) a_{[3]}^{Z, Y}]\\
		&= \vartheta(a^{X, Z}_{[2]}) \otimes \varepsilon_{H}\big(\pi_{X} S_{X, X}(a_{[1]}^{X, X})\big)\vartheta(a_{[3]}^{Z, Y}) \quad (\text{by $(1) + (2)$})\\
		&= \vartheta(a^{X, Z}_{[1]}) \otimes \vartheta(a^{Z, Y}_{[2]}).
	\end{align*}
	Statement $(4)$ follows.
	\epf
	\bl\label{(Lem : coDelata)}         
	For any $X, Y, Z \in \ob(\KK)$, there exists an algebra map in $\YD^{H}_{H}$
	\begin{align*}                           
		\underline{\Delta}^{Z}_{X, Y} \colon \,^{co\pi_{X}}\KK(X, Y) &\longrightarrow \,^{co\pi_{X}}\KK(X, Z) \otimes \,^{co\pi_{Z}}\KK(Z, Y)\\
		a^{X, Y} &\longmapsto a^{X, Z}_{[1]} \otimes \vartheta(a_{[2]}^{Z, Y})
	\end{align*}                             
	such that the following diagrams commute 
	\[\begin{tikzcd}                         
		^{co\pi_{X}}\KK(X, Y) \ar[r, "\underline{\Delta}^{Z}_{X, Y}"] \ar[d, "\underline{\Delta}^{T}_{X, Y}"]&^{co\pi_{X}}\KK(X, Z) \otimes \,^{co\pi_{Z}}\KK(Z, Y) \ar[d, "\underline{\Delta}^{T}_{X, Z} \otimes 1"]\\
		^{co\pi_{X}}\KK(X, T) \otimes \,^{co\pi_{T}}\KK(T, Y) \ar[r, "1 \otimes \underline{\Delta}^{Z}_{T, Y}"] &^{co\pi_{X}}\KK(X, T) \otimes \,^{co\pi_{T}}\KK(T, Z) \otimes \,^{co\pi_{Z}}\KK(Z, Y)
	\end{tikzcd}\]                           
	\[\begin{tikzcd}                         
		^{co\pi_{X}}\KK(X, Y) \ar[rd, equals, shift left = 1ex] \ar[d, "\underline{\Delta}^{X}_{X, Y}"] &\\
		^{co\pi_{X}}\KK(X, X) \otimes \,^{co\pi_{X}}\KK(X, Y) \ar[r, "\varepsilon_{X} \otimes 1"] &\,^{co\pi_{X}}\KK(X, Y)
	\end{tikzcd}\]                           
	\[\begin{tikzcd}                         
		^{co\pi_{X}}\KK(X, Y) \ar[rd, equals, shift left = 1ex] \ar[d, "\underline{\Delta}^{Y}_{X, Y}"] &\\
		^{co\pi_{X}}\KK(X, Y) \otimes \,^{co\pi_{X}}\KK(Y, Y) \ar[r, "1 \otimes \varepsilon_{Y}"] &\,^{co\pi_{X}}\KK(X, Y).
	\end{tikzcd}\]                           
	\el                                      
	\bpf
	We first show that $\underline{\Delta}_{X, Y}^{Z}$ is a well-defined algebra morphism in $\YD^{H}_{H}$. Let $a^{X, Y} \in \, 	^{co\pi_{X}}\KK(X, Y)$, we have
	\begin{align*}                           
		(\pi_{X} \otimes 1)(\Delta^{X}_{X, Z} \otimes 1)\big(a_{[1]}^{X, Z} \otimes \vartheta(a^{Z, Y}_{[2]})\big) = \pi_{X}(a_{[1]}^{X, X}) \otimes a_{[2]}^{X, Z} \otimes i_{Z, Y}\pi_{Z}S_{Z, Z}(a_{[3]}^{Z, Z})a^{Z, Y}_{[4]}.
	\end{align*}                             
	Note that \[\pi_{X}(a^{X, X}_{[1]}) \otimes a^{X, Y}_{[2]} = 1 \otimes a^{X, Y}.\]
	Applying $(1 \otimes 1 \otimes \Delta_{Z, Y}^{Z})(1 \otimes \Delta^{Z}_{X, Y})$ to this equality we get
	\begin{equation*}                        
		\pi_{X}(a^{X, X}_{[1]}) \otimes a^{X, Z}_{[2]} \otimes a^{Z, Z}_{[3]} \otimes a^{Z, Y}_{[4]} = 1 \otimes  a^{X, Z}_{[1]} \otimes a^{Z, Z}_{[2]} \otimes a^{Z, Y}_{[3]}.
	\end{equation*}                          
	Thus                                     
	\begin{align*}                           
		(\pi_{X} \otimes 1)(\Delta^{X}_{X, Z} \otimes 1)\big(a_{[1]}^{X, Z} \otimes \vartheta(a^{Z, Y}_{[2]})\big) &= 1 \otimes a_{[1]}^{X, Z} \otimes i_{Z, Y}\pi_{Z}S_{Z, Z}(a_{[2]}^{Z, Z})a^{Z, Y}_{[3]}\\
		&= 1 \otimes a_{[1]}^{X, Z} \otimes \vartheta(a_{[2]}^{X, Y}),
	\end{align*}                             
	and $\underline{\Delta}_{X, Y}^{Z}$ is well defined.
	It is not difficult to verify that this map is both $H$-linear and $H$-colinear. We now show that it is in fact an algebra morphism in $\YD^{H}_{H}$:
	\begin{align*}                           
		&\underline{\Delta}_{X, Y}^{Z}(a^{X, Y})\underline{\Delta}_{X, Y}^{Z}(b^{X, Y}) = (a_{[1]}^{X, Z} \otimes \vartheta(a_{[2]}^{Z, Y})).(b_{[1]}^{X, Z} \otimes \vartheta(b_{[2]}^{Z, Y})) \quad(\text{the braided product})\\
		&= a_{[1]}^{X, Z} b_{[1]}^{X, Z} \otimes \Big(\vartheta(a_{[2]}^{Z, Y}) \leftarrow \pi_{Z}(b_{[2]}^{Z, Z})\Big)\vartheta(b_{[3]}^{Z, Y})\\
		&= a_{[1]}^{X, Z} b_{[1]}^{X, Z} \otimes i_{Z, Y}\big(S_{H}\pi_{Z}(b_{[2]}^{Z, Z})\big)\vartheta(a_{[2]}^{Z, Y})i_{Z, Y}\big(\pi_{Z}(b_{[3]}^{Z, Z})\big)\vartheta(b_{[4]}^{Z, Y})\\
		&= a_{[1]}^{X, Z} b_{[1]}^{X, Z} \otimes i_{Z, Y}\big(\pi_{Z}S_{Z, Z}(b_{[2]}^{Z, Z})\big)\vartheta(a_{[2]}^{Z, Y})i_{Z, Y}\pi_{Z}(b_{[3]}^{Z, Z})i_{Z, Y}\pi_{Z}S_{Z, Z}(b_{[4]}^{Z, Z})b_{[5]}^{Z, Y}\\
		&= a_{[1]}^{X, Z} b_{[1]}^{X, Z} \otimes i_{Z, Y}\big(\pi_{Z}S_{Z, Z}(b_{[2]}^{Z, Z})\big)i_{Z, Y}\big(\pi_{Z}S_{Z, Z}(a_{[2]}^{Z, Z})\big)a_{[3]}^{Z, Y}b_{[3]}^{Z, Y}\\
		&= a_{[1]}^{X, Z} b_{[1]}^{X, Z} \otimes i_{Z, Y}\pi_{Z}S_{Z, Z}(a_{[2]}^{Z, Z}b_{[2]}^{Z, Z})a_{[3]}^{Z, Y}b_{[3]}^{Z, Y}\\
		&= a_{[1]}^{X, Z} b_{[1]}^{X, Z} \otimes \vartheta\big(a_{[2]}^{Z, Y}b_{[2]}^{Z, Y}\big)\\
		&= \underline{\Delta}_{X, Y}^{Z} (a^{X, Y}b^{X, Y}).
	\end{align*}
	The commutativity of the first diagram then follows directly from Lemma~\ref{Lem : prop of v}.(4).
	
	Next, we have
	\begin{align*}
		(\varepsilon_{X} \otimes 1)\underline{\Delta}^{X}_{X, Y}(a^{X, Y}) &= (\varepsilon_{X} \otimes 1)\big(a^{X, X}_{[1]} \otimes i_{X, Y} \pi_{X} S_{X, X}(a^{X, X}_{[2]})a^{X, Y}_{[3]}\big)\\
		&= \varepsilon_{X}(a^{X, X}_{[1]}) i_{X, Y} \pi_{X} S_{X, X}(a^{X, X}_{[2]})a^{X, Y}_{[3]}\\
		&= i_{X, Y}\Big[\varepsilon_{X} \big(S_{X, X}(a^{X, X}_{[1]})\big) \pi_{X}\big(S_{X, X}(a^{X, X}_{[2]})\big)\Big]a^{X, Y}_{[3]}\\
		&= i_{X, Y}\Big[\varepsilon_{X} \big(S_{X, X}(a^{X, X}_{[1]})\big)\Big]a^{X, Y}_{[2]}\\
		&= a^{X, Y}.
	\end{align*}
	On the other hand,
	\begin{align*}
		(1 \otimes \varepsilon_{Y})\underline{\Delta}^{Y}_{X, Y}(a^{X, Y}) &= (1 \otimes \varepsilon_{Y})\big(a^{X, Y}_{[1]} \otimes i_{Y, Y} \pi_{Y} S_{Y, Y}(a^{Y, Y}_{[2]})a^{Y, Y}_{[3]}\big)\\
		&= a^{X, Y}_{[1]} \varepsilon_{Y}(a_{[2]}^{Y, Y})\varepsilon_{Y}(a_{[3]}^{Y, Y})\\
		&= a^{X, Y}.
	\end{align*}
	Consequently, the final two diagrams commute.
	\epf

	\bl \label{Lem : coS}
	For $X, Y \in \ob(\KK)$, there exists a morphism in $\YD^{H}_{H}$
	\begin{align*}
		\underline{S}_{X, Y} \colon \, ^{co\pi_{X}}\KK(X, Y) &\longrightarrow \,^{co\pi_{Y}}\KK(Y, X)\\
		a^{X, Y} &\longmapsto S_{X, Y}(a_{[1]}^{X, Y}) i_{Y, X}\pi_{Y}(a^{Y, Y}_{[2]})
	\end{align*}
	such that the following diagrams commute
	\[\begin{tikzcd}
		^{co\pi_{X}}\KK(X, X) \ar[d, "\underline{\Delta}^{Y}_{X, X}"] \ar[r, "\varepsilon_{X}"] &k \ar[r, "u_{X, Y}"] &\,^{co\pi_{X}}\KK(X, Y)\\
		^{co\pi_{X}}\KK(X, Y) \otimes \,^{co\pi_{Y}}\KK(Y, X) \ar[rr, "1 \otimes \underline{S}_{Y, X}"] && \,^{co\pi_{X}}\KK(X, Y) \otimes \,^{co\pi_{Y}}\KK(X, Y) \ar[u, "m"]
	\end{tikzcd}\]
	\[\begin{tikzcd}
		^{co\pi_{X}}\KK(X, X) \ar[d, "\underline{\Delta}^{Y}_{X, X}"] \ar[r, "\varepsilon_{X}"] &k \ar[r, "u_{Y, X}"] &\,^{co\pi_{Y}}\KK(Y, X)\\
		^{co\pi_{X}}\KK(X, Y) \otimes \,^{co\pi_{Y}}\KK(Y, X) \ar[rr, "\underline{S}_{X, Y} \otimes 1"] && \,^{co\pi_{X}}\KK(X, Y) \otimes \,^{co\pi_{Y}}\KK(X, Y) \ar[u, "m"].
	\end{tikzcd}\]
	\el
	\bpf 
	We see that $\underline{S}_{X, Y}$ is a well-defined morphism in $\YD^{H}_{H}$ by direct computations. Now we must verify that the two diagrams above commute. We begin with the first diagram:
	for $a^{X, X} \in \,^{co\pi_{X}}\KK(X, X)$,
	\begin{align*}
		&m(1 \otimes \underline{S}_{Y, X})\underline{\Delta}^{Y}_{X, X}(a^{X, X}) = a^{X, Y}_{[1]} \underline{S}_{Y, X}\big(\vartheta(a^{Y, X}_{[2]})\big)\\
		&= a^{X, Y}_{[1]} S_{Y, X}\big(\vartheta(a^{Y, X}_{[3]})\big)i_{X, Y}\pi_{X}\big(i_{X, X}\pi_{Y}S_{Y, Y}(a^{Y, Y}_{[2]})a^{X, X}_{[4]}\big) \quad (\text{by Lemma \ref{Lem : prop of v}.$(3)$})\\
		&= a^{X, Y}_{[1]} S_{Y, X}\big(\vartheta(a^{Y, X}_{[3]})\big)i_{X, Y}\pi_{Y}S_{Y, Y}(a^{Y, Y}_{[2]})i_{X, Y}\pi_{X}(a^{X, X}_{[4]})\\
		&= a^{X, Y}_{[1]} S_{Y, X}\big(i_{Y, X}\pi_{Y}S_{Y, Y}(a^{Y, Y}_{[3]})a^{Y, X}_{[4]}\big)i_{X, Y}\pi_{Y}S_{Y, Y}(a^{Y, Y}_{[2]})i_{X, Y}\pi_{X}(a^{X, X}_{[5]})\\
		&= a^{X, Y}_{[1]} S_{Y, X}(a^{Y, X}_{[4]})S_{Y, X}\big(i_{Y, X}\pi_{Y}S_{Y, Y}(a^{Y, Y}_{[3]})\big)i_{X, Y}\pi_{Y}S_{Y, Y}(a^{Y, Y}_{[2]})i_{X, Y}\pi_{X}(a^{X, X}_{[5]})\\
		&= a^{X, Y}_{[1]} S_{Y, X}(a^{Y, X}_{[2]})i_{X, Y}\pi_{X}(a^{X, X}_{[3]})\\
		&= i_{X, Y}\pi_{X}(a^{X, X})\\
		&= 1\varepsilon_{X}(a^{X, X}) \quad (\text{because $a^{X, X} \in \,^{co\pi_{X}}\KK(X, X)$}).
	\end{align*}
	For the second diagram,
	\begin{align*}
		m(\underline{S}_{Y, X} \otimes 1)\underline{\Delta}^{Y}_{X, X}(a^{X, X}) &= \underline{S}_{X, Y}(a^{X, Y}_{[1]}) \vartheta(a^{Y, X}_{[2]})\\
		&= S_{X, Y}(a^{X, Y}_{[1]})i_{Y, X}\pi_{Y}(a^{Y, Y}_{[2]})\vartheta(a^{Y, X}_{[3]})\\
		&= S_{X, Y}(a^{X, Y}_{[1]})i_{Y, X}\pi_{Y}(a^{Y, Y}_{[2]})i_{Y, X}\pi_{Y}S_{Y, Y}(a^{Y, Y}_{[3]})a^{Y, X}_{[4]}\\
		&= S_{X, Y}(a^{X, Y}_{[1]})a^{Y, X}_{[2]}\\
		&= \varepsilon_{X}(a^{X, X})1. \qedhere
	\end{align*} 
	\epf
	\btp \label{Def: Cogr coinva}
	Let $H$ be a Hopf algebra with bijective antipode and let $(\KK, \pi, i)$ be a cogroupoid triple over $H$. Then $^{co\pi}\KK$ is a cogroupoid in $\YD_{H}^{H}$, called the cogroupoid of coinvariants, defined as follows:
	\begin{enumerate}
		\item $\ob(^{co\pi}\KK) = \ob(\KK),$
		\item for any $X, Y \in \ob(\KK)$, $^{co\pi}\KK(X, Y) = \,^{co\pi_{X}}\KK(X, Y)$,
		\item the structural maps $\underline{\Delta}^{\bullet}_{\bullet, \bullet}$, $\varepsilon_{\bullet}$  and $\underline{S}_{\bullet, \bullet}$ are defined in Lemma \ref{(Lem : coDelata)} and Lemma \ref{Lem : coS}.
	\end{enumerate}
	\etp
	\subsection{Bosonizations}
	Let $H$ be a Hopf algebra with bijective antipode, and let $A$ be a Hopf algebra in $\YD^{H}_{H}$. We use Sweedler's notation as follows:
	\[\Delta_A(a) = a_{[1]} \otimes a_{[2]}, \quad \Delta_H(h) = h_{(1)}\otimes h_{(2)}, \quad  \beta^{H}(a) = a_{(0)} \otimes a_{(1)}.\]
	The \textsl{bosonization} (or the Radford biproduct)  [\citen{Maj94, Rad85} ] $H\#A$ is then the ordinary Hopf algebra that has $H\otimes A$ as underlying vector space, has the unit and counit of the ordinary tensor product of algebras and coalgebras, and comultiplication, product and antipode given by 
	\[ h\#a \cdot k\# b = hk_{(1)} \# (a \leftarrow k_{(2)})b, \quad \Delta(h\#a)= (h_{(1)} \# a_{[1](0)}) \otimes (h_{(2)}a_{[1](1)} \# a_{[2]}) \]
	\begin{align*} S(h\#a) &= (1\#S_A(a_{(0)}))\cdot (S_H(ha_{(1)})\#1) \\
		& = S_H(h_{(2)}a_{(2)}) \# [S_A(a_{(0)})  \leftarrow S_H(h_{(1)}a_{(1)})]
	\end{align*}
	for any $h, k \in H$ and $a, b \in A$.
	
	We will now construct the bosonization of a cogroupoid over $\YD^{H}_{H}:$ let $\Cc$ be a $\YD^{H}_{H}$-cogroupoid. For any $a^{X, Y} \in \Cc(X, Y)$, we also write \[\beta^{H}(a^{X, Y}) = a_{(0)}^{X, Y} \otimes a_{(1)}^{X, Y}.\]
	Recall that, in this case, the compatibility condition $\eqref{compatibiliy}$ is characterized, for $a^{X, Y} \in \Cc(X, Y), h \in H$, by
	\begin{align}
		(a^{X,Y} \leftarrow h)_{(0)} \otimes (a^{X,Y} \leftarrow h)_{(1)} = a^{X,Y}_{(0)} \leftarrow h_{(2)} \otimes S_{H}(h_{(1)})a^{X,Y}_{(1)}h_{(3)}.\label{compatibiliy2}
	\end{align}
	The property that $\Delta^{Z}_{X,Y}$ is an algebra morphism in $\YD_{H}^{H}$ is expressed by
	\begin{align}
		\Delta^{Z}_{X,Y}(a^{X,Y}b^{X,Y}) &= (a^{X,Z}b^{X,Z})_{[1]} \otimes (a^{Z,Y}b^{Z,Y})_{[2]} \nonumber\\
		&= a_{[1]}^{X,Z}b_{[1](0)}^{X,Z} \otimes \big(a_{[2]}^{Z,Y}\leftarrow b_{[1](1)}^{X,Z}\big)b_{[2]}^{Z,Y}, \label{Delta_mor}
	\end{align}
	and the property that $\Delta^{Z}_{X, Y}$ is $H$-colinear is expressed by
	\begin{align}
		a^{X, Z}_{[1](0)} \otimes a^{Z, Y}_{[2](0)} \otimes a^{X, Z}_{[1](1)}a^{Z, Y}_{[2](1)} = a^{X, Z}_{(0)[1]} \otimes a^{Z, Y}_{(0)[2]} \otimes a^{X, Y}_{(1)} \label{H_colinear}.
	\end{align}
	Applying $(\beta^{H} \otimes 1 \otimes 1)$ and $( 1 \otimes \beta^{H} \otimes 1)$ respectively, to \eqref{H_colinear}, we get
	\begin{align}
		&a^{X, Z}_{[1](0)} \otimes a^{X, Z}_{[1](1)} \otimes a^{Z, Y}_{[2](0)} \otimes a^{X, Z}_{[1](2)}a^{Z, Y}_{[2](1)} = a^{X, Z}_{(0)[1](0)} \otimes a^{X, Z}_{(0)[1](1)}\otimes a^{Z, Y}_{(0)[2]} \otimes a^{X, Y}_{(1)}; \label{H_colinear1}\\
		&a^{X, Z}_{[1](0)} \otimes a^{Z, Y}_{[2](0)} \otimes a^{Z, Y}_{[2](1)} \otimes a^{X, Z}_{[1](1)}a^{Z, Y}_{[2](2)} = a^{X, Z}_{(0)[1]} \otimes a^{Z, Y}_{(0)[2](0)}\otimes a^{Z, Y}_{(0)[2](1)} \otimes a^{X, Y}_{(1)}  \label{H_colinear2}.
	\end{align}
	\bl \label{algebra}
	For any $X, Y \in \ob(\Cc)$, $H \# \Cc(X, Y)$ is a $k$-algebra with the product defined by
	\[ \widetilde{m}\big(h\#a^{X, Y} \otimes k\# b^{X, Y}\big) := (h\#a^{X, Y}) \cdot (k\# b^{X, Y}) = hk_{(1)} \# (a^{X, Y} \leftarrow k_{(2)})b^{X, Y}\]
	and unit $1 \# 1$.
	
	The maps $\widetilde{\Delta}^{Z}_{X,Y} : H \# \Cc(X, Z) \to H \# \Cc(X, Z) \otimes H \# \Cc(Z, Y)$, defined by
	\[\widetilde{\Delta}^{Z}_{X,Y}(h \# a^{X, Y}) = (h_{(1)} \# a^{X, Z}_{[1](0)}) \otimes (h_{(2)}a^{X, Z}_{[1](1)} \# a^{Z, Y}_{[2]}),\]
	and 
	\begin{align*}
		\widetilde{\varepsilon}_{X} : H  \# \Cc(X, X) &\longrightarrow k\\
		h \# a^{X, X} &\longmapsto \varepsilon_{H}(h) \varepsilon_{X}(a^{X, X})
	\end{align*}
	are morphisms of algebras and satisfy the $k$-cocategory axioms.
	\el 
	\bpf
	A direct computation shows that the product is associative.
	We will show that the map $\widetilde{\Delta}^{Z}_{X,Y}$ is a morphism of algebras:
	\begin{align*}
		\widetilde{\Delta}^{Z}_{X,Y}&\big[(h\#a^{X, Y}) \cdot (k\# b^{X, Y})\big] = \widetilde{\Delta}^{Z}_{X,Y}\big[hk_{(1)} \# (a^{X, Y} \leftarrow k_{(2)})b^{X, Y}\big]\\
		&= (hk_{(1)})_{(1)} \# \big((a^{X, Z} \leftarrow k_{(2)})b^{X, Z}\big)_{[1](0)}\\ &\quad \otimes (hk_{(1)})_{(2)}\big((a^{X, Z} \leftarrow k_{(2)})b^{X, Z}\big)_{[1](1)} \# \big((a^{Z, Y} \leftarrow k_{(2)})b^{Z, Y}\big)_{[2]}\\
		&= h_{(1)}k_{(1)} \# \big[(a^{X, Z} \leftarrow k_{(3)})_{[1]}b_{[1](0)}^{X, Z}\big]_{(0)} \\
		&\quad \otimes h_{(2)}k_{(2)} \big[(a^{X, Z} \leftarrow k_{(3)})_{[1]}b_{[1](0)}^{X, Z}\big]_{(1)} \# \big[(a^{Z, Y} \leftarrow k_{(3)})_{[2]} \leftarrow b_{[1](1)}^{X, Z}\big]b_{[2]}^{Z, Y}\\
		&(\text{by applying \eqref{Delta_mor} to $(a^{X, Z} \leftarrow k_{(2)})$ and $b^{X,Z}$})\\
		&= h_{(1)}k_{(1)} \# \big[(a^{X, Z}_{[1]} \leftarrow k_{(3)})b_{[1](0)}^{X, Z}\big]_{(0)} \\
		&\quad \otimes h_{(2)}k_{(2)} \big[(a^{X, Z}_{[1]} \leftarrow k_{(3)})b_{[1](0)}^{X, Z}\big]_{(1)} \# \big[(a^{Z, Y}_{[2]} \leftarrow k_{(4)}) \leftarrow b_{[1](1)}^{X, Z}\big]b_{[2]}^{Z, Y}\\
		&(\text{$\Delta^{Z}_{X,Y}$ is $H$-linear})\\
		&= h_{(1)}k_{(1)} \# (a^{X, Z}_{[1]} \leftarrow k_{(3)})_{(0)}b_{[1](0)}^{X, Z} \otimes h_{(2)}k_{(2)} (a^{X, Z}_{[1]} \leftarrow k_{(3)})_{(1)}b_{[1](1)}^{X, Z} \# \big[a^{Z, Y}_{[2]} \leftarrow k_{(4)}b_{[1](2)}^{X, Z}\big]b_{[2]}^{Z, Y}\\
		&= h_{(1)}k_{(1)} \# (a^{X, Z}_{[1](0)} \leftarrow k_{(4)})b_{[1](0)}^{X, Z} \otimes h_{(2)}k_{(2)} S_{H}(k_{(3)})a^{X, Z}_{[1](1)}k_{(5)}b_{[1](1)}^{X, Z} \# \big[a^{Z, Y}_{[2]} \leftarrow k_{(6)}b_{[1](2)}^{X, Z}\big]b_{[2]}^{Z, Y}\\
		&\text{(by applying $\eqref{compatibiliy2}$ to $a^{X, Z}_{[1]}$ and $k_{(3)}$)}\\
		&=h_{(1)}k_{(1)} \# (a^{X, Z}_{[1](0)} \leftarrow k_{(2)})b_{[1](0)}^{X, Z} \otimes h_{(2)}a^{X, Z}_{[1](1)}k_{(3)}b_{[1](1)}^{X, Z} \# \big[a^{Z, Y}_{[2]} \leftarrow k_{(4)}b_{[1](2)}^{X, Z}\big]b_{[2]}^{Z, Y}\\
		&= (h_{(1)} \# a^{X, Z}_{[1](0)})\cdot (k_{(1)} \# b^{X, Z}_{[1](0)}) \otimes (h_{(2)}a^{X, Z}_{[1](1)} \# a^{Z, Y}_{[2]})\cdot (k_{(2)}b^{X, Z}_{[1](1)} \# b^{Z, Y}_{[2]}).
	\end{align*}
	Next, for any $X, Y, Z, T \in \ob(\Cc)$ and $a^{X, Y} \in \Cc(X, Y), h\in H$, we have
	\begin{align*}
		\big(\widetilde{\Delta}^{T}_{X,Z} \otimes 1\big)&\widetilde{\Delta}^{Z}_{X,Y} (h \# a^{X, Y}) = \big(\widetilde{\Delta}^{T}_{X,Z} \otimes 1\big)\big((h_{(1)} \# a^{X, Z}_{[1](0)}) \otimes (h_{(2)}a^{X, Z}_{[1](1)} \# a^{Z, Y}_{[2]})\big)\\
		&=(h_{(1)} \# a^{X, T}_{[1](0)[1](0)}) \otimes  (h_{(2)}a^{X, T}_{[1](0)[1](1)} \# a^{T, Z}_{[1](0)[2]}) \otimes \big(h_{(3)}a^{X, Z}_{[1](1)} \# a^{Z, Y}_{[2]}\big).
	\end{align*}
	Applying the $H$-colinearity of $\Delta^{T}_{X,Z}$ \eqref{H_colinear1} to $a^{X, Z}_{[1]}$, we obtain
	\begin{align*}
		\big(\widetilde{\Delta}^{T}_{X,Z} \otimes 1\big)&\widetilde{\Delta}^{Z}_{X,Y} (h \# a^{X, Y}) =(h_{(1)} \# a^{X, T}_{[1](0)}) \otimes  (h_{(2)}a^{X, T}_{[1](1)} \# a^{T, Z}_{[2](0)}) \otimes \big(h_{(3)}a^{X, T}_{[1](2)}a^{T, Z}_{[2](1)} \# a^{Z, Y}_{[3]}\big)\\
		&= \big(1 \otimes \widetilde{\Delta}^{Z}_{T,Y}\big)\widetilde{\Delta}^{T}_{X,Y} (h \# a^{X, Y}).
	\end{align*}
	It is not difficult to check that, for any $X \in \ob(\Cc)$, the map $\widetilde{\varepsilon}_{X}$ is a morphism of algebras and satisfies the $k$-cocategory axioms as follows:
	\begin{align*}
		\big(1 \otimes \widetilde{\varepsilon}_{Y}\big)&\widetilde{\Delta}^{Y}_{X,Y} (h \# a^{X, Y}) = \big(1 \otimes \widetilde{\varepsilon}_{Y}\big)\big((h_{(1)} \# a^{X, Y}_{[1](0)}) \otimes (h_{(2)}a^{X, Y}_{[1](1)} \# a^{Y, Y}_{[2]})\big)\\
		&= (h_{(1)} \# a^{X, Y}_{[1](0)}) \otimes \varepsilon_{H}\big((h_{(2)}a^{X, Y}_{[1](1)}\big) \# \varepsilon_{Y}(a^{Y, Y}_{[2]})\\
		&= h \# a^{X, Y}.
	\end{align*}
	Similarly, $\big(\widetilde{\varepsilon}_{X} \otimes 1\big)\widetilde{\Delta}^{X}_{X,Y} (h \# a^{X, Y}) = h \# a^{X, Y}.$ 
	\epf
	\bl \label{Lem : antipode}
	For any $X, Y \in \ob(\Cc)$, the linear maps $\widetilde{S}_{X,Y} : H \# \Cc(X, Y) \longrightarrow H \# \Cc(Y, X)$, defined by
	\begin{align*}
		\widetilde{S}_{X,Y} (h \# a^{X, Y}) &= \big(1 \# S_{X, Y}(a_{(0)}^{X, Y})\big) \cdot \big(S_{H}(ha_{(1)}^{X, Y}) \# 1\big)\\
		&=S_{H}(h_{(2)}a_{(2)}^{X, Y}) \# \big[S_{X, Y}(a_{(0)}^{X, Y}) \leftarrow S_{H}(h_{(1)}a_{(1)}^{X, Y})\big],
	\end{align*}
	satisfy the additional cogroupoid axioms.
	\el
	\bpf
	For $a^{X, X} \in \Cc(X, X), h \in H$, we have
	\begin{align*}
		\widetilde{m}(\widetilde{S}_{X,Y} \otimes 1) &\Delta^{Y}_{X,X}(h \# a^{X, X}) = \widetilde{S}_{X,Y}\big(h_{(1)} \# a^{X, Y}_{[1](0)}\big) \cdot (h_{(2)}a^{X, Y}_{[1](1)} \# a^{Y, X}_{[2]})\\
		&= \big[S_{H}\big(h_{(2)}a_{[1](2)}^{X, Y}\big) \# S_{X, Y}\big(a_{[1](0)}^{X, Y}\big) \leftarrow S_{H}(h_{(1)}a_{[1](1)}^{X, Y})\big]\cdot (h_{(3)}a^{X, Y}_{[1](3)} \# a^{Y, X}_{[2]})\\
		&= S_{H}\big(h_{(2)}a_{[1](2)}^{X, Y}\big)h_{(3)}a^{X, Y}_{[1](3)} \# \big[S_{X, Y}\big(a_{[1](0)}^{X, Y}\big) \leftarrow S_{H}(h_{(1)}a_{[1](1)}^{X, Y})h_{(4)}a^{X, Y}_{[1](4)} \big]a^{Y, X}_{[2]}\\
		&= \varepsilon_{H}(h) \varepsilon_{X}(a^{X, X}).\\
		\widetilde{m}(1 \otimes \widetilde{S}_{Y, X}) &\Delta^{Y}_{X,X}(h \# a^{X, Y}) = \big(h_{(1)} \# a^{X, Y}_{[1](0)}\big) \cdot \widetilde{S}_{X,Y}\big(h_{(2)}a^{X, Y}_{[1](1)} \# a^{Y, X}_{[2]}\big)\\
		&= \big(h_{(1)} \# a^{X, Y}_{[1](0)}\big) \cdot \big[S_{H}\big(h_{(3)}a^{X, Y}_{[1](2)}a^{Y, X}_{[2](2)}\big) \# S_{Y, X}\big(a^{Y, X}_{[2](0)}\big) \leftarrow S_{H}\big(h_{(2)}a^{X, Y}_{[1](1)}a^{Y, X}_{[2](1)}\big)\big]
	\end{align*}
	By applying $(1 \otimes 1 \otimes \Delta_{H})$ to \eqref{H_colinear}, we get
	\begin{align*}
		&\widetilde{m}(1 \otimes \widetilde{S}_{Y, X}) \Delta^{Y}_{X,X}(h \# a^{X, Y}) = \big(h_{(1)} \# a^{X, Y}_{(0)[1]}\big) \cdot \big[S_{H}\big(h_{(3)}a^{X, X}_{(2)}\big) \# S_{Y, X}\big(a^{Y, X}_{(0)[2]}\big)\leftarrow S_{H}\big(h_{(2)}a^{X, X}_{(1)}\big)\big]\\
		&= h_{(1)}\big(S_{H}\big(h_{(3)}a^{X, X}_{(2)}\big)\big)_{(1)} \# \left(a^{X, Y}_{(0)[1]}\leftarrow \big(S_{H}\big(h_{(3)}a^{X, X}_{(2)}\big)\big)_{(2)} \right)\left(S_{Y, X}\big(a^{Y, X}_{(0)[2]}\big)\leftarrow S_{H}\big(h_{(2)}a^{X, X}_{(1)}\big)\right)\\
		&= h_{(1)}S_{H}\big(h_{(4)}a^{X, X}_{(3)}\big) \# \left(a^{X, Y}_{(0)[1]}\leftarrow S_{H}\big(h_{(3)}a^{X, X}_{(2)}\big) \right)\left(S_{Y, X}\big(a^{Y, X}_{(0)[2]}\big)\leftarrow S_{H}\big(h_{(2)}a^{X, X}_{(1)}\big)\right)\\
		&= h_{(1)}S_{H}\big(h_{(3)}a^{X, X}_{(2)}\big) \# \left(a^{X, Y}_{(0)[1]}S_{Y, X}\big(a^{Y, X}_{(0)[2]}\big)\right)\leftarrow S_{H}\big(h_{(2)}a^{X, X}_{(1)}\big)\quad (\text{because $m_{\Cc(X, Y)}$ is $H$-linear})\\
		&= \varepsilon_{H}(h)\varepsilon_{X}(a^{X, X}). \qedhere
	\end{align*}
	\epf
	We summarize the previous constructions in the following result:
	\btp \label{Thm-Def: boso}
	Let $H$ be a Hopf algebra with bijective antipode. Let $\Cc$ be a $\YD^{H}_{H}$-cogroupoid. The bosonization $H \# \Cc$ is the $k$-cogroupoid defined as follows:
	\begin{itemize}
		\item $\ob(H \# \Cc) = \ob(\Cc)$,
		\item for any $X, Y \in \ob(\Cc)$, the algebra $(H \# \Cc)(X, Y)$ is $H \# \Cc(X, Y)$ with the structural maps $\widetilde{\Delta}^{\bullet}_{\bullet, \bullet}, \widetilde{\varepsilon}_{\bullet}$ and $\widetilde{S}_{\bullet, \bullet}$ defined in Lemmas \ref{algebra} and \ref{Lem : antipode}.
	\end{itemize}
	\etp

	\bt
	Let $H$ be a Hopf algebra with bijective antipode. Let $\Cc$ be a faithfully flat $\YD^{H}_{H}$-cogroupoid. Then, for any $X, Y \in \ob(\CC)$, we have $k$-linear monoidal equivalences
	\[(\YD^H_{H})^{\Cc(X,X)} \cong^{\otimes} (\YD^H_{H})^{\Cc(Y,Y)} \quad \text{and} \quad \MM^{H \# \CC(X, X)} \cong^{\otimes} \MM^{H \# \CC(Y, Y)}.\]
	\et
	\bpf
	The first equivalence follows directly from Theorem \ref{Thm: monoid func}, while the second one follows from \cite[Theorem 2.12]{MR3285340}, since the bosonization is a $k$-cogroupoid by Theorem-Definition \ref{Thm-Def: boso}.
	\epf
	Since, by Remark~\ref{rk:YD-coquasi}, the category of $H$-comodules over a coquasitriangular Hopf algebra $H$ embeds into the category of Yetter--Drinfeld modules over $H$, we deduce the following result:
	\bc \label{Cor : equiva boso}
	Let $H$ be a coquasitriangular Hopf algebra. Let $\Cc$ be a faithfully flat $\MM^{H}$-cogroupoid. Then, for any $X, Y \in \ob(\CC)$, there exists a $k$-linear monoidal equivalence 
	\[\MM^{H \# \CC(X, X)} \cong^{\otimes} \MM^{H \# \CC(Y, Y)}.\]
	\ec
	\brq
	Note that when $H$ is a coquasitriangular Hopf algebra and $\CC$ is a $\MM^{H}$-cogroupoid, then by Theorem \ref{Thm: monoid func}, for any $X, Y \in \ob(\CC)$, we obtain a $k$-linear equivalence:
	\[\big(\MM^{H}\big)^{\CC(X, X)} \cong^{\otimes} \big(\MM^{H}\big)^{\CC(Y, Y)}.\]
	Moreover, by \cite[Propostion 1.2]{HN24} or \cite [Theorem 9.4.12]{majid1997foundations}, we have 
	\[\big(\MM^{H}\big)^{\CC(X, X)} \cong^{\otimes} \MM^{H \# \CC(X, X)} \]
	Thus, the monoidal equivalence in the above proposition follows as well.
	\erq
	So far, by Theorem–Definition~\ref{Def: Cogr coinva}, we have seen that given a cogroupoid triple $\KK$, we obtain a braided cogroupoid over $\YD_{H}^{H}$, namely its cogroupoid of coinvariants. The following theorem, where the notion of isomorphism between cogroupoids is the obvious one, shows that $\KK$ can be reconstructed from its coinvariants via bosonization.
	\bt \label{Thm :  isoboso}
	Let $H$ be a Hopf algebra, and let $(\KK, \pi, i)$ be a cogroupoid triple over $H$. Then, for any $X, Y \in \ob(\KK)$, we have an isomorphism
	\begin{align*}
		\Theta \colon \KK(X, Y) &\longrightarrow H \# \,^{co\pi_{X}}\KK(X, Y)\\
		a^{X, Y} &\longmapsto \pi_{X}(a^{X, X}_{[1]}) \# \vartheta(a^{X, Y}_{[2]}),
	\end{align*}
	where $\vartheta$ is the map defined in Lemma \ref{Lem : prop of v}, and this provides an isomorphism between the cogroupoids $\KK$ and $H \# \,^{co\pi}\KK$.
	\et
	\bpf
	To begin with, we verify that $\Theta$ is a morphism of algebras: for $a^{X, Y}, b^{X, Y} \in \KK(X, Y),$
	\begin{align*}
		\Theta (a^{X, Y}b^{X, Y}) &= \pi_{X}(a^{X, X}_{[1]}b_{[1]}^{X, X}) \# \vartheta (a_{[2]}^{X, Y}b_{[2]}^{X, Y})\\
		&= \pi_{X}(a^{X, X}_{[1]}b_{[1]}^{X, X}) \# i_{X, Y}\pi_{X} S_{X, X}(b_{[2]}^{X, X})\vartheta(a_{[2]}^{X, Y})b_{[3]}^{X, X} \quad (\text{by Lemma \ref{Lem : prop of v}}),
	\end{align*}
	and 
	\begin{align*}
		&\Theta (a^{X, Y})\Theta(b^{X, Y}) = \big(\pi_{X}(a^{X, X}_{[1]}) \# \vartheta(a_{[2]}^{X, Y})\big)\big(\pi_{X}(b^{X, X}_{[1]}) \# \vartheta(b_{[2]}^{X, Y})\big)\\
		&= \pi_{X}(a^{X, X}_{[1]})\pi_{X}(b^{X, X}_{[1]})_{(1)} \# \big(\vartheta(a_{[2]}^{X, Y}) \leftarrow \pi_{X}(b^{X, X}_{[1]})_{(2)}\big)\vartheta(b_{[2]}^{X, Y})\\
		&= \pi_{X}(a^{X, X}_{[1]})\pi_{X}(b^{X, X}_{[1]}) \# \big(\vartheta(a_{[2]}^{X, Y}) \leftarrow \pi_{X}(b^{X, X}_{[2]})\big)\vartheta(b_{[3]}^{X, Y})\\
		&= \pi_{X}(a^{X, X}_{[1]})\pi_{X}(b^{X, X}_{[1]}) \# S_{Y, X}i_{Y,X}\pi_{X}(b^{X, X}_{[2]})\vartheta(a_{[2]}^{X, Y}) i_{X,Y}\pi_{X}(b^{X, X}_{[3]})\vartheta(b_{[4]}^{X, Y})\\
		&= \pi_{X}(a^{X, X}_{[1]})\pi_{X}(b^{X, X}_{[1]}) \# S_{Y, X}i_{Y,X}\pi_{X}(b^{X, X}_{[2]})\vartheta(a_{[2]}^{X, Y}) i_{X,Y}\pi_{X}(b^{X, X}_{[3]})i_{X, Y}\pi_{X}S_{X, X}(b_{[4]}^{X, X})b_{[5]}^{X, Y}\\
		&= \pi_{X}(a^{X, X}_{[1]})\pi_{X}(b^{X, X}_{[1]}) \# S_{Y, X}i_{Y,X}\pi_{X}(b^{X, X}_{[2]})\vartheta(a_{[2]}^{X, Y}) b_{[3]}^{X, Y}\\
		&= \Theta (a^{X, Y}b^{X, Y}).
	\end{align*}
	Define
	\begin{align*}
		\Theta^{-1} \colon H \# \,^{co\pi_{X}}\KK(X, Y) &\longrightarrow \KK(X, Y)\\
		h \# a^{X, Y} &\longmapsto i_{X, Y}(h)a^{X, Y}.
	\end{align*}
	We have
	\begin{align*}
		\Theta^{-1} \circ \Theta(a^{X, Y}) &= i_{X, Y}\pi_{X}(a^{X, X}_{[1]})i_{X, Y} \pi_{X} S_{X, X}(a_{[2]}^{X, X}) a_{[3]}^{X, Y}\\
		&= i_{X, Y}\pi_{X}\big(a^{X, X}_{[1]}S_{X, X}(a_{[2]}^{X, X})\big)a_{[3]}^{X, Y}\\
		&= a^{X, Y},
	\end{align*}
	and
	\begin{align*}
		&\Theta\circ \Theta^{-1} (h \# a^{X, Y}) = \Theta\big(i_{X, Y}(h)a^{X, Y}\big)\\
		&= \pi_{X}(i_{X, Y}(h)^{X, X}_{[1]}a^{X, X}_{[1]}) \# i_{X, Y} \pi_{X} S_{X, X}(i_{X, Y}(h)^{X, X}_{[2]}a_{[2]}^{X, X}) i_{X, Y}(h)^{X, Y}_{[3]}a_{[3]}^{X, Y}\\
		&= \pi_{X}(i_{X, X}(h_{(1)})a^{X, X}_{[1]}) \# i_{X, Y} \pi_{X} S_{X, X}(i_{X, X}(h_{(2)})a_{[2]}^{X, X}) i_{X, Y}(h_{(3)})a_{[3]}^{X, Y}\\
		&= h_{(1)}\pi_{X}(a^{X, X}_{[1]}))  \# i_{X, Y} \pi_{X} S_{X, X}(a_{[2]}^{X, X}) i_{X, Y} S_{H}(h_{(2)})i_{X, Y}(h_{(3)})a_{[3]}^{X, Y}\\
		&= h \pi_{X}(a^{X, X}_{[1]})) \# i_{X, Y} \pi_{X} S_{X, X}(a_{[2]}^{X, X}) a_{[3]}^{X, Y}.
	\end{align*}
	Moreover, since $a^{X, Y} \in \,^{co\pi_{X}}\KK(X, Y)$, we have $(\pi_{X} \otimes 1)\Delta^{X}_{X, Y}(a^{X, Y}) = 1 \otimes a^{X, Y}$. Thus
	\[\Theta\circ \Theta^{-1} (h \# a^{X, Y}) =  h \# i_{X, Y} \pi_{X} S_{X, X}(a_{[1]}^{X, X}) a_{[2]}^{X, Y} = h \# a^{X, Y}.\]
	It follows that $\Theta$ is bijective. To conclude, we note that $\Theta$ also fits into the following commutative diagram:
	\[\begin{tikzcd}
		\KK(X, Y) \ar[r, "\Delta^{Z}_{X, Y}"] \ar[d, "\Theta"] &\KK(X, Z) \otimes \KK(Z, Y) \ar[d, "\Theta \otimes \Theta"]\\
		H \# \,^{co\pi_{X}}\KK(X, Y) \ar[r, "\widetilde{\Delta}^{Z}_{X, Y}"] &\big(H \# \,^{co\pi_{X}}\KK(X, Z) \big)\otimes \big(H \# \,^{co\pi_{X}}\KK(Z, Y)\big).
	\end{tikzcd}
	\]
	Indeed, \[(\Theta \otimes \Theta)\Delta^{Z}_{X, Y}(a^{X, Y}) = \big(\pi_{Z}(a^{X, X}_{[1]}) \# \vartheta(a^{X, Z}_{[2]}) \big)\otimes \big(\pi_{X}(a^{Z, Z}_{[3]}) \# \vartheta(a^{Z, Y}_{[4]})\big).\]
	On the other hand,
	\begin{align*}
		&\widetilde{\Delta}^{Z}_{X, Y} \Theta(a^{X, Y}) = \widetilde{\Delta}^{Z}_{X, Y}\big(\pi_{X}(a^{X, X}_{[1]}) \# \vartheta(a^{X, Y}_{[2]}\big)\\
		&=\Big(\pi_{X}(a^{X, X}_{[1]})_{(1)} \# \vartheta(a^{X, Z}_{[2]})_{(0)}\Big) \otimes \Big(\pi_{X}(a^{X, X}_{[1]})_{(2)}\vartheta(a^{X, Z}_{[2]})_{(1)} \# \vartheta(a^{Z, Y}_{[3]})\Big)
	\end{align*}
	Using the first statement of Lemma~\ref{Lem : prop of v}, the $H$-coaction on $\vartheta(a^{X, Z})$ is given by:
	\[\vartheta(a^{X, Z}) \mapsto \vartheta(a_{[2]}^{X, Z}) \otimes \pi_{X}S_{X, X}(a^{X, X}_{[1]})\pi_{Z}(a^{Z, Z}_{[3]}).\]
	Thus
	\begin{align*}
		&\widetilde{\Delta}^{Z}_{X, Y} \Theta(a^{X, Y})\\
		&= \Big(\pi_{X}(a^{X, X}_{[1]})_{(1)} \# \vartheta(a^{X, Z}_{[3]})\Big) \otimes \Big(\pi_{X}(a^{X, X}_{[1]})_{(2)} \pi_{X}S_{X, X}(a^{X, X}_{[2]})\pi_{Z}(a^{Z, Z}_{[4]}) \# \vartheta(a^{Z, Y}_{[5]})\Big)\\
		&= \Big(\pi_{X}(a^{X, X}_{[2]}) \# \vartheta(a^{X, Z}_{[4]})\Big) \otimes \Big(\pi_{X}(a^{X, X}_{[2]}) \pi_{X}S_{X, X}(a^{X, X}_{[3]})\pi_{Z}(a^{Z, Z}_{[5]}) \# \vartheta(a^{Z, Y}_{[6]})\Big)\\
		&= \Big(\pi_{X}(a^{X, X}_{[2]}) \# \vartheta(a^{X, Z}_{[2]})\Big) \otimes \Big(\pi_{Z}(a^{Z, Z}_{[3]}) \# \vartheta(a^{Z, Y}_{[4]})\Big),
	\end{align*}
	which is exactly the expression of $(\Theta \otimes \Theta)\Delta^{Z}_{X, Y}(a^{X, Y})$, and this completes our proof.
	\epf
	%
	
	\subsection{Examples : braided SL\texorpdfstring{$_{n}$}{SLn} cogroupoid} \label{Ex : GLn}
	In this section, we construct the braided $\mathrm{SL}_n$ cogroupoid from the cogroupoid associated with the multiparameter quantum $\mathrm{GL}_n$ $(n \in \mathbb{N}^{*})$ and its cogroupoid of coinvariants. To this end, we begin by recalling the $2$-cocycle cogroupoid of a Hopf algebra introduced in \cite[Section~3.3]{MR3285340}.
	\subsubsection{The $2$-cocycle cogroupoid of a Hopf algebra}
	Let $A$ be a Hopf algebra. Recall \cite{MR1213985} that a $2$-cocycle on $A$ is a convolution invertible linear map $\sigma \colon A \otimes A \to k$ satisfying
	\[\sigma(x_{(1)}, y_{(1)}) \sigma(x_{(2)}y_{(2)}, z) = \sigma(y_{(1)}, z_{(1)})\sigma(x, y_{(2)}z_{(2)})\]
	and $\sigma(x, 1) = \sigma(1, x) = \varepsilon(x)$, for all $x, y, z \in A$. The set of $2$-cocyles on $A$ is denoted by $Z^{2}(A)$.
	
	We also recall that $_{\sigma}A$ denotes the algebra whose underlying vector space is $A$, endowed with the product
	\[[x]\ast [y] = \sigma(x_{(1)}, y_{(1)})x_{(2)}y_{(2)} \quad \text{for $x, y \in A$}.\]
	
	Let $\sigma, \tau \in Z^{2}(A)$. The algebra $A(\sigma, \tau)$ is the algebra having $A$ as underlying vector space and whose product is defined by
	\[[x]\ast [y] = \sigma(x_{(1)}, y_{(1)})x_{(2)}y_{(2)}\tau^{-1}(x_{(3)}, y_{(3)}).\]
	
	For any $\sigma, \tau, \rho \in Z^{2}(A)$, there exist algebra maps:
	\begin{align} \label{cocycle structural}
		\Delta^{\rho}_{\sigma, \tau} = \Delta \colon A(\sigma, \tau) &\longrightarrow A(\sigma, \rho) \otimes A(\rho, \tau) \nonumber\\
		x &\longmapsto x_{(1)} \otimes x_{(2)} \nonumber\\
		\varepsilon_{\sigma} = \varepsilon \colon A(\sigma, \sigma) &\longrightarrow k \\
		S_{\sigma, \tau} \colon A(\sigma, \tau) &\longrightarrow A(\tau, \sigma)^{op}\nonumber\\
		x &\longmapsto \sigma(x_{(1)}, S(x_{(2)})) S(x_{(3)}) \tau^{-1}(S(x_{(4)}), x_{(5)}).\nonumber
	\end{align}
	The $2$-cocycle cogroupoid of $A$, denoted by $\mathcal{Z}^{2}(A)$, is defined as follows \cite{MR3285340} : 
	\begin{itemize}
		\item [(i)] $\ob(\mathcal{Z}^{2}(A)) = Z^{2}(A)$,
		\item[(ii)] for $\sigma, \tau \in Z^{2}(A)$, the algebra $\mathcal{Z}^{2}(A)(\sigma, \tau)$ is the algebra $A(\sigma, \tau)$ defined above,
		\item[(iii)]the structural maps $\Delta^{\bullet}_{\bullet,\bullet}, \varepsilon_{\bullet}$ and $S_{\bullet, \bullet}$ are given in \eqref{cocycle structural}.
	\end{itemize}
	
	We note that when $\tau = 1$, the algebra $A(\sigma, 1)$ is simply the algeba $_{\sigma}A$ given above, and when $\tau = \sigma$, the algebra $A(\sigma, \sigma)$ is the Hopf algebra $A^{\sigma}$ introduced by Doi in \cite{MR1213985}.
	
	\brq \label{Rk : twisted cocycle}
	Let $A$ be a Hopf algebra and let $\Gamma$ be a group.
	Let $\rho \colon A \to k\Gamma$ be a Hopf algebra map. Via $\rho$, we may view $A$ as an algebra in the category of $k\Gamma$-comodules, equivalently as a $\Gamma$-graded algebra. Furthermore, $A$ becomes a $\Gamma \times \Gamma$ - graded algebra
	\[A = \bigoplus_{g, h \in \Gamma}\, _{g}A_{h},\]
	where
	\[_{g}A_{h} = \{a \in A \mid \rho(a_{(1)}) \otimes a_{(2)} = g \otimes a; a_{(1)} \otimes \rho(a_{(2)}) = a \otimes h\}.\]
	Let $\sigma, \tau \in Z^{2}(\Gamma, k^{*})$. Then $\sigma$ and $\tau$ might be viewed as elements of $\mathcal{Z}^{2}(A)$ by composing with $\rho \otimes \rho$, and then the algebra $A(\sigma, \tau)$ is defined to be the vector space $A$ endowed with the multiplication
	\[[x]\ast [y] = \sigma(\rho(x_{(1)}), \rho(y_{(1)}))x_{(2)}y_{(2)}\tau^{-1}(\rho(x_{(3)}), \rho(y_{(3)})) \quad \text{for $x, y \in A$}.\]
	If $x \in \,_{g}A_{h}$ and $y \in \,_{k}A_{l}$ for $g, h, k, l \in \Gamma$, then this product in $A(\sigma, \tau)$  simplifies to
	\[[x]\ast [y] = \sigma(g, k)\tau^{-1}(h, l)xy.\]
	Define a map
	\begin{align*}
		\omega\colon \big(\Gamma \times \Gamma\big) \times \big(\Gamma \times \Gamma\big) &\longrightarrow k^{*}\\
		((g, h), (k, l)) &\longmapsto \sigma(g, k)\tau^{-1}(h, l).
	\end{align*}
	Then $\omega$ is a $2$-cocycle on the group $\Gamma \times \Gamma$. With this notation, the algebra $A(\sigma, \tau)$  is precisely the $2$-cocycle twist algebra $_{\omega}A$.
	\erq

	\subsubsection{The multiparametric $\mathrm{GL}_{q;n}$ cogroupoid}\label{coinGL_n}
	Now, we aim to construct a multiparametric GL$_{q;n}$ cogroupoid ($q \in k^{*}, n \in \N^{*}$) using the construction of $2$-cocycle cogroupoids recalled in the previous subsection.
	
	\bd \label{Def : OqSLn}
	Let $q \in k^{*}$. The single parameter quantum $n \times n$ matrix algebra, denoted by $\Oq$, is the algebra presented by generators $x_{ij}$ for $i, j = 1, \dots, n$, subject to the relations
	\begin{align}
		x_{im} x_{ik} &= q x_{ik} x_{im} &(k < m)\\
		x_{jk} x_{ik} &= q x_{ik} x_{jk} &(i < j)\\
		x_{jm} x_{ik} &=  x_{ik} x_{jm} &(i < j, k > m)\\
		x_{jm} x_{ik} - x_{ik} x_{jm} &= (q - q^{-1})x_{im}x_{jk} &(i <j, k <m). \label{formule det}
	\end{align}
	\ed
	The quantum determinant in $\Oq$ is expressed by
	\[D_{q} = \sum_{\sigma \in S_{n}} (-q)^{-\ell(\sigma)} x_{1\sigma(1)} x_{2\sigma(2)} \cdots  x_{n\sigma(n)}\]
	where $\ell(\sigma)$ denotes the length of the permutation $\sigma$. The quantum determinant is central in $\Oq$, and the single parameter quantum SL$_{n}$ and quantum GL$_{n}$ are defined by
	\[\OQ = \Oq / (D_{q} - 1) \quad \text{and} \quad \Ogl = \Oq [D_{q}^{-1}].\]
	%
	
	Fix $i, j \in \{1, \cdots, n\}$. Let $M_{ji}$ be the $(n - 1) \times (n - 1)$ quantum minor obtained by deleting row $j$ and column $i$ from $(x_{ij})_{1 \leq i < j < \leq n}$ :
	\[M_{ji} = \sum_{\sigma \in S_{n - 1}} (-q)^{-\ell (\sigma)} x_{a_{1}, b_{\sigma(1)}} \cdots  x_{a_{n -1}, b_{\sigma(n - 1)}} \]
	where $a_1 < \cdots < a_{n-1}$ are distinct from $j$, and $b_1 < \cdots < b_{n-1}$ are distinct from $i$.
	
	It is well known (see e.g. \cite{KSBook97,BG02}) that these algebras $\Ogl$ and $ \OQ$ are Hopf algebras, with the antipode given by
	\begin{align*}
		S_{q} \colon \Ogl &\longrightarrow \Ogl\\
		x_{ij} &\longmapsto (-q)^{j - i} M_{ji}D_{q}^{-1}.
	\end{align*}
	
	A matrix $\bar{p} = (p_{ij}) \in M_{n}(k)$ is called an AST-matrix if it satisfies $p_{ii} = 1$ and $p_{ij}p_{ji} = 1$ for all $i$ and $j$, following Artin-Shelter-Tate \cite{MR1127037}. We denote by $\mathrm{AST}(n)$ the set of AST matrices of size $n$.
	\bd
	Let $q \in k^{*}$ and $\bar{p} \in \mathrm{AST}(n)$. The algebra $\Opgl$ is the algebra presented by generators $x_{ij}, 1 \leq i,j \leq n$ and $(\Dp_{q})^{-1}$, subject to the relations
	\begin{align*}
		&x_{im}x_{ik} = q p_{mk}x_{ik} x_{im}\quad (k < m)\\
		&x_{jk}x_{ik}  = qp_{ij} x_{ik} x_{jk} x_{ik} \quad (i < j)\\
		&p_{ji} x_{jm} x_{ik} =  p_{mk} x_{ik} x_{jm} \quad (i < j, k > m)\\
		&p_{ji}x_{jm}x_{ik} - p_{mk}x_{ik}x_{jm} = (q - q^{-1})x_{im}x_{jk} \quad (i <j, k <m)\\
		1 &= (\Dp_{q})^{-1}\Bigg(\sum_{\sigma \in S_{n}} (-q)^{-\ell(\sigma)}\Bigg(\prod_{\substack{1 \leq i< j \leq n\\ \sigma(i) > \sigma(j)}} (p_{\sigma(j) \sigma(i)})\Bigg)x_{1\sigma(1)} x_{2\sigma(2)} \cdots  x_{n\sigma(n)}\Bigg) \\
		& = \Bigg(\sum_{\sigma \in S_{n}} (-q)^{-\ell(\sigma)}\Bigg(\prod_{\substack{1 \leq i< j \leq n\\ \sigma(i) > \sigma(j)}} (p_{\sigma(j) \sigma(i)})\Bigg)x_{1\sigma(1)} x_{2\sigma(2)} \cdots  x_{n\sigma(n)}\Bigg)(\Dp_{q})^{-1}.
	\end{align*}
	\ed
	
	We note that our definition of $\Ogl$ corresponds to $\mathcal{O}_{q^{-1}}(\mathrm{GL}_{n}(k))$ in the conventions of Brown–Goodearl \cite{BG02}, and that our definition of $\mathcal{O}_{q}^{ \bar{p}}(\mathrm{GL}_{n}(k))$, with $\tilde{p}_{ij} = p_{ij}^{-1}$, corresponds to $\mathcal{O}_{q^{-1},
		\widetilde{p}}(\mathrm{GL}_{n}(k))$ in the sense of Hayashi \cite{HAYASHI1992146}.
	\bd\label{Oqpr}
	Let $q \in k^{*}$ and $\bar{p}, \bar{r} \in \mathrm{AST}(n)$. The algebra $\Opr$ is the algebra presented by generators $x_{ij}, 1 \leq i,j \leq n$ and $(\Dpr_{q})^{-1}$, subject to relations
	\begin{align}
		& x_{im}x_{ik} = q r_{mk}x_{ik} x_{im}\quad (k < m) \label{line}\\
		&x_{jk}x_{ik}  = q p_{ij} x_{ik} x_{jk} \quad (i < j) \label{column}\\
		&p_{ji} x_{jm} x_{ik} = r_{mk} x_{ik} x_{jm} \quad (i < j, k > m) \label{cross1}\\
		&p_{ji}x_{jm}x_{ik} - r_{mk}x_{ik}x_{jm} = (q - q^{-1})x_{im}x_{jk} \quad (i <j, k <m) \rangle. \label{cross2}\\
		1 &= (\Dpr_{q})^{-1}\Bigg(\sum_{\sigma \in S_{n}} (-q)^{-\ell(\sigma)}\prod_{1 \leq i< j \leq n} p_{ij}\prod_{\substack{1 \leq i< j \leq n\\ \sigma(i) < \sigma(j)}}r_{\sigma(j), \sigma(i)} \prod_{i = 1}^{n} x_{i \sigma(i)}\Bigg) \label{detpr} \\
		& = \Bigg(\sum_{\sigma \in S_{n}} (-q)^{-\ell(\sigma)}\prod_{1 \leq i< j \leq n} p_{ij}\prod_{\substack{1 \leq i< j \leq n\\ \sigma(i) < \sigma(j)}}r_{\sigma(j), \sigma(i)} \prod_{i = 1}^{n} x_{i \sigma(i)}\Bigg)(\Dpr_{q})^{-1}. \nonumber
	\end{align}
	\ed
	
	When $\bar{p} = \bar{r}$, it is clear that $\mathcal{O}^{\bar{p}, \bar{p}}_{q}(\mathrm{GL}_{n}(k)) = \mathcal{O}^{\bar{p}}_{q}(\mathrm{GL}_{n}(k))$.
	
	From now on, we set $A = \mathcal{O}_{q}(\mathrm{GL}_{n}(k))$. The following lemma is verified by a straightforward computation:
	\bl \label{lem:morphism pi}
	Let $\Z^{n} = \langle g_{1}, \cdots, g_{n}\rangle$ be the free abelian group of rank $n$. There exists a unique algebra map:
	\begin{align*}
		\rho\colon A &\longrightarrow k\Z^{n}\\
		x_{ij} &\longmapsto \delta_{ij}g_{i}.
	\end{align*}
	\el

	Let $\bar{p}, \bar{r} \in \text{AST}(n)$. We consider the unique bicharacters
	\begin{align} \label{cocycle of kZ}
		\psi_{\bar{p}} \colon \Z^{n} \times \Z^{n} &\longrightarrow k^{*} &\psi_{\bar{r}} \colon \Z^{n} \times \Z^{n} &\longrightarrow k^{*}  \nonumber\\
		(g_{i}, g_{j}) &\mapsto \begin{cases}
			p_{ji} &\text{for} \, i < j\\
			1 & \text{for} \, i \geq j.
		\end{cases}
		&(g_{i}, g_{j}) &\mapsto \begin{cases}
			r_{ji} &\text{for} \, i < j\\
			1 & \text{for} \, i \geq j.
		\end{cases}
	\end{align} 
	This defines convolution-invertible $2$-cocyles on the group algebra $k\Z^{n}$. Hence
	\[\tau_{p}:= \psi_{\bar{p}}\circ (\rho \otimes \rho) \colon A \otimes A \to k^{*}\] is a $2$-cocycle on $A$, and $\tau_{r} := \psi_{\bar{r}} \circ (\rho \otimes \rho)$ is as well.
	
	\bp \label{Prop: isoApr}
	Let $q \in k^{*}$ and $\bar{p}, \bar{r} \in \mathrm{AST}(n)$. The algebra $\Opr$ is isomorphic to the twisted algebra $A(\tau_{p}, \tau_{r})$.
	\ep
	\bpf
	In $A(\tau_{p}, \tau_{r})$, we have 
	\begin{itemize}
		\item for $k <m$,
		\begin{align*}
			[x_{im}]*[x_{ik}] = \psi_{p}(g_{i}, g_{i})[x_{im}x_{ik}] \psi^{-1}_{r}(g_{m}, g_{k}) = [x_{im}x_{ik}] = q[x_{ik}x_{im}],
		\end{align*}
		and 
		\begin{align*}
			[x_{ik}]*[x_{im}] = \psi_{p}(g_{i}, g_{i})[x_{ik}x_{im}] \psi^{-1}_{r}(g_{k}, g_{m}) = r^{-1}_{mk}[x_{ik}x_{im}].
		\end{align*}
		Thus $[x_{im}]*[x_{ik}] = qr_{mk}[x_{ik}]*[x_{im}]$;
		\item for $i < j$ and $k <m$,
		\begin{align*}
			[x_{jm}]*[x_{ik}] = \psi_{p}(g_{j}, g_{i})[x_{jm}x_{ik}] \psi^{-1}_{r}(g_{m}, g_{k}) = [x_{jm}x_{ik}],
		\end{align*}
		\begin{align*}
			[x_{ik}]*[x_{jm}] = \psi_{p}(g_{i}, g_{j})[x_{ik}x_{jm}] \psi^{-1}_{r}(g_{k}, g_{m}) = p_{ji}r^{-1}_{mk}[x_{ik}x_{jm}].
		\end{align*}
		and 
		\begin{align*}
			[x_{im}]*[x_{jk}] = \psi_{p}(g_{i}, g_{j})[x_{im}x_{jk}] \psi^{-1}_{r}(g_{m}, g_{k}) = p_{ji}[x_{im}x_{jk}].
		\end{align*}
		Hence $p_{ji}[x_{jm}]*[x_{ik}] - r_{mk}[x_{ik}]*[x_{jm}] = (q - q^{-1})[x_{im}]*[x_{jk}]$.
		Similarly, we also have \[[x_{jk}]*[x_{ik}] = qp_{ij}[x_{ik}]*[x_{jk}] \quad (i <j),\] and \[p_{ji} [x_{jm}]*[x_{ik}] = r_{mk} [x_{ik}]*[x_{jm}] \quad (i < j, k > m).\]
		\item for a fixed $\sigma \in S_{n}$,
		proceeding by induction on $k \in \llbracket 1, n \rrbracket$, we can show that
		\begin{align*}
			&[x_{1\sigma(1)}]*[x_{2\sigma(2)}]*\cdots*[x_{k\sigma(k)}] \\
			&= \prod_{1 \leq i < j \leq k} \psi_{p}(g_{i}, g_{j})[x_{1\sigma(1)}x_{2\sigma(2)} \cdots x_{k\sigma(k)} ]\prod_{1 \leq i < j \leq k} \psi^{-1}_{r}(g_{\sigma(i)}, g_{\sigma(j)})\\
			&= \prod_{1 \leq i < j \leq k} p_{ji}[x_{1\sigma(1)}x_{2\sigma(2)} \cdots x_{k\sigma(k)}] \prod_{\substack{1 \leq i< j \leq k\\ \sigma(i) < \sigma(j)}}r^{-1}_{\sigma(j), \sigma(i)}.
		\end{align*}
		Then we obtain 
		\[[D_{q}] =\sum_{\sigma \in S_{n}} (-q)^{-\ell(\sigma)} (\prod_{1 \leq i < j \leq n} p_{ij})(\prod_{\substack{1 \leq i< j \leq k\\ \sigma(i) < \sigma(j)}}r_{\sigma(j), \sigma(i)})[x_{1\sigma(1)}]*\cdots*[x_{n\sigma(n)}].\]
	\end{itemize}
	These relations ensure that there exists a unique algebra morphism
	\begin{align*}
		\Opr &\longrightarrow A(\tau_{p}, \tau_{r})\\
		x_{ij} &\longmapsto [x_{ij}]\\
		(\Dpr_{q})^{-1} &\longmapsto [D_{q}]^{-1}
	\end{align*}
	which is indeed an isomorphism. This follows by combining Remark \ref{Rk : twisted cocycle} with the results of \cite[Section~3]{MR2945585}.
	\epf
	When $\bar{p} = \bar{r}$, a simple substitution in Definition~\ref{Oqpr}, together with the following lemma, shows that we recover the Hopf algebra $A(\tau_{p}, \tau_{p}) = A^{\tau_{p}}$, which is isomorphic to $\mathcal{O}_{q}^{\bar{p}}(\mathrm{GL}_{n}(k))$.
	\bl \label{Lem : p,r}
	Let $\sigma \in S_{n}$ and $\bar{p} = (p_{ij}) \in \mathrm{AST}(n)$. Then
	\[\prod_{1 \leq i< j \leq n} p_{ij}\prod_{\substack{1 \leq i< j \leq n\\ \sigma(i) < \sigma(j)}}p_{\sigma(j), \sigma(i)} = \prod_{\substack{1 \leq i< j \leq n\\ \sigma(i) > \sigma(j)}} p_{\sigma(j), \sigma(i)}.\]
	\el
	\bpf
	The proof is a simple verification and is left to the reader.
	\epf
	The next lemma shows that $\Dpr_{q}$ is a normal element, which will be used in the next subsection (in Proposition~\ref{Prop : coinGLn}).
	\bl \label{Lem : D_commute}
	In $\Opr$, we have, for $1 \leq i, j \leq n$,
	\[(\Dpr_{q})^{-1}x_{ij} = \big(\prod^{n}_{\ell = 1} p_{\ell i}r_{j \ell}\big)x_{ij}(\Dpr_{q})^{-1}.\]
	\el
	\bpf
	Since $D_{q}$ is a group-like and central element in $\mathcal{O}_{q}(\mathrm{M}_{n}(k))$, we observe that, in the twisted algebra $A(\tau_{p}, \tau_{r})$,
	\begin{align*}
		[x_{ij}] * [D_{q}] &= \prod^{n}_{\ell = 1}\psi_{p}(g_{i}, g_{\ell}) [x_{ij}D_{q}] \prod^{n}_{\ell = 1}\psi^{-1}_{r}(g_{j}, g_{\ell})\\
		&= \prod_{\ell > i}p_{\ell i}\prod_{\ell > j}r^{-1}_{\ell j}[D_{q}x_{ij}].
	\end{align*}
	On the other hand, 
	\begin{align*}
		[D_{q}] * [x_{ij}] &= \prod^{n}_{\ell = 1}\psi_{p}(g_{k}, g_{i}) [x_{ij}D_{q}] \prod^{n}_{\ell = 1}\psi^{-1}_{r}(g_{\ell}, g_{j})\\
		&= \prod_{\ell < i}p_{i\ell}\prod_{\ell < j}r^{-1}_{j\ell}[D_{q}x_{ij}]\\
		&= \prod^{n}_{\ell= 1}p_{i\ell}r^{-1}_{j\ell} \,[x_{ij}] * [D_{q}].
	\end{align*}
	Hence, \[[D_{q}]^{-1}* [x_{ij}] = \big(\prod^{n}_{\ell = 1} p_{\ell i}r_{j \ell}\big)x_{ij}[D_{q}]^{-1}.\]
	The proof is then completed using the isomorphism of Proposition \ref{Prop: isoApr}.
	\epf
	\bl \label{Lem : GL structure}
	We denote $\mathcal{O}_{\bar{p}, \bar{r}} =\Opr$.
	\begin{enumerate}
		\item For any $\bar{p}, \bar{r}, \bar{s} \in \mathrm{AST}(n)$, there exists an algebra map
		\begin{align*}
			\Delta^{\bar{s}}_{\bar{p}, \bar{r}} \colon \mathcal{O}_{\bar{p}, \bar{r}} &\longrightarrow \mathcal{O}_{\bar{p}, \bar{s}} \otimes \mathcal{O}_{\bar{s}, \bar{r}}\\
			x_{ij} &\longmapsto \sum^{n}_{k = 1} x_{ik} \otimes x_{kj}.
		\end{align*}
		such that for any $\bar{t}\in \mathrm{AST}(n)$,
		\[(\Delta^{\bar{t}}_{\bar{p}, \bar{s}} \otimes 1)\Delta^{\bar{s}}_{\bar{p}, \bar{r}} = (1 \otimes \Delta^{\bar{t}}_{\bar{s}, \bar{r}})\Delta^{\bar{s}}_{\bar{p}, \bar{r}}.\]
		\item For any $\bar{p} \in \mathrm{AST}(n)$, there exists an algebra map
		\begin{align*}
			\varepsilon_{\bar{p}} \colon \mathcal{O}_{\bar{p}, \bar{p}}&\longrightarrow k\\
			x_{ij} &\longmapsto \delta_{ij}.
		\end{align*}
		such that 
		\[(1 \otimes \varepsilon_{\bar{r}}) \circ \Delta^{\bar{r}}_{\bar{p}, \bar{r}} = \id = (\varepsilon_{\bar{p}} \otimes 1) \circ \Delta^{\bar{p}}_{\bar{p}, \bar{r}}.\]
		\item For any $\bar{p}, \bar{r} \in \mathrm{AST}(n)$, there exists an algebra map
		\begin{align*}
			S_{\bar{p}, \bar{r}} \colon \mathcal{O}_{\bar{p}, \bar{r}} &\longrightarrow \mathcal{O}_{\bar{r}, \bar{p}}^{op}\\
			x_{ij} &\longmapsto (-q)^{j - i} (\prod_{1 \leq s\leq t \leq n}p_{ts}r^{-1}_{ts}) (\prod_{k < i}p^{-1}_{ik})(\prod_{k < j}r_{jk})M^{\bar{r}, \bar{p}}_{ji}(D_{q}^{\bar{r}, \bar{p}})^{-1}.
		\end{align*}
		where $M^{\bar{r}, \bar{p}}_{ji}$ is the multiparameter quantum minor,
		such that
		\[m(1 \otimes S_{\bar{r}, \bar{p}})\Delta^{\bar{s}}_{\bar{p}, \bar{p}} = u \circ \varepsilon_{\bar{p}} \quad \text{and} \quad m(S_{\bar{p}, \bar{r}} \otimes 1)\Delta^{\bar{s}}_{\bar{p}, \bar{p}} = u \circ \varepsilon_{\bar{p}}.\]
	\end{enumerate}
	\el
	\bpf
	Statements $(1)$ and $(2)$ follow from Proposition~\ref{Prop: isoApr} together with the transport of the $2$-cocycle cogroupoid structure. For $(3)$: using the deformation by $2$-cocycles, there exists an algebra map,
	\begin{align*}
		S_{\bar{p}, \bar{r}} \colon A(\tau_{p}, \tau_{r}) &\longrightarrow A(\tau_{r}, \tau_{p})^{op}\\
		x_{ij} &\longmapsto [S_{q}(x_{ij})]
	\end{align*}
	where we still denote $A = \Ogl$ and $S_{q}$ is the antipode of $\Ogl$. Hence
	\[[S_{q}(x_{ij})] = (-q)^{j - i}[M_{ji}D^{-1}_q].\] Moreover, in $A(\tau_{r}, \tau_{p})$
	\begin{align*}
		[M_{ji}]*[D^{-1}_q] &= \prod^{n}_{k = 1}\psi_{\bar{r}}(g_{k}, g_{j})^{-1}[M_{ji}D^{-1}_q] \prod^{n}_{k = 1}\psi_{\bar{p}}(g_{k}, g_{i})\\
		&= \prod_{k < j}r^{-1}_{jk}\prod_{k < i}p_{ik} [M_{ji}D^{-1}_q].
	\end{align*}
	and
	\begin{align*}
		[D_q]\ast[D^{-1}_q] = \tau_{r}\big(D_{q}, D^{-1}_{q}\big)\tau^{-1}_{p}\big(D_{q}, D^{-1}_{q}\big)
		&= \prod^{n}_{s, t = 1}\psi_{\bar{r}}(g_{s}, g_{t})^{-1}\psi_{\bar{p}}(g_{s}, g_{t})\\
		&= \prod_{1 \leq t < s \leq n}p_{ts}r^{-1}_{ts}.
	\end{align*}
	Thus
	\[[M_{ji}D^{-1}_q] = (\prod_{1 \leq s\leq t \leq n}p_{ts}r^{-1}_{ts}) (\prod_{k < i}p^{-1}_{ik})(\prod_{k < j}r_{jk}) [M_{ji}]*[D_q]^{-1}.\]
	The result then follows by applying the isomorphism of Proposition \ref{Prop: isoApr}.
	\epf
	\brq \label{Rk : permute Mij-D}
	Using the same reasoning as in the proof of Lemma \ref{Lem : GL structure}.(3).  
	We can also check that, in $\Opr$, we have 
	\[M^{\bar{r}, \bar{p}}_{ji}(D_{q}^{\bar{r}, \bar{p}})^{-1} = \prod_{k \neq i}p_{ik} \prod_{k \neq j}r^{-1}_{kj}(D_{q}^{\bar{r}, \bar{p}})^{-1} M^{\bar{r}, \bar{p}}_{ji}.\]
	This remark will be useful in Lemma \ref{Lem: SL structure}.
	\erq
	These results lead us to the following definition:
	\bd
	Let $q \in k^{*}$. The multiparametric GL$_{q; n}$-cogroupoid, denoted by \textbf{GL}$_{q; n}$, is defined as follows:
	\begin{itemize}
		\item $\ob(\textbf{GL}_{q; n}) = \mathrm{AST}(n)$
		\item for $\bar{p}, \bar{r} \in \mathrm{AST}(n)$, the algebra $\textbf{GL}_{q; n}(\bar{p}, \bar{r})$ is the algebra $\Opr$ defined in Definition \ref{Oqpr}.
		\item the structure maps $\Delta^{\bullet}_{\bullet, \bullet}, \varepsilon_{\bullet}$ and $S_{\bullet,\bullet}$ given in the previous lemma.
	\end{itemize}
	\ed
	
	When $q = 1$, this is precisely the multiparametric GL$_n$ cogroupoid in \cite[Section 3.4]{MR3285340}. It is immediate to verify that the following result holds.
	\bl \label{lem:i and pi}
Let $q \in k^{*}$, let $\bar{p}, \bar{r} \in \mathrm{AST}(n)$. Consider $\mathbb{Z}=\langle z\rangle$, the infinite cyclic group. Then there exist, respectively, an algebra morphism and a Hopf algebra morphism:
			\begin{align} \label{pi p}
				i_{\bar{p}, \bar{r}} \colon k\mathbb{Z} &\longrightarrow \Opr,
				& \pi_{\bar{p}} \colon \Opgl &\longrightarrow k\mathbb{Z}, \nonumber\\
				z &\longmapsto \Dpr_{q},
				& x_{ij} &\longmapsto 
				\begin{cases}
					z \delta_{ij}, & \text{if } i = 1, \\
					\delta_{ij}, & \text{if } i > 1,
				\end{cases} \\
				& 
				& (\Dpr_{q})^{-1} &\longmapsto z^{-1} \nonumber
			\end{align}
			such that $\pi_{\bar{p}} \circ i_{\bar{p}, \bar{p}} = 1_{k\Z}$. 
				\el 
It follows from this lemma that, considering the family of algebra morphisms
$i := \{\, i_{\bar{p}, \bar{r}} \mid \bar{p}, \bar{r} \in \mathrm{AST}(n) \,\}$,
and the family of Hopf algebra morphisms
$\pi := \{\, \pi_{\bar{p}} \mid \bar{p} \in \mathrm{AST}(n) \,\}$,
we obtain a $\mathrm{GL}_{q;n}$-cogroupoid triple
$(\mathbf{GL}_{q;n}, \pi, i)$, which will be used in the next subsection.
			\subsubsection{Braided SL$_{n}$ cogroupoid} \label{subsec : SL_n}
			Here, we introduce the algebra of multiparameter deformations of SL$_{n}$ by generators and relations. We then show that it can be placed within the cogroupoid framework and that it is, in fact, the cogroupoid of coinvariants associated with the cogroupoid triple (\textbf{GL}$_{q; n}$, $\pi$, $i$).
			\bd
			Let $q \in k^{*}$ and $\bar{p}, \bar{r} \in \mathrm{AST}(n)$. The algebra $\mathcal{O}_{q}^{\bar{p}, \bar{r}}(\mathrm{SL}_{n}(k))$ is the algebra presented by generators $X_{ij}, 1 \leq i,j \leq n$ subject to relations
			\begin{align*}
				X_{1m}X_{1k} &= q(\prod^{n}_{\ell = 1, \ell \neq k} r_{\ell m}r^{-1}_{\ell k}) X_{1k}X_{1m} \quad &\text{$(k < m)$}\\
				X_{im}X_{ik} &= qr_{m k} X_{ik}X_{im} \quad& \text{$(i > 1, k < m)$}\\
				X_{jk}X_{1k} &= q(\prod^{n}_{\ell = 2}p_{j\ell}) (\prod^{n}_{\ell = 1}r_{\ell k}) X_{1k}X_{jk} \quad &\text{$(j > 1)$}\\
				X_{jk}X_{ik} &= qp_{ij} X_{ik}X_{jm} \quad &\text{$(j > i > 1)$}\\
				X_{jm}X_{1k} &= (\prod^{n}_{\ell = 2}p_{j\ell})(\prod^{n}_{\ell = 1, \ell \neq k} r_{\ell m}) X_{1k}X_{jm} \quad &\text{($j > 1, k > m$)}\\
				p_{ji}X_{jm}X_{ik} &= r_{mk} X_{ik}X_{jm} \quad &\text{($j > i > 1, k > m$)}\\
				(\prod^{n}_{\ell = 2}p_{\ell j})(\prod^{n}_{\ell = 1}r_{m \ell}) X_{jm}X_{1k} - &r_{mk} X_{1k}X_{jm} = (q - q^{-1})X_{1m}X_{jk} \quad &\text{($j > 1, k < m$)}\\
				p_{ji}X_{jm}X_{ik} - &r_{mk} X_{ik}X_{jm} = (q - q^{-1})X_{im}X_{jk} \quad &\text{($j > i > 1, k < m$)}\\
				\sum_{\sigma \in S_{n}} (-q)^{-\ell(\sigma)}&(\prod_{1 \leq i< j \leq n} p_{ij})(\prod_{\substack{1 \leq i< j \leq n\\ \sigma(i) < \sigma(j)}}r_{\sigma(j), \sigma(i)}) X_{1 \sigma(1)} \cdots X_{n\sigma(n)} = 1.
			\end{align*}
			We also write $\mathcal{O}_{q}^{\bar{p}}(\mathrm{SL}_{n}(k)) := \mathcal{O}_{q}^{\bar{p}, \bar{p}}(\mathrm{SL}_{n}(k))$ when $\bar{p} = \bar{r}$.
			\ed
			When $n = 2$, consider the matrices
			\[\bar{p} = \begin{pmatrix}
				1 &p\\
				p^{-1} & 1
			\end{pmatrix} \quad \text{and} \quad \bar{r} = \begin{pmatrix}
				1 &r\\
				r^{-1} & 1
			\end{pmatrix}.\] If $p = r$, then the algebra $\mathcal{O}_{q}^{\bar{r}, \bar{r}}(\mathrm{SL}_{2}(k))$ coincides with the two-parameter braided quantum algebra $\mathcal{O}_{qr^{-1},\, qr}(\mathrm{SL}_{2}(k))$ defined in \cite[Definition~6.1]{bichon2024cohomological}. Moreover, if $p \neq r$, we can show that the algebra $\mathcal{O}_{q}^{\bar{p}, \bar{r}}(\mathrm{SL}_{2}(k))$ is still isomorphic to $\mathcal{O}_{qr^{-1},\, qr}(\mathrm{SL}_{2}(k))$.
			\bp\label{prop:kZ coaction SLn}
			Let $q \in k^{*}$ and $\bar{p}, \bar{r} \in \mathrm{AST}(n)$. The algebra $\mathcal{O}_{q}^{\bar{p}, \bar{r}}(\mathrm{SL}_{n}(k))$ has a right $k\Z$-comodule structure
			\begin{align*}
				\mathcal{O}_{q}^{\bar{p}, \bar{r}}(\mathrm{SL}_{n}(k)) \quad &\longrightarrow \quad \mathcal{O}_{q}^{\bar{p}, \bar{r}}(\mathrm{SL}_{n}(k)) \otimes k\Z\\
				\begin{pmatrix}
					X_{11} & X_{12} & \cdots & X_{1j} & \cdots & X_{1n} \\
					X_{21} & X_{22} & \cdots & X_{2j} & \cdots & X_{2n} \\
					\vdots & \vdots & \ddots & \vdots &        & \vdots \\
					X_{i1} & \cdots & \cdots & X_{ij} & \cdots & X_{in} \\
					\vdots &        &        & \vdots & \ddots & \vdots \\
					X_{n1} & X_{n2} & \cdots & X_{nj} & \cdots & X_{nn}
				\end{pmatrix}
				&\longmapsto \begin{pmatrix}
					X_{11} \otimes 1 & X_{12} \otimes z^{-1 }& \cdots & X_{1j} \otimes z^{-1} & \cdots & X_{1n} \otimes z^{-1}\\
					X_{21} \otimes z & X_{22} \otimes 1 & \cdots & X_{2j} \otimes 1 & \cdots & X_{2n} \otimes 1 \\
					\vdots & \vdots & \ddots & \vdots &        & \vdots \\
					X_{i1} \otimes z & \cdots & \cdots & X_{ij} \otimes 1 & \cdots & X_{in} \otimes 1 \\
					\vdots &        &        & \vdots & \ddots & \vdots \\
					X_{n1} \otimes z & X_{n2}\otimes 1 & \cdots & X_{nj} \otimes 1 & \cdots & X_{nn} \otimes 1
				\end{pmatrix}
			\end{align*}
			and a right $k\Z$-module structure
			\begin{align*}
				\mathcal{O}_{q}^{\bar{p}, \bar{r}}(\mathrm{SL}_{n}(k)) \otimes k\Z  &\longrightarrow  \mathcal{O}_{q}^{\bar{p}, \bar{r}}(\mathrm{SL}_{n}(k))\\
				X_{ij} \otimes z 
				&\longmapsto 
				(\prod^{n}_{\ell = 1} p_{\ell i} r_{j \ell}) X_{ij},
			\end{align*}
			where $\Z = \langle z \rangle$ is the infinite cyclic group. 
			In this way, $\mathcal{O}_{q}^{\bar{p}, \bar{r}}(\mathrm{SL}_{n}(k))$ is an algebra object in the category $\YD_{k\Z}^{k\Z}$ of Yetter-Drinfeld modules.
			\ep
			\bpf
			This is a direct verification.
			\epf
			\bp \label{Prop : IsoSL}
			The smash product  $k\Z \# \mathcal{O}_{q}^{\bar{p}, \bar{r}}(\mathrm{SL}_{n}(k))$ is isomorphic to $\Opr$ via 
			\begin{align*}
				g \colon \Opr &\longrightarrow k\Z \# \mathcal{O}_{q}^{\bar{p}, \bar{r}}(\mathrm{SL}_{n}(k))\\
				x_{ij} &\longmapsto \begin{cases}
					z \# X_{1j}\\
					1 \# X_{ij} \quad \text{for $i > 1$},
				\end{cases}\\
				(\Dpr_{q})^{-1} &\longmapsto z^{-1} \# 1,
			\end{align*}
				\ep
				\bpf The construction of $g$ is a straightforward verification. Define
				\begin{align*}
					h \colon k\Z \# \mathcal{O}_{q}^{\bar{p}, \bar{r}}(\mathrm{SL}_{n}(k)) &\longrightarrow \Opr&\\
					1 \# X_{ij} &\longmapsto \begin{cases} (\Dpr_{q})^{-1}x_{1j}  \quad &\text{if $i = 1$}\\
						x_{ij} \quad \quad &\text{if $i > 1$},
					\end{cases}\\
					z \# 1 &\longmapsto \Dpr_{q}.&
				\end{align*}
				It is not difficult to check that $h$ is well defined and extends to an algebra morphism; moreover, it is the inverse of $g$.
				\epf

					Moreover, (\textbf{GL}$_{q;n}$, $\pi$, $i$) is a cogroupoid triple. Applying Theorem \ref{Thm :  isoboso}, we obtain the following result:
					\bp \label{Prop : IsocoGL}
					Let $q \in k^{\ast}$, and let $\bar{p}, \bar{r} \in  \mathrm{AST}(n)$. The algebra $\Opr$ is isomorphic to $k\Z \# \,^{co\pi_{\bar{p}}}\Opr$ via
					\begin{align*}
						\Phi \colon \Opr &\longrightarrow k\Z \# ^{co\pi_{\bar{p}}}\Opr\\
						x_{ij} &\longmapsto \begin{cases}
							z \# (\Dpr_{q})^{-1}x_{ij} \quad \text{if $i = 1$},\\
							1 \# x_{ij} \quad \quad \text{if $i > 1$},
						\end{cases}\\
						(\Dpr_{q})^{-1} &\longmapsto z^{-1} \# 1.
					\end{align*}
					where $\pi_{\bar{p}}$ is the Hopf algebra morphism as defined in \eqref{pi p}.
					\ep
					In view of Propositions~\ref{Prop : IsoSL} and \ref{Prop : IsocoGL}, we aim to establish a connection between $\mathcal{O}_{q}^{\bar{p}, \bar{r}}(\mathrm{SL}_{n}(k))$ and $^{co\pi_{\bar{p}}}\Opr$. To this end, we prove the following proposition:

					\bp \label{Prop : coinGLn}
					Let $q \in k^{\ast}$ and $\bar{p}, \bar{r} \in \mathrm{AST}(n)$. Consider the Hopf algebra map $\pi_{\bar{p}}$ defined in \eqref{pi p}. Then the coinvariant subalgebra $^{co\pi_{\bar{p}}}\Opr$ is isomorphic to $\mathcal{O}_{q}^{\bar{p}, \bar{r}}(\mathrm{SL}_{n}(k))$ via
					\begin{align*}
						f \colon\mathcal{O}_{q}^{\bar{p}, \bar{r}}(\mathrm{SL}_{n}(k)) &\longrightarrow \,^{co\pi_{\bar{p}}}\Opr\\
						X_{1k} &\longmapsto (\Dpr_{q})^{-1}x_{1k}\\
						X_{ik} &\longmapsto x_{jk} \quad \text{for $i> 1$}.
					\end{align*}
					\ep
					\bpf We verify that $f$ is a well-defined algebra morphism. To do so,
					we first observe that, for any $j > 1$, if $i > 1$, we have
					\begin{align*}
						(\pi_{\bar{p}} \otimes 1)\Delta (x_{ij}) &= 1 \otimes x_{ij};
					\end{align*}
					and, if $i = 1$,
					\[(\pi_{\bar{p}} \otimes 1)\Delta \big((\Dpr_{q})^{-1}x_{1j}\big) = (\pi_{\bar{p}} \otimes 1)\sum^{n}_{k= 1}(\Dpr_{q})^{-1}x_{1k} \otimes (\Dpr_{q})^{-1}x_{kj} = 1 \otimes (\Dpr_{q})^{-1}x_{1j}.\]
					Now, using Lemma \ref{Lem : D_commute}, we have
					\begin{enumerate}
						\item for $m > k$, 
						\begin{align*}
							\big((\Dpr_{q})^{-1}x_{1m}\big)\big((\Dpr_{q})^{-1}x_{1k}\big) &= \big(\prod^{n}_{\ell = 1}p_{1 \ell}r_{\ell m}\big) (\Dpr_{q})^{-2}x_{1m}x_{1k}\\
							(\text{by \eqref{line}})&= q(\prod^{n}_{\ell = 1}p_{1 \ell}r_{\ell m}) r_{mk}(\Dpr_{q})^{-2}x_{1k}x_{1m};\\
							\big((\Dpr_{q})^{-1}x_{1k}\big)\big((\Dpr_{q})^{-1}x_{1m}\big) &= \big(\prod^{n}_{\ell = 1}p_{1 \ell}r_{\ell k}\big) (\Dpr_{q})^{-2}x_{1k}x_{1m}.
						\end{align*}
						Thus \[\big((\Dpr_{q})^{-1}x_{1m}\big)\big((\Dpr_{q})^{-1}x_{1k}\big) = q\big(\prod^{n}_{\ell = 1, \ell \neq k} r_{\ell m} r^{-1}_{\ell k}\big)\big((\Dpr_{q})^{-1}x_{1k}\big)\big((\Dpr_{q})^{-1}x_{1m}\big).\]
						\item for $j > 1$,
						\begin{align*}
							x_{jk}\big((\Dpr_{q})^{-1}x_{1k}\big) &= \big(\prod^{n}_{\ell = 1}p_{j \ell}r_{\ell k}\big)(\Dpr_{q})^{-1}x_{jk}x_{1k}\\
							(\text{by \eqref{column}})	&= q\big(\prod^{n}_{\ell = 1}p_{j \ell}r_{\ell k}\big)p_{1j}(\Dpr_{q})^{-1}x_{1k}x_{jk} = q(\prod^{n}_{\ell = 2}p_{j \ell})(\prod^{n}_{\ell = 1}r_{\ell k})(\Dpr_{q})^{-1}x_{1k}x_{jk}.
						\end{align*}
						\item for $j > 1, k > m$,
						\begin{align*}
							x_{jm}(\Dpr_{q})^{-1}x_{1k}) &= \big(\prod^{n}_{\ell = 1}p_{j \ell}r_{\ell m}\big) (\Dpr_{q})^{-1}x_{jm}x_{1k}\\
							(\text{by \eqref{cross1}})&= q\big(\prod^{n}_{\ell = 1}p_{j \ell}r_{\ell m}\big)p_{1j}r_{mk}(\Dpr_{q})^{-1}x_{1k}x_{jm} \\
							&= q\big(\prod^{n}_{\ell = 2}p_{j \ell}\big)\big(\prod^{n}_{\ell = 1, \ell \neq k}r_{\ell k}\big)(\Dpr_{q})^{-1}x_{1m}x_{jk}.
						\end{align*}
						\item for $j > 1, k < m$,
						\begin{align*}
							(\Dpr_{q})^{-1}x_{jm}x_{1k} = \prod^{n}_{\ell = 1} p_{\ell j}r_{m \ell} x_{jm} (\Dpr_{q})^{-1}x_{1k},
						\end{align*}
						then by \eqref{cross2}, we obtain
						\begin{align*}
							(\prod^{n}_{\ell = 2} p_{\ell j})(\prod^{n}_{\ell = 1}r_{m \ell}) x_{jm} ((\Dpr_{q})^{-1}x_{1k}) - r_{mk}(\Dpr_{q})^{-1}x_{1k}x_{jm} = (q - q^{-1})(\Dpr_{q})^{-1}x_{1m}x_{jk};
						\end{align*}
					\end{enumerate}
					Combining with \eqref{detpr},
					these computations guarantee the existence and uniqueness of an algebra morphism $f$, as announced. Using Proposition \ref{Prop : IsoSL} and Proposition \ref{Prop : IsocoGL}, we have
					\begin{align*}
						k\Z \# \mathcal{O}_{q}^{\bar{p}, \bar{r}}(\mathrm{SL}_{n}(k)) &\xrightarrow[\;\sim\;]{\; g^{-1}\;} \Opr \xrightarrow[\;\sim\;]{\;\Phi\;} k\Z \# \,^{co\pi_{\bar{p}}}\Opr\\
						1 \# X_{ij} &\longmapsto \quad \begin{cases}
							(\Dpr_{q})^{-1}x_{1j} \\
							x_{ij}
						\end{cases}\longmapsto \begin{cases}
							1 \# (\Dpr_{q})^{-1}x_{1j} \quad \text{if $i = 1$}\\
							1 \# x_{ij} \quad \qquad \qquad \text{if $i > 1$}.
						\end{cases}\\
						z \# 1 &\longmapsto \qquad \Dpr_{q} \qquad\longmapsto  \qquad z \# 1.
					\end{align*}
					Then the algebra map
					\begin{align*}
						1 \# f \colon 	k\Z \# \mathcal{O}_{q}^{\bar{p}, \bar{r}}(\mathrm{SL}_{n}(k)) \longrightarrow k\Z \# \,^{co\pi_{\bar{p}}}\Opr
					\end{align*}
					is an isomorphism; hence, so is $f$.
					\epf

					We now transport the structure of the coinvariant cogroupoid of $\textbf{GL}_{q;n}$ via the isomorphism of Proposition~\ref{Prop : coinGLn} and obtain the following results:
					\bl \label{Lem: SL structure}
					\begin{enumerate}
						\item For any $\bar{p}, \bar{r}, \bar{s} \in \mathrm{AST}(n)$, there exists an algebra map in $\YD^{k\Z}_{k\Z}$
						\begin{align*}
							\overline{\Delta}^{\bar{s}}_{\bar{p}, \bar{r}} \colon \mathcal{O}_{q}^{\bar{p}, \bar{r}}(\mathrm{SL}_{n}(k)) &\longrightarrow \mathcal{O}_{q}^{\bar{p}, \bar{s}}(\mathrm{SL}_{n}(k)) \otimes \mathcal{O}_{q}^{\bar{s}, \bar{r}}(\mathrm{SL}_{n}(k))\\
							X_{ij} &\longmapsto \sum^{n}_{k = 1} X_{ik} \otimes X_{kj}.
						\end{align*}
						such that, for any $\bar{t}\in \mathrm{AST}(n)$, the following diagram commutes
						\[\footnotesize\begin{tikzcd}                         
							\mathcal{O}_{q}^{\bar{p}, \bar{r}}(\mathrm{SL}_{n}(k)) \ar[r, "\underline{\Delta}^{\bar{s}}_{\bar{p}, \bar{r}}"] \ar[d, "\underline{\Delta}^{\bar{s}}_{\bar{p}, \bar{r}}"] & \mathcal{O}_{q}^{\bar{p}, \bar{s}}(\mathrm{SL}_{n}(k)) \otimes \mathcal{O}_{q}^{\bar{s}, \bar{r}}(\mathrm{SL}_{n}(k)) \ar[d, "\underline{\Delta}^{\bar{t}}_{\bar{p}, \bar{s}} \otimes 1"]\\
							\mathcal{O}_{q}^{\bar{p}, \bar{s}}(\mathrm{SL}_{n}(k)) \otimes \mathcal{O}_{q}^{\bar{s}, \bar{r}}(\mathrm{SL}_{n}(k)) \ar[r, "1 \otimes \underline{\Delta}^{\bar{t}}_{\bar{s}, \bar{r}}"] &\mathcal{O}_{q}^{\bar{p}, \bar{s}}(\mathrm{SL}_{n}(k)) \otimes \mathcal{O}_{q}^{\bar{s}, \bar{t}}(\mathrm{SL}_{n}(k))  \otimes \mathcal{O}_{q}^{\bar{t}, \bar{r}}(\mathrm{SL}_{n}(k)).
						\end{tikzcd}\]
						\item For $\bar{p} \in \mathrm{AST}(n)$, the linear map 
						\begin{align*}
							\underline{\varepsilon}_{\bar{p}} \colon 	\mathcal{O}_{q}^{\bar{p}}(\mathrm{SL}_{n}(k)) &\longrightarrow k\\
							X_{ij} &\longmapsto \delta_{ij}
						\end{align*} is an algebra map in $\YD^{k\Z}_{k\Z}$, such that, for any $\bar{r} \in \mathrm{AST}(n)$, the following diagrams commute
						\[\footnotesize\begin{tikzcd}                         
							\mathcal{O}_{q}^{\bar{p}, \bar{r}}(\mathrm{SL}_{n}(k)) \ar[rd, equals, shift left = 1ex] \ar[d, "\underline{\Delta}^{\bar{p}}_{\bar{p}, \bar{r}}"] & &\mathcal{O}_{q}^{\bar{p}, \bar{r}}(\mathrm{SL}_{n}(k)) \ar[rd, equals, shift left = 1ex] \ar[d, "\underline{\Delta}^{\bar{r}}_{\bar{p}, \bar{r}}"]&\\
							\mathcal{O}_{q}^{\bar{p}}(\mathrm{SL}_{n}(k)) \otimes \mathcal{O}_{q}^{\bar{p}, \bar{r}}(\mathrm{SL}_{n}(k)) \ar[r, "\underline{\varepsilon}_{\bar{p}} \otimes 1"] &\mathcal{O}_{q}^{\bar{p}, \bar{r}}(\mathrm{SL}_{n}(k)) &\mathcal{O}_{q}^{\bar{p}, \bar{r}}(\mathrm{SL}_{n}(k)) \otimes \mathcal{O}_{q}^{\bar{r}}(\mathrm{SL}_{n}(k)) \ar[r, "1 \otimes \underline{\varepsilon}_{\bar{r}}"] &\mathcal{O}_{q}^{\bar{p}, \bar{r}}(\mathrm{SL}_{n}(k)) 
						\end{tikzcd}\]
						\item For any $\bar{p}, \bar{r} \in \mathrm{AST}(n)$, let $\overline{M}^{\bar{r}, \bar{p}}_{ji}$ be the $(n - 1) \times (n - 1)$ quantum minor 
						\[\overline{M}^{\bar{p}, \bar{r}}_{ji} = (\prod_{ k < l; k, l \neq j} r_{kl})(\prod_{\substack{ k < l; k,l \neq i\\ \sigma(k) < \sigma(l)}}p_{\sigma(l), \sigma(k)}) X_{a_{1}b_{\sigma(1)}} \cdots X_{a_{n -1}b_{\sigma(n -1)}},\]
						where $a_{1} < \cdots < a_{n-1}$, $b_1 < \cdots < b_{n - 1}$ are the row and column indices distinct from $j$ and $i$, respectively.
						 There exists an algebra map
						\[\underline{S}_{\bar{p}, \bar{r}} \colon \mathcal{O}_{q}^{\bar{p}, \bar{r}}(\mathrm{SL}_{n}(k)) \longrightarrow \mathcal{O}_{q}^{\bar{r}, \bar{p}}(\mathrm{SL}_{n}(k))^{op}\] defined by
						\begin{align*}
							\underline{S}_{\bar{p}, \bar{r}}\big(X_{ij}\big) = (\prod_{1 \leq s < t \leq n}p_{ts}r^{-1}_{ts})\begin{cases} (-q)^{1 - i}(\prod_{1 \leq k < i \leq n} p_{ki})\overline{M}^{\bar{r}, \bar{p}}_{ji} \quad \text{if $j = 1$}\\
								(-q)^{j - i}(\prod_{1 \leq i < k \leq n} p_{ik})(\prod_{1 \leq j < k \leq n} r_{kj}) \overline{M}^{\bar{r}, \bar{p}}_{ji} \quad \text{if $j \neq 1$}
							\end{cases}
						\end{align*}
					 such that the following diagrams commute
						\[\footnotesize\begin{tikzcd}                         
							\mathcal{O}_{q}^{\bar{p}, \bar{p}}(\mathrm{SL}_{n}(k)) \ar[r, "\varepsilon_{\bar{p}}"] \ar[d, "\underline{\Delta}^{\bar{r}}_{\bar{p}, \bar{p}}"] &k \ar[r, "u"] &\mathcal{O}_{q}^{\bar{p}, \bar{r}}(\mathrm{SL}_{n}(k))\\
							\mathcal{O}_{q}^{\bar{p}, \bar{r}}(\mathrm{SL}_{n}(k)) \otimes \mathcal{O}_{q}^{\bar{r}, \bar{p}}(\mathrm{SL}_{n}(k)) \ar[rr, "1 \otimes \underline{S}_{\bar{r}, \bar{p}}"] &&\mathcal{O}_{q}^{\bar{p}, \bar{r}}(\mathrm{SL}_{n}(k)) \otimes \mathcal{O}_{q}^{\bar{p}, \bar{r}}(\mathrm{SL}_{n}(k)) \ar[u, "m"]
						\end{tikzcd}\]
						\[\footnotesize\begin{tikzcd}                         
							\mathcal{O}_{q}^{\bar{p}, \bar{p}}(\mathrm{SL}_{n}(k)) \ar[r, "\varepsilon_{\bar{p}}"] \ar[d, "\underline{\Delta}^{\bar{r}}_{\bar{p}, \bar{p}}"] &k \ar[r, "u"] &\mathcal{O}_{q}^{\bar{r}, \bar{p}}(\mathrm{SL}_{n}(k))\\
							\mathcal{O}_{q}^{\bar{p}, \bar{r}}(\mathrm{SL}_{n}(k)) \otimes \mathcal{O}_{q}^{\bar{r}, \bar{p}}(\mathrm{SL}_{n}(k)) \ar[rr, "\underline{S}_{\bar{p}, \bar{r}} \otimes 1"] &&\mathcal{O}_{q}^{\bar{r}, \bar{p}}(\mathrm{SL}_{n}(k)) \otimes \mathcal{O}_{q}^{\bar{r}, \bar{p}}(\mathrm{SL}_{n}(k)) \ar[u, "m"]
						\end{tikzcd}\]
					\end{enumerate}
					\el
					\bpf
					The first two statements follow by direct verification using \ref{(Lem : coDelata)} and Proposition \ref{Prop : coinGLn}. For the last statement, Lemma \ref{Lem : coS} ensures that there exists a morphism in $\YD_{k\Z}^{k \Z}$,
					\[\underline{S}_{\bar{p}, \bar{r}} \colon \, ^{co \pi_{\bar{p}}}\Opr \longrightarrow \,^{co\pi_{\bar{r}}}\mathcal{O}_{q}^{\bar{r}, \bar{p}}(\mathrm{GL}_{n}(k)),\]
					such that, for $i \in \N$,
					\begin{align*}
						\underline{S}_{\bar{p}, \bar{r}} (x_{i1}) &= \sum^{n}_{k = 1}S_{\bar{p}, \bar{r}}(x_{ik}) i_{\bar{r}, \bar{p}}\pi_{\bar{r}}(x_{k1}) \\
						&= S_{\bar{p}, \bar{r}}(x_{i1})\Dpr_{q}\\
						&= (-q)^{1 - i}\big(\prod_{1 \leq s < t \leq n}p_{ts}r^{-1}_{ts}\big)\big(\prod_{1 \leq k < i \leq n}p_{ki}\big) M_{1i}^{\bar{r}, \bar{p}}.
					\end{align*}
					For $i,j \in \N$ and $j \neq 1$,
					\begin{align*}
						\underline{S}_{\bar{p}, \bar{r}} (x_{ij}) &= \sum^{n}_{k = 1}S_{\bar{p}, \bar{r}}(x_{ik}) i_{\bar{r}, \bar{p}}\pi_{\bar{r}}(x_{kj}) \\
						&= S_{\bar{p}, \bar{r}}(x_{ij})\\
						&= (-q)^{j - i}\big(\prod_{1 \leq s < t \leq n}p_{ts}r^{-1}_{ts}\big)\big(\prod_{1 \leq k < i \leq n}p_{ki}\big) \big(\prod_{1 \leq k < j \leq n}r_{jk}\big)M_{ji}^{\bar{r}, \bar{p}}(\Dpr_{q})^{-1}\\
						&= (-q)^{j - i}\big(\prod_{1 \leq s < t \leq n}p_{ts}r^{-1}_{ts}\big)\big(\prod_{1 \leq i < k \leq n}p_{ik}\big) \big(\prod_{1 \leq j < k \leq n}r_{kj}\big)(\Dpr_{q})^{-1}M_{ji}^{\bar{r}, \bar{p}} \quad \text{by Remark \ref{Rk : permute Mij-D}}.
					\end{align*}
					It is not difficult to verify that this is indeed an algebra antimorphism in $\YD_{k\Z}^{k \Z}$. Hence, by applying the isomorphism in Proposition \ref{Prop : coinGLn}, we obtain the desired morphism.
					\epf
					We therefore propose the following definition:
					\bd
					The braided multiparameter $\mathrm{SL}_{q;n}$-cogroupoid, denoted by $\mathbf{SL}_{q;n}$, is defined as follows:
					\begin{itemize}
						\item $\ob(\mathbf{SL}_{q;n}) = \mathrm{AST}(n)$
						\item for $\bar{p}, \bar{r} \in \mathrm{AST}(n)$, the algebra $\mathbf{SL}_{q;n}(\bar{p}, \bar{r})$ is the algebra $\mathcal{O}_{q}^{\bar{p}, \bar{r}}(\mathrm{SL}_{n}(k))$  defined in Definition \ref{Oqpr}.
						\item the structure maps $\underline{\Delta}^{\bullet}_{\bullet, \bullet}$, $\underline{\varepsilon}_{\bullet}$ and $\underline{S}_{\bullet, \bullet}$ are given by the lemma above.
					\end{itemize}
					This is the cogroupoid of coinvariants of $\mathbf{GL}_{q; n}$.
					\ed
					\bc
					Let $q \in k^{*}$. For $\bar{p}, \bar{r} \in \mathrm{AST}(n)$, we have a $k$-linear equivalence of monoidal categories
					\begin{align*}
						(\YD_{k\Z}^{k \Z})^{\mathcal{O}_{q}^{\bar{p}, \bar{p}}(\mathrm{SL}_{n}(k))} &\cong^{\otimes} (\YD_{k\Z}^{k \Z})^{\mathcal{O}_{q}^{\bar{r}, \bar{r}}(\mathrm{SL}_{n}(k))}.
					\end{align*}
					\ec
				
					\section{Transmutation} \label{Sec : transmutation}
					Throughout this subsection, let $H$ be a coquasitriangular Hopf algebra and \textbf{r} a fixed universal $r$-form of $H$. Following the approach of S. Majid, we consider the passage from the $k$-cogroupoid $\CC$ to a braided cogroupoid over $\MM^{H}$. This transformation is known as transmutation. 
					\subsection{Construction}
					To begin, we recall Majid’s transmutation procedure \cite{Maj93} (we follow the formulation in \cite[Chapter 10, Proposition 36]{KSBook97}). Let $A$ be a Hopf algebra and let $\pi : A \to H$ be a Hopf algebra morphism. Then the vector space $A$ becomes a Hopf algebra $A_{\pi}$ in the category $\MM^{H}$ of right $H$-comodules, with the product given by
					\[a\bullet b = a_{[2]}b_{[2]}\rr\big[\pi \big(S_{A}(a_{[1]})a_{[3]}\big), \pi\big(S_{A}(b_{[1]})\big)\big],\]
					the $H$-coaction
					\[a \mapsto a_{[2]} \otimes \pi\big(S_{A}(a_{[1]})a_{[3]}\big),\]and antipode 
					\[S^{\pi}(a) = S_{A}(a_{[2]}) \rr\big[\pi\big(S^{2}_{A}(a_{[3]})S_{A}(a_{[1]})\big), \pi(a_{[4]})\big].\]
					\brq \label{Rk: twist trans}
					Let $H = k\Gamma$, where $\Gamma$ is an abelian group endowed with a bicharacter $\psi \colon \Gamma \times \Gamma \to k^{*}$, and let $\pi : A \to k\Gamma$ be a Hopf algebra morphism. As in Remark \ref{Rk : twisted cocycle}, the algebra $A$ admits a $\Gamma \times \Gamma$-graded algebra structure
					$A = \bigoplus_{g, h \in \Gamma} \,_{g}A_{h}.$
					For $a \in  \,_{g}A_{h}$ and $b \in \,_{k}A_{\ell}$, the product in $A_{\pi}$ is given by
					\[a \bullet b = ab \psi(g^{-1}h, k^{-1})\]
					Hence, $A_\pi$ identifies with the $2$-cocycle twisted algebra $_{\nu}A$, where $\nu$ is the $2$-cocycle  on the group $\Gamma \times \Gamma$ defined by
					\begin{align*}
						\nu \colon (\Gamma \times \Gamma) \times (\Gamma \times \Gamma) &\longrightarrow k\\
						\big((g, h), (k, \ell)\big) &\longmapsto \psi(g^{-1}h, k^{-1}).
					\end{align*}
					\erq
				
					In what follows, we adopt Sweedler's notation for cogroupoids, as introduced in Section \ref{Sec : Bosonization}: for $a^{X, Y} \in \Cc(X, Y)$, we write
					\[\Delta^{Z}_{X, Y}(a^{X, Y}) = a^{X, Z}_{[1]} \otimes a^{Z, Y}_{[2]}\] and we have, by Proposition \ref{Proprosition: S_X,Y},
					\[\Delta^{Z}_{Y, X}\big(S_{X, Y}(a^{X, Y})\big) = S_{Z, Y}(a^{Z, Y}_{[2]}) \otimes S_{X, Z}(a^{X, Z}_{[1]}).\]
					
					The following result generalizes \cite[Proposition 36]{KSBook97}: 
					
					\bt\label{Thm : Transmutation}
					Let $\Cc$ be a $k$-cogroupoid, let $H$ be a coquasitriangular Hopf algebra equipped with a fixed universal $r$-form $\normalfont{\textbf{r}}$, and let \[\pi = \big(\pi_{X} \colon \Cc(X, X) \to H \mid  X \in \ob(\Cc)\big)\] be a family of Hopf algebra morphisms. Then $\Cc$ becomes an $\MM^{H}$-cogroupoid, denoted by $\Cc^{\pi}$, with structures given as follows: for any $X, Y \in \ob(\CC)$, and for $a^{X, Y}, b^{X, Y} \in \CC(X, Y)$,
					the $k$-algebra $\CC(X, Y)$ equipped with the new product 
					\[[a^{X, Y}] \cdot [b^{X, Y}] = a_{[2]}^{X, Y}b^{X, Y}_{[2]} \normalfont\textbf{r}\big[\pi_{X}\big(S_{X, X}(a^{X, X}_{[1]})\big)\pi_{Y}(a^{Y,Y}_{[3]}); \pi_{X}\big(S_{X, X}(b^{X, X}_{[1]})\big)\big],\] the unit $u_{X, Y}$
					and the $H$-coaction
					\[\beta([a^{X, Y}]) = a_{[2]}^{X, Y} \otimes \pi_{X}\big(S_{X, X}(a^{X, X}_{[1]})\big)\pi_{Y}(a^{Y,Y}_{[3]}),\]
					is an algebra in $\MM^{H}$, denoted by $\CC^{\pi}(X, Y)$.\\
					The algebra morphisms $\varepsilon_{\bullet}$, $\Delta^{\bullet}_{\bullet, \bullet}$ of $\CC^{\pi}$ coincide with those of $\CC$ as a $k$-cogroupoid, and for any $X, Y \in \ob(\CC)$, the morphism
					\[S^{\pi}_{X, Y}: \CC^{\pi}(X, Y) \longrightarrow \CC^{\pi}(Y, X)\] is defined by
					\[S^{\pi}_{X, Y}([a^{X, Y}]) = S_{X, Y}(a^{X, Y}_{[2]})\, \normalfont\textbf{r}\big[\pi_{Y}\big(S^{2}_{Y,Y}(a^{Y,Y}_{[3]})\big)\pi_{X}\big(S_{X, X}(a^{X, X}_{[1]})\big); \pi_{Y}(a^{Y,Y}_{[4]})\big].\]
					\et
					\bpf
					\sloppy
					By a straightforward computation, we can show that for any $X, Y \in \ob(\CC)$, $\CC(X, Y)$ is an algebra in $\MM^{H}$, equipped with the given product and $H$-coaction. Then, for any $X, Y, Z, T \in \ob(\CC)$, the tensor product $\CC(X, Y) \otimes \CC(Z, T)$ is also an algebra in $\MM^{H}$ with respect to the tensor product coaction and the product 
					\begin{align}\label{tensor product}
						& (a^{X, Y} \otimes b^{Z, T}) \bullet (c^{X, Y} \otimes d^{Z, T}) \nonumber\\
						&= a^{X, Y} \cdot c^{X, Y}_{[2]} \otimes b^{Z, T}_{[2]} \cdot d^{Z, T}\, \, \textbf{r}[\pi_{Z}\big(S_{Z, Z}(b^{Z, Z}_{[1]})\big)\pi_{T}(b^{T, T}_{[3]}); \pi_{X}\big(S_{X, X}(c^{X, X}_{[1]})\big)\pi_{Y}(c^{Y, Y}_{[3]})].
					\end{align}
					Now, we check that $\Delta^{Z}_{X, Y}$ is an algebra morphism of $\MM^{H}$: for $a^{X, Y}, b^{X, Y} \in \CC(X, Y),$ we have
					\begin{align*}
						\Delta^{Z}_{X, Y}(a^{X, Y} \cdot b^{X, Y}) &= \Delta^{Z}_{X, Y}\big(a_{[2]}^{X, Y} b_{[2]}^{X, Y}\big) \, \textbf{r}\big[\pi_{X}\big(S_{X, X}(a^{X, X}_{[1]})\big)\pi_{Y}(a^{Y,Y}_{[3]}); \pi_{X}\big(S_{X, X}(b^{X, X}_{[1]})\big)\big]\\
						&= a_{[2]}^{X, Z} b_{[2]}^{X, Z} \otimes a_{[3]}^{Z, Y} b_{[3]}^{Z, Y}\, \textbf{r}\big[\pi_{X}\big(S_{X, X}(a^{X, X}_{[1]})\big)\pi_{Y}(a^{Y,Y}_{[4]}); \pi_{X}\big(S_{X, X}(b^{X, X}_{[1]})\big)\big].
					\end{align*}
					On the other hand, we have
					\begin{align*}
						(a_{[1]}^{X, Z} \otimes &a_{[2]}^{Z, Y}) \bullet (b_{[1]}^{X, Z} \otimes b_{[2]}^{Z, Y})\\
						&= a_{[1]}^{X, Z}\cdot b_{[2]}^{X, Z} \otimes a_{[3]}^{Z, Y}\cdot b_{[4]}^{Z, Y} \, \textbf{r}\big[\pi_{Z}\big(S_{Z, Z}(a^{Z, Z}_{[2]})\big)\pi_{Y}(a^{Y,Y}_{[4]}); \pi_{X}\big(S_{X, X}(b^{X, X}_{[1]})\big) \pi_{Z}\big(b_{[3]}^{Z, Z}\big)\big]\\
						&= a_{[2]}^{X, Z} b_{[3]}^{X, Z} \otimes a_{[6]}^{Z, Y} b_{[6]}^{Z, Y}\, \textbf{r}\big[\pi_{X}\big(S_{X, X}(a^{X, X}_{[1]})\big)\pi_{Z}(a^{Z,Z}_{[3]}); \pi_{X}\big(S_{X, X}(b^{X, X}_{[2]})\big)\big]\\
						&\hspace*{4cm}\textbf{r}\big[\pi_{Z}\big(S_{Z, Z}(a^{Z, Z}_{[5]})\big)\pi_{Y}(a^{Y,Y}_{[7]}); \pi_{Z}\big(S_{Z, Z}(b^{Z, Z}_{[5]})\big)\big]\\
						&\hspace*{4cm}\textbf{r}\big[\pi_{Z}\big(S_{Z, Z}(a^{Z, Z}_{[4]})\big)\pi_{Y}(a^{Y,Y}_{[8]}); \pi_{X}\big(S_{X, X}(b^{X, X}_{[1]})\big)\pi_{Z}\big(b^{Z, Z}_{[4]}\big)\big]\\
						&= a_{[2]}^{X, Z} b_{[3]}^{X, Z} \otimes a_{[5]}^{Z, Y} b_{[5]}^{Z, Y}\, \textbf{r}\big[\pi_{X}\big(S_{X, X}(a^{X, X}_{[1]})\big)\pi_{Z}(a^{Z,Z}_{[3]}); \pi_{X}\big(S_{X, X}(b^{X, X}_{[2]})\big)\big]\\
						&\hspace*{4cm}\textbf{r}\big[\pi_{Z}\big(S_{Z, Z}(a^{Z, Z}_{[4]})\big)\pi_{Y}(a^{Y,Y}_{[6]}); \pi_{X}\big(S_{X, X}(b^{X, X}_{[1]})\big)\varepsilon_{Z}\big(b^{Z, Z}_{[4]}\big)\big]\\
						&= a_{[2]}^{X, Z} b_{[2]}^{X, Z} \otimes a_{[3]}^{Z, Y} b_{[3]}^{Z, Y}\, \textbf{r}\big[\pi_{X}\big(S_{X, X}(a^{X, X}_{[1]})\big)\pi_{Y}(a^{Y,Y}_{[4]}); \pi_{X}\big(S_{X, X}(b^{X, X}_{[1]})\big)\big]\\
						&= \Delta^{Z}_{X, Y}(a^{X, Y} \cdot b^{X, Y}) .
					\end{align*}
					Moreover, $\Delta^{Z}_{X, Y}$ is indeed $H$-colinear:
					\begin{align*}
						\beta \circ \Delta^{Z}_{X, Y}(a^{X, Y}) &= \beta(a^{X, Z}_{[1]} \otimes a^{Z, Y}_{[2]})\\
						&= a^{X, Z}_{[2]} \otimes a^{Z, Y}_{[5]} \otimes \pi_{X}\big(S_{X, X}(a^{X, X}_{[1]})\big)\pi_{Z}(a^{Z,Z}_{[3]}) \pi_{Z}\big(S_{Z, Z}(a^{Z, Z}_{[4]})\big) \pi_{Y}(a^{Y,Y}_{[6]})\\
						&= a^{X, Z}_{[2]} \otimes a^{Z, Y}_{[3]} \otimes \pi_{X}\big(S_{X, X}(a^{X, X}_{[1]})\big)\pi_{Y}(a^{Y,Y}_{[4]})\\
						&=\big(\Delta^{Z}_{X, Y} \otimes 1\big) \beta (a^{X, Y}).
					\end{align*}
					
					There remain to be checked the properties of $S^{\pi}_{X, Y}$ that endow $\CC^{\pi}$ with an $\MM^{H}$-cogroupoid structure: 
					\begin{itemize}
						\item $S^{\pi}_{X, Y}$ is $H$-colinear: for any $a^{X, Y} \in \CC(X, Y)$, we have
						\begin{align*}
							\beta \circ S^{\pi}_{X, Y}(a^{X, Y}) &= \beta\big(S_{X, Y}(a^{X, Y}_{[2]})\big) \textbf{r}\big[\pi_{Y}\big(S^{2}_{Y,Y}(a^{Y,Y}_{[3]})\big)\pi_{X}\big(S_{X, X}(a^{X, X}_{[1]})\big); \pi_{Y}(a^{Y,Y}_{[4]})\big]\\
							&= S_{X, Y}(a^{X, Y}_{[3]}) \otimes \pi_{Y}\big(S^{2}_{Y, Y}(a^{Y, Y}_{[4]})\pi_{X}\big(S_{X, X}(a^{X, X}_{[2]})\big)\\ &\quad \textbf{r}\big[\pi_{Y}\big(S^{2}_{Y,Y}(a^{Y,Y}_{[5]})\big)\pi_{X}\big(S_{X, X}(a^{X, X}_{[1]})\big); \pi_{Y}(a^{Y,Y}_{[6]})\big].
						\end{align*}
						On the other hand, we have
						\begin{align*}
							\big(S^{\pi}_{X, Y} \otimes 1\big)\circ \beta (a^{X, Y}) &= S^{\pi}_{X, Y}(a^{X, Y}_{[2]}) \otimes  \pi_{X}\big(S_{X, X}(a^{X, X}_{[1]})\big)\pi_{Y}(a^{Y,Y}_{[3]})\\
							&= S_{X, Y}(a^{X, Y}_{[3]}) \otimes \pi_{X}\big(S_{X, X}(a^{X, X}_{[1]})\big)\pi_{Y}(a^{Y,Y}_{[6]})\\ & \quad \textbf{r}\big[\pi_{Y}\big(S^{2}_{Y,Y}(a^{Y,Y}_{[4]})\big)\pi_{X}\big(S_{X, X}(a^{X, X}_{[2]})\big); \pi_{Y}(a^{Y,Y}_{[5]})\big]\\
							&= S_{X, Y}(a^{X, Y}_{[3]}) \otimes \pi_{X}\big(S_{X, X}(a^{X, X}_{[1]})\big)\pi_{Y}(a^{Y,Y}_{[7]})\\
							&\quad \textbf{r}\big[\pi_{Y}\big(S^{2}_{Y,Y}(a^{Y,Y}_{[4]})\big); \pi_{Y}(a^{Y,Y}_{[5]})\big] \, \textbf{r}\big[\pi_{X}\big(S_{X, X}(a^{X, X}_{[2]})\big); \pi_{Y}(a^{Y,Y}_{[6]})\big]\\
							&= S_{X, Y}(a^{X, Y}_{[3]}) \otimes \pi_{Y}(a^{Y,Y}_{[6]})\pi_{X}\big(S_{X, X}(a^{X, X}_{[2]})\big)\\
							&\quad \textbf{r}\big[\pi_{Y}\big(S^{2}_{Y,Y}(a^{Y,Y}_{[4]})\big); \pi_{Y}(a^{Y,Y}_{[5]})\big] \, \textbf{r}\big[\pi_{X}\big(S_{X, X}(a^{X, X}_{[1]})\big); \pi_{Y}(a^{Y,Y}_{[7]})\big].
						\end{align*}
						Since $\pi_{X}$ is a Hopf algebra morphism, we have
						\begin{align*}
							\textbf{r}\big[\pi_{Y}\big(S^{2}_{Y,Y}(a^{Y,Y}_{[4]})\big); \pi_{Y}(a^{Y,Y}_{[5]})\big] &= \textbf{r}\big[S_{H}^{2}\big(\pi_{Y}(a^{Y,Y}_{[4]})\big); \pi_{Y}(a^{Y,Y}_{[5]})\big]\\
							&= \textbf{r}^{-1}\big[S_{H}\big(\pi_{Y}(a^{Y,Y}_{[4]})\big); \pi_{Y}(a^{Y,Y}_{[5]})\big] (\text{by  Proposition \ref{Prop : r3}.(1)})
						\end{align*} 
						Applied to the case where $\pi_{Y}(a_{[4]}^{Y, Y}) = x$ in $H$, the equality (2) of Proposition \ref{Prop : r3} becomes
						\begin{align*}
							\textbf{r}^{-1}\big[S_{H}\big(\pi_{Y}(a^{Y,Y}_{[4]})\big); \pi_{Y}(a^{Y,Y}_{[5]})\big]\pi_{Y}(a^{Y,Y}_{[6]}) &= \textbf{r}^{-1}\big[S_{H}\big(\pi_{Y}(a^{Y,Y}_{[5]})\big); \pi_{Y}(a^{Y,Y}_{[6]})\big]\pi_{Y}\big(S^{2}_{Y,Y}(a^{Y,Y}_{[4]})\big)\\
							&= \textbf{r}\big[\pi_{Y}\big(S^{2}_{Y, Y}(a^{Y,Y}_{[5]})\big); \pi_{Y}(a^{Y,Y}_{[6]})\big]\pi_{Y}\big(S^{2}_{Y,Y}(a^{Y,Y}_{[4]})\big).
						\end{align*}
						Hence
						\begin{align*}
							\big(S^{\pi}_{X, Y} \otimes 1\big)\circ \beta (a^{X, Y}) &= S_{X, Y}(a^{X, Y}_{[3]}) \otimes \pi_{Y}\big(S^{2}_{Y,Y}(a^{Y,Y}_{[4]})\big)\pi_{X}\big(S_{X, X}(a^{X, X}_{[2]})\big)\\
							&\quad \textbf{r}\big[\pi_{Y}\big(S^{2}_{Y,Y}(a^{Y,Y}_{[5]})\big); \pi_{Y}(a^{Y,Y}_{[6]})\big] \, \textbf{r}\big[\pi_{X}\big(S_{X, X}(a^{X, X}_{[1]})\big); \pi_{Y}(a^{Y,Y}_{[7]})\big]\\
							&= S_{X, Y}(a^{X, Y}_{[3]}) \otimes \pi_{Y}\big(S^{2}_{Y,Y}(a^{Y,Y}_{[4]})\big)\pi_{X}\big(S_{X, X}(a^{X, X}_{[2]})\big)\\
							&  \quad \textbf{r}\big[\pi_{Y}\big(S^{2}_{Y,Y}(a^{Y,Y}_{[5]})\big)\pi_{X}\big(S_{X, X}(a^{X, X}_{[1]})\big); \pi_{Y}(a^{Y,Y}_{[6]})\big]\\
							&= \beta \circ S^{\pi}_{X, Y}(a^{X, Y}),
						\end{align*}
						and we conclude that $S^{\pi}_{X,Y}$ is $H$-colinear.
						\item For $a^{X, Y} \in \CC(X, Y)$, we compute
						\begin{align*}
							a_{[1]}^{X, Y} \cdot S^{\pi}_{Y, X}(a^{Y, X}_{[2]}) &= a_{[1]}^{X, Y} \cdot S_{Y, X}(a^{Y, X}_{[3]}) \, \textbf{r}\big[\pi_{X}\big(S^{2}_{X, X}(a^{X,X}_{[4]})\big)\pi_{Y}(a^{Y, Y}_{[2]}); \pi_{X}\big(a^{X,X}_{[5]}\big)\big]\\
							&= a_{[1][2]}^{X, Y}S_{Y, X}(a^{Y, X}_{[4]}) \, \textbf{r}\big[\pi_{X}\big(S_{X, X}(a^{X,X}_{[1][1]})\big)\pi_{Y}(a^{Y, Y}_{[1][3]}); \pi_{X}\big(S^{2}_{X, X}(a^{X, X}_{[5]})\big)\big]\\
							&\hspace*{4cm}\textbf{r}\big[\pi_{X}\big(S^{2}_{X, X}(a^{X,X}_{[6]})\big)\pi_{Y}(a^{Y, Y}_{[3]}); \pi_{X}\big(a^{X,X}_{[5]}\big)\big]\\
							&= a_{[2]}^{X, Y}S_{Y, X}(a^{Y, X}_{[5]}) \, \textbf{r}\big[\pi_{X}\big(S_{X, X}(a^{X,X}_{[1]})\big)\pi_{Y}(a^{Y, Y}_{[3]}); \pi_{X}\big(S^{2}_{X, X}(a^{X, X}_{[6]})\big)\big]\\
							&\hspace*{4cm}\textbf{r}\big[\pi_{X}\big(S^{2}_{X, X}(a^{X,X}_{[7]})\big)\pi_{Y}(a^{Y, Y}_{[4]}); \pi_{X}\big(a^{X,X}_{[8]}\big)\big]\\
							&= a_{[2]}^{X, Y}S_{Y, X}(a^{Y, X}_{[5]}) \, \textbf{r}\big[\pi_{X}\big(S_{X, X}(a^{X,X}_{[1]})\big)\pi_{Y}(a^{Y, Y}_{[3]}); \pi_{X}\big(S^{2}_{X, X}(a^{X, X}_{[6]})\big)\big]\\
							&\quad \textbf{r}^{-1}\big[\pi_{X}\big(S_{X, X}(a^{X,X}_{[7]})\big); \pi_{X}\big(a^{X,X}_{[8]}\big)\big] \, \textbf{r}\big[\pi_{Y}(a^{Y, Y}_{[4]}); \pi_{X}\big(a^{X,X}_{[9]}\big)\big].
						\end{align*}
						By Proposition \ref{Prop : r3}.(2), we get
						\[\textbf{r}^{-1}\big[\pi_{X}\big(S_{X, X}(a^{X,X}_{[7]})\big); \pi_{X}\big(a^{X,X}_{[8]}\big)\big]\pi_{X}\big(S^{2}_{X, X}(a^{X, X}_{[6]})\big) = \textbf{r}^{-1}\big[\pi_{X}\big(S_{X, X}(a^{X,X}_{[6]})\big); \pi_{X}\big(a^{X,X}_{[7]}\big)\big]\pi_{X}(a^{X, X}_{[8]}).\]
						Then
						\begin{align*}
							a_{[1]}^{X, Y} \cdot S^{\pi}_{Y, X}(a^{Y, X}_{[2]})
							&= a_{[2]}^{X, Y}S_{Y, X}(a^{Y, X}_{[5]}) \, \textbf{r}\big[\pi_{X}\big(S_{X, X}(a^{X,X}_{[1]})\big)\pi_{Y}(a^{Y, Y}_{[3]}); \pi_{X}\big(a^{X, X}_{[8]}\big)\big]\\
							&\hspace*{2cm}\quad \textbf{r}^{-1}\big[\pi_{X}\big(S_{X, X}(a^{X,X}_{[6]})\big); \pi_{X}\big(a^{X,X}_{[7]}\big)\big] \, \textbf{r}\big[\pi_{Y}(a^{Y, Y}_{[4]}); \pi_{X}\big(a^{X,X}_{[9]}\big)\big] \\
							&= a_{[2]}^{X, Y}S_{Y, X}(a^{Y, X}_{[4]}) \, \textbf{r}\big[\pi_{X}\big(S_{X, X}(a^{X,X}_{[1]})\big)\pi_{Y}(\varepsilon_{Y}(a^{Y, Y}_{[3]})); \pi_{X}\big(a^{X, X}_{[7]}\big)\big]\\
							&\hspace*{2cm}\quad \textbf{r}^{-1}\big[\pi_{X}\big(S_{X, X}(a^{X,X}_{[5]})\big); \pi_{X}\big(a^{X,X}_{[6]}\big)\big]\\
							&= \varepsilon_{X}(a^{X, X})1.
						\end{align*}
						Similarly, we have
						\begin{align*}
							S^{\pi}_{X, Y}(a_{[1]}^{X, Y}) \cdot a^{Y, X}_{[2]} &= S_{X, Y}(a^{X, Y}_{[2]}) \cdot a^{Y, X}_{[5]} \textbf{r}\big[\pi_{Y}\big(S^{2}_{Y,Y}(a^{Y,Y}_{[3]})\big)\pi_{X}\big(S_{X, X}(a^{X, X}_{[1]})\big); \pi_{Y}(a^{Y,Y}_{[4]})\big]\\
							&= S_{X, Y}(a^{X, Y}_{[3]})a^{Y, X}_{[8]} \, \textbf{r}\big[\pi_{Y}\big(S^{2}_{Y,Y}(a^{Y,Y}_{[4]})\big)\pi_{X}\big(S_{X, X}(a^{X, X}_{[2]})\big); \pi_{Y}\big(S_{Y,Y}(a^{Y,Y}_{[2]})\big)\big]\\
							&\hspace*{4cm}\textbf{r}\big[\pi_{Y}\big(S^{2}_{Y,Y}(a^{Y,Y}_{[5]})\big)\pi_{X}\big(S_{X, X}(a^{X, X}_{[1]})\big); \pi_{Y}(a^{Y,Y}_{[6]})\big]\\
							&= S_{X, Y}(a^{X, Y}_{[2]})a^{Y, X}_{[5]} \textbf{r}\big[\pi_{Y}\big(S^{2}_{Y,Y}(a^{Y,Y}_{[3]})\big)\pi_{X}\big(S_{X, X}(a^{X, X}_{[1]})\big); \varepsilon_{Y}\big(a^{Y,Y}_{[4]}\big)\big]\\
							&= \varepsilon_{X}(a^{X, X})1.
						\end{align*}
					\end{itemize}
					This completes the proof.
					\epf

					\subsection{Monoidal equivalences}
					We again consider a Hopf algebra $A$, and a coquasitriangular Hopf algebra $H$ endowed with an $r$-form \textbf{r}. Let $\pi \colon A \longrightarrow H$ be a Hopf algebra morphism. 
					We define a $k$-cogroupoid $\Cc$ with two objects $X, Y$ by setting $\Cc(X, X) = \Cc(X, Y) = \Cc(Y, X) = \Cc(Y, Y) = A$. Consider the morphisms of Hopf algebras $\pi_{X} = u_{H} \circ \varepsilon_{A}$ and $\pi_{Y} = \pi$. Then, using Theorem \ref{Thm : Transmutation}, this data gives rise to a $\MM^{H}$-cogroupoid with $\Cc^{\pi}(X, X) = A$ and $\Cc^{\pi}(Y, Y) = A_{\pi}$. Applying Theorem \ref{Thm : equi}, we obtain 
					\bt
					Let $\pi : A \longrightarrow H$ be a Hopf algebra morphism from a Hopf algebra $A$ to a coquasitriangular Hopf algebra $H$. Then there exists a $k$-linear monoidal equivalence 
						\[F_{\pi}: (\MM^{H})^{A} \xrightarrow{\;\cong^{\otimes}\;} (\MM^{H})^{A_{\pi}}.\]
					\et
					Such a monoidal equivalence was obtained in \cite{HN24}, where the following question naturally arose: if $L$ is a Hopf subalgebra of $H$, can we restrict the above monoidal equivalence to
					\[(\MM^{L})^{A} \cong^{\otimes} (\MM^{L})^{A_{\pi}} ?\]
					In general, there is no reason to expect that the functor $F_{\pi}$ sends an $L$-comodule to an object that still carries an $L$-coaction. Therefore, 
					the equivalence on $(\MM^H)^{A}$ does not automatically restrict to the subcategory $(\MM^{L})^{A}$. Nevertheless, such a monoidal equivalence does exist in the case of group algebras, as shown by Habbestad and Neshveyev in the setting of Hopf $2$-cocycle \cite[Corollary~2.6]{HN24}. This observation leads us to consider a Hopf subalgebra $L$ and to investigate the conditions under which our equivalence holds.
					
					To understand this issue in detail, we will describe the functor $F_{\pi}$ explicitly, and then examine some examples to determine when such a restriction of the equivalence does occur.
					
					For $V \in (\MM^{H})^{A}$, we fix Sweedler’s notation:
					\[\Delta_A \colon A \to A \otimes A, \quad a \mapsto a_{[1]} \otimes a_{[2]}\] 
					\[\beta_{V}^{A} \colon V \to V \otimes A, \quad v \mapsto v_{[0]} \otimes v_{[1]}\]
					\[\beta_{V}^{H} \colon V \to V \otimes H, \quad v \mapsto v_{(0)} \otimes v_{(1)}.\]
					\bp\label{Prop: func Fpi}
					Let $V \in (\MM^{H})^{A}$. Define a right $A_{\pi}$-comodule structure on $V$ by keeping the original right $A$-comodule structure, that is, $\beta_{V}^{A_{\pi}} = v_{[0]} \otimes v_{[1]}$. The map
					\begin{align*}
						\beta_{V}^{H_{\pi}} \colon V &\longrightarrow V \otimes H\\
						v &\longmapsto v_{[0](0)} \otimes v_{[0](1)}\pi(v_{[1]})
					\end{align*}
					provides $V$ with a new right $H$-comodule structure. Then $V$ becomes an object in $(\MM^{H})^{A_{\pi}}$, that we denote by $V_{\pi} := (V, \beta^{A_{\pi}}, \beta^{H_{\pi}})$. This gives rise  to a functor
					\begin{align*}
						F_{\pi} : (\MM^{H})^{A} \rightarrow (\MM^{H})^{A_{\pi}}; \quad 	V \mapsto V_{\pi}, 	f\mapsto f.
					\end{align*}
					\ep %
					\bpf
					A straightforward computation shows that $\beta_{V}^{H_{\pi}} $ is a well-defined right $H$-comodule structure. We verify the compatibility of the two structures, which means that $\beta_{V}^{A_{\pi}}\colon V \to V \otimes A^{\pi}$ is $H$-colinear. We first remark that, since $\beta_{V}^{A}$ is $H$-colinear, for $v \in V$,
					\begin{equation}
						v_{[0](0)} \otimes v_{[1]} \otimes v_{[0](1)} = v_{(0)[0]} \otimes v_{(0)[1]} \otimes v_{(1)}. \label{Cpati_com}
					\end{equation}
					Hence we also have
					\begin{equation}
						\beta_{V}^{H_{\pi}}(v) = v_{(0)[0]} \otimes v_{(1)}\pi(v_{(0)[1]}) \label{beta'}.
					\end{equation}
					Now, we have
					\begin{align*}
						(\beta_{V}^{A_{\pi}} \otimes 1) \circ \beta_{V}^{H_{\pi}} (v) &= v_{(0)[0]} \otimes v_{(0)[1]} \otimes v_{(1)}\pi(v_{(0)[2]})\\
						&= v_{[0](0)} \otimes v_{[1]} \otimes v_{[0](1)}\pi(v_{[2]}) \quad (\text{applying $1 \otimes \Delta_{A} \otimes 1$ to \eqref{Cpati_com}}).
					\end{align*}
					On the other hand,
					\begin{align*}
						\beta^{H}_{V \otimes A_{\pi}} \circ \beta_{V}^{A_{\pi}}(v) &= \beta^{H}_{V \otimes A_{\pi}}\big(v_{[0]} \otimes v_{[1]}\big)\\
						&= v_{[0](0)} \otimes v_{[3]} \otimes v_{[0](1)} \pi(v_{[1]})\pi\big(S_{A}(v_{[2]})v_{[4]}\big)\\
						& \quad (\text{using coassociativity of the right $A$-comodule structure})\\
						&= v_{[0](0)} \otimes v_{[3]} \otimes v_{[0](1)} \pi\big(v_{[1]}S_{A}(v_{[2]})v_{[4]}\big)\\
						&= v_{[0](0)} \otimes v_{[1]} \otimes v_{[0](1)} \pi(v_{[2]}).
					\end{align*}
					which implies that $(\beta_{V}^{A_{\pi}} \otimes 1) \circ \beta_{V}^{H_{\pi}} = \beta^{H}_{V \otimes A_{\pi}} \circ \beta_{V}^{A_{\pi}}$. Thus $\beta^{A_{\pi}}_{V}$ is $H$-colinear. The verification of the last assertion is immediate.
					\epf
					\bp \label{prop:monoidal Api}
					The functor \begin{align*}
						F_{\pi}\colon (\MM^{H})^{A} &\longrightarrow (\MM^{H})^{A_{\pi}}
					\end{align*} is a monoidal equivalence of categories whose monoidal structure is given by the natural isomorphisms defined, for $V , W \in (\MM^{H})^{A}$, by
					\begin{align*}
						\widetilde{F}_{V, W} \colon V_{\pi} \otimes W_{\pi} &\longrightarrow (V \otimes W)_{\pi}\\
						v \otimes w &\longmapsto \rr\Big[\pi(v_{[1]}), w_{(1)}\Big] v_{[0]} \otimes w_{(0)}.
					\end{align*}
					\ep
					\bpf
					We check that for $V , W \in (\MM^{H})^{A}$, $\widetilde{F}_{V, W}$ is a morphism of right $A_{\pi}$-comodules; that is,
					\[\beta^{A}_{V \otimes W} \circ \widetilde{F}_{V, W} = (\widetilde{F}_{V, W} \otimes 1) \circ \beta^{A_{\pi}}_{V \otimes W}.\]  We start with
					\begin{align*}
						\beta^{A}_{V \otimes W} \circ \widetilde{F}_{V, W}(v \otimes w) &= \beta^{A}_{V \otimes W}(v_{[0]} \otimes w_{(0)}) \rr\Big[\pi(v_{[1]}), w_{(1)}\Big]\\
						&= v_{[0]} \otimes w_{(0)[0]} \otimes v_{[1]}w_{(0)[1]} \rr\Big[\pi(v_{[2]}), w_{(1)}\Big]\\
						&= v_{[0]} \otimes w_{[0](0)} \otimes v_{[1]}w_{[1]} \rr\Big[\pi(v_{[2]}), w_{[0](1)}\Big]\\
						&\quad (\text{by applying \eqref{Cpati_com} to $w$}).
					\end{align*}
					On the other hand, we have
					\begin{align*}
						&\beta^{A_{\pi}}_{V \otimes W} (v \otimes w)\\
						&=  v_{[0]} \otimes w_{[0](0)} \otimes v_{[2]} \bullet w_{[2]}  \rr\big[\pi\big(S_{A}(v_{[1]})v_{[3]}\big), w_{[0](1)}\pi(w_{[1]})\big]\\
						& \qquad \quad (\text{multiplication in $A_{\pi}$})\\
						&= v_{[0]} \otimes w_{[0](0)} \otimes v_{[3]} w_{[3]}  \rr\big[\pi\big(S_{A}(v_{[2]})v_{[4]}\big), \pi\big(S_{A}(w_{[2]})\big)\big] \rr\big[\pi\big(S_{A}(v_{[1]})v_{[5]}\big), w_{[0](1)}\pi(w_{[1]})\big]\\
						&= v_{[0]} \otimes w_{[0](0)} \otimes v_{[2]} w_{[1]} \rr\big[\pi\big(S_{A}(v_{[1]})v_{[3]}\big), w_{[0](1)}\big].
					\end{align*}
					Thus
					\begin{align*}
						(\widetilde{F}_{V, W} \otimes 1)\beta^{A_{\pi}}_{V \otimes W} (v \otimes w) &= v_{[0]} \otimes w_{[0](0)} \otimes v_{[3]} w_{[1]}\rr\big[\pi(v_{[1]}), w_{[0](1)}\big] \rr\big[\pi\big(S_{A}(v_{[2]})v_{[4]}\big), w_{[0](2)}\big] \\
						&= v_{[0]} \otimes w_{[0](0)} \otimes v_{[3]} w_{[1]}\rr\big[\pi(v_{[1]})\pi\big(S_{A}(v_{[2]})v_{[4]}\big), w_{[0](1)}\big] \\
						&= v_{[0]} \otimes w_{[0](0)} \otimes v_{[1]} w_{[1]}\rr\big[\pi(v_{[2]}), w_{[0](1)}\big] .
					\end{align*}
					Therefore, the desired equality holds. For naturality, if $f : V \to V'$ and $g: W \to W'$ are morphisms in $(\MM^{H})^{A}$ then 
					\[\Phi_{V', W'} \circ (f \otimes g) = (f \otimes g) \circ \Phi_{V,W}\]
					because the morphisms of comodules preserve coaction components and $\rr$ depends only on those components.
					
					It is not difficult to see that $\widetilde{F}_{V, W}$ is bijective with inverse
					\begin{align*}
						\widetilde{F}_{V, W}^{-1} \colon (V \otimes W)_{\pi} &\longrightarrow V_{\pi} \otimes W_{\pi}\\
						v \otimes w &\longmapsto \rr\big[\pi\big(S_{A}(v_{[1]})\big), w_{(1)}\big] v_{[0]} \otimes w_{(0)}.
					\end{align*}
					
					Let $U, V, W \in (\MM^{H})^{A}.$ We also have
					\begin{align*}
						\widetilde{F}_{U \otimes V, W}(\widetilde{F}_{U,V} \otimes 1)(u \otimes v \otimes w) &= \rr\big[\pi(u_{[1]}), v_{(1)}\big] \widetilde{F}_{U \otimes V, W}(u_{[0]} \otimes v_{(0)} \otimes w)\\
						&= \rr\big[\pi(u_{[2]}), v_{(1)}\big] \rr\big[\pi(u_{[1]}v_{(0)[1]}), w_{(1)}\big]u_{[0]} \otimes v_{(0)[0]} \otimes w_{0}\\
						&= \rr\big[\pi(u_{[2]}), v_{[0](1)}\big] \rr\big[\pi(u_{[1]}v_{[1]}), w_{(1)}\big]u_{[0]} \otimes v_{[0](0)} \otimes w_{0}\\
						&\quad (\text{by applying \eqref{Cpati_com} to $v$}),
					\end{align*}
					and
					\begin{align*}
						\widetilde{F}_{U, V \otimes W}(1 \otimes \widetilde{F}_{V, W})(u \otimes v \otimes w) &= \rr \big[\pi(v_{[1]}), w_{(1)}\big]\widetilde{F}_{U, V \otimes W}(u \otimes v_{[0]} \otimes w_{(0)})\\
						&=  \rr \big[\pi(v_{[1]}), w_{(2)}\big] \rr \big[\pi(u_{[1]}), v_{[0](1)}w_{(1)}\big]u_{[0]}\otimes v_{[0](0)} \otimes w_{(0)}\\
						&=  \rr \big[\pi(u_{[1]}v_{[1]}), w_{(1)}\big]\rr \big[\pi(u_{[2]}), v_{[0](1)}\big]u_{[0]}\otimes v_{[0](0)} \otimes w_{(0)}.
					\end{align*}
					Thus we obtain
					\begin{align*}
						\widetilde{F}_{U \otimes V, W}(\widetilde{F}_{W,V} \otimes 1) = \widetilde{F}_{U, V \otimes W}(1 \otimes \widetilde{F}_{V, W}),
					\end{align*}
					and this finishes our proof.
					\epf
					\brq
					Let $L$ be a Hopf subalgebra of $H$, and let $V \in (\MM^{L})^{A}$. In Proposition \ref{Prop: func Fpi}, we see that $V \in (\MM^{L})^{A_{\pi}}$ if, and only if, for $v \in V,$ \[\beta_{V}^{H_{\pi}}(v) = v_{[0](0)} \otimes v_{[0](1)}\pi(v_{[1]}) \in V \otimes L.\] 
					\erq

					In the sequel, we shall discuss a sufficient condition for this, although there is no reason for it to hold in general. Recall that the adjoint coaction on $A$ is given by 
					\[\mathrm{ad} \colon A \to A \otimes A, \mathrm{ad}(a) = a_{[1]} \otimes S_{A}(a_{[1]})a_{[3]}.\]
					We observe that, if $H$ is commutative, then $(A, \mathrm{ad})$ forms a Hopf algebra object in the category $\MM^{H, \varepsilon \otimes \varepsilon}$.
					\bp \label{Prop : equisubcat}
					Let $L$ be a central Hopf subalgebra of $H$ such that \[\mathrm{ad}_{\pi}(A) := (1 \otimes \pi \circ \mathrm{ad})(A) \subset A \otimes L.\]
					Then $A_{\pi}$ is a Hopf algebra in $\MM^{L, \mathbf{r}}$, with coaction given by
					\begin{align*}
						\beta^{L}_{A_{\pi}} \colon A_{\pi} &\longrightarrow A_{\pi} \otimes L\\
						a &\mapsto a_{[2]} \otimes \pi\big(S_{A}(a_{[1]})a_{[3]}\big).
					\end{align*}
					Moreover, regarding $\big(A, \mathrm{ad}_{\pi}\big)$ as a Hopf algebra in the category $\MM^{L, \varepsilon \otimes \varepsilon}$, we have a monoidal equivalence
					\[(\MM^{L, \mathbf{r}})^{A_{\pi}} \cong^{\otimes} (\MM^{L, \varepsilon \otimes \varepsilon})^{A}.\]
					\ep
					\bpf
					It is immediate from the condition that $A_{\pi}$ is a Hopf algebra in $\MM^{L, \mathbf{r}}$. A straightforward verification shows that the identity functor
					\begin{align*}
						F \colon (\MM^{L, \mathbf{r}})^{A_{\pi}} \longrightarrow  (\MM^{L, \varepsilon \otimes \varepsilon})^{A}; \quad V \longmapsto V, 
						f \longmapsto f.
					\end{align*}
					is well defined and yields a monoidal equivalence between the two categories. The monoidal structure is given by
					\begin{align*}
						\Phi_{V,W} \colon F(V) \otimes F(W) &\longrightarrow F(V \otimes W)\\
						v \otimes w &\longmapsto \rr\Big[\pi(S_{A}(v_{[1]})), w_{[0](1)}\pi(S_{A}(w_{[1]}))\Big] v_{[0]} \otimes w_{[0](0)}.
					\end{align*}
					The unit map is the identity $k \to k$. Indeed, we see that $\Phi_{V,W}$ is bijective with inverse
					\[\Phi_{V, W}^{-1}(v \otimes w) = \rr\Big[\pi(v_{[1]}), w_{[0](1)}\pi(S_{A}(w_{[1]}))\Big]v_{[0]} \otimes w_{[0](0)}.\]
					Now, we will check that $\Phi_{V,W}$ is a morphism of right $A$-comodules; that is, 
					\[(\Phi_{V,W} \otimes 1) \circ \beta^{A}_{V \otimes W} = \beta^{A_{\pi}}_{V \otimes W} \circ \Phi_{V, W}\]
					where $\beta^{A}_{V \otimes W}$ and $\beta^{A_{\pi}}_{V \otimes W}$ denote the right $A$-coaction and $A_{\pi}$-coaction, respectively. 
					
					We first have \[\beta^{A}_{V \otimes W}(v \otimes w) = v_{[0]} \otimes w_{[0]} \otimes v_{[1]}w_{[1]}.\] Apply $\Phi_{V, W} \otimes 1$ to that tensor, we obtain
					\[(\Phi_{V, W} \otimes 1)\beta^{A}_{V \otimes W}(v \otimes w) = v_{[0]} \otimes w_{[0](0)} \otimes v_{[2]}w_{[2]} \rr\Big[\pi(S_{A}(v_{[1]})), w_{[0](1)}\pi(S_{A}(w_{[1]}))\Big].\]
					On the other hand,
					\begin{align*}
						&\beta^{A_{\pi}}_{V \otimes W} \Phi_{V,W}(v \otimes w) = \rr \Big[\pi(S_{A}(v_{[1]})), w_{[0](1)}\pi(S_{A}(w_{[1]}))\Big]\beta^{A_{\pi}}_{V \otimes W}\Big(v_{[0]} \otimes w_{[0](0)}\Big)\\
						&= v_{[0]} \otimes w_{[0](0)[0](0)} \otimes v_{[2]} \bullet w_{[0](0)[1]} \\
						&\quad \rr\Big[\pi(S_{A}(v_{[5]})v_{[3]}), w_{[0](0)[0](1)}\Big] \rr \Big[\pi(S_{A}(v_{[1]})), w_{[0](1)}\pi(S_{A}(w_{[1]}))\Big].
					\end{align*}
					Since $\beta^{A_{\pi}}_{W}$ is $L$-colinear, we have 
					\begin{equation} \label{A-compatile}
						w_{[0](0)} \otimes w_{[2]} \otimes w_{[0](1)}\pi\Big(S_{A}(w_{[1]})w_{[3]}\Big) = w_{(0)[0]} \otimes w_{(0)[1]} \otimes w_{(1)}.
					\end{equation}
					Then apply $\beta^{L}_{W} \otimes 1 \otimes 1$ to that equality, we get 
					\begin{equation} \label{A-compatible2}
						w_{[0](0)} \otimes w_{[0](1)} \otimes w_{[2]} \otimes w_{[0](2)}\pi\Big(S_{A}(w_{[1]})w_{[3]}\Big) = w_{(0)[0](0)} \otimes  w_{(0)[0](1)} \otimes w_{(0)[1]} \otimes w_{(1)}.
					\end{equation}
					Thus, by applying \eqref{A-compatible2} to $w_{[0]}$,
					\begin{align*}
						&\beta^{A_{\pi}}_{V \otimes W} \Phi_{V,W}(v \otimes w) 
						= v_{[0]} \otimes w_{[0](0)} \otimes v_{[2]} \bullet w_{[2]} \rr\Big[\pi(S_{A}(v_{[1]})v_{[3]}), w_{[0](1)}\Big] \\
						&\qquad \rr \Big[\pi\big(S_{A}(v_{[4]})\big), w_{[0](2)}\pi(S_{A}(w_{[1]})w_{[3]})\pi\big(S_{A}(w_{[4]})\big)\Big]\\
						&= v_{[0]} \otimes w_{[0](0)} \otimes v_{[2]} \bullet w_{[2]} \rr\Big[\pi(S_{A}(v_{[1]})v_{[3]}), w_{[0](1)}\Big] 
						\rr \Big[\pi\big(S_{A}(v_{[4]})\big), w_{[0](2)}\pi(S_{A}(w_{[1]}))\Big]\\
						&= v_{[0]} \otimes w_{[0](0)} \otimes v_{[3]} w_{[3]} \rr\Big[\pi(S_{A}(v_{[2]})v_{[4]}), \pi(S_{A}(w_{[2]}))\Big] \\
						&\quad \rr\Big[\pi(S_{A}(v_{[1]})v_{[5]}), w_{[0](1)}\Big] \rr \Big[\pi\big(S_{A}(v_{[6]})\big), w_{[0](2)}\pi(S_{A}(w_{[1]}))\Big]\\
						&=  v_{[0]} \otimes w_{[0](0)} \otimes v_{[2]} w_{[3]} \rr\Big[\pi(S_{A}(v_{[1]})v_{[3]}), w_{[0](1)}\pi(S_{A}(w_{[2]}))\Big] \\
						& \quad \rr \Big[\pi\big(S_{A}(v_{[4]})\big), w_{[0](2)}\pi(S_{A}(w_{[1]}))\Big]\\
						&= v_{[0]} \otimes w_{[0](0)} \otimes v_{[2]} w_{[2]} \rr\Big[\pi(S_{A}(v_{[1]})v_{[3]})\pi\big(S_{A}(v_{[4]})\big), w_{[0](1)}\pi(S_{A}(w_{[1]}))\Big] \\
						&= v_{[0]} \otimes w_{[0](0)} \otimes v_{[2]} w_{[2]} \rr\Big[\pi(S_{A}(v_{[1]})), w_{[0](1)}\pi(S_{A}(w_{[1]}))\Big].
					\end{align*}
					This completes the proof.
					\epf
					We now summarize our results:
					\bc \label{Cor : equiboso}
					Let $(H, \mathbf{r})$ be a coquasitriangular Hopf algebra, and $A$ be a Hopf algebra. Suppose that $\pi \colon A \longrightarrow H$ be a Hopf algebra morphism, and let $L \subset H$ be a central Hopf subalgebra such that \[(1 \otimes \pi \circ \mathrm{ad})(A) \subset A \otimes L.\]
					Then the transmuted Hopf algebra $A_{\pi}$ is a Hopf algebra in $\MM^{L, \mathbf{r}}$, and there is a $k$-linear monoidal equivalence of comodule categories
					\[\MM^{L \# A_{\pi}} \cong^{\otimes} \MM^{L \ltimes A}.\]
					where $L \ltimes A$ is the semidirect product of $L$ and $A$, with coproduct defined by
					\[\Delta_{L \ltimes A}(x \otimes a) = \big(x_{(1)} \otimes a_{[2]}\big) \otimes \big(x_{(2)} \pi\big(S_{A}(a_{[1]}) a_{[3]}\big) \otimes a_{[4]}\big),\] for any $x \in L, a \in A$.
					\ec

					We note that this result generalizes that of Habbestad and Neshveyev in \cite[Theorem 2.5]{HN24} where $H = \mathbb{C} T$ is the group algebra and $L = \mathbb{C}[T/T_{0}]$ in their setting.
					\brq
					By our result on braided cogroupoids, we obtain the $k$-linear monoidal  equivalence $F_{\pi}$. Then, by Corollary \ref{Cor : equiva boso}, there exists a $k$-linear monoidal equivalence
					\[\MM^{H \# A} \cong^{\otimes} \MM^{H \# A_{\pi}}.\]
					Observe that $H \# A_{\pi}$ can be viewed as a twisted algebra with respect to the $2$-cocycle
					\begin{align*}
						\sigma\big((g \otimes a), (h \otimes b)\big) = \varepsilon_{H}(g)\varepsilon_{A}(b)\rr^{-1}(\pi(a), h)
					\end{align*}
					as described by Habbestad and Neshveyev in \cite[Theorem 1.5, Remark 1.6]{HN24}. Howeover, in contrast to Remark 1.6, here we do not need require $H$ to be cocommutative for $\sigma$ to be a well-defined Hopf $2$-cocycle since the monoidal equivalence above holds without any assumption on $H$. Moreover, as $\rr$ is an $r$-form, the composition $\rr^{-1} \circ (\pi \otimes 1)$ is well defined as an invertible skew pairing on $(A, H)$ \cite[Definition 1.3]{MR1213985}. Hence, by \cite[Proposition 1. 5]{MR1213985}, the bilinear map $\sigma$ is a $2$-cocycle.
					\erq
					\subsection{Examples} 
					\subsubsection{Braided multiparameter quantum $SL_{n}$}\label{SLn}
					In this section, we consider $\mathcal{O}_{q}(\mathrm{SL}_{n}(k))$ introduced in Definition~\ref{Def : OqSLn}. We fix a datum $(\Gamma,\psi,(g_{1},\dots,g_{n}))$, where $\Gamma$ is an abelian group endowed with a bicharacter $\psi \colon \Gamma \times \Gamma \to k^{*}$, and $g_{1},\dots,g_{n} \in \Gamma$ are elements that satisfy $\prod_{i=1}^{n} g_{i}=1$. As in Lemma \ref{lem:morphism pi}, we have a Hopf algebra morphism
					\begin{align}
						\pi \colon \mathcal{O}_{q}(\mathrm{SL}_{n}(k))&\longrightarrow k\Gamma \label{pi}\\ 
						x_{ij} &\longmapsto \delta_{ij}g_{i}. \nonumber
					\end{align}
					We aim to transmute the Hopf algebra $\OQ$ to obtain a Hopf algebra $\mathcal{O}^{\pi}_{q}(\mathrm{SL}_{n}(k))$ in the category $\MM^{k\Gamma, \psi}$.
					\bdep \label{def:Ogamma_SLn}
					Let $q \in k^{*}$, and let $(\Gamma,\psi,(g_{1},\dots,g_{n}))$ be the datum introduced above. Define $\mathcal{O}^{\Gamma}_{q}(\mathrm{SL}_{n}(k))$ to be the algebra presented by the elements $a_{ij}$, for $1 \leq i,j \leq n$, subject to the relations
					\begin{align*}
						&a_{im} a_{ik} = q \psi(g^{-1}_{m}g_{k}, g_{i})a_{ik} a_{im} \quad \quad (k < m)\\
						&a_{jk} a_{ik} = q \psi(g_{j}g^{-1}_{k}, g_{i})\psi(g^{-1}_{i}g_{k}, g_{j})a_{ik} a_{jk} \quad(i < j)\\
						&a_{jm} a_{ik} =  \psi(g_{j}g^{-1}_{m}, g_{i}) \psi(g^{-1}_{i}g_{k}, g_{j}) a_{ik} a_{jm} \quad (i < j, k > m)\\
						\psi(g^{-1}_{j}g_{m}, g_{i})\psi(g_{i}, g_{j}) a_{jm} a_{ik} - &\psi(g_{k}, g_{j}) a_{ik}a_{jm} = (q - q^{-1})\psi(g_{m}, g_{j}) a_{im}a_{jk} \quad (i <j, k <m).\\
						&\sum_{\sigma \in S_n} (-q)^{-\ell(\sigma)}\Big(\prod_{1 \leq i < j \leq n} \psi(g^{-1}_{i}g_{\sigma(i)}, g_{j})\Big) \prod_{i = 1}^{n} a_{i \sigma(i)} = 1.
					\end{align*}
					Then $\mathcal{O}^{\Gamma}_{q}(\mathrm{SL}_{n}(k))$ admits a structure of $k\Gamma$-comodule algebra with coaction
					\begin{align*}
						\beta \colon \mathcal{O}^{\Gamma}_{q}(\mathrm{SL}_{n}(k)) &\longrightarrow \mathcal{O}^{\Gamma}_{q}(\mathrm{SL}_{n}(k)) \otimes k\Gamma\\
						a_{ij} &\longmapsto a_{ij} \otimes g^{-1}_{i}g_{j}.
					\end{align*}
					\edep
					\bp \label{prop : Opi_SLn}
					Let $q \in k^{\ast}$. Consider the Hopf algebra morphism $\pi \colon \mathcal{O}_{q}(\mathrm{SL}_{n}(k))\longrightarrow k\Gamma$ in \eqref{pi}. Then the algebra $\mathcal{O}^{\Gamma}_{q}(\mathrm{SL}_{n}(k))$ is isomorphic to the transmutation $\mathcal{O}^{\pi}_{q}(\mathrm{SL}_{n}(k))$ of $\mathcal{O}_{q}(\mathrm{SL}_{n}(k))$.
					\ep
					To prepare for the proof of this proposition, we first prove the following lemma:
					\bl \label{lem: transmute det}
					Let $\sigma \in S_n$. For any $k \geq 1$, in the transmuted algebra $\mathcal{O}^{\pi}_{q}(\mathrm{SL}_{n}(k))$, we have 
					\[\Big[\prod_{i = 1}^{k} x_{i \sigma(i)}\Big] = \Bigg(\prod_{i = 1}^{k} [x_{i \sigma(i)}]\Bigg) \prod_{1 \leq i < j \leq k} \psi(g^{-1}_{i}g_{\sigma(i)}, g_{j})\]
					where $x_{ij}$ denotes the generators of $\OQ$, and $[x_{ij}]$ their
					images in $\mathcal{O}^{\pi}_{q}(\mathrm{SL}_{n}(k))$ under the transmutation.
					\el
					\bpf
					We will proceed by induction: for $k = 1$, this is trivial: $[x_{ij}] = [x_{ij}]$. For $k = 2$, we have $[x_{1\sigma(1)}]. [x_{2\sigma(2)}] = \psi(g^{-1}_{1}g_{\sigma(1)}, g^{-1}_{2}) [x_{1\sigma(1)}x_{2\sigma(2)}]$. Hence
					\[[x_{1\sigma(1)}x_{2\sigma(2)}] =\psi(g^{-1}_{1}g_{\sigma(1)}, g_{2})  [x_{1\sigma(1)}].[x_{2\sigma(2)}]\]
					which proves the desired identity for $k = 2$.
					Assume the formula holds for $k -1$:
					\[\Big[\prod_{i = 1}^{k-1} x_{i \sigma(i)}\Big] = \Bigg(\prod_{i = 1}^{k-1} [x_{i \sigma(i)}]\Bigg) \prod_{1 \leq i < j \leq k -1} \psi(g^{-1}_{i}g_{\sigma(i)}, g_{j}).\]
					We have 
					\begin{align*}
						\prod^{k}_{i = 1} [x_{i \sigma(i)}] &=\Bigg( \prod^{k - 1}_{i = 1} [x_{i \sigma(i)}]\Bigg) [x_{k \sigma(k)}] \\
						&= \Big[\prod_{i = 1}^{k-1} x_{i \sigma(i)}\Big] \Bigg(\prod_{1 \leq i < j \leq k -1} \psi(g^{-1}_{i}g_{\sigma(i)}, g^{-1}_{j})\Bigg) [x_{k \sigma(k)}]\\
						&= \psi\Big(\prod^{k - 1}_{i = 1} g^{-1}_{i}g_{\sigma(i)}, g^{-1}_{k}\Big)\Big[\prod_{i = 1}^{k-1} x_{i \sigma(i)} x_{k \sigma(k)}\Big] \Bigg(\prod_{1 \leq i < j \leq k -1} \psi(g^{-1}_{i}g_{\sigma(i)}, g^{-1}_{j})\Bigg)\\
						&= \prod_{1 \leq i < k} \psi \Big(g^{-1}_{i}g_{\sigma(i)}, g^{-1}_{k}\Big) \prod_{1 \leq i < j \leq k -1} \psi(g^{-1}_{i}g_{\sigma(i)}, g^{-1}_{j}) \Big[\prod_{i = 1}^{k} x_{i \sigma(i)}\Big]\\
						&= \prod_{1 \leq i < j \leq k} \psi(g^{-1}_{i}g_{\sigma(i)}, g^{-1}_{j}) \Big[\prod_{i = 1}^{k} x_{i \sigma(i)}\Big].
					\end{align*}
					Thus, \[\Big[\prod_{i = 1}^{k} x_{i \sigma(i)}\Big] = \Bigg(\prod_{i = 1}^{k} [x_{i \sigma(i)}]\Bigg) \prod_{1 \leq i < j \leq k} \psi(g^{-1}_{i}g_{\sigma(i)}, g_{j}).\]
					This completes the induction.
					\epf
					\begin{proof}[Proof of Proposition \ref{prop : Opi_SLn}] 
						In $\mathcal{O}^{\pi}_{q}(\mathrm{SL}_{n}(k))$, we have 
						\begin{enumerate}
							\item for $k < m$,
							\[[x_{ik}].[x_{im}] = \psi(g^{-1}_{i}g_{k}, g^{-1}_{i}) [x_{ik} x_{im}],\]
							\[[x_{im}].[x_{ik}] = [x_{im} x_{ik}] \psi(g^{-1}_{i}g_{m}, g^{-1}_{i}) = q\psi(g^{-1}_{i}g_{m}, g^{-1}_{i}) [x_{ik} x_{im}],\]
							so \begin{align*}
								[x_{im}].[x_{ik}] &= q \psi(g^{-1}_{m}g_k, g_i) [x_{ik}]. [x_{im}].
							\end{align*}
							\item For $i <j$, 
							\[[x_{ik}]. [x_{jk}] = \psi(g^{-1}_{i}g_{k}, g^{-1}_{j}) [x_{ik} x_{jk}]\]
							\[[x_{jk}].[x_{ik}] = \psi(g^{-1}_{j}g_{k}, g^{-1}_{i}) [x_{jk} x_{ik}] = q\psi(g^{-1}_{j}g_{k}, g^{-1}_{i})[x_{ik} x_{jk}],\]
							so \begin{align*}
								[x_{jk}] .[x_{ik}] &= q\psi(g^{-1}_{j}g_{k}, g^{-1}_{i}) \psi(g^{-1}_{i}g_{k}, g_{j})[x_{ik}] .[x_{jk}].
							\end{align*}
							\item For $i < j$ and $k > m$,
							\[[x_{ik}]. [x_{jm}] = \psi(g^{-1}_{i}g_{k}, g^{-1}_{j}) [x_{ik} x_{jm}]\]
							\[[x_{jm}] .[x_{ik}] = \psi(g^{-1}_{j}g_{m}, g^{-1}_{i}) [x_{jm} x_{ik}] = \psi(g^{-1}_{j}g_{m}, g^{-1}_{i})[x_{ik} x_{jm}],\]
							so \begin{align*}
								[x_{jm}] .[x_{ik}] &= \psi(g_{j}g^{-1}_{m}, g_{i}) \psi(g^{-1}_{i}g_{k}, g_{j})[x_{ik}] .[x_{jk}].
							\end{align*}
							\item For $i <j$ and $k <m$,
							\[[x_{jm}] .[x_{ik}] = \psi(g^{-1}_{j}g_{m}, g^{-1}_{i}) [x_{jm} x_{ik}]\]
							\[[x_{ik}]. [x_{jm}] = \psi(g^{-1}_{i}g_{k}, g^{-1}_{j}) [x_{ik} x_{jm}]\]
							\[[x_{im}].[x_{jk}] = \psi(g^{-1}_{i}g_{m}, g^{-1}_{j}) [x_{im} x_{jk}].\]
							Using \eqref{formule det}, we have
							\[\psi(g^{-1}_{j}g_{m}, g_{i}) [x_{jm}].[x_{ik}] - \psi(g^{-1}_{i}g_{k}, g_{j})[x_{ik}].[x_{jm}] = (q - q^{-1})\psi(g^{-1}_{i}g_{m}, g_{j})[x_{im}].[x_{jk}].\]
							Therefore,
							\[\psi(g^{-1}_{j}g_{m}, g_{i})\psi(g_{i}, g_{j})[x_{jm}].[x_{ik}] - \psi(g_{k}, g_{j})[x_{ik}].[x_{jm}] = (q - q^{-1}) \psi(g_{m}, g_{j})[x_{im}].[x_{jk}].\]
							\item In $\OQ$, we have \[D_{q} = \sum_{\sigma \in S_{n}} (-q)^{-\ell(\sigma)} x_{1\sigma(1)} x_{2\sigma(2)} \cdots  x_{n\sigma(n)}.\] Thus, by Lemma \ref{lem: transmute det}, we obtain
							\begin{align*}
								\sum_{\sigma \in S_{n}} (-q)^{-\ell(\sigma)} \Bigg(\prod_{1 \leq i < j \leq k} \psi(g^{-1}_{i}g_{\sigma(i)}, g_{j})\Bigg)\prod_{i = 1}^{k} [x_{i \sigma(i)}] = 1.
							\end{align*}
						\end{enumerate}
						These relations $(1) - (5)$ define an algebra morphism 
						\begin{align*}
							\mathcal{O}^{\Gamma}_{q}(\mathrm{SL}_{n}(k)) &\longrightarrow \mathcal{O}^{\pi}_{q}(\mathrm{SL}_{n}(k))\\
							a_{ij} &\longmapsto [x_{ij}]
						\end{align*}
						To conclude the proof we remark that since $\mathcal{O}^{\pi}_{q}(\mathrm{SL}_{n}(k))$ can be viewed as the algebra twisted by a $2$-cocycle, the result in \cite[Section 3.2]{MR2945585} again ensures that this algebra morphism is an isomorphism.
					\end{proof}
					Using Propositions~\ref{prop:monoidal Api} and Corollary~\ref{Cor : equiboso}, we deduce the following result:
					\bc
					Let $q \in k^{*}$ and let $(\Gamma,\psi,(g_{1},\dots,g_{n}))$ be the datum defined above. Put $\Gamma_{0} = \langle g^{-1}_{i}g_{j}\mid 1 \leq i, j \leq n\rangle \subset \Gamma$. Then there are $k$-linear monoidal equivalences
					\[(\MM^{k\Gamma})^{\mathcal{O}^{\Gamma}_{q}(\mathrm{SL}_{n}(k))} \cong^{\otimes} (\MM^{k\Gamma})^{\mathcal{O}_{q}(\mathrm{SL}_{n}(k))} \quad \text{and} \quad \MM^{k\Gamma_{0} \# \mathcal{O}^{\Gamma}_{q}(\mathrm{SL}_{n}(k))} \cong^{\otimes} \MM^{k\Gamma_{0} \ltimes \mathcal{O}_{q}(\mathrm{SL}_{n}(k))}.\]
					\ec
					We now establish a connection between the bosonization of $\mathcal{O}^{\Gamma}_{q}(\mathrm{SL}_{n}(k))$ by $k\Gamma$ and $\Opgl$, for $\bar{p} \in \mathrm{AST}(n)$, introduced in Section \ref{Sec : Bosonization}.
					\bp \label{Prop : morphismGLn}
					Let $\bar{p} = (p_{ij})_{1 \leq i,j \leq n} \in M_{n}(k)$ with $p_{ij} = \psi(g^{-1}_{i}g_{j}, g_{n})$. We have a morphism of Hopf algebras
					\begin{align*}
						\Opgl &\longrightarrow  k\Gamma \#\mathcal{O}^{\Gamma}_{q}(\mathrm{SL}_{n}(k))\\
						x_{ij} &\longmapsto g^{-1}_{n}g_{i} \# a_{ij}\\
						(D^{\bar{p}}_{q})^{-1} &\longmapsto g_{n}^{n} \# 1.
					\end{align*}
					\ep
					\bpf
					The proof is lengthy and somewhat technical, so we defer the details to the appendix.
					\epf
		
					While this proposition relates $\mathcal{O}^{\Gamma}_{q}(\mathrm{SL}_{n}(k))$ to $\mathcal{O}^{\bar{p}}_{q}(\mathrm{SL}_{n}(k))$ from Section~\ref{subsec : SL_n}, there is no reason for the two braided Hopf algebras to coincide, even when $\Gamma=\mathbb{Z}$. Examining Proposition~\ref{prop:kZ coaction SLn} shows that one may restrict to the datum $(\mathbb{Z}, \psi, (g_{1}, \dots, g_{n}))$, where $\mathbb{Z}=\langle z\rangle$ is the infinite cyclic group endowed with the bicharacter $\psi$ uniquely determined by $\psi(z,z)=\xi \in k^{*}$, and where $g_{1}=z^{n-1}$ and $g_{i}=z^{-1}$ for $2 \leq i \leq n$.
					We thus consider the matrix  \[
					\bar{p} =
					\begin{pmatrix}
						1 & \xi^{n} & \cdots & \xi^{n} \\
						\xi^{-n} & 1 & \cdots & 1 \\
						\vdots & \vdots & \ddots & \vdots \\
						\xi^{-n} & 1 & \cdots & 1
					\end{pmatrix}
					\in \mathrm{AST}(n).\]
					Then we obtain the following result:
					\bc \label{cor:knz}
					The bosonization
					$k(n\mathbb{Z})\# \mathcal{O}^{\Z}_{q}(\mathrm{SL}_{n}(k))$
					is isomorphic to $\mathcal{O}^{\bar{p}}_{q}(\mathrm{GL}_{n}(k))$ via the map
					\begin{align*}
						\mathcal{O}^{\bar{p}}_{q}(\mathrm{GL}_{n}(k))
						&\longrightarrow
						k(n\mathbb{Z})\# \mathcal{O}^{\Z}_{q}(\mathrm{SL}_{n}(k)),\\
						x_{ij}
						&\longmapsto \begin{cases}
							z^{n}\# a_{1j},\\
							1 \# a_{ij} \quad \text{for $i > 1$}
						\end{cases}\\
						(D^{\bar{p}}_{q})^{-1}
						&\longmapsto
						z^{-n}\# 1.
					\end{align*}
					\ec
					\bpf
					The morphism is a special case of Proposition~\ref{Prop : morphismGLn}. Moreover, it is an isomorphism, with inverse given by
					\begin{align*}
						\theta \colon	k(n\mathbb{Z})\# \mathcal{O}^{\Z}_{q}(\mathrm{SL}_{n}(k))	
						&\longrightarrow
						\mathcal{O}^{\bar{p}}_{q}(\mathrm{GL}_{n}(k)),\\
						1	\# a_{1j}
						&\longmapsto \begin{cases}
							(D^{\bar{p}}_{q})^{-1}x_{ij},\\
							x_{ij}\quad \text{for $i > 1$}
						\end{cases}\\
						z^{n}\# 1
						&\longmapsto
						D^{\bar{p}}_{q}. 
					\end{align*}
					We leave the verification to the reader.
					\epf
					
					\brq \label{Rem : semidirect}
					Let $q \in k^{*}$. The semidirect product $k(n\Z) \ltimes \mathcal{O}_{q}(\mathrm{SL}_{n}(k))$ is also isomorphic to $\mathcal{O}_{q}(\mathrm{GL}_{n}(k))$ via 
					\begin{align*}
						\Ogl &\xrightarrow{\; \simeq\;} k(n\Z) \ltimes \OQ\\
						x_{ij} &\longmapsto \begin{cases}
							z^{n} \# x_{ij} \quad \text{for $ 1 \leq i < n, \forall j$}\\
							1 \# x_{nj} \quad \text{otherwise}
						\end{cases}\\
						D^{-1}_{q} &\longmapsto z^{-n} \# 1.
					\end{align*}
					\erq	
					To conclude this subsection, we investigate the relationship between the transmuted algebras of $\mathcal{O}_{q}(\mathrm{SL}_{n}(k))$ and the algebra of coinvariants of $\mathcal{O}_{q}^{\bar{p}}(\mathrm{GL}_{n}(k)).$
					
					\bt \label{thm:comparasion}
					Let $q \in k^{*}$. Consider the datum $(\mathbb{Z}, \xi, (g_{1}, \dots, g_{n}))$ with $g_{1} = z^{n-1}$ and $g_{i} = z^{-1}$ for $2 \leq i \leq n$. Consider the matrix  \[
					\bar{p} =
					\begin{pmatrix}
						1 & \xi^{n} & \cdots & \xi^{n} \\
						\xi^{-n} & 1 & \cdots & 1 \\
						\vdots & \vdots & \ddots & \vdots \\
						\xi^{-n} & 1 & \cdots & 1
					\end{pmatrix}
					\in \mathrm{AST}(n).\]
					Then there exists an isomorphism of algebras
					\begin{align*}
						\mathcal{O}^{\bar{p}}_{q}(\mathrm{SL}_{n}(k)) &\longrightarrow \mathcal{O}^{\Z}_{q}(\mathrm{SL}_{n}(k))\\
						X_{ij} &\longmapsto a_{ij}.
					\end{align*}
					
						\et
						\bpf
						This is a direct consequence of Proposition \ref{Prop : IsoSL} and Corollary \ref{cor:knz}.
						\epf
						\subsubsection{Braided bilinear cogroupoid}
						First, let us we recall the cogroupoid $\B$ in \cite{MR3285340}. Let $E \in \GL_{m}(k)$ and let $F \in \GL_{n}(k)$. The algebra $\B(E, F)$ is the universal algebra with generators $a_{ij}, 1 \leq i \leq m, 1 \leq j \leq n$, satisfying the relations:
						\begin{equation}F^{-1}a^{t}Ea = I_{n}; \quad aF^{-1}a^{t}E = I_{m}. \label{B(E, F)}
						\end{equation}
						When $E = F$, we write $\B(E) := \B(E, E)$, this is the Hopf algebra introduced in \cite{DUBOISVIOLETTE1990175}, and is interpreted as the quantum symmetry group of the bilinear form corresponding to $E$.
						
						For any $E \in \GL_{m}(k), F \in \GL_{n}(k)$, and $G \in  \GL_{p}(k)$, \cite{MR3285340} there are algebra maps:
						\begin{align} \label{B(E, F) maps}
							\Delta^{G}_{E, F}: \B{(E, F)} &\longrightarrow \B(E, G) \otimes \B(G, F), \nonumber\\
							a_{ij} &\longmapsto \sum^{p}_{\ell = 1} a_{i\ell} \otimes a_{\ell j} \nonumber\\
							\varepsilon_{E} : \B(E) &\longrightarrow k,\\
							a_{ij} &\longmapsto \delta_{ij} \nonumber\\
							S_{E, F} : \B{(E, F)} &\longrightarrow \B(F, E)^{op} \nonumber\\
							a &\longmapsto E^{-1}a^{t}F. \nonumber
						\end{align}
						
						The cogroupoid $\B$ is defined as follows \cite{MR3285340}: 
						\begin{itemize}
							\item [(i)] $\ob(\B) = \{E \in \GL_{m}(k), m \geq 1\}$.
							\item [(ii)] for $E, F \in \ob(\B),$ the algebra $\B(E, F)$ is the algebra defined in \eqref{B(E, F)}.
							\item [(iii)] the structural maps $\Delta^{\bullet}_{\bullet, \bullet}$, $\varepsilon_{\bullet}$ and $S_{\bullet, \bullet}$ are defined in \eqref{B(E, F) maps}.
						\end{itemize}
						\brq \label{remark: basis}
						Let $E \in \GL_{m}(k), F \in \GL_{n}(k)$ with $m,n  \geq 2$. We note that (see e.g. \cite[Proposition 3.4]{MR1998031}) $\B(E, F) \neq 0$ if and only if  Tr($E^{-1}E^{t}) =$ Tr($F^{-1}F^{t})$. In this case, the subcogroupoid of $\B$ with objects $E$ and $F$ is faithfully flat.
						\erq
						We now aim to construct a braided version of the bilinear cogroupoid $\B$ using Theorem~\ref{Thm : Transmutation} (related considerations are studied in \cite{HN24}, in the compact framework). To this end, we begin with a lemma and a definition:
						\bl \label{Lem:morphism kZ}
						Let $E \in \GL_{m}(k)$ and let $\Gamma$ be an abelian group. Let $\{g_{1}, \dots, g_{m}\}$ be a family of elements of $\Gamma$ such that, for all $1 \leq i, j \leq m,$ we have $g_{i}g_{j} = 1$ whenever $E_{ij} \neq 0$ or $E^{-1}_{ij} \neq 0.$ There exists a Hopf algebra morphism
						\[\pi_{E} : \B(E) \longrightarrow k\Gamma\]
						defined by $\pi_{E}(a_{ij}) = \delta_{ij}g_i$ for all $1 \leq i, j \leq m$.
						\el

						\bd \label{Def : datum}
						Let $E \in \GL_{m}(k)$ and $F \in \GL_{n}(k)$. An $(E, F)$-\textbf{datum} $\mathbf{\Gamma}$ is a tuple $ \big(\Gamma, \psi; \bar{g}, \bar{h}\big)$ where 
						\begin{itemize}
							\item $\Gamma$ is an abelian group endowed with the bicharater $\psi \colon \Gamma \times \Gamma \to k^{*}$,
							\item $\bar{g} = (g_{1}, \cdots, g_{m})$ and $\bar{h} = (h_{1}, \cdots, h_{n})$ are families of elements of $\Gamma$ such that $g_{i}g_{j} = 1$ whenever $E_{ij} \neq 0$ or $E^{-1}_{ij} \neq 0$ and $h_{i}h_{j} = 1$ whenever $F_{ij} \neq 0$ or $F^{-1}_{ij} \neq 0.$
						\end{itemize}
						
						When $E = F$ and $\bar{g} = \bar{h}$, we denote the datum by the triple $\big(\Gamma, \psi; \bar{g}\big)$ and call it an $E$-\textbf{datum}. 
						
						For multiple matrices, let $\mathcal{G} = \bigsqcup_{n \geq 1} \GL_n(k)$. A $\mathcal{G}$-\textbf{datum} $\mathbf{\Gamma}:= (\Gamma, \psi;( \bar{g}^{E})_{E \in \mathcal{G}})$ is defined analogously, where $\bar{g}^{E} = (g^{E}_{1}, \cdots, g^{E}_{n})$ is a family of elements of $\Gamma$ depending on $E  \in \GL_{n}(k)$.
						\ed
						\bd \label{Def : Bpi}
						Let $E \in \GL_{m}(k)$ and $F \in \GL_{n}(k)$. Let $\big(\Gamma, \psi; \bar{g}, \bar{h}\big)$ be an $(E, F)$-datum. The algebra $\B^{\Gamma}(E, F)$ is the algebra presented by generators $a_{ij}$ $(1 \leq i \leq m, 1 \leq j \leq n)$ subject to the relations:
						\begin{equation}
							F^{-1}\tilde{a}^{t}\widetilde{E}a = I_{n}, \quad aF^{-1}\tilde{a}^{t}\widetilde{E} = I_{n} \label{Trans BE (1)}
						\end{equation}
						where, for all $1 \leq i \leq m$, $1 \leq j \leq n$, $\tilde{a}_{ij} = a_{ij}\psi(h_j, g^{-1}_i)$ and $\widetilde{E}_{ij} = E_{ij}\psi(g^{-1}_{i}, g_{j})$. Note that $\widetilde{E}$ is an invertible matrix whose inverse is given by $(\widetilde{E}^{-1})_{ij} = (E^{-1})_{ij}\psi(g_{i}, g_{j})$.
						\ed

						As in the ordinary case, when $E= F$, we write $\B^{\Gamma}(E)$ for $\B^{\Gamma}(E, E)$, the algebra associated with the $E$-datum $(\Gamma, \psi; \bar{g})$.
						\bp \label{Prop : B Gamma}
						Let $E \in \GL_{m}(k)$ and $F \in \GL_{n}(k)$ and let $\big(\Gamma, \psi; \bar{g}, \bar{h}\big)$ be an $(E, F)$-datum. The algebra $\B^{\Gamma}(E, F)$ has a $k\Gamma$-comodule algebra structure whose coaction is defined by
						\begin{align*}
							\beta : \B^{\Gamma}(E, F) &\longrightarrow \B^{\Gamma}(E, F) \otimes k\Gamma\\
							a_{ij} &\longmapsto a_{ij} \otimes g^{-1}_{i}h_{j}.
						\end{align*}
						\ep
						\bpf
						This is a straightforward verification.
						\epf
						\bp
						Let $E \in \GL_{m}(k)$ and $F \in \GL_{n}(k)$ and let $\big(\Gamma, \psi; \bar{g}, \bar{h}\big)$ be an $(E, F)$-datum. We consider the following two Hopf algebra morphisms:
						\begin{align*}
							\pi_{E} : \B(E) &\longrightarrow k\Gamma, \hspace*{2cm}\pi_{F} : \B(F) \longrightarrow k\Gamma\\
							a_{ij}&\longmapsto \delta_{ij}g_{i} \hspace*{3cm} a_{ij} \longmapsto \delta_{ij}h_{i}.
						\end{align*}
						The algebra $\B^{\Gamma}(E, F)$ is isomorphic to the transmutation $\B^{\pi}(E, F)$ of $\B(E, F)$.
						\ep
						\bpf
						In $\B(E, F)$, we have
						\begin{align*}
							\sum_{k,l}a_{ki}E_{kl}a_{lj} = F_{ij}.
						\end{align*}
						Using Theorem \ref{Thm : Transmutation}, we obtain, in $\B^{\pi}(E, F),$
						\[[a_{ki}] \cdot [a_{lj}] = [a_{ki}a_{lj}]\psi(g^{-1}_{k}h_{i}, g^{-1}_{l}).\]
						Thus,
						\begin{align*}
							F_{ij} = \big[\sum_{k,l}a_{ki}E_{kl}a_{lj}\big] &= \sum_{k,l}E_{kl}\big[a_{ki}a_{lj}\big]\\
							&= \sum_{k,l}E_{kl}[a_{ki}] \cdot [b_{lj}]\psi(g^{-1}_{k}h_{i}, g_{l})\\
							&= \sum_{k,l}E_{kl}[a_{ki}] \cdot [b_{lj}]\psi(g^{-1}_{k}, g_{l})\psi(h_{i}, g_{l})\\
							&= \sum_{k,l}E_{kl}\psi(g^{-1}_{k}, g_{l})\psi(h_{i}, g^{-1}_{k})[a_{ki}] \cdot [b_{lj}]\quad (\text{for $E_{kl} \neq 0$}).
						\end{align*}
						Similarly, we also have, in $\B^{\pi}(E, F)$,
						\begin{align*}
							(E^{-1})_{ij} = \big[\sum_{k,l}a_{ik}(F^{-1})_{kl}a_{jl}\big] &= \sum_{k,l}(F^{-1})_{kl}\big[a_{ik}a_{jl}\big]\\
							&= \sum_{k,l}(F^{-1})_{kl}[a_{ik}] \cdot [a_{jl}] \psi(g^{-1}_{i}h_{k}, g_{j})\\
							&= \sum_{k,l}(F^{-1})_{kl}[a_{ik}] \cdot [a_{jl}] \psi(g^{-1}_{i}, g_{j})\psi(h_{k}, g_{j})\\
							&= \sum_{k,l}(F^{-1})_{kl}[a_{ik}] \cdot [a_{jl}]\psi(h^{-1}_{l}, g_{j}) \psi(g^{-1}_{i}, g_{j}) \quad (\text{for $(F^{-1})_{kl} \neq 0$}).
						\end{align*}
						Thus, we obtain
						\[(E^{-1})_{ij}\psi(g_{i}, g_{j}) = \sum_{k,l}(F^{-1})_{kl}[a_{ik} ]\cdot [a_{jl}]\psi(h_{l}, g^{-1}_{j}),\]
						and then these relations ensure that there exists a unique algebra morphism
						\begin{align*}
							f : \B^{\Gamma}(E, F) &\longrightarrow \B^{\pi}(E, F)\\
							a_{ij} &\longmapsto [a_{ij}]
						\end{align*}
						The map $f$ is indeed an isomorphism, as follows once again from Remark \ref{Rk: twist trans} and \cite[Section 3]{MR2945585}.
						\epf
						%
						%
						\bl \label{Lem : structure B(E)}
						\begin{enumerate}
							\item For any $E \in \GL_{m}(k), F \in \GL_{n}(k)$, and $G \in  \GL_{p}(k)$, let $(\Gamma, \psi; \bar{g}, \bar{h}, \bar{k})$ be an $(E, F, G)$-datum. There exist algebra maps in $\MM^{k\Gamma}$:
							\begin{align*}
								\Delta^{G}_{E, F}: \B^{\Gamma}{(E, F)} &\longrightarrow \B^{\Gamma}(E, G) \otimes \B^{\Gamma}(G, F)\\
								a_{ij} &\longmapsto \sum^{p}_{\ell = 1} a_{i\ell} \otimes a_{\ell j};\\
								\varepsilon_{E} : \B^{\Gamma}(E) &\longrightarrow k\\
								a_{ij} &\longmapsto \delta_{ij}
							\end{align*}
							and for any $H \in \GL_{m}(k)$, the following diagrams commute
							\[\begin{tikzcd}[scale = 0.6]
								\B^{\Gamma}(E, F) \ar[d, "\Delta^{H}_{E, F}"]\ar[r, "\Delta^{G}_{E, F}"] &\B^{\Gamma}(E, G) \otimes \B^{\Gamma}(G, F) \ar[d, "\Delta^{H}_{E, G} \otimes 1"]\\
								\B^{\Gamma}(E, H) \otimes \B^{\Gamma}(H, F) \ar[r, "1 \otimes \Delta^{G}_{H, F}"] &\B^{\Gamma}(E, H) \otimes \B^{\Gamma}(H, G) \otimes \B^{\Gamma}(G, R)
							\end{tikzcd}\]
							\[\footnotesize\begin{tikzcd}                         
								\B^{\Gamma}(E, F) \ar[rd, equals, shift left = 1ex] \ar[d, "\Delta^{F}_{E, F}"] & &\B^{\Gamma}(E, F) \ar[rd, equals, shift left = 1ex] \ar[d, "\Delta^{E}_{E, F}"]&\\
								\B^{\Gamma}(E, F)  \otimes \B^{\Gamma}(F, F)  \ar[r, "1 \otimes \varepsilon_{F}"] &\B^{\Gamma}(E, F) &\B^{\Gamma}(E, E) \otimes \B^{\Gamma}(E, F) \ar[r, "\varepsilon_{E} \otimes 1"] &B^{\Gamma}(E, F)
							\end{tikzcd}\]
							\item For any $E \in \GL_{m}(k), F \in \GL_{n}(k)$, let $(\Gamma, \psi; \bar{g}, \bar{h})$ be an $(E, F)$-datum. There exists an algebra map in $\MM^{k\Gamma}$:
							\begin{align*}
								S^{\Gamma}_{E, F} \colon \B^{\Gamma}(E, F) &\longrightarrow \B^{\Gamma}(F, E)^{op, c}\\
								a &\longmapsto E^{-1}\tilde{a}^{t}\widetilde{F}
							\end{align*}
							such that the following diagrams commute
							\[\begin{tikzcd}[scale = 0.6]
								\B^{\Gamma}(E, E) \ar[d, "\Delta^{F}_{E, E}"]\ar[r, "\varepsilon_{E}"] &k \ar[r, "u"] &\B^{\Gamma}(E, F)\\
								\B^{\Gamma}(E, F) \otimes \B^{\Gamma}(F, E) \ar[rr, "1 \otimes S^{\Gamma}_{F, E}"] &&\B^{\Gamma}(E, F) \otimes \B^{\Gamma}(E, F) \ar[u, "m"]
							\end{tikzcd}\]
							\[\begin{tikzcd}[scale = 0.6]
								\B^{\Gamma}(E, E) \ar[d, "\Delta^{F}_{E, E}"]\ar[r, "\varepsilon_{E}"] &k \ar[r, "u"] &\B^{\Gamma}(F, E)\\
								\B^{\Gamma}(E, F) \otimes \B^{\Gamma}(F, E) \ar[rr, "S^{\Gamma}_{F, E} \otimes 1"] &&\B^{\Gamma}(F, E) \otimes \B^{\Gamma}(F, E) \ar[u, "m"]
							\end{tikzcd}\]
						\end{enumerate}
						\el
						\bpf
						This is a direct consequence of Proposition~\ref{Prop : B Gamma} and the transport of the $\MM^{k \Gamma}$-cogroupoid structures given in Theorem~\ref{Thm : Transmutation}.
						\epf
						\brq
						It follows that, for any $E \in \GL_n(k)$ together with an associated $E$-datum $\mathbf{\Gamma}$, the algebra $\B^{\Gamma}(E)$ is a Hopf algebra in the category $\MM^{k\Gamma, \psi}$.
						\erq
						\bdep \label{Prodef : Bpi}
						Let $\mathcal{G} = \bigsqcup_{n \geq 1} \GL_n(k)$ and let $\mathbf{\Gamma}$ be a $\mathcal{G}$-datum. The $\MM^{k\Gamma}$-cogroupoid, called the braided bilinear cogroupoid $\B^{\Gamma}$, is defined as follows:
						\begin{enumerate}
							\item $\ob(\B^{\Gamma}) = \mathcal{G}$;
							\item for $E, F \in \ob(\B^{\Gamma})$ the algebra $\B^{\Gamma}(E, F)$ is defined in Definition \ref{Def : Bpi}.
							\item the structural maps $\Delta^{\bullet}_{\bullet, \bullet}$, $\varepsilon_{\bullet}$ and $S^{\Gamma}_{\bullet, \bullet}$ are defined in Lemma \ref{Lem : structure B(E)}.
						\end{enumerate}
						\edep
						\bc \label{Cor: equi Bpi}
						Let $E \in \GL_{n}(k)$, $F \in \GL_{m}(k)$ be such that $\mathrm{tr}(E^{-1}E^{t}) = \mathrm{tr}(F^{-1}F^{t})$ and let $\mathbf{\Gamma} =\big(\Gamma, \psi; \bar{g}, \bar{h}\big)$ be an $(E, F)$-datum. Consider the subgroups $\Gamma_{0} = \langle g^{-1}_{i}g_{j} \rangle$ of $\Gamma$.
						Then there exist $k$-linear monoidal equivalences of categories
						\[
						(\MM^{k\Gamma})^{\B^{\Gamma}(E)} \cong^{\otimes} (\MM^{k\Gamma})^{\B^{\Gamma}(F)},
						\]
						\[
						\MM^{k\Gamma_{0} \# \B^{\Gamma}(E)} \cong^{\otimes} \MM^{k\Gamma_{0} \ltimes \B(E)}.
						\]
						\ec
						\bpf
						By our assumptions and Proposition--Definition~\ref{Prodef : Bpi}, we have a braided bilinear cogroupoid $\B^{\Gamma}$ with two objects $E$ and $F$, which, by Remark~\ref{remark: basis}, is a faithfully flat $k\Gamma$-cogroupoid whenever $\mathrm{tr}(E^{-1}E^{t}) = \mathrm{tr}(F^{-1}F^{t})$. The first monoidal equivalence then follows directly from Theorem~\ref{Thm: monoid func}, while the second equivalence is an immediate consequence of Corollary~\ref{Cor : equiboso}.
						\epf

						We now relate the bosonization of $\B^{\Gamma}(E)$ to some familiar Hopf algebras.
						Recall from \cite{MR3275027} that, for $A, B \in {\rm GL}_{n}(k)$, the algebra $\mathcal{G}(A, B)$ is defined as follows:
						\[\mathcal{G}(A, B) = k \langle (x_{ij})_{1 \leq i, j \leq n}, d \mid x^{t}Ax = Ad, xBx^{t} = Bd, dd^{-1} =1= d^{-1}d\rangle \]
						We also obtain the following results concerning bosonization:
						\bp\label{Prop : boso BpiE}
						Let $E \in \GL_{n}(k)$ and let $(\Gamma, \psi; \bar{g})$ be an $E$-datum. Let $g \in \Gamma$, we define the matrices $E'$ and $F'$ by
						\[
						E'_{ij} = \psi(g_{i}, g)\, E_{ij} \quad \text{and} \quad E''_{ij} = \psi(g_{i}^{-1}, g)\, E^{-1}_{ij},
						\]
						for $1 \leq i, j \leq n$. Then there exists a Hopf algebra morphism
						\begin{align*}
							\mathcal{G}(E', E'') &\longrightarrow k\Gamma \# \B^{\Gamma}(E)\\
							x_{ij} &\longmapsto gg_{i} \# a_{ij}\\
							d &\longmapsto g^{2} \# 1.
						\end{align*}
						\ep
						\bpf
						There exists a unique morphism of algebras
						\begin{align*}
							\phi \colon k \langle x_{ij}, d \rangle &\longrightarrow k\Gamma \# \B^{\Gamma}(E)\\
							x_{ij} &\longmapsto gg_{i} \# a_{ij}\\
							d &\longmapsto g^{2} \# 1.
						\end{align*}
						We have 
						\begin{align*}
							\phi \big(\sum_{k, \ell} x_{ki} E'_{k\ell} x_{\ell j}\big) &= \sum_{k, \ell} (gg_{k} \# a_{ki})E'_{k\ell} (gg_{\ell} \# a_{\ell j})\\
							&= \sum_{k, \ell} g^{2}g_{k}g_{\ell} \# a_{ki}E_{k\ell}a_{\ell j} \psi(g_{k}, g)\psi(g^{-1}_{k}g_{i}, gg_{\ell})\\
							&= \sum_{k, \ell} g^{2}g_{k}g_{\ell} \# a_{ki}E_{k\ell}a_{\ell j} \psi(g^{-1}_{k}, g_{\ell})\psi(g_{i}, gg_{\ell})\\
							&= \sum_{k, \ell} g^{2} \# a_{ki} \psi(g_{i}, g^{-1}_{k})E_{k\ell}\psi(g_{k}, g_{k})a_{\ell j} \psi(g_{i}, g) \quad \text{for $E_{k\ell} \neq 0$}\\
							&= (g^{2} \# 1)\big(1 \# \sum_{k, \ell} \tilde{a}_{ki} \tilde{E}_{k\ell} a_{\ell j}\big)\psi(g_{i}, g)\\
							&= E'_{ij}\phi(d).
						\end{align*}
						By a similar computation, we also obtain
						\[\phi \big(\sum_{k\ell} x_{ik}E''_{k \ell}x_{j \ell}\big) = F'_{ij}\phi(d).\]
						Hence, we obtain the desired algebra morphism. A straightforward verification shows that it is compatible with the Hopf algebra structures, which completes the proof.
						\epf
						Consider \[
						E_{q} =
						\begin{pmatrix}
							0 & 1 \\
							-q^{-1} & 0
						\end{pmatrix}, 
						\]
						it is well known that $\Oo = \B(E_{q})$. Now let $\Gamma$ be the infinite cyclic group $\mathbb{Z} = \langle z \rangle$, endowed with the bicharacter $\psi$ uniquely determined by $\psi(z, z) = \xi \in k^{*}$, and consider the $E_{q}$-datum $\big(\Z, \xi; \{z, z^{-1}\}\big)$. 
						One can readily check that $\B^{\Z}(E_{q}) = \mathcal{O}_{q \xi^{-2}, q \xi^{2}}(\mathrm{SL}_{2}(k))$. Under the assumptions of the above proposition, we have
						\[E'_{q} =
						\begin{pmatrix}
							0 & \xi \\
							-q^{-1}\xi^{-1} & 0
						\end{pmatrix} \quad \text{and} \quad E''_{q} =\begin{pmatrix}
							0 & \xi^{-1} \\
							-q^{-1}\xi & 0
						\end{pmatrix}.
						\]
						A straightforward computation shows that $\mathcal{G}(E'_{q}, E''_{q}) = \mathcal{O}_{q \xi^{-2}, q \xi^{2}}(\mathrm{GL}_{2}(k))$, which is defined in \cite{Tak90}.
						
						%
						%
						For all the results that follow, we work with $\mathbb{Z} = \langle z \rangle$, the infinite cyclic group endowed with the bicharacter $\psi \colon \mathbb{Z} \times \mathbb{Z} \to k$ uniquely determined by $\psi(z, z) = \xi$.
						\bc \label{Cor : equibosoBE}
						Let $E \in \GL_{n}(k)$, and consider the subset of integers
						\[
						\mathcal{I}_{E} := \big\{ \bar{d} = (d_{1}, \cdots, d_{n}) \mid \; d_{i} + d_{j} = 1\,  \text{whenever } E_{ij} \neq 0 \text{ or } E^{-1}_{ij} \neq 0; 1 \leq i, j \leq n \big\}.
						\] 
						For any $\bar{d} \in \mathcal{I}_{E}$, the triple $\big(\Z, \xi; z^{2\bar{d} -1}\big)$, with $z^{2\bar{d} -1} := (z^{2d_{1} -1}, \cdots, z^{2d_{n} -1})$, is an $E$-datum, and there exists a Hopf algebra isomorphism
						\begin{align*}
							\varphi \colon	\mathcal{G}(E', F') &\longrightarrow k2\mathbb{Z} \# \B^{\Z}(E), \\
							x_{ij} &\longmapsto z^{2d_{i}} \# a_{ij}, \\
							d &\longmapsto z^{2} \# 1.
						\end{align*}
						where $E'$ and $F'$ are the matrices defined by
						\[
						E'_{ij} = \xi^{2d_{i} - 1} E_{ij}, \quad F'_{ij} = \xi^{-2d_{i} + 1} E^{-1}_{ij}, \quad 1 \leq i, j \leq n, \ \bar{d} \in \mathcal{I}_{E}.
						\]
						\ec
						\bpf
						Applying Proposition~\ref{Prop : boso BpiE} with $g = z$ and $g_{i} = z^{2d_{i} - 1}$, there exists a unique Hopf algebra morphism $\varphi$ as announced. Define
						\begin{align*}
							\phi \colon	k2\Z \# \B^{\Z}(E) &\longrightarrow \mathcal{G}(E', F')\\
							1 \# a_{ij} &\longmapsto d^{-d_{i}}x_{ij}\\
							z^{2} \# 1 &\longmapsto d. 
						\end{align*}
						Note that it is straightforward to check that $\phi\big((1 \# a_{ij}).(z^{2} \# 1)\big) = \phi(1 \# a_{ij})\phi((z^{2} \# 1))$, so $\phi$ is a well-defined algebra morphism and is indeed the inverse of $\varphi$.
						\epf
						\brq \label{Rk: k2Z cogroup}
						In the case of the infinite cyclic group $\mathbb{Z} = \langle z \rangle$, consider the subsets of integers $\mathcal{I}_{E}$ for $E \in \mathcal{G} = \bigsqcup_{n \geq 1} \GL_{n}(k)$, as given in Corollary~\ref{Cor : equibosoBE}. For $E \in \GL_{m}(k)$, $F \in \GL_n(k)$ and for $\bar{d}^{E}= (d^{E}_{1}, \cdots, d^{E}_{m}) \in \mathcal{I}_{E}$, $\bar{d}^{F}= (d^{F}_{1}, \cdots, d^{F}_{n}) \in \mathcal{I}_{F}$, the map
						\begin{align*}
							\mathrm{ad}_{E, F} \colon \B(E, F) &\longrightarrow \B(E, F) \otimes k2\mathbb{Z}, \\
							a_{ij} &\longmapsto a_{ij} \otimes z^{2(d^{F}_{j} - d^{E}_{i})}
						\end{align*}
						defines a $k2\mathbb{Z}$-coaction on $\B(E, F)$. In this setting, we may also view $\B$ as a cogroupoid over $\mathcal{M}^{k2\mathbb{Z}, \varepsilon \otimes \varepsilon}$.
						\erq
						In light of this remark, we conclude this section with the following result:

						\bc \label{Cor : G(E, F)}
						Let $q \in k^{*}$. Let $E \in \GL_{n}(k)$ such that $\operatorname{tr}(E^{-1}E^{\mathrm{t}}) = -q - q^{-1}.$ Consider the set
						\[
						\mathcal{I}_{E} := \big\{ \bar{d} = (d_{1}, \cdots, d_{n}) \mid \; d_{i} + d_{j} = 1\,  \text{whenever } E_{ij} \neq 0 \text{ or } E^{-1}_{ij} \neq 0; 1 \leq i, j \leq n \big\},\] and the $E$-datum $(\Z, \xi; z^{2\bar{d}-1})$.
						Then we have a $k$-linear equivalence of monoidal categories
						\[\MM^{\mathcal{G}(E', F')} \cong^{\otimes} \MM^{\mathcal{O}_{q}({\rm GL}_2(k))},\]
						where $E'$ and $F'$ are the matrices defined by
						\[
						E'_{ij} = \xi^{2d_{i} - 1} E_{ij}, \quad F'_{ij} = \xi^{-2d_{i} + 1} E^{-1}_{ij}, \quad 1 \leq i, j \leq n, \ \bar{d} \in \mathcal{I}_{E}.
						\]
						\ec
						\bpf
						We have
						\begin{align*}
							\MM^{\mathcal{G}(E', F')} &\cong^{\otimes} \MM^{k2\Z \# \B^{\Z}(E)} \quad (\text{since Corollary \ref{Cor : equibosoBE},})\\
							&\cong^{\otimes} \MM^{k2\Z \ltimes  \B(E)} \quad (\text{by Proposition \ref{Prop : B Gamma} and Corollary \ref{Cor : equiboso}})\\
							&\cong^{\otimes} \MM^{k2\Z \ltimes  \mathcal{O}_{q}(\text{SL}_{2}(k))} \quad (\text{by Remark \ref{Rk: k2Z cogroup}})\\
							&\cong^{\otimes} \MM^{\mathcal{O}_{q}(\text{GL}_{2}(k))} \quad (\text{by Remark \ref{Rem : semidirect} for $n = 2$}).
						\end{align*}
						This establishes our desired equivalence.
				
						\epf
						We also observe that this corollary recovers part of C. Mrozinski's theorem \cite[Theorem 1.1]{MR3275027}.
						%

						\appendix
						\renewcommand{\thesection}{}
						\section*{Appendix : Proof of Proposition \ref{Prop : morphismGLn}}
						Let $q \in k^{*}$, and consider the datum $\big(\Gamma, \psi, (g_{1}, \cdots, g_{n})\big)$, where $\Gamma$ is an abelian group endowed with a bicharacter $\psi \colon \Gamma \times \Gamma \to k^{*}$, and $g_{1}, \dots, g_{n} \in \Gamma$ are elements satisfying $\prod_{i = 1}^{n} g_{i} = 1$. Let $\bar{p} =(p_{ij})_{1 \leq i, j \leq n} \in M_{n}(k)$ with $p_{ij} = \psi(g^{-1}_{i}g_{j}, g_{n})$. We conclude this paper by proving the existence of a Hopf algebra morphism $$\Opgl \longrightarrow  k\Gamma \#\mathcal{O}^{\Gamma}_{q}(\mathrm{SL}_{n}(k))$$ thereby providing the proof of Proposition \ref{Prop : morphismGLn}. We begin with the following proposition:
						
						\begin{lemma*}
						There exists a unique algebra morphism
							\begin{align*}
						 \mathcal{O}^{\bar{p}}_{q}(\mathrm{M}_{n}(k)) &\longrightarrow k\Gamma \#\mathcal{O}^{\Gamma}_{q}(\mathrm{SL}_{n}(k))\\
							x_{ij} &\longmapsto
							g^{-1}_{n}g_{i} \# a_{ij},
						\end{align*}
						\end{lemma*}
						\bpf
						Consider the morphism of algebras
						\begin{align*}
							f \colon k\langle x_{ij} \rangle &\longrightarrow k\Gamma \#\mathcal{O}^{\Gamma}_{q}(\mathrm{SL}_{n}(k))\\
						x_{ij} &\longmapsto
						g^{-1}_{n}g_{i} \# a_{ij}.
						\end{align*}
						We have
						\begin{enumerate}
							\item for $k < m$ and for any $i <n$,
							\begin{align*}
								f(x_{ik}x_{im}) &= (g^{-1}_{n}g_{i} \# a_{ik})(g^{-1}_{n}g_{i} \# a_{im})\\
								&= g^{-2}_{n}g^{2}_{i} \# a_{ik}a_{im} \psi(g^{-1}_{i}g_{k}, g^{-1}_{n}g_{i}).\\
								f(x_{im}x_{ik}) &= (g^{-1}_{n}g_{i} \# a_{im})(g^{-1}_{n}g_{i} \# a_{ik})\\
								&= g^{-2}_{n}g^{2}_{i} \# a_{im}a_{ik} \psi(g^{-1}_{i}g_{m}, g^{-1}_{n}g_{i}).\\
								&= g^{-2}_{n}g^{2}_{i} \# q \psi(g^{-1}_{m}g_{k}, g_{i})\psi(g^{-1}_{i}g_{m}, g^{-1}_{n}g_{i}) a_{ik}a_{im}.
							\end{align*}
							Hence, $f(x_{im}x_{ik}) = q \psi(g_{k}g^{-1}_{m}, g_{n})f(x_{ik}x_{im}) = qp_{mk}f(x_{ik}x_{im}).$
							And for $k <m$,
					\begin{align*}
						f(x_{nk}x_{nm}) &= (1 \# a_{nk})(1\# a_{nm}) = 1 \# a_{nk}a_{nm},\\
								f(x_{nm}x_{nk}) &= (1 \# a_{nm})(1 \# a_{nk})\\
								&= 1 \# a_{nm}a_{nk} \\
								&= 1 \# q \psi(g^{-1}_{m}g_{k}, g_{n})a_{nk}a_{nm}\\
								&= qp_{mk} f(x_{nk}x_{nm})
							\end{align*}
Thus, for $k < m$ and for any $i \leq n$, $f(x_{im}x_{ik}) = qp_{mk} f(x_{ik}x_{im})$.
							\item For $i < j < n$ and for any $ k \leq n$,
							\begin{align*}
								f(x_{ik}x_{jk}) &= (g^{-1}_{n}g_{i} \# a_{ik})(g^{-1}_{n}g_{j} \# a_{jk})\\
								&= g^{-2}_{n}g_{i}g_{j} \# a_{ik}a_{jk} \psi(g^{-1}_{i}g_{k}, g^{-1}_{n}g_{j}),\\
								f(x_{jk}x_{ik}) &= (g^{-1}_{n}g_{j} \# a_{jk})(g^{-1}_{n}g_{i} \# a_{ik})\\
								&= g^{-2}_{n}g_{i}g_{j} \# a_{jk}a_{ik} \psi(g^{-1}_{j}g_{k}, g^{-1}_{n}g_{i})\\
								&= g^{-2}_{n}g_{i}g_{j} \# q\psi(g^{-1}_{j}g_{k}, g^{-1}_{i}) \psi(g^{-1}_{i}g_{k}, g_{j}) \psi(g^{-1}_{j}g_{k}, g^{-1}_{n}g_{i})a_{ik}a_{jk} \\
								&= g^{-2}_{n}g_{i}g_{j} \# q\psi(g^{-1}_{i}g_{k}, g_{j}) \psi(g^{-1}_{j}g_{k}, g^{-1}_{n})a_{ik}a_{jk}.
							\end{align*}
							So 
								$f(x_{jk}x_{ik}) = q\psi(g_{j}g^{-1}_{i}, g_{n}) f(x_{ik}x_{jk}) = qp_{ij}f(x_{ik}x_{jk}).$
					
						We also have, for $i <n$, $j = n$ and for any $1 \leq k \leq n$,
							\begin{align*}
								f(x_{ik}x_{nk}) &= (g^{-1}_{n}g_{i} \# a_{ik})(1 \# a_{nk}) = g^{-1}_{n}g_{i} \# a_{ik}a_{nk}.\\
								f(x_{nk}x_{ik}) &= (1 \# a_{nk})(g^{-1}_{n}g_{i} \# a_{ik})\\
								&= g^{-1}_{n}g_{i} \# a_{nk}a_{ik} \psi(g^{-1}_{n}g_{k}, g^{-1}_{n}g_{i})\\
								&= g^{-1}_{n}g_{i} \# q\psi(g^{-1}_{n}g_{k}, g^{-1}_{i}) \psi(g^{-1}_{i}g_{k}, g_{n}) \psi(g^{-1}_{n}g_{k}, g^{-1}_{n}g_{i})a_{ik}a_{nk}\\
								&= g^{-1}_{n}g_{i} \# q\psi(g^{-1}_{i}g_{k}, g_{n}) \psi(g^{-1}_{n}g_{k}, g^{-1}_{n})a_{ik}a_{nk}.
							\end{align*}
							Thus \begin{align*}
								f(x_{nk}x_{ik}) &= g^{-1}_{n}g_{i} \# q \psi(g^{-1}_{i}g_{n}, g_{n})f(x_{ik}x_{nk}) = qp_{in}f(x_{ik}x_{nk}).
							\end{align*}
							and hence, for $i < j \leq n$ and for any $ k \leq n$, we obtain \[f(x_{jk}x_{ik}) = q\psi(g_{j}g^{-1}_{i}, g_{n}) f(x_{ik}x_{jk}) = qp_{ij}f(x_{ik}x_{jk}).\]
							\item For $i < j < n$ and $k > m$, 
							\begin{align*}
								f(x_{ik}x_{jm}) &= (g^{-1}_{n}g_{i} \# a_{ik})(g^{-1}_{n}g_{j} \# a_{jm})\\
								&= g^{-2}_{n}g_{i}g_{j} \# \psi(g^{-1}_{i}g_{k}, g^{-1}_{n}g_{j})a_{ik}a_{jm}.
							\end{align*}
								On the other hand,
								\begin{align*}
								f(x_{jm}x_{ik}) &= (g^{-1}_{n}g_{j} \# a_{jm})(g^{-1}_{n}g_{i} \# a_{ik})\\
								&= g^{-2}_{n}g_{i}g_{j} \# \psi(g^{-1}_{j}g_{m}, g^{-1}_{n}g_{i})a_{jm}a_{ik}\\
								&= g^{-2}_{n}g_{i}g_{j} \# \psi(g_{j}g^{-1}_{m}, g_{i}) \psi(g^{-1}_{i}g_{k}, g_{j})\psi(g^{-1}_{j}g_{m}, g^{-1}_{n}g_{i})a_{ik}a_{jm}\\
								&= g^{-2}_{n}g_{i}g_{j} \# \psi(g^{-1}_{i}g_{k}, g_{j})\psi(g^{-1}_{j}g_{m}, g^{-1}_{n})a_{ik}a_{jm}.
							\end{align*}
							Thus \[\psi(g^{-1}_{i}g_{k}, g_{n}) f(x_{ik}x_{jm}) = \psi(g^{-1}_{j}g_{m}, g_{n})f(x_{jm}x_{ik}),\]
						and	hence \[p_{ji}f(x_{jm}x_{ik}) = p_{mk}f(x_{ik}x_{jm}).\]
							For $i < n$ and $k > m$, we also have
							\begin{align*}
								f(x_{ik}x_{nm}) &= (g^{-1}_{n}g_{i} \# a_{ik})(1 \# a_{nm}) = g^{-1}_{n}g_{i} \# a_{ik}a_{nm},
							\end{align*}
							and
							\begin{align*}
								f(x_{nm}x_{ik}) &= (1 \# a_{nm})(g^{-1}_{n}g_{i} \# a_{ik})\\
								&=  g^{-1}_{n}g_{i} \# \psi(g^{-1}_{n}g_{m}, g^{-1}_{n}g_{i})a_{nm}a_{ik}\\
								&= g^{-1}_{n}g_{i} \# \psi(g_{n}g^{-1}_{m}, g_{i}) \psi(g^{-1}_{i}g_{k}, g_{n})\psi(g^{-1}_{n}g_{m}, g^{-1}_{n}g_{i})a_{ik}a_{nm}\\
								&= g^{-1}_{n}g_{i} \#  \psi(g^{-1}_{i}g_{k}, g_{n})\psi(g^{-1}_{n}g_{m}, g^{-1}_{n})a_{ik}a_{nm}
							\end{align*}
							Thus
						$p_{ni}f(x_{nm}x_{ik}) = p_{mk}f(x_{ik}x_{nm}).$
							Finally, \[p_{ji}f(x_{jm}x_{ik}) = p_{mk}f(x_{ik}x_{jm}) \quad \text{for $i < j \leq n$ and $k >m$}.\]
							\item For $i < j < n$ and $k < m$,
							\begin{align*}
								&f(x_{jm}x_{ik}) - f(x_{ik}x_{jm})= (g^{-1}_{n}g_{j} \# a_{jm})(g^{-1}_{n}g_{i} \# a_{ik}) - (g^{-1}_{n}g_{i} \# a_{ik})(g^{-1}_{n}g_{j} \# a_{jm}) \\
								&= g^{-2}_{n}g_{i}g_{j} \# \big(\psi(g^{-1}_{j}g_{m}, g^{-1}_{n}g_{i})a_{jm}a_{ik} - \psi(g^{-1}_{i}g_{k}, g^{-1}_{n}g_{j})a_{ik}a_{jm}\big).
							\end{align*}
							Therefore, we have
							\begin{align*}
								&p_{ji}f(x_{jm}x_{ik}) - p_{mk} f(x_{ik}x_{jm})\\
								&= g^{-2}_{n}g_{i}g_{j} \# \big(\psi(g^{-1}_{j}g_{i}, g_{n})\psi(g^{-1}_{j}g_{m}, g^{-1}_{n}g_{i})a_{jm}a_{ik} - \psi(g^{-1}_{m}g_{k}, g_{n})\psi(g^{-1}_{i}g_{k}, g^{-1}_{n}g_{j})a_{ik}a_{jm}\big)
							\end{align*}
							The factor in $\mathcal{O}^{\Gamma}_{q}(\mathrm{SL}_{n}(k))$ equals
							\begin{align*}
								&\psi(g_{i}, g_{n})\psi(g^{-1}_{j}, g_{i})\psi(g_{m}, g^{-1}_{n}g_{i})a_{jm}a_{ik} - \psi(g^{-1}_{m}, g_{n})\psi(g_{k}, g_{j})\psi(g^{-1}_{i}, g^{-1}_{n}g_{j})a_{ik}a_{jm}\\
								&= \psi(g^{-1}_{i}, g^{-1}_{n}g_{j})\psi(g_{m}, g^{-1}_{n})\Big(\psi(g^{-1}_{j}g_{m}, g_{i})\psi(g_{i}, g_{j})a_{jm}a_{ik} -\psi(g_{k}, g_{j})a_{ik}a_{jm}\Big)\\
								&= (q - q^{-1})\psi(g_{i}g^{-1}_{m}, g_{n})\psi(g^{-1}_{i}, g_{j}) \psi(g_{m}, g_{j})a_{im}a_{jk}\\
								&= (q - q^{-1})\psi(g^{-1}_{i}g_{m}, g^{-1}_{n}g_{j})a_{im}a_{jk}.
							\end{align*}
							Thus
							\begin{align*}
								p_{ji}f(x_{jm}x_{ik}) - p_{mk} f(x_{ik}x_{jm}) &= g^{-2}_{n}g_{i}g_{j} \# (q - q^{-1})\psi(g^{-1}_{i}g_{m}, g^{-1}_{n}g_{j})a_{im}a_{jk}\\ 
								&= (q - q^{-1})(g^{-1}_{n}g_{i} \# a_{im})(g^{-1}_{n}g_{j} \# a_{jk})\\
								&= (q - q^{-1})f(x_{im}x_{jk}).
							\end{align*}
							The case $i < n, j = n$ and $k <m$ is handled in exactly the same way.
								\end{enumerate}
	These computations establish the existence of a unique algebra morphism \[f \colon
			\mathcal{O}^{\bar{p}}_{q}(\mathrm{M}_{n}(k)) \longrightarrow k\Gamma \#\mathcal{O}^{\Gamma}_{q}(\mathrm{SL}_{n}(k))\] as desired.
							\epf
\begin{proposition*}
	There exists a Hopf algebra morphism
	\begin{align*}
 \bar{f} \colon \mathcal{O}^{\bar{p}}_{q}(\mathrm{GL}_{n}(k)) &\longrightarrow k\Gamma \#\mathcal{O}^{\Gamma}_{q}(\mathrm{SL}_{n}(k))\\
		x_{ij} &\longmapsto
		g^{-1}_{n}g_{i} \# a_{ij},\\
		D^{\bar{p}}_{q} &\longmapsto g_{n}^{-n} \# 1.
	\end{align*}
\end{proposition*}
\bpf
Using the morphism $f$ in the preceding lemma, we can extend $f$ to an algebra morphism $\bar{f}$, as claimed. Indeed, $f$ also preserves the final defining relation of
$ \mathcal{O}^{\bar{p}}_{q}(\mathrm{GL}_{n}(k))$, namely,
							\[f\Big((D^{\bar{p}}_{q})^{-1}\sum_{\sigma \in S_{n}} \prod_{\substack{1 \leq i< j \leq n\\ \sigma(i) > \sigma(j)}} (-q^{-1}p_{\sigma(j) \sigma(i)}) x_{1\sigma(1)} x_{2\sigma(2)} \cdots  x_{n\sigma(n)} \Big) = 1 \# 1. \]
							To verify this, we compute:
							\begin{align*}
								&f\Big(\sum_{\sigma \in S_{n}} \Big(\prod_{\substack{1 \leq i< j \leq n\\ \sigma(i) > \sigma(j)}}-q^{-1}p_{\sigma(j) \sigma(i)}\Big)x_{1\sigma(1)} x_{2\sigma(2)} \cdots  x_{n\sigma(n)} \Big) \\
								&= \sum_{\sigma \in S_{n}} \Big(\prod_{\substack{1 \leq i< j \leq n\\ \sigma(i) > \sigma(j)}} -q^{-1}p_{\sigma(j) \sigma(i)}\Big)\big(g^{-1}_{n}g_{1} \# x_{1\sigma(1)}\big)\cdots \big(g^{-1}_{n}g_{n -1} \# x_{n -1\sigma(n -1)}\big)\big(1 \# x_{n\sigma(n)}\big).
							\end{align*}
							Applying the same argument as in Lemma~\ref{lem: transmute det}, we find that
							\begin{align*}
								&=\sum_{\sigma \in S_{n}} \Bigg(\prod_{\substack{1 \leq i< j \leq n\\ \sigma(i) > \sigma(j)}} (-q^{-1}p_{\sigma(j) \sigma(i)})\Bigg) \big(g^{-n}_{n} \# \prod_{1 \leq i < j \leq n -1} \psi(g^{-1}_{i}g_{\sigma(i)}, g^{-1}_{n}g_{j}) \prod^{n -1}_{i = 1} x_{i\sigma(i)}\big)\big(1 \# x_{n\sigma(n)}\big)\\
								&= \sum_{\sigma \in S_{n}} -q^{-\ell(\sigma)}\Bigg(\prod_{\substack{1 \leq i< j \leq n\\ \sigma(i) > \sigma(j)}} (p_{\sigma(j) \sigma(i)})\Bigg)\Bigg(\prod_{1 \leq i < j \leq n -1} \psi(g^{-1}_{i}g_{\sigma(i)}, g^{-1}_{n}g_{j}) \Bigg) \big(g^{-n}_{n} \# \prod^{n}_{i = 1} x_{i\sigma(i)}\big).
							\end{align*}
						All that remains is to look at the coefficient:
						\begin{align*}
							&\Bigg(\prod_{\substack{1 \leq i< j \leq n\\ \sigma(i) > \sigma(j)}} (p_{\sigma(j) \sigma(i)})\Bigg)\Bigg(\prod_{1 \leq i < j \leq n -1} \psi(g^{-1}_{i}g_{\sigma(i)}, g^{-1}_{n}g_{j}) \Bigg)\\
							&= \Bigg(\prod_{\substack{1 \leq i< j \leq n\\ \sigma(i) > \sigma(j)}} \psi(g_{\sigma(i)} g^{-1}_{\sigma(j)}, g_{n})\Bigg)\Bigg(\prod_{1 \leq i < j \leq n -1} \psi(g^{-1}_{i}g_{\sigma(i)}, g^{-1}_{n})\Bigg)\Bigg(\prod_{1 \leq i < j \leq n -1} \psi(g^{-1}_{i}g_{\sigma(i)}, g_{j})\Bigg)\\
							&= \Bigg(\prod_{\substack{1 \leq i< j \leq n\\ \sigma(i) > \sigma(j)}} \psi(g_{\sigma(i)} g^{-1}_{\sigma(j)}, g_{n})\Bigg)\Bigg(\prod_{1 \leq i <j \leq n} \psi(g^{-1}_{i}g_{\sigma(i)}, g^{-1}_{n})\Bigg)\Bigg(\prod_{1 \leq i < j \leq n} \psi(g^{-1}_{i}g_{\sigma(i)}, g_{j})\Bigg)\\
							&= \Bigg(\prod_{\substack{1 \leq i< j \leq n\\ \sigma(i) > \sigma(j)}} \psi(g_{\sigma(i)} g^{-1}_{\sigma(j)}, g_{n})\Bigg)\Bigg(\prod^{n -1}_{i = 1} \psi(g_{i}g^{-1}_{\sigma(i)}, g_{n})^{n - i}\Bigg)\Bigg(\prod_{1 \leq i < j \leq n} \psi(g^{-1}_{i}g_{\sigma(i)}, g_{j})\Bigg)\\
							&= \prod_{1 \leq i < j \leq n} \psi(g^{-1}_{i}g_{\sigma(i)}, g_{j}),
						\end{align*}
						(since the first product is indeed equal to $1$; the verification is given in the lemma immediately following this proposition).
						Finally, we get 
						\begin{align*}
							&f\Bigg[\sum_{\sigma \in S_{n}} \Bigg(\prod_{\substack{1 \leq i< j \leq n\\ \sigma(i) > \sigma(j)}} (-q^{-1}p_{\sigma(j) \sigma(i)})\Bigg)x_{1\sigma(1)} x_{2\sigma(2)} \cdots  x_{n\sigma(n)}\Bigg]\\
							&= \sum_{\sigma \in S_{n}} -q^{\ell(\sigma)}\prod_{1 \leq i < j \leq n} \psi(g^{-1}_{i}g_{\sigma(i)}, g_{j}) \big(g_{n}^{-n} \#  \prod^{n}_{i = 1} x_{i\sigma(i)}\big)\\
							&= (g_{n}^{-n} \# 1)\Big(1 \# \sum_{\sigma \in S_{n}} -q^{\ell(\sigma)}(\prod_{1 \leq i < j \leq n} \psi(g^{-1}_{i}g_{\sigma(i)}, g_{j}) ) \prod^{n}_{i = 1} x_{i\sigma(i)}\big)\Big)\\
							&= (g_{n}^{-n} \# 1)\\
							&= f(D^{\bar{p}}_{q}).
						\end{align*}
						Thus, we obtain an algebra morphism
						\[\bar{f} \colon \Opgl \longrightarrow  k\Gamma \#\mathcal{O}^{\Gamma}_{q}(\mathrm{SL}_{n}(k))\] as announced. This is also a morphism of coalgebras. Indeed, we have
						\begin{align*}
							(\bar{f} \otimes \bar{f}) \circ \Delta(x_{ij}) &= \sum^{n}_{k = 1} (g^{-1}_{n}g_{i} \# a_{ik}) \otimes (g^{-1}_{n}g_{k} \# a_{kj}).
						\end{align*}
						On the other hand,
						\begin{align*}
							\widetilde{\Delta} \circ \bar{(f)}(x_{ij}) &= \widetilde{\Delta} \big(g^{-1}_{n}g_{i} \# a_{ij}\big)\\
							&= \sum^{n}_{k = 1} (g^{-1}_{n}g_{i} \# a_{ik}) \otimes (g^{-1}_{n}g_{i}g_{i}^{-1}g_{k} \# a_{kj})\\
							&= \sum^{n}_{k = 1} (g^{-1}_{n}g_{i} \# a_{ik}) \otimes (g^{-1}_{n}g_{k} \# a_{kj}).
						\end{align*}
						where $\widetilde{\Delta}$ is the coproduct defined in Section $4$ on bosonization. Similarly, it is not difficult to verify that $\bar{f}$ is compatible with the remaining Hopf algebra structure maps, and we conclude the proof.
						\epf
						\begin{lemma*}\label{Lem : tricky}
							Let $\Gamma$ be an abelian group. Fix $n \geq 2$, elements $g_{1}, \dots, g_{n} \in \Gamma$ and a permutation $\sigma \in S_{n}$. Then
							\[\prod_{\substack{1 \leq i< j \leq n\\ \sigma(i) > \sigma(j)}} g_{\sigma(i)} g^{-1}_{\sigma(j)} = \prod^{n -1}_{i = 1}(g^{-1}_{i}g_{\sigma(i)})^{n - i}\]
						\end{lemma*}
						\bpf
						Let \[P = \prod_{\substack{1 \leq i< j \leq n\\ \sigma(i) > \sigma(j)}}g_{\sigma(i)} g^{-1}_{\sigma(j)} \quad \text{and} \quad Q = \prod^{n -1}_{i = 1} (g^{-1}_{i}g_{\sigma(i)})^{n - i}\]
						For a fixed $k \in \{1, \dots, n\}$, let $p = \sigma^{-1}(k)$ denote the position of $k$ in $\sigma$. We study the exponent of $g_{k}$ that appears in $P$ and $Q$, denoted by $\lambda_{1}$ and by $\lambda_{2}$, respectively.
						
						In $P$, we study the cardinal of 
						\[E = \{(i, j) \mid i < j,\ \sigma(i) = k,\ \sigma(j) < k\},
						\]
						Or equivalently,
						\[E = \{\, j \in \llbracket 1, n\rrbracket \mid j > p, \sigma(j) < k \,\}.\]
						We have
						\begin{align*}
							\sigma(E) &= \{\sigma(j) \in \llbracket 1, n\rrbracket \mid j > p, \sigma(j) < k\}\\
							&= \{j' \in \llbracket 1, n\rrbracket \mid \sigma^{-1}(j') > p, j' < k\}.
						\end{align*}
						and \[\mid \sigma(E) \mid = \mid E \mid\]
						because $\sigma$ is bijective and $E$ is finite.
						
						Similarly, we denote by 
						\begin{align*}
							F &= \{(i, j) \mid i < j,\ \sigma(j) = k,\ \sigma(i) > k\}\\
							&= \{ i \in \llbracket 1, n \rrbracket \mid i < p,\ \sigma(i) > k\}.
						\end{align*}
						and
						\begin{align*}
							\mid F \mid = \mid \sigma(F) \mid &= \mid \{ \sigma(i) \in \llbracket 1, n \rrbracket \mid i < p,\ \sigma(i) > k\}\mid\\
							&= \mid \{ i' \in \llbracket 1, n \rrbracket \mid \sigma^{-1}(i') < p,\ i' > k\}\mid.
						\end{align*}
						We obtain 
						\[\lambda_{1} = \mid E \mid - \mid F \mid = \mid \sigma(E) \mid - \mid \sigma(F) \mid.\]
						Now, we denote by
						\[L(k) = \{j' \in \llbracket 1, n\rrbracket \mid \sigma^{-1}(j') < p, j' < k\},\]
						we have $\lvert L(k) \rvert + \lvert \sigma(E) \rvert = \lvert \{j' \in \llbracket 1, n\rrbracket \mid \sigma^{-1}(j') \neq p, j' < k\}\rvert$. But, for $i \in \llbracket 1, n\rrbracket$, $\sigma^{-1}(j') = p$ if and only of $ j' = k$. Thus
						\begin{equation}
							\lvert L(k) \rvert + \lvert \sigma(E) \rvert = \lvert \{j' \in \llbracket 1, n\rrbracket \mid j' < k\}\rvert = k - 1. \label{1}
						\end{equation}
						Also, we have 
						\begin{align*}
							L(k) \bigcup \sigma(F) = \{j' \in \llbracket 1, n\rrbracket \mid \sigma^{-1}(j') < p\}
						\end{align*}
						and $L(k) \bigcap \sigma(F) = \emptyset$. Thus
						\[\lvert L(k) \rvert + \lvert \sigma(F) \rvert = p - 1.\]
						Therefore, we obtain
						\begin{align*}
							\lambda_{1} = \lvert \sigma(E) \rvert - \lvert \sigma(F) \rvert &= \lvert \sigma(E) \rvert - (p - 1 - L(k))\\
							& = \lvert \sigma(E) \rvert - p + 1 + L(k) \\
							&= k - 1 - p + 1 \quad (\text{by \eqref{1}})\\
							&= k - \sigma^{-1}(k).
						\end{align*}
						In $Q$, the exponent of $g_{k}$ is 
						\begin{align*}
							\lambda_{2} &= -(n - k) \lvert \{i \in \llbracket 1, n \rrbracket \mid i = k\} \rvert + ( n - \sigma^{-1}(k))\lvert \{i \in \llbracket 1, n \rrbracket \mid \sigma(i) = k\} \rvert\\
							&= k - \sigma^{-1}(k)\\
							&= \lambda_{1}.
						\end{align*}
						Finally, for each fixed $k$, the exponent of $g_{k}$ in $P$ equals that in $Q$. Since this holds for any $k$, we obtain the result desired.
						\epf
				
						\bibliographystyle{plain}
						\bibliography{biblio.bib}
						
					\end{document}